\pdfoutput=1
\documentclass{DuBoch-axSIH}\usepackage{mathrsfs}\usepackage{amscd}

\newtheorem{theorem}{Theorem}%[section]
\newtheorem{proposition}[theorem]{Proposition}
\newtheorem{lemma}[theorem]{Lemma}
\newtheorem{corollary}[theorem]{Corollary}
\newAtheorem{Atheorem}{Theorem}
\newAtheorem{Alemma}[Atheorem]{Lemma}
\newAtheorem{Aproposition}[Atheorem]{Proposition}

\theoremstyle{definition}

\newtheorem{constructions}[theorem]{\noindent{\font\SweD=cmssi10\SweD Constructions}\bf}

\newtheorem{definitions}[theorem]{\font\SweD=cmssi10\SweD Definitions\bf}
\newtheorem{def:al schema}[theorem]{\font\SweD=cmssi10\SweD Definitional schema\bf}
\newtheorem{def:al schemas}[theorem]{\font\SweD=cmssi10\SweD Definitional schemata\bf}
\newtheorem{axiom schema}[theorem]{{\font\SweD=cmssi10\SweD Axiom schema}\bf}
\newtheorem{example}[theorem]{\font\SweD=cmssi10\SweD Example\bf}

\newtheorem{problem}[theorem]{\font\SweD=cmssi10\SweD Problem\bf}
\newtheorem{remark}[theorem]{\font\SweD=cmssi10\SweD Remark\bf}
\newtheorem{remarks}[theorem]{\font\SweD=cmssi10\SweD Remarks\bf}

\newcommand\Cal{\mathcal}
\newcommand\bosy{\boldsymbol}
\newcommand\mfrk{\mathfrak}\newcommand\mfrak{\mathfrak}
\renewcommand\math[1]{\hbox{$\kern0.4mm#1\kern0.4mm$}}% the same as above but also putting an extra space of .4mm:s before and after
\newcommand\mathss[3]{\hbox{\kern0.17mm\kern.#1mm$#3$\kern.#2mm\kern0.17mm}}% the same as above but allowing to choose the space before and after; e.g. `` ... that \mathss23{x=y} holds ... '' is the same as `` ... that \PouN$ \kern0.17mm\kern.2mm  x=y  \kern0.3mm\kern.17mm$  holds ... '' , " \mathss{25}{2}{\roman x}. " is the same as " \math{\roman x}. "
\newcommand\aall[1]{\forall\kern.8mm{#1}\kern.2mm\,;}
\newcommand\eexi[1]{\exists\kern.9mm{#1}\kern.2mm\,;}
\newcommand\impss[2]{\kern.#1mm\kern.17mm\Rightarrow\kern.17mm\kern.#2mm}
\newcommand\equivss[2]{\kern.#1mm\kern.17mm\Leftrightarrow\kern.17mm\kern.#2mm}
\def\afr#1_#2{\mathfrak#1_{\kern.2mm\lower.15mm\hbox{\font\SweD =cmr6\SweD #2}}} % e.g., \afr x_0 = \mathfrak x_0
\def\Alf{\hbox{\font\SweD =cmmi10 scaled\magstep1\SweD \char'013}\kern0.15mm}% bigger alpha
\def\Eps{\hbox{\font\SweD =cmmi10 scaled\magstep1\SweD \char'017}\kern0.15mm}% bigger varepsilon
\newcommand\twEps{\hbox{\font\Å=cmmi10 scaled\magstep1\Å\char'042}} % gives 12 point eps in math
\def\Iota{\kern.15mm\hbox{\font\SweD =cmmi10 scaled\magstep1\SweD \char'023}\kern0.2mm}% bigger iota
\newcommand\sIota{\kern.1mm\hbox{\font\SweD =cmmi8\SweD \char'023}\kern0.15mm} % bigger iota for sub- and superscripts
\def\Nu{\hbox{\font\SweD =cmmi10 scaled\magstep1\SweD \char'027}\kern0.25mm}% bigger nu
\newcommand\twer[1]{\lower.45mm\hbox{\font\Å=cmr12\Å#1}} %
\newcommand\tweq{\raise.13mm\hbox{\font\Å=cmmi12\Åq}} %
\newcommand\Yps{\varUpsilon}\newcommand\vPi{\varPi}

\newcommand\Phii{\char'010}\newcommand\Psii{\char'011}
 \newcommand\iotaa{\char'023}%
\newcommand\nuu{{\char'027}}
\def\pii{\char'031}\def\sigmaa{\char'033}

\def\rbrakf{\kern.9mm]\kern.2mm\lower.8mm\hbox{\font\SweD =cmr5\SweD f}\kern.2mm} % [f,g]_f : the function x mapsto (fx,gx) , e.g. $[\,\sp f\sp,\sp g\rbrakf$
\def\rbrakff{\kern.9mm]\kern.2mm\lower.8mm\hbox{\font\SweD =cmr5\SweD ff}\kern.2mm} % [f,g]_{ff}
\newcommand\rbracff[1]{\kern.3mm\kern.#1mm]\kern.2mm\lower.88mm\hbox{\font\å=cmr5\åf\kern-.4mmf}\kern.2mm} % use e.g. as in $[\KP1 a\ar 0\,,\sp a\ar 1\rbracff3$ to denote the function [a_0,a_1]_{ff} : x |--> <a_0(x),a_1(x)> ; here the "3" puts an extra 0.3 mm space between "_1" and `]´
\newcommand\Pif[1]{\Pi\lower.68mm\hbox{\font\å=cmr5\åf}\kern.6mm\kern.#1mm} % use e.g. as in $\Pif0\bosy f$ , $\Pif2\seqss11{f\sp,\sp g}=[\KP1 f\sp,\sp g\rbracff6$
\def\ftimes{\hbox{${}\times\kern-2.7mm\lower.9mm\hbox{\font\SweD =cmr5\SweD f}\kern1.93mm$}} % f\ftimes g : the function (x,y) mapsto (fx,gy)

\def\bbN{\mathbb N} % N = \N = set of nonempty natural numbers
\def\bbNo{\mathbb N\kern.15mm\lower.65mm\hbox{\font\SweD =cmr6\SweD 0}\kern.1mm} % N_0 = \No = set of natural numbers
\def\sbbNo{{\mathbb N\kern.07mm\lower.45mm\hbox{\font\SweD =cmr5\SweD 0}\kern.1mm}} % the preceding for sub- and superscripts
\def\bbR{\mathbb R} % = \Re = set of real numbers
\newcommand\ovbbR{\kern.33mm\raise1.1mm\hbox{$^{\overline{\kern1.35mm}}$}\kern-1.85mm\mathbb R} % the extended real line $\overline{\mathbb R}$
\newcommand\bbI{\mathbb I} % the standard closed unit interval [0,1]
\newcommand\ssbb[3]{\hspace{.#1mm}\mathbb#3\hspace{.#2mm}} % e.g. $(\ssbb42 R)$ gives $(\hspace{.4mm}\mathbb R\hspace{.2mm})$ , \ssbb00 R = \mathbb R
\def\lbb#1_#2{\mathbb#1\kern.2mm\raise.52mm\hbox{$_{_{#2}}$}} % e.g., to get \mathbb Z_- use \lbb Z_-
\def\rbb#1^#2{\mathbb#1\kern.2mm\lower.55mm\hbox{$^{^{#2}}$}} % e.g., to get \mathbb R^+ use \rbb R^+
\def\Rlplus{\mathbb R\kern.2mm\lower.33mm\hbox{\font\SweD =cmr5\SweD +}} % \Rlplus = R_+ = { t : t real and t \ge 0 } = \lbb R_+
\def\fbbR{\raise1.2mm\hbox{\font\SweD =cmr5\SweD f}\kern.3mm\mathbb R} % the standard field structure of real numbers
\def\tfbbR{\raise1.2mm\hbox{\font\SweD =cmr5\SweD tf}\kern.3mm\mathbb R} % standard topological field of real numbers
\def\stfbbR{{\raise.5mm\hbox{\font\SweD =cmr5\SweD t\kern-.1mmf}\kern.3mm\mathbb R}} % the preceding for sub- and superscripts
\newcommand\vbbR[1]{\raise1.65mm\hbox{\font\Å=cmss5\Åv}\kern.2mm\mathbb R\kern.3mm\kern.#1mm} % use e.g. as in $\vbbR5^{3.}$ , $\vbbR{45}^{\ssmb N}$ , $\vbbR5^k$ to get R^3 , R^N , R^k considered as a real vector space
\newcommand\tvbbR[1]{\raise1.35mm\hbox{\font\SweD=cmr5\SweD t}\raise1.65mm\hbox{\font\SweD=cmss5\SweD v}\kern.2mm\mathbb R\kern.3mm\kern.#1mm} % use e.g. as in $\tvbbR5^{3.}$ , $\tvbbR{45}^{\ssmb N}$ , $\tvbbR5^k$ to get R^3 , R^N , R^k considered as a real topological vector space
\newcommand\smbbR[1]{\raise1.6mm\hbox{\font\SweD=cmr5\SweD sm}\kern.1mm\mathbb R\kern.3mm\kern.#1mm} % use e.g. as in $\smbbR5^{3.}$ , $\smbbR{45}^{\ssmb N}$ , $\smbbR5^k$ to get R^3 , R^N , R^k considered as a real smooth manifold
\newcommand\smbbRR{\raise1.67mm\hbox{\font\Å=cmss5\Åsm}\mathbb R} %
\def\fbbC{\raise1.23mm\hbox{\font\SweD =cmr5\SweD f}\kern.1mm\mathbb C} % standard field of complex numbers
\def\tfbbC{\raise1.23mm\hbox{\font\SweD =cmr5\SweD tf}\kern.1mm\mathbb C} % standard topological field of complex numbers
\def\stfbbC{{\raise.5mm\hbox{\font\SweD =cmr5\SweD t\kern-.1mmf}\kern.3mm\mathbb C}} % the preceding for sub- and superscripts
\def\tfbbH{\raise1.2mm\hbox{\font\SweD =cmr5\SweD tf}\kern.3mm\mathbb H} % standard topological division ring of quaternions
\def\taubb_#1{\tau\kern-.15mm\lower.7mm\hbox{\font\SweD =msbm5\SweD #1}\kern.3mm} % tuottaa tau_{bb #1} , e.g. \taubb_R = the standard topology of the real line
\def\nsTbb_#1{\hbox{\font\SweD =eusm9\SweD T}\lower.7mm\hbox{\font\SweD =msbm5\SweD #1}\kern.3mm} % tuottaa 9 point script T_{bb #1} , e.g. \nsTbb_R = the standard topology of the real line , alternative notation for the above
\def\barscTbb_#1{\kern.3mm\lower.4mm\hbox{$^{^{\overline{\kern1.4mm}}}$}\kern-2.05mm\hbox{\font\Å=eusb7\ÅT}\kern.25mm\lower.7mm\hbox{\font\Å=msbm5\Å#1}\kern.3mm} % use as e.g. in $\barscTbb_R$ , $\barscTbb_C$ , $\barscTbb_H$ ; the first is the topology of the extended real line , the second that of the complex sphere , and the last that of the quaternionic sphere
\newcommand\TopexR{\hbox{\font\å=eusm9\åT}\lower.27mm\hbox{$_{^{\overline{\mathbb R\kern-.4mm}}}$}\kern.388mm} % topology of the extended real line
\def\bartau_bb#1{\bar\tau\kern-.15mm\lower.7mm\hbox{\font\SweD =msbm5\SweD #1}\kern.3mm} % \bar\tau_{mathbb #1} , e.g. \bartaubb_R = the standard topology of the extended real line
\def\cinfty{\raise1.35mm\hbox{\font\SweD =cmr5\SweD c}\kern-.15mm\infty} % the complex infinity
\def\inftyyplus{{\vphantom{p_{p_p}}\infty\RHB{.5}{\fiveroman+}}} % \infty^+ for superscript, e.g. $C^{\,\inftyyplus}(O)$
\newcommand\imag{{\kern-.18mm\raise1.25mm\hbox{\font\Å=cmr5\Å\char'052}\kern-1mm\hbox{\font\Å=cmr10\Å\char'020}}} % the imaginary unit
\newcommand\iimag{{\kern-.18mm\raise.75mm\hbox{\font\Å=cmr5\Å\char'052}\kern-.95mm\hbox{\font\Å=cmr7\Å\char'020}\kern.05mm}} % the imaginary unit for sub- and superscripst
\newcommand\thepi{{\kern-.25mm\raise.67mm\hbox{\font\Å=cmr10\Å\char'056}\kern-.72mm\pi}} % the pi
\newcommand\tthepi{{\kern-.15mm\raise.45mm\hbox{\font\fff=cmr7\fff\char'056}\kern-.63mm\pi}} % the pi for sub- and superscripst
\newcommand\rmdss[2]{\kern1mm\kern.#1mm\roman d\kern.#2mm\,}% use e.g. as in $\int_Af\rmdss30\mu$
\newcommand\rvdss[2]{\kern1mm\kern.#1mm\hbox{\font\Å=cmr10\Åd}\kern-2.1mm\lower.99mm\hbox{\font\Å=cmmi5\Å\char'051}\kern.#2mm\kern.3mm} % gives underset{<--}\roman d, use in notation for Stietjes integral of vector valued function w.r.t a scalar one, similarly as "\rmdss"
\newcommand\sigmalg[1]{\hbox{\font\SweD =cmmi12\SweD \char'033}\kern-.2mm\lower.77mm\hbox{\font\SweD =cmr6\SweD a}\lower.8mm\hbox{\font\SweD =cmr5\SweD l}\lower.45mm\hbox{\font\SweD =cmr5\SweD g}\kern.#1mm\kern.4mm} % sigma algebra of a topology ; use e.g. as in $\sigmalg0(\sp\nsTbb_R\!\expnota^2.]_{ti})$ , $\sigmalg5\Cal A$ , $\sigmalg5\Cal B$ , $\sigmalg0\Cal T$ , $\sigmalg5\Cal C$ , $\sigmalg7\Cal U$ , $\sigmalg4\Cal W$
\newcommand\sigmAlg[1]{\hbox{\font\SweD =cmmi12\SweD \char'033}\kern-.2mm\lower.8mm\hbox{\font\SweD =cmr5\SweD Al}\lower.45mm\hbox{\font\SweD =cmr5\SweD g}\kern.#1mm\kern.4mm} % sigma algebra of a topology ; use e.g. as in $\sigmalg0(\sp\nsTbb_R\!\expnota^2.]_{ti})$ , $\sigmalg5\Cal A$ , $\sigmalg5\Cal B$ , $\sigmalg0\Cal T$ , $\sigmalg5\Cal C$ , $\sigmalg7\Cal U$ , $\sigmalg4\Cal W$
\def\Lebmea^#1{\kern.25mm\mu_{\kern.3mm\hbox{\font\SweD =cmr5\SweD Leb}}^{\vphantom n\kern.5mm{#1}}} % \Lebmea^N , \Lebmea^{} = the complete Lebesgue measure on R^N , R , resp.
\def\Lebmef^#1{\kern.25mm\mathfrak m_{\kern.3mm\hbox{\font\SweD =cmr5\SweD Leb}}^{\vphantom n\kern.5mm{#1}}} % another notation for the above; using mathfrak m in place mu
\def\LeBmef^#1{\kern.25mm\mathfrak m_{\kern.3mm\hbox{\font\SweD =cmr5\SweD LeB}}^{\vphantom n\kern.5mm{#1}}} % the Lebesgue measure restricted on the class of Borel sets
\newcommand\openIval[1]{\null\kern.35mm]\kern.7mm#1\kern.8mm[\kern.35mm\null} % \openIval{s,t} = ] s,t [

\newcommand\rsigma[1]{\raise.15mm\hbox{$^{\sigma}$\kern-.6mm\kern.#1mm}} % use e.g. as in {\it is \rsigma1finite}, is \rsigma6finite, is \rsigma1{\it finite}, {\it is a \rsigma1ideal}, is a \rsigma6ideal, is a \rsigma1{\it ideal}, {\it is \rsigma0additive}, is \rsigma3additive, is \rsigma0{\it additive}, {\it is \rsigma0compact}, is \rsigma3compact, is \rsigma0{\it compact}
\newcommand\loint{\int\kern-1.2mm\lower.45mm\hbox{$\underline{\phantom.}$}\kern1mm}% = \int_-
\newcommand\upint{\kern.5mm\raise1.85mm\hbox{$\overline{\phantom.}$}\kern-1.2mm\int}% = ^-\int
\newcommand\plusint{\kern.5mm\raise1.65mm\hbox{\font\SweD=cmr5\SweD +}\kern-1mm\int} % = ^+\int
\newcommand\Reint{\kern.5mm\raise1.65mm\hbox{\font\SweD=msbm5\SweD R}\kern-1mm\int} % = ^R\int

\newcommand\sefRC{\{\kern0.15mm\raise1.2mm\hbox{\font\SweD =cmr5\SweD f}\kern.3mm\mathbb R\,,\kern-0.3mm\raise1.23mm\hbox{\font\SweD =cmr5\SweD f}\kern.1mm\mathbb C\,\}} % = {^fR,^fC} , use e.g. $\bold K\in\sefRC$
\def\setRC{\{\kern0.15mm\raise1.2mm\hbox{\font\SweD =cmr5\SweD tf}\kern.3mm\mathbb R\,,\kern-0.3mm\raise1.23mm\hbox{\font\SweD =cmr5\SweD tf}\kern.1mm\mathbb C\,\}} % = {^{tf}R,^{tf}C} , use e.g. $\bosy K\in\setRC$
\def\Reit#1{_{\lower.2mm\hbox{\kern.2mm\hskip.#1mm\font\SweD =msbm5\SweD R\kern.15mm\font\SweD =cmr5\SweD t}\kern.15mm}} % E_{\mathbb R\roman t} = E\Reit3 = the realification of a complex topological vector space E ; use e.g. also $F\Reit0,G\Reit0,H\Reit1,\varUpsilon\Reit0,\varPi\Reit2$ , the number n puts a space of .n mm:s before the subscript
\def\Reif#1{_{\lower.2mm\hbox{\kern.#1mm\font\SweD =msbm5\SweD R\kern.2mm\font\SweD =cmr5\SweD i\kern-.3mm f}\kern.15mm}} % E_{\mathbb R\roman if} = E\Reif8 = the realification of a complex topological vector space E ; use e.g. also $F\Reif0,G\Reif0,H\Reif3,\varUpsilon\Reif0,\varPi\Reif5$ , the number n puts a space of .n mm:s before the subscript
\def\Ceit#1{_{\lower.2mm\hbox{\hskip.#1mm\font\SweD =msbm5\SweD C\kern.15mm\font\SweD =cmr5\SweD t}\kern.15mm}} % E_{\mathbb C\roman t} = E\Ceit3 = the complexification of a real topological vector space E ; use e.g. also $F\Ceit0,G\Ceit0,H\Ceit1,\varUpsilon\Ceit0,\varPi\Ceit2$ , the number n puts a space of .n mm:s before the subscript
\def\Ceif#1{_{\lower.2mm\hbox{\kern.#1dd\font\SweD =msbm5\SweD C\kern.2mm\font\SweD =cmr5\SweD i\kern-.3mm f}\kern.15mm}} % E_{\mathbb C\roman if} = E\Ceif8 = the complexification of a real topological vector space E ; use e.g. also $F\Ceif0,G\Ceif0,H\Ceif3,\varUpsilon\Ceif0,\varPi\Ceif5$ , the number n puts a space of .n dd:s before the subscript
\def\tvsps#1(#2){\hbox{\font\Å=cmss10\Åt.\kern-.3mm v\kern-.3mm .s\kern.8mm}(\kern.#1pt\boldsymbol#2\kern0.37mm)} % the class of topological vector spaces over a topological field, use e.g. as in $\tvsps0(K)$ , $\tvsps5(\tfbbR)$
\def\TVSps#1(#2){\roman{TVS}\kern0.7mm(\kern.#1pt\boldsymbol#2\kern0.37mm)} % the class of Hausdorff topological vector spaces over a topological field, use e.g. as in $\TVSps0(K)$ , $\TVSps5(\tfbbR)$
\def\LCSps#1(#2){\roman{LCS}\kern0.7mm(\kern.#1pt\boldsymbol#2\kern0.37mm)} % the class of Hausdorff locally convex spaces over a topological field, use e.g. as in $\LCSps0(K)$ , $\LCSps5(\tfbbR)$
\def\catLCS#1(#2){\hbox{\font\Å=cmssbx8\ÅLCS}\kern0.7mm(\kern.#1pt\boldsymbol#2\kern0.37mm)} % the category of Hausdorff locally convex spaces, use similarly as the above
\def\BaSps#1(#2){\roman{BaS}\kern0.7mm(\kern.#1pt\boldsymbol#2\kern0.37mm)} % the class of Banchable locally convex spaces over a topological field, use e.g. as in $\BaSps0(K)$ , $\BaSps5(\tfbbR)$
\def\FrSps#1(#2){\roman{FrS}\kern0.7mm(\kern.#1pt\boldsymbol#2\kern0.37mm)} % the class of Banchable locally convex spaces over a topological field, use e.g. as in $\FrSps0(K)$ , $\FrSps5(\tfbbR)$
\def\HilbSps#1(#2){\roman H\kern.15mm\lower.7mm\hbox{\font\SweD =cmr5\SweD ilb}\kern0.4mm\roman S\kern.6mm(\kern.#1pt\boldsymbol#2\kern0.37mm)} % the class of Hilbertable locally convex spaces over a topological field, use e.g. as in $\HilbSps0(K)$ , $\HilbSps5(\tfbbR)$
\def\TVS{\roman{TVS}\kern0.4mm}%
\def\LCS{\roman{LCS}\kern0.4mm}%
\def\BaS{\roman{BaS}\kern0.4mm}%
\def\HilbLCS{\roman{H\kern.15mm\lower.7mm\hbox{\font\SweD =cmr5\SweD ilb}\kern.25mmLCS}\kern0.4mm} % \HilbLCS(K) = the class of locally convex spaces over K with topology determined by a Hilbert norm
\def\Nbh{\Cal N_{\font\SweD =cmmi6\lower.15mm\hbox{\kern.1mm\SweD bh\kern.15mm}}}% \Nhb(x,T) = set of T-neighborhoods of x
\newcommand\neiBoo{\Cal N\kern-.4mm\lower.7mm\hbox{\font\SweD=cmr7\SweD o}\kern.8mm} % set of zero neighbourhoods in a topological vector space
\newcommand\bouSet{\Cal B\kern-.1mm\lower.7mm\hbox{\font\Å=cmssi6\Ås}\kern.8mm} % set of bounded sets in a topological vector space
\newcommand\bouSeT{\Cal B\kern.1mm\lower.7mm\hbox{\font\Å=cmss6\Ås}\kern.7mm} % set of bounded sets in a convergence vector space
\newcommand\semiNor{\Cal S\kern.2mm\lower.75mm\hbox{\font\SweD=cmr5\SweD N}\kern.8mm} % set of continuous seminorms in a topological vector space
\newcommand\BSnorm{\mathscr B\kern-.25mm\lower.65mm\hbox{\font\Å=cmsl6\Ås\kern-.15mm\font\Å=cmsl6\Ån}\kern.85mm} % the set of bounded seminorms in a topological vector space, use e.g. as in $\BSnorm\vPi$ , $\BSnorm E$ , $\BSnorm F$ , $\BSnorm G$ , $\BSnorm\Yps$ , $\BSnorm H$
\newcommand\Bqnorm{\mathscr B\kern-.25mm\lower.6mm\hbox{\font\Å=cmssqi5\Åq}\lower.85mm\hbox{\kern-.15mm\font\Å=cmsl6\Ån}\kern.85mm} % the set of bounded quasi-seminorms in a topological vector space
\newcommand\bouNor{\Cal B\kern.2mm\lower.75mm\hbox{\font\SweD=cmr5\SweD N}\kern.8mm} % set of bounded seminorms in a topological vector space
\newcommand\vsquotient{\lower.85mm\hbox{\kern-.6mm\font\x=cmr5\x vs\kern.85mm}} % $X\sp/\vsquotient N\aar 0$ = quotient of a module X by the submodule set N_0
\newcommand\tvsquotient{\lower.85mm\hbox{\kern-.35mm\font\x=cmr5\x tvs\kern.85mm}} % $E\ssp/\vsquotient N\aar 0$ = quotient of a topologized module E by the submodule set N_0

\def\SemiNor{\Cal S_{_N}\kern0.15mm}% \SemiNor E = set of continuous seminorms in E
\def\dimHa{{\rm dim_{_{\kern.2mm Ha}}}}% the Hamel dimension of a structured vector space

\def\lfbb_#1{_{\lower.25mm\hbox{\kern.3mm\font\SweD =msbm5\SweD #1\kern.15mm}}} % use e.g. as in C^i(O)\lfbb_C = the real space of complex valued C^i functions on O
\newcommand\Lis{\mathscr L\lower.75mm\hbox{\kern-.75mm\font\Å=cmmi6\Åi\kern-.865mmi\font\Å=cmmi7\Ås\kern.4mm}} % \Lis(E,F) = set of linear homeomorphisms E to F ; requires "\usepackage{mathrsfs}" in the preamble
\newcommand\Linb{\Cal L\lower.7mm\hbox{\kern.3mm\font\SweD =cmmi6\SweD b\kern.25mm}} % \Linb(E,F) gives \Cal L_b(E,F) = the topological vector space of continuous linear maps E to F with the topology of uniform convergence on bounded sets
\newcommand\Lingamma{\Cal L\lower.43mm\hbox{\kern.27mm\font\Å=cmtex5\Å\char'011\kern.45mm}} % \Lingamma(E,F) gives \Cal L_\gamma(E,F) = the topological vector space of continuous linear maps E to F with the topology of uniform convergence on ...
\newcommand\dualgamma[2]{_{\kern-.1mm\kern.#1mm\raise.1mm\hbox{\font\Å=cmtex5\Å\char'011}}\kern-.#1mm\kern-0.6mm^{\kern.#2mm\prime\kern.2mm}_{\kern-.#2mm\kern.#1mm}} % use e.g. as in use e.g. as in $\vPi\dualgamma33$ , $E\dualgamma53$ , $F\dualgamma15$ , $)\dualgamma51$ to get " _{\kern.#1 mm\gamma}^{\kern.#2 mm\prime} "
\newcommand\vbdualgamma[1]{_{\kern-.3mm\kern.#1mm\raise.1mm\hbox{\font\Å=cmtex5\Å\char'011}\kern.25mm\lower.25mm\hbox{\font\Å=cmss5\Åv\font\Å=cmmi5\Åb}}\kern-.#1mm\kern-2.6mm^{\prime}_{\kern.#1mm\kern2.45mm}} % to get " _{\kern.#1 mm \gamma vb}' " use e.g. as in $\vbscrDLi'(\sp\vcal E\ssp) = \Linb(\ssp\vbscrD\,(\ssp\Delta^{\ssp 1}\ssp\vcal E\vbdualgamma3)\,,\ssp\bosy K\ssp)$ , $\Cal L\,(\ssp\vbCzero4^{0.}(\ssp\Delta^{\ssp 1}\ssp(\ssp\Lambda^{\ssp 2.\,}\Taubun^*M\sp\sbi U\sp)\vbdualgamma7)\,,\tfbbR\,)$ , $\vcal F\vbdualgamma0$ , $\vcal G\vbdualgamma5$ ; the number puts an extra space before " \gamma vb " 
\newcommand\vbdualc[1]{_{\kern-.3mm\kern.#1mm\kern.25mm\lower.25mm\hbox{\font\Å=cmss5\Åcv\font\Å=cmmi5\Åb}}\kern-.#1mm\kern-2.6mm^{\prime}_{\kern.#1mm\kern2.45mm}} % to get " _{\kern.#1 mm cvb}' " use e.g. as in ... 
\def\dualbeta{^{\kern0.4mm\prime}_{\kern-.2mm\raise.95mm\hbox{$_{_\beta}$}}} % E\dualbeta = E'_\beta = strong topological dual of E
\newcommand\dlbetss[2]{_{\kern.#2mm\hbox{\font\Å=cmssi5\Å\char'031}}^{\kern.#1mm\kern.35mm\prime}} % $E\dlbetss00 = E'_\beta$ , taking numbers pq in place of 00 puts extra spaces of .p mm:s in front of ' and .q mm:s in front of the 5 point \beta
\newcommand\dlmuss[2]{_{\kern.#2mm\raise.37mm\hbox{\font\Å=cmmib5\Å\char'026}}^{\kern.#1mm\kern.35mm\prime}}
\newcommand\dlkss[2]{_{\kern.#2mm\kern.05mm\hbox{\font\Å=cmssqi5\Åk}}^{\kern.#1mm\kern.35mm\prime}}
\newcommand\dlckss[2]{_{\kern.#2mm\hbox{\font\Å=cmssi5\Åck}}^{\kern.#1mm\kern.35mm\prime}}
\newcommand\dlsigstar[2]{_{\kern.#2mm\hbox{\font\Å=cmmi6\Å\char'033}^*\kern-.2mm}^{\kern.#1mm\kern.35mm\prime}}
\newcommand\dualsigma[1]{_\sigma^{\kern.4mm\kern.#1mm\prime}} %
\newcommand\dlsigss[2]{^{\kern.4mm\kern.#1mm\prime}_{\kern.#2mm\sigma}} % $E\dlsigss00 = E'_\sigma$ , taking numbers pq in place of 00 puts extra spaces of .p mm:s in front of ' and .q mm:s in front of the 7 point \sigma

\def\fRe{\kern.1mm\raise1.2mm\hbox{\font\SweD =cmr5\SweD f\kern.1mm}\roman{Re}\kern.9mm}% the real      part of a function
\def\fIm{\kern.1mm\raise1.2mm\hbox{\font\SweD =cmr5\SweD f\kern.1mm}\roman{Im}\kern.9mm}% the imaginary part of a function

\newcommand\vtopsum[1]{\kern.7mm\kern.#1mm\text{-}\kern.4mm\raise1.57mm\hbox{\font\Å=cmss5\Åt}\kern-.75mm\sum} % topological sum of a family $f$ in a topologized module $E$, use e.g. as in $E\vtopsum3\,f$
\def\prodsubtext#1{\prod\kern-0.3mm\lower.9mm\hbox{\font\SweD =cmr5\SweD #1}\kern.7mm}
\def\coprodsubtext#1{\coprod\kern-0.3mm\lower.9mm\hbox{\font\SweD =cmr5\SweD #1}\kern.7mm}
\def\prodc{\prod{_{_{\kern-.3mm\bold c\kern.15mm}}}}% cartesian propduct of sets
\def\vsprod_#1_#2{\prod\kern-0.3mm{}_{_{\roman{#1}\sp{#2}\,}}} % for example, \vsprod_TVS_\Re produces \prod_{\roman{TVS} I\!\!R}
\def\vscoprod_#1_#2{\coprod\kern-0.3mm{}_{_{\roman{#1}\sp{#2}\,}}} % see the preceding
\def\ssexpar#1#2^#3)_#4{\kern.5mm\kern.#1mm^{#3\kern.15mm\kern.#2mm)}\kern.15mm\raise.65mm\hbox{\font\Å=cmr5\Å#4}} % use e.g. as in $E\ssexpar43^I)_{tvs}$ , $\tfbbR\ssexpar40^\sbbNo)_{lcs}$ , $\fbbR\ssexpar22^A)_{\snn rn}$
\def\expnota^#1]_#2{\,^{#1\,]{_{}}_{\roman{#2}}}} % F\expnota^\Omega_{tvs} = the topological product vector space of all functions Omega to F
\def\expnote^#1]#2{\kern.8mm^{#1}\kern.3mm\raise1.2mm\hbox{\font\SweD =cmsy5\SweD \char'143}\raise1.1mm\hbox{\font\SweD =cmr5\SweD #2}} %
\def\expar#1]#2{\kern1mm\raise1.5mm\hbox{\font\SweD =cmsy5\SweD \char'042}\kern.6mm#1\kern.8mm\raise1.4mm\hbox{\font\SweD =cmsy5\SweD \char'145}\raise1.3mm\hbox{\font\SweD =cmr5\SweD #2}} %
\newcommand\sbig[2]{\lower.1mm\hbox{\font\Å=cmr11\Å#1}\kern.#2mm} % use e.g. as in $\mathscr D\,\sbig(2\yi N\ssbb67 R,\spp\bosy K\ssp)$
\newcommand\sast{\kern.2mm\raise.9mm\hbox{\font\å=cmsy5\å\char'003}} % `ast´ * for scrips e.g. as in $\LL^{p\sast}$
\newcommand\subw[1]{\kern.07mm\lower.85mm\hbox{\lower.0#1in\hbox{\font\Å=cmssqi5\Åw }}} % use e.g. as in $)\subw0$ , $\big)\subw1$
\newcommand\subvw[1]{\kern.07mm\lower.85mm\hbox{\lower.0#1in\hbox{\font\Å=cmssi5\Åv\kern-.4mm w }}} % use e.g. as in $)\subvw0$ , $\big)\subvw1$
\newcommand\subsigrs[2]{_{\lower.3mm\hbox{\lower.#1mm\hbox{\kern-.2mm\kern.#2mm\font\Å=cmmi7\Å\char'033}}}} % use e.g. as in $E\subsigrs04=F\subsigrs00=\vPi\subsigrs03=((\sp E\dlbetss22\sp)\subsigrs03)\dlsigss11$
\newcommand\subsigma{\kern.07mm\lower.85mm\hbox{\font\Å=cmmi6\Å\char'033}} % use e.g. as in ???
\def\flbb_#1{\kern.3mm\lower.8mm\hbox{\font\SweD=msbm5\SweD #1}\kern.15mm} % tuottaa _{bb #1}
\def\co{\hbox{\font\Å=cmmi12\Åc}\lower.8mm\hbox{\font\Å=cmr6\Åo}\kern.4mm}% \co(I,F) = the supremumnormable space of functions I to F with become small outside finite sets
\def\LL^#1{L\kern0.15mm\raise.4mm\hbox{$^{#1}$}\kern0.15mm}% = L^p with p a bit lifted
\def\Lrs#1#2^#3{L\kern0.25mm\kern0.#2mm\lower.#1mm\hbox{\raise.4mm\hbox{$^{#3}$}}\kern0.25mm} % \Lrs00^p = \LL^p
\def\mvLrs#1#2^#3{\raise1.53mm\hbox{\font\å=cmtt5\åm}\raise1.53mm\hbox{\font\å=cmssq5\åv}\kern-.3mm\lower.12mm\hbox{\font\å=cmitt12\åL}\kern0.25mm\kern0.#2mm\lower.#1mm\hbox{\raise.4mm\hbox{$^{#3}$}}\kern0.25mm} % use e.g. as in $\mvLrs02^p(\ssp\mu\,,\spp\vPi\ssp)$
\def\mLrs#1#2^#3{\raise1.63mm\hbox{\font\Å=cmtex5\Åm}\kern-.33mm\lower.12mm\hbox{\font\Å=cmitt12\ÅL}\kern0.25mm\kern0.#2mm\lower.#1mm\hbox{\raise.4mm\hbox{$^{#3}$}}\kern0.25mm} % use e.g. as in $\mvLrs03^p(\ssp\mu\sp)$
\def\mvsLrs#1#2^#3{\raise1.53mm\hbox{\font\å=cmtt5\åm}\raise1.53mm\hbox{\font\å=cmssq5\åv}\kern-.3mm\lower.12mm\hbox{$\hbox{\font\å=cmitt12\åL}_{\lower.2mm\hbox{\font\å=cmitt6\ås}}$}\kern-1.0mm\kern0.#2mm\lower.#1mm\hbox{\raise.4mm\hbox{$^{#3}$}}\kern0.25mm} % use e.g. as in $\mvsLrs02^p(\ssp\mu\,,\spp\vPi\ssp)$
\def\LLrs#1#2^#3{\lower.15mm\hbox{\font\å=cmitt12\åL}\kern0.25mm\kern0.#2mm\lower.#1mm\hbox{\raise.4mm\hbox{$^{#3}$}}\kern0.25mm} % use e.g. as in $\LLrs02^p(\ssp\Omega\,\sbi\Yps\sp,\spp\vPi\ssp)$
\def\sLrs#1#2^#3{\hbox{\font\å=cmitt12\åL}_{\lower.2mm\hbox{\font\å=cmitt6\ås}}\kern-1.0mm\kern0.#2mm\lower.#1mm\hbox{\raise.4mm\hbox{$^{#3}$}}\kern0.25mm} % use e.g. as in $\sLrs02^p(\ssp\Omega\,\sbi\Yps\sp,\spp\vPi\ssp)$
\def\lll^#1{\ell\kern0.7mm\raise.2mm\hbox{$^{#1}$}\kern0.15mm}% = l^p with p a bit lifted
\def\Cinfty{C\kern.6mm\raise.15mm\hbox{$^\infty$}\kern.07mm} % C^{\infty}
\def\scrDSw#1'{\raise1.68mm\hbox{\font\Å=cmssi5\Å#1}\kern-.35mm\mathscr D\kern-.35mm\lower.84mm\hbox{\font\Å=cmsl5\ÅS\kern-.3mmw}\kern-1.55mm'\kern1.1mm} % use e.g. as in $\scrDSw\sp'(\ssp\Omega\,\sbi\Yps\ssp,\spp\vPi\ssp)$ or $\scrDSw{vc}'(\sp\ebit\Pii\ssp)$ to get the space of Schwartz distributions in a vector column
\def\Besovrss#1#2#3_#4^#5{\lower.1mm\hbox{\font\Å=cmss10\ÅB}\kern.6mm\lower.25mm\hbox{\raise.#1mm\hbox{$^{\kern.2mm\kern.#3mm#5\,}_{\kern.#2mm#4}$}}} %
\def\Besovprss#1#2#3_#4^#5{\kern.82mm\lower.8mm\hbox{\font\Å=cmb8\Å.}\kern-1.65mm\lower.1mm\hbox{\font\Å=cmss10\ÅB}\kern.6mm\lower.25mm\hbox{\raise.#1mm\hbox{$^{\kern.2mm\kern.#3mm#5\,}_{\kern.#2mm#4}$}}} %
\def\vcBesovrss#1#2#3_#4^#5{\raise1.62mm\hbox{\font\Å=cmssi5\Åvc}\kern-.35mm\lower.1mm\hbox{\font\Å=cmssi10\ÅB}\kern.75mm\raise.#1mm\hbox{$^{\kern.45mm\kern.#3mm#5\,}_{\kern.#2mm#4}$}} %
\def\LiTrFrss#1#2#3_#4^#5{\lower.1mm\hbox{\font\Å=cmss10\ÅF}\kern.4mm\lower.25mm\hbox{\raise.#1mm\hbox{$^{\kern.6mm\kern.#3mm#5\,}_{\kern.#2mm#4}$}}} %
\def\vcLiTrFrss#1#2#3_#4^#5{\raise1.62mm\hbox{\font\Å=cmss5\Åvc}\kern-.13mm\lower.1mm\hbox{\font\Å=cmss10\ÅF}\kern.4mm\raise.#1mm\hbox{$^{\kern.6mm\kern.#3mm#5\,}_{\kern.#2mm#4}$}} %
\def\Besrss#1#2#3^#4{B\kern-.25mm\lower.78mm\hbox{\font\SweD =cmr6\SweD es}\kern-1.35mm\raise.#1pt\hbox{$_{\kern1.5mm}^{\kern.#2pt{#4}\kern.2mm\kern.#3pt}$}} % Bessel potential space, use e.g. as in $\Besrss300^{\emath s}(\spp\Omega\spp)\,$, $\Besrss360^{\emath r}(\spp\Omega\spp)\,$, $\Besrss700^{\plusinftyy}(\spp\Omega\spp)\,$, $\Besrss700^{\minusinftyy}(\spp\Omega\spp)\,$, $\Besrss040^t(\spp\Omega\spp)\,$, $\Besrss930^p(\spp\Omega\spp)\,$, $\Besrss900^q(\spp\Omega\spp)\,$, $\Besrss700^\iota(\spp\Omega\spp)\,$, $\Besrss030^l(\spp\Omega\spp)\,$, $\Besrss320^{\emath n}(\spp\Omega\spp)\,$, $\Besrss700^m(\spp\Omega\spp)\,$, $\Besrss000^{\ssmb S}(\spp\Omega\spp) $
\def\HBsrss#1#2#3^#4{H\kern-.35mm\lower.84mm\hbox{\font\Å=cmmi5\ÅB}\kern-.3mm\lower.8mm\hbox{\font\Å=cmmi6\Ås}\kern-1.75mm\raise.#1pt\hbox{$_{\kern2.2mm}^{\kern.#2pt{#4}\kern.2mm\kern.#3pt}$}} % Bessel potential space, same as above, another notation, use e.g. as in $\HBsrss300^{\emath s}(\spp\Omega\spp)\,$, $\HBsrss360^{\emath r}(\spp\Omega\spp)\,$, $\HBsrss700^{\plusinftyy}(\spp\Omega\spp)\,$, $\HBsrss700^{\minusinftyy}(\spp\Omega\spp)\,$, $\HBsrss040^t(\spp\Omega\spp)\,$, $\HBsrss930^p(\spp\Omega\spp)\,$, $\HBsrss900^q(\spp\Omega\spp)\,$, $\HBsrss700^\iota(\spp\Omega\spp)\,$, $\HBsrss030^l(\spp\Omega\spp)\,$, $\HBsrss320^{\emath n}(\spp\Omega\spp)\,$, $\HBsrss700^m(\spp\Omega\spp)\,$, $\HBsrss000^{\ssmb S}(\spp\Omega\spp) $
\def\vcBesrss#1#2#3^#4{\raise1.65mm\hbox{\font\SweD =cmr5\SweD v}\kern-.1mm\raise1.65mm\hbox{\font\SweD =cmmi5\SweD c}\kern-.3mmB\kern-.25mm\lower.78mm\hbox{\font\SweD =cmr6\SweD es}\kern-1.35mm\raise.#1pt\hbox{$_{\kern1.5mm}^{\kern.#2pt{#4}\kern.2mm\kern.#3pt}$}} % Bessel potential space in a vector column, use e.g. as in $\vcBesrss300^{\emath s}(\vcalO\sp)\,$, $\vcBesrss360^{\emath r}(\vcalO\sp)\,$, $\vcBesrss700^{\plusinftyy}(\vcalO\sp)\,$, $\vcBesrss700^{\minusinftyy}(\vcalO\sp)\,$, $\vcBesrss040^t(\vcalO\sp)\,$, $\vcBesrss930^p(\vcalO\sp)\,$, $\vcBesrss900^q(\vcalO\sp)\,$, $\vcBesrss700^\iota(\vcalO\sp)\,$, $\vcBesrss030^l(\vcalO\sp)\,$, $\vcBesrss320^{\emath n}(\vcalO\sp)\,$, $\vcBesrss700^m(\vcalO\sp)\,$, $\vcBesrss000^{\ssmb S}(\vcalO\sp) $
\def\vcHBsrss#1#2#3^#4{\raise1.65mm\hbox{\font\Å=cmr5\Åv}\kern-.1mm\raise1.65mm\hbox{\font\Å=cmmi5\Åc}\kern-.3mmH\kern-.35mm\lower.84mm\hbox{\font\Å=cmmi5\ÅB}\kern-.3mm\lower.8mm\hbox{\font\Å=cmmi6\Ås}\kern-1.75mm\raise.#1pt\hbox{$_{\kern2.2mm}^{\kern.#2pt{#4}\kern.2mm\kern.#3pt}$}} % Bessel potential space in a vector column, same as above, another notation, use e.g. as in $\vcHBsrss300^{\emath s}(\vcalO\sp)\,$, $\vcHBsrss360^{\emath r}(\vcalO\sp)\,$, $\vcHBsrss700^{\plusinftyy}(\vcalO\sp)\,$, $\vcHBsrss700^{\minusinftyy}(\vcalO\sp)\,$, $\vcHBsrss040^t(\vcalO\sp)\,$, $\vcHBsrss930^p(\vcalO\sp)\,$, $\vcHBsrss900^q(\vcalO\sp)\,$, $\vcHBsrss700^\iota(\vcalO\sp)\,$, $\vcHBsrss030^l(\vcalO\sp)\,$, $\vcHBsrss320^{\emath n}(\vcalO\sp)\,$, $\vcHBsrss700^m(\vcalO\sp)\,$, $\vcHBsrss000^{\ssmb S}(\vcalO\sp) $
\def\HBsprss#1#2#3^#4{H\kern-.35mm\lower.84mm\hbox{\font\Å=cmmi5\ÅB}\kern-.3mm\lower.82mm\hbox{\font\Å=cmmi6\Ås}\lower.88mm\hbox{\font\Å=cmb8\Å\char'056}\kern-2.58mm\raise.#1pt\hbox{$_{\kern2.68mm}^{\kern.#2pt{#4}\kern.2mm\kern.#3pt}$}} % Bessel potential space with zero boundary values, use e.g. as in $\HBsprss300^{\emath s}(\spp\Omega\spp)\,$, $\HBsprss360^{\emath r}(\spp\Omega\spp)\,$, $\HBsprss700^{\plusinftyy}(\spp\Omega\spp)\,$, $\HBsprss700^{\minusinftyy}(\spp\Omega\spp)\,$, $\HBsprss040^t(\spp\Omega\spp)\,$, $\HBsprss930^p(\spp\Omega\spp)\,$, $\HBsprss900^q(\spp\Omega\spp)\,$, $\HBsprss700^\iota(\spp\Omega\spp)\,$, $\HBsprss030^l(\spp\Omega\spp)\,$, $\HBsprss320^{\emath n}(\spp\Omega\spp)\,$, $\HBsprss700^m(\spp\Omega\spp)\,$, $\HBsprss000^{\ssmb S}(\spp\Omega\spp) $
\def\vcHBsprss#1#2#3^#4{\raise1.65mm\hbox{\font\Å=cmr5\Åv}\kern-.1mm\raise1.65mm\hbox{\font\Å=cmmi5\Åc}\kern-.3mmH\kern-.35mm\lower.84mm\hbox{\font\Å=cmmi5\ÅB}\kern-.3mm\lower.82mm\hbox{\font\Å=cmmi6\Ås}\lower.88mm\hbox{\font\Å=cmb8\Å\char'056}\kern-2.58mm\raise.#1pt\hbox{$_{\kern2.68mm}^{\kern.#2pt{#4}\kern.2mm\kern.#3pt}$}} % Bessel potential space with zero boundary values in a vector column, use e.g. as in $\vcHBsprss300^{\emath s}(\spp\Omega\spp)\,$, $\vcHBsprss360^{\emath r}(\spp\Omega\spp)\,$, $\vcHBsprss700^{\plusinftyy}(\spp\Omega\spp)\,$, $\vcHBsprss700^{\minusinftyy}(\spp\Omega\spp)\,$, $\vcHBsprss040^t(\spp\Omega\spp)\,$, $\vcHBsprss930^p(\spp\Omega\spp)\,$, $\vcHBsprss900^q(\spp\Omega\spp)\,$, $\vcHBsprss700^\iota(\spp\Omega\spp)\,$, $\vcHBsprss030^l(\spp\Omega\spp)\,$, $\vcHBsprss320^{\emath n}(\spp\Omega\spp)\,$, $\vcHBsprss700^m(\spp\Omega\spp)\,$, $\vcHBsprss000^{\ssmb S}(\spp\Omega\spp) $
\def\Hdotrss#1#2#3^#4{H\kern-.1mm\lower.1mm\hbox{\font\SweD =cmb9\SweD \char'056}\kern-.15mm\raise.#1pt\hbox{$_{\kern.4mm}^{\kern.#2pt{#4}\kern.2mm\kern.#3pt}$}} % Bessel potential space with zero boundary values, use e.g. as in $\Hdotrss300^{\emath s}(\spp\Omega\spp)\,$, $\Hdotrss360^{\emath r}(\spp\Omega\spp)\,$, $\Hdotrss700^{\plusinftyy}(\spp\Omega\spp)\,$, $\Hdotrss700^{\minusinftyy}(\spp\Omega\spp)\,$, $\Hdotrss040^t(\spp\Omega\spp)\,$, $\Hdotrss930^p(\spp\Omega\spp)\,$, $\Hdotrss900^q(\spp\Omega\spp)\,$, $\Hdotrss700^\iota(\spp\Omega\spp)\,$, $\Hdotrss030^l(\spp\Omega\spp)\,$, $\Hdotrss320^{\emath n}(\spp\Omega\spp)\,$, $\Hdotrss700^m(\spp\Omega\spp)\,$, $\Hdotrss000^{\ssmb S}(\spp\Omega\spp) $
\def\vcHdotrss#1#2#3^#4{\raise1.65mm\hbox{\font\SweD =cmr5\SweD v}\kern-.1mm\raise1.65mm\hbox{\font\SweD =cmmi5\SweD c}\kern-.3mmH\kern-.1mm\lower.1mm\hbox{\font\SweD =cmb9\SweD \char'056}\kern-.15mm\raise.#1pt\hbox{$_{\kern.4mm}^{\kern.#2pt{#4}\kern.2mm\kern.#3pt}$}} % Bessel potential space with zero boundary values in a vector column, use e.g. as in $\vcHdotrss300^{\emath s}(\vcalO\sp)\,$, $\vcHdotrss360^{\emath r}(\vcalO\sp)\,$, $\vcHdotrss700^{\plusinftyy}(\vcalO\sp)\,$, $\vcHdotrss700^{\minusinftyy}(\vcalO\sp)\,$, $\vcHdotrss040^t(\vcalO\sp)\,$, $\vcHdotrss930^p(\vcalO\sp)\,$, $\vcHdotrss900^q(\vcalO\sp)\,$, $\vcHdotrss700^\iota(\vcalO\sp)\,$, $\vcHdotrss030^l(\vcalO\sp)\,$, $\vcHdotrss320^{\emath n}(\vcalO\sp)\,$, $\vcHdotrss700^m(\vcalO\sp)\,$, $\vcHdotrss000^{\ssmb S}(\vcalO\sp) $
\def\HCerss#1#2#3^#4{H\kern-.35mm\lower.82mm\hbox{\font\SweD =cmmi5\SweD C}\kern-.75mm\raise.#1pt\hbox{$_{\kern.9mm}^{\kern.#2pt{#4}\kern.2mm\kern.#3pt}$}} % Bessel potential space of continuous functions, use e.g. as in $\HCerss300^{\emath s}(\spp\Omega\spp)\,$, $\HCerss360^{\emath r}(\spp\Omega\spp)\,$, $\HCerss700^{\plusinftyy}(\spp\Omega\spp)\,$, $\HCerss040^t(\spp\Omega\spp)\,$, $\HCerss930^p(\spp\Omega\spp)\,$, $\HCerss900^q(\spp\Omega\spp)\,$, $\HCerss700^\iota(\spp\Omega\spp)\,$, $\HCerss034^l(\spp\Omega\spp)\,$, $\HCerss320^{\emath n}(\spp\Omega\spp)\,$, $\HCerss700^m(\spp\Omega\spp)\,$, $\HCerss002^{\ssmb S}(\spp\Omega\spp) $
\def\vcHCerss#1#2#3^#4{\raise1.65mm\hbox{\font\SweD =cmr5\SweD v}\kern-.1mm\raise1.65mm\hbox{\font\SweD =cmmi5\SweD c}\kern-.3mmH\kern-.35mm\lower.82mm\hbox{\font\SweD =cmmi5\SweD C}\kern-.75mm\raise.#1pt\hbox{$_{\kern.9mm}^{\kern.#2pt{#4}\kern.2mm\kern.#3pt}$}} % Bessel potential space of continuous functions in a vector column, use e.g. as in $\vcHCerss300^{\emath s}(\vcalO\sp)\,$, $\vcHCerss360^{\emath r}(\vcalO\sp)\,$, $\vcHCerss700^{\plusinftyy}(\vcalO\sp)\,$, $\vcHCerss040^t(\vcalO\sp)\,$, $\vcHCerss930^p(\vcalO\sp)\,$, $\vcHCerss900^q(\vcalO\sp)\,$, $\vcHCerss700^\iota(\vcalO\sp)\,$, $\vcHCerss034^l(\vcalO\sp)\,$, $\vcHCerss320^{\emath n}(\vcalO\sp)\,$, $\vcHCerss700^m(\vcalO\sp)\,$, $\vcHCerss002^{\ssmb S}(\vcalO\sp) $
\def\HCdotrss#1#2#3^#4{H\kern-.35mm\lower.82mm\hbox{\font\SweD =cmmi5\SweD C}\lower.85mm\hbox{\font\SweD =cmb8\SweD \char'056}\kern-1.55mm\raise.#1pt\hbox{$_{\kern1.6mm}^{\kern.#2pt{#4}\kern.2mm\kern.#3pt}$}} % Bessel potential space of continuous functions with zero boundary values, use e.g. as in $\HCdotrss300^{\emath s}(\spp\Omega\spp)\,$, $\HCdotrss360^{\emath r}(\spp\Omega\spp)\,$, $\HCdotrss700^{\plusinftyy}(\spp\Omega\spp)\,$, $\HCdotrss040^t(\spp\Omega\spp)\,$, $\HCdotrss930^p(\spp\Omega\spp)\,$, $\HCdotrss900^q(\spp\Omega\spp)\,$, $\HCdotrss700^\iota(\spp\Omega\spp)\,$, $\HCdotrss034^l(\spp\Omega\spp)\,$, $\HCdotrss320^{\emath n}(\spp\Omega\spp)\,$, $\HCdotrss700^m(\spp\Omega\spp)\,$, $\HCdotrss002^{\ssmb S}(\spp\Omega\spp) $
\def\vcHCdotrss#1#2#3^#4{\raise1.65mm\hbox{\font\SweD =cmr5\SweD v}\kern-.1mm\raise1.65mm\hbox{\font\SweD =cmmi5\SweD c}\kern-.3mmH\kern-.35mm\lower.82mm\hbox{\font\SweD =cmmi5\SweD C}\lower.85mm\hbox{\font\SweD =cmb8\SweD \char'056}\kern-1.55mm\raise.#1pt\hbox{$_{\kern1.6mm}^{\kern.#2pt{#4}\kern.2mm\kern.#3pt}$}} % Bessel potential space of continuous functions with zero boundary values in a vector column, use e.g. as in $\vcHCdotrss300^{\emath s}(\vcalO\sp)\,$, $\vcHCdotrss360^{\emath r}(\vcalO\sp)\,$, $\vcHCdotrss700^{\plusinftyy}(\vcalO\sp)\,$, $\vcHCdotrss040^t(\vcalO\sp)\,$, $\vcHCdotrss930^p(\vcalO\sp)\,$, $\vcHCdotrss900^q(\vcalO\sp)\,$, $\vcHCdotrss700^\iota(\vcalO\sp)\,$, $\vcHCdotrss034^l(\vcalO\sp)\,$, $\vcHCdotrss320^{\emath n}(\vcalO\sp)\,$, $\vcHCdotrss700^m(\vcalO\sp)\,$, $\vcHCdotrss002^{\ssmb S}(\vcalO\sp) $
\def\Holrss#1#2#3^#4{H\kern-.23mm\lower.8mm\hbox{\font\SweD =cmr6\SweD o}\lower.83mm\hbox{\font\SweD =cmr5\SweD l}\kern-1.2mm\raise.#1pt\hbox{$_{\kern1.3mm}^{\kern.#2pt{#4}\kern.0#3in\kern.2mm}$}} % space of holomorphic functions, use e.g. as in $\Holrss720^p(\spp\Omega\spp)\,$, $\Holrss720^\plusinftyy(\spp\Omega\spp)\,$, $\Holrss520^{\eightmath s}(\spp\Omega\spp)\,$, $\Holrss500^{\iota}(\spp\Omega\spp)\,$, $\Holrss000^{}(\spp\Omega\spp)\,$, $\Holrss020^0(\spp\Omega\spp)\,$, $\Holrss020^1(\spp\Omega\spp)\,$, $\Holrss020^2(\spp\Omega\spp)$
\def\vcHolrss#1#2#3^#4{\raise1.65mm\hbox{\font\SweD =cmr5\SweD v}\kern-.1mm\raise1.65mm\hbox{\font\SweD =cmmi5\SweD c}\kern-.3mmH\kern-.23mm\lower.8mm\hbox{\font\SweD =cmr6\SweD o}\lower.83mm\hbox{\font\SweD =cmr5\SweD l}\kern-1.2mm\raise.#1pt\hbox{$_{\kern1.3mm}^{\kern.#2pt{#4}\kern.0#3in\kern.2mm}$}} % space of holomorphic functions in a vector column, use e.g. as in $\vcHolrss720^p(\vcal Q\sp)\,$, $\vcHolrss720^\plusinftyy(\vcal Q\sp)\,$, $\vcHolrss520^{\eightmath s}(\vcal Q\sp)\,$, $\vcHolrss500^{\iota}(\vcal Q\sp)\,$, $\vcHolrss000^{}(\vcal Q\sp)\,$, $\vcHolrss020^0(\vcal Q\sp)\,$, $\vcHolrss020^1(\vcal Q\sp)\,$, $\vcHolrss020^2(\vcal Q\sp)$
\def\vbHolrss#1#2#3^#4{\raise1.65mm\hbox{\font\SweD =cmr5\SweD v}\kern.06mm\raise1.2mm\hbox{\font\SweD =cmmi5\SweD b}\kern-.36mmH\kern-.23mm\lower.8mm\hbox{\font\SweD =cmr6\SweD o}\lower.83mm\hbox{\font\SweD =cmr5\SweD l}\kern-1.2mm\raise.#1pt\hbox{$_{\kern1.3mm}^{\kern.#2pt{#4}\kern.0#3in\kern.2mm}$}} % space of holomorphic functions in a vector bundle, use e.g. as in $\vbHolrss720^p(\sp\vcal E\,)\,$, $\vbHolrss720^\plusinftyy(\sp\vcal E\,)\,$, $\vbHolrss520^{\eightmath s}(\sp\vcal E\,)\,$, $\vbHolrss500^{\iota}(\sp\vcal E\,)\,$, $\vbHolrss000^{}(\sp\vcal E\,)\,$, $\vbHolrss020^0(\sp\vcal E\,)\,$, $\vbHolrss020^1(\sp\vcal E\,)\,$, $\vbHolrss020^2(\sp\vcal E\,)$
\newcommand\Holombd{H\kern-.25mm\lower.82mm\hbox{\font\SweD =cmmi5\SweD b}\lower.82mm\hbox{\font\SweD =cmr5\SweD d}\kern.35mm} % $\Holombd(\spp\Omega\spp)$

\def\bold#1{{\bf#1}}
\def\roman#1{{\rm#1}}
\def\limu_#1{\lim\kern-5.5mm\lower1.5mm\hbox{$_{#1}\ $}}
\def\oseoy{\raise1.9mm\hbox{\kern.5mm\font\SweD =cmr5\SweD o}\kern-1.7mm y}% kahdessa kohtaa esiintyy
\def\O{{}^{}\Cal O}
\newcommand\uniqset{\raise1.6mm\hbox{\font\SweD=cmr5\SweD uniq}\kern.3mm\roman{set}\kern.9mm} % \uniqset x:P = ^{uniq}set x:P = the unique set x with P if such exists, otherwise = the universe
\def\Univ{\hbox{\font\SweD =cmssbx10\SweD U}{}} % the class of all sets
\def\Pows{\Cal P\kern-.7mm\lower.15mm\hbox{$_s$}\kern.3mm} % \Pows A = { U : U subseteq A }
\def\lei{      {}_{ {}^{\,\downarrow\text{\hskip-2.1mm}       }  }  \cap       }
\def\lei{\hbox{\kern.45mm$_{^\downarrow}\kern-1.280mm\cap\kern.85mm$}}
\def\leip{\hbox{\kern.6mm$_{^\downarrow}\kern-1.280mm\cap\kern1mm$}}
\def\leiss#1#2{\hbox{\kern.5mm\kern.#1mm$_{^\downarrow}\kern-1.280mm\cap\kern.9mm\kern.#2mm$}} % \leiss01 = kern.05mm \lei \kern.15mm
\def\supp{\roman{supp\kern1mm}} % \supp f = standard support of a function f
\def\vfsupp{\text{\,-\,\raise1.4mm\hbox{\font\SweD =cmr5\SweD f}supp\kern1mm}} % (\Cal T,o)\vfsupp f = (T,o)-supp f = the preceding generalized; o for 0 and dom f inc bigcup Cal T with Cal T a topology
\def\Card{{}^{}{\rm Card}{}^{}\,}
\newcommand\card[1]{\roman{card}\kern.#1mm\kern.4mm} % 
\newcommand\inve{\kern.27mm\raise1.4mm\hbox{\font\SweD =cmr5\SweD -\kern-.28mm-}\raise1.35mm\hbox{\font\SweD =cmss5\SweD i\kern.18mm}} % relational inverse
\newcommand\sinve{\raise.85mm\hbox{\font\Å=cmr5\Å-\kern-.28mm-}\raise.8mm\hbox{\font\Å=cmss5\Åi\kern.18mm}} % the above for scripts
\newcommand\invss[2]{\kern.27mm\kern.#1mm\raise1.4mm\hbox{\font\SweD =cmr5\SweD -\kern-.28mm-}\raise1.35mm\hbox{\font\SweD =cmss5\SweD i\kern.#2mm\kern.18mm}} % same as above but allows adjustment of spaces before and after; use e.g. as in $x\invss00$ , $\mu\invss11\image\sp\lbb R_+$ , $\chi\invss42\images\spp\Cal B$ , $(\ssp f\invss20\fvalue x\ssp)$ , $R\ssp\inve=R\invss40$
\newcommand\invxs[2]{\kern.27mm\kern.#1mm\raise1.4mm\hbox{\font\SweD =cmr5\SweD -\kern-.28mm-in\kern.#2mm\kern.18mm}} % alternative notation for the above
\newcommand\fvalss[2]{\hbox{\kern.3mm\kern.#1mm\font\SweD =cmr10\SweD \char'022\kern-.1mm}\kern.#2mm} % same as fvalue but allows adjustment of spaces before and after
\def\fvalue{\hbox{\kern.2mm\font\SweD =cmr10\SweD \char'022\kern-.2mm}} % f\value x=f(x)
\def\ffvalue{\hbox{\kern.2mm\font\SweD =cmr7\SweD \char'022\kern-.2mm}} % f\value x=f(x)
\def\image{\hbox{\font\SweD =cmr10\SweD \char'022\kern-1mm\char'022}} % f\image A=f[A]
\def\iimage{\hbox{\font\SweD =cmr7\SweD \kern.3mm\char'022\kern-.7mm\char'022\kern-.3mm}} % f\image A=f[A]
\def\images{\hbox{\font\SweD =cmr10\SweD \char'022\kern-1mm\char'022\kern-1mm\char'022}} % f\images\Cal A={f[A]:A in Cal A}
\def\weco{\kern.15mm\hbox{\font\SweD =cmtt10\SweD \char'054}\kern.4mm} % Weihe comma: x\weco y=(x,y)_{Weihe}
\def\cdotn{\kern-.2mm\cdot\kern-.2mm} % same as cdot but with smaller spaces before and after
\def\setminusn{\kern-.2mm\setminus\kern-.2mm} % same as setminus but with smaller spaces before and after
\def\capss#1#2{\kern.1mm\kern.#1mm\cap\kern.1mm\kern.#2mm} % \capss21 = kern.3mm \cap \kern.2mm
\def\cupss#1#2{\kern.1mm\kern.#1mm\cup\kern.1mm\kern.#2mm} % \cupss21 = kern.3mm \cup \kern.2mm
\def\cuppp{\kern0.15mm\cup\kern0.15mm} % \cup with larger spaces around
\def\timesn{\kern-.2mm\times\kern-.2mm} % same as times but with smaller spaces before and after
\def\Times{\kern.7mm\hbox{\font\SweD =cmbsy10\SweD \char'002}\kern.7mm}
\def\ttimes{\hbox{\kern-.2mm${}\times\kern-2.5mm\lower.8mm\hbox{\font\SweD =cmr5\SweD t}\kern1.8mm$}} % product topology = \rist
\def\ttimesn{\hbox{\kern-.2mm${}\times\kern-2.5mm\lower.8mm\hbox{\font\SweD =cmr5\SweD t}\kern1.4mm$}} % same as above but with smaller space after 
\def\ktimes{\hbox{\kern-.2mm${}\times\kern-2.5mm\lower1mm\hbox{\font\SweD =cmr5\SweD k}\kern1.5mm$}} % compactly generated product topology
\def\ltimes{\hbox{${}\times\kern-2.45mm\lower.8mm\hbox{\font\SweD =cmr5\SweD l}\kern1.6mm$}} % \varLambda\ltimes\varGamma = the convergence product of varLambda and varGamma
\def\vstimes{\kern.95mm\raise.45mm\hbox{\font\SweD =cmbsy6\SweD \char'002}\kern-2.3mm\lower.9mm\hbox{\font\SweD =cmr5\SweD vs}\kern1.05mm} % X\vstimes Y = the vector space product of X and Y
\def\atimes{\kern.8mm\hbox{\font\SweD =cmsy10\SweD \char'002}\kern-1.85mm\lower.65mm\hbox{\font\SweD =cmr5\SweD a}\kern1.4mm} % A\atimes B = algebra product of A and B
\newcommand\meatimes{\kern1mm\raise.5mm\hbox{\font\SweD =cmbsy5\SweD \char'012}\kern1mm} % small \otimes for product measure
\newcommand\circss[2]{\kern.1mm\kern.#1mm\circ\kern.#2mm\kern.1mm} % use e.g. as in $u\circss00 v$ to get .1mm larger spaces around `\circ' than in $u\circ v$
\newcommand\mcircss[2]{\kern.#1mm\kern.7mm{\circ}\kern-1.85mm\lower.75mm\hbox{\font\Å=cmr5\Åm}\kern.8mm\kern.#2mm} % to denote the composition of maps use e.g. as in $\tilde f\mcircss23\tilde g$ , $\tilde g\mcircss61\tilde f$ , $\tilde f\mcircss25\tilde\ell$ , $\tilde g\mcircss65\tilde\ell$ , $\tilde\ell\mcircss53\tilde u$
\def\Circ{\kern.9mm\hbox{\font\SweD =cmbsy10\SweD \char'016}\kern.9mm}
\def\cardplus{\hbox{$\kern.77mm+\kern-1.95mm\raise.23mm\hbox{$_{_{\roman c}}$}\kern1.33mm$}}%
\def\ordplus{\hbox{$\kern.78mm+\kern-1.97mm\raise.23mm\hbox{$_{_{\roman o}}$}\kern1.22mm$}}%

\def\ccdot{\hbox{$\kern.77mm\cdot\kern-1mm\raise.45mm\hbox{$_{_{\roman c}}$}\kern.38mm$}} % i\ccdot j = the cardinal product of i and j
\def\svs#1{\sbi{\fiveroman{svs\,}#1}} % for example (x+tu)\svs E = (x+tu)_{svs E} = " x+tu " in the ambiguous notation
\newcommand\vvs[1]{{_{\kern-0.1mm}}_{\hbox{\font\SweD =cmr5\SweD vs\kern.5mm}#1}} % for example (x+tu)\vvs X = (x+tu)_{vs X} = " x+tu " in the ambiguous notation in an unstructured module X

\def\minus{\kern.2mm\lower1.05mm\hbox{$^-$}}
\def\pplus{\raise.22mm\hbox{\font\SweD =cmr5\SweD \char'053}}% 5 point +
\def\mminus{\raise.18mm\hbox{\font\SweD =cmsy5\SweD \char'000}}% 5 point -
\def\plusinftyy{{\raise.18mm\hbox{\font\SweD =cmr5\SweD \char'053}\infty}}% +infty for sub- and superscripts with smaller +
\def\minusinftyy{{\raise.18mm\hbox{\font\SweD =cmsy5\SweD \char'000}\infty}}% -infty for sub- and superscripts with smaller -
\def\inftyy{{\raise.15mm\hbox{\font\SweD =cmsy7\SweD \char'061}}} % a little raised infty for superscript
\def\inftyyplus{\raise.15mm\hbox{\font\SweD =cmsy7\SweD \char'061}\raise.65mm\hbox{\font\SweD =cmr5\SweD +}} % a little raised infty^+ for superscript
\def\plusinfty{\lower1.05mm\hbox{$^+$}\infty}
\def\minusinfty{\lower1.05mm\hbox{$^-$}\infty}
\def\ebit#1{\kern.1mm\hbox{\font\SweD =cmmib8\SweD #1}\kern.2mm} % use e.g. as $\ebit A$ to get 8 point bold italic A
\def\ebiF{\kern.1mm\hbox{\font\SweD =cmmib8\SweD F}\kern.5mm} % 8 point bold italic F
\def\ebiT{\kern.1mm\hbox{\font\SweD =cmmib8\SweD T}\kern.6mm} % 8 point bold italic T
\def\ebiU{\kern.1mm\hbox{\font\SweD =cmmib8\SweD U}\kern.5mm} % 8 point bold italic U
\def\ebiV{\kern.1mm\hbox{\font\SweD =cmmib8\SweD V}\kern.7mm} % 8 point bold italic V
\newcommand\biiit[3]{\kern.1mm\hbox{\font\SweD =cmmib#2\SweD#3}\kern.#1pt\kern.#1pt} % e.g. $\biiit48 T$ gives bold italic 8 point T with .4 pt + .4 pt extra space after
\newcommand\nCalT{\hbox{\font\SweD =cmbsy9\SweD T}^{\kern.5mm}} % 9 point bold Cal T
\newcommand\ebiGamma{\kern.1mm\hbox{\font\SweD =cmmib8\SweD \char'000}\kern.5mm} % 8 point bold italic Gamma
\def\bmii#1#2{\hbox{\font\SweD =cmmib#1\SweD #2}}

\def\fssi#1{\hbox{\font\Å=cmssi10\Å#1}\kern0.15mm} % cmssi in math
\newcommand\nfss[1]{\hbox{\font\Å=cmss9\Å#1}} %
\def\smb#1{\hbox{\font\sweD =cmmi8\sweD #1\kern.3mm}} % eightpoint (capital) math symbols
\def\ssmb#1{\hbox{\font\SweD =cmmi6\SweD #1}} % smaller capital for index
\def\eCal#1{\kern.1mm\hbox{\font\sweD =cmbsy8\sweD #1\kern.4mm}} % 8point bold calligraphic #1
\def\ecal#1{\kern.1mm\hbox{\font\sweD =cmsy8\sweD #1\kern.3mm}} % 8point calligraphic #1
\def\ncal#1{\kern.1mm\hbox{\font\sweD =cmsy9\sweD #1\kern.3mm}} % 9point calligraphic #1
\def\vcal#1{\kern-.1mm\vec{\kern.2mm\hbox{\font\sweD =cmsy7\sweD #1}\kern.3mm}} % e.g., vector bundle \vcal E
\newcommand\vCal[2]{\kern-.#1pt\kern-.#1pt\vec{\kern.#1pt\kern.#1pt\hbox{\font\Å=cmbsy7\Å#2}\kern-.#1pt\kern-.#1pt}\kern.#1pt\kern.#1pt} %
\newcommand\vbscr[2]{\kern-.#1pt\kern-.#1pt\vec{\kern.#1pt\kern.#1pt\hbox{\font\Å=eusb7\Å#2}\kern-.#1pt\kern-.#1pt}\kern.#1pt\kern.#1pt} % to get a notation for vector bundle, use e.g. as in $\vCal0 E=\vbscr2 E=(\ssp\ncal E\ssp,\sp M\sp,\sp\varrho\,,\sp\varTheta\ssp)$ , $(\sp\vbscr2 E\ssp),\vbscr1 F,\vbscr4 G,\vbscr1 H$ , the number #1 moves the `vec´ an amount of .#1pt + .#1pt to the left
\newcommand\scrm[2]{\hbox{\font\SweD=eusm#1\SweD#2}} % using e.g. as in $\scrm9 T$ gives a 9 point scriptstyle T in math
\newcommand\scrb[2]{\hbox{\font\SweD=eusb#1\SweD#2}} % using e.g. as in $\scrb9 T$ gives a 9 point bold scriptstyle T in math
\newcommand\scrmt[1]{\hbox{\font\SweD=eusm10\SweD#1}} % using e.g. as in $\scrmt T$ gives a 10 point scriptstyle T in math
\newcommand\scriptm[2]{\hbox{\font\SweD=eusm#1\SweD#2}} % using e.g. as in $\scriptm9T$ gives a 9 point scriptstyle T in math
\newcommand\vcalO{\kern-.1mm\vec{\kern.3mm\hbox{\font\sweD =cmsy7\sweD O}}\kern.2mm} % gives \vec{7 point Cal O} in math
\newcommand\vcolQ{\vec{\hbox{\font\Å=cmbsy7\Å\char'121}}} % vector column vec bold 7 point cal Q

\newcommand\conccc{\lower.3mm\hbox{${\bf{\hat{\phantom w}}}$}} % <a,b,c>\conccc<x,y,z>=<a,b,c,x,y,z>
\newcommand\ssconc[2]{\kern.#1mm%\lower.3mm
                                \hbox{${\bf{\hat{\phantom w}}}$}\kern.#2mm} % x\ssconc00 y = x\conccc y
\def\id{\kern.3mm\roman{id}\kern.7mm}
\def\idv{\hbox{\font\SweD =cmr10\SweD id}\kern.25mm\lower.7mm\hbox{\font\SweD =cmr6\SweD v}\kern.6mm} % \idv E = id_v E = id(v_s E)
\def\idm{\hbox{\font\SweD =cmr10\SweD id}\kern.25mm\lower.7mm\hbox{\font\SweD =cmr6\SweD m}\kern.6mm} % id_m M=id(Dom M)
\def\seq#1{\langle#1\rangle}
\newcommand\seqss[3]{\langle\kern.7mm\kern.#1mm#3\kern.#2mm\kern.8mm\rangle} % for example $\seqss00{fx:x\in O} = \langle fx:x\in O\rangle $ and $\seqss23{\roman k\,i:i\in\mathbb N} = \langle\kern.2mm\roman f\,i:i\in\mathbbN\kern.3mm\rangle
\def\Seq#1{\big\langle#1\big\rangle}
\def\vecs{\upsilon\kern-0.3mm\lower.15mm\hbox{$_s$}\kern0.3mm} % underlying set of a structured vector space
\newcommand\svecs{\upsilon\kern-.3mm\lower.55mm\hbox{\font\Å=cmmi5\Ås}\kern.35mm} % the above for sub- and superscripts
\def\vecss{\hbox{\font\SweD =cmitt10\SweD v}\kern-0.1mm\lower.15mm\hbox{$_s$}\kern0.2mm} % underlying set of a vector space / vector (= linear) structure
\newcommand\svecss{\hbox{\font\SweD =cmitt7\SweD v}\kern-0.1mm\lower.45mm\hbox{\font\SweD =cmmi5\SweD s}\kern0.2mm} % the above for sub- and superscripts
\def\nullv{0\lower.7mm\hbox{\font\SweD =cmr6\SweD v}\kern.6mm} % \nullv X = zero vector of the vector space X
\def\nullsv{0\lower.7mm\hbox{\font\SweD =cmr6\SweD sv}\kern.6mm} % \nullsv E = zero element in a structured module E
\def\Bnull_#1{\hbox{\font\SweD =cmssbx10\SweD 0}{_{\kern-0.1mm}}_{#1\kern.15mm}} % \Bnull_E = \nullsv E
\def\Bzero_#1{\hbox{\font\SweD =cmbx10\SweD 0}{_{\kern-0.1mm}}_{#1}} % \bzero X = zero in vector space X
\def\bnull#1{\hbox{\font\SweD =cmssbx10\SweD 0}{}_{\font\SweD =cmmi6\lower.15mm\hbox{\kern-.1mm\SweD #1\kern.15mm}}} % \bnull E = \Bnull_E = \nullsv E
\def\bnulla#1_#2{\hbox{\font\SweD =cmssbx10\SweD 0}{}_{\font\SweD =cmmi6\lower.15mm\hbox{\SweD #1\kern-.1mm}}\lower.3mm\hbox{$_{_{#2}}$}} % \bnulla E_1 = zero in vector space E_1
\def\bzero#1{\hbox{\font\SweD =cmbx10\SweD 0}{}_{\font\SweD =cmmi6\lower.15mm\hbox{\kern-.1mm\SweD #1\kern.15mm}}} % \bzero X = zero in vector space X
\def\dom{{{}^{}{\rm dom}\,{}_{{}^{}}}}
\def\domm{\kern0.15mm{\rm dom}{^{\kern.3mm\hbox{\font\SweD =cmr6\SweD 2}}}\,}
\def\domr#1{\roman{dom}^{\font\SweD =cmr6\raise.0mm\hbox{\kern.3mm\SweD #1}}}
\def\domy^#1{\roman{dom}\kern.1mm\raise1.5mm\hbox{\font\Å=cmr5\Å#1}\kern.7mm}%

\def\rng{{}^{}{\rm rng}\,{}_{{}^{}}}
\def\RHB#1#2{\raise#1mm\hbox{$#2$}} % raised (by #1 mm) horizontal box of #2
\def\LHB#1#2{\lower#1mm\hbox{$#2$}} % lowered (by #1 mm) horizontal box of #2
\def\fivemath#1{\hbox{\font\SweD =cmmi5\SweD #1}}
\def\fiveroman#1{\hbox{\font\SweD =cmr5\SweD #1\kern.1mm}}
\def\sixmath#1{\hbox{\font\SweD =cmmi6\SweD #1}}
\def\sixroman#1{\hbox{\font\SweD =cmr6\SweD #1\kern.1mm}}
\def\eightmath#1{\hbox{\font\SweD =cmmi8\SweD {#1}\kern.1mm}}% 8 point math italic
\newcommand\emath[1]{\hbox{\font\SweD =cmmi8\SweD {#1}\kern.1mm}}% 8 point math italic, same as \eightmath
\newcommand\emth[2]{\lower.0#2mm\hbox{\raise.#2pt\hbox{\font\SweD =cmmi8\SweD #1}}} % \emth s0=\eightmath s , \emth s3=\eightmath s raised approx .1mm
\def\eightroman#1{\hbox{\font\SweD =cmr8\SweD {#1}\kern.1mm}}% 8 point roman
\def\erm#1{{\font\SweD =cmr8\SweD #1}}% 8 point roman for text, e.g. \erm{DF\,}--\,space 
\newcommand\esl[1]{\hbox{\font\Å=cmsl8\Å#1}} % gives 8 point sl
\newcommand\eit[1]{\hbox{\font\Å=cmti8\Å#1}} % gives 8 point italic
\newcommand\efss[1]{\hbox{\font\Å=cmss8\Å#1}} % gives 8 point "sans serif"
\newcommand\eightsl{\font\Å=cmsl8\Å}

\newcommand\subtext[1]{_{\kern0.15mm\lower.30mm\hbox{\font\Å=cmr5\Å#1}}}
\newcommand\subText[1]{_{\lower.15mm\hbox{\font\SweD =cmr5\SweD #1}}} % $\gamma\subText{given}$ is the same as $\gamma_{given}$ but ``given'' in the 5 point roman font and a bit more lowered
\newcommand\yxbtext[3]{_{\kern.#2mm\lower.#1mm\hbox{\lower.15mm\hbox{\font\SweD =cmr5\SweD #3}}}} % \yxbtext2{55}{fun} is about the same as \subtext{fun} but .2mm:s lower and .55mm:s further to the right
\def\subtexT#1{\raise.2mm\hbox{$_{_{\kern0.15mm\hbox{\font\SweD =cmr5\SweD #1}}}$}}
\newcommand\suba{\lower.87mm\hbox{\font\Å=cmr5\Åa}} % notation for absolute value as in $|\,t\,|\suba$
\def\lllnor_#1{\lower.87mm\hbox{\kern.2mm\font\Å=cmmi5\Å\char'140}\raise.23mm\hbox{\kern.5mm$_{#1}$}} % use e.g. as in $\|\,x\,\|\lllnor_p$ , 
\def\sLnorm_#1{\kern.35mm\lower.83mm\hbox{\font\Å=cmmi5\ÅL}\kern.5mm\raise.2mm\hbox{$_{#1}$}} % use e.g. as in $\|\,x\,\|\sLnorm_q$ , $\big(\ssp\|\,x\,\|\sLnorm_2\spp\big)\,^2$
\def\Lnorss#1#2^#3_#4{\lower.86mm\hbox{\kern.2mm\font\Å=cmssi5\ÅL}\kern.2mm\kern.#1mm\raise.3mm\hbox{$_{#3}$}\kern-.2mm\kern.#2mm\lower.3mm\hbox{$_{#4}$}} % use e.g. as in $\|\,x\,\|\Lnorss33^p_\mu$
\def\sWnorm_#1{\kern.35mm\lower.8mm\hbox{\font\Å=cmmi5\ÅW}\kern.65mm\raise.2mm\hbox{$_{#1}$}} % use e.g. as in $\|\,x\,\|\sWnorm_{\emath s,\ssp p}$ , $\big(\ssp\|\,x\,\|\sWnorm_{1\sp,\ssp 2}\sp\big)\,^2$
\def\ssNor#1_#2{\kern.4mm\kern.#1pt\lower.32mm\hbox{$_{#2}$}} % use e.g. "\|\ssNor3_i" to get "\|_i" with .3 point extra space between `\|´ `i´
\def\sNor#1{\kern.25mm\lower.38mm\hbox{$_{#1}$}}
\def\sNorr#1{\kern-.2mm\lower.38mm\hbox{$_{#1}$}}
\def\sNoreset_#1{\kern.13mm\lower.83mm\hbox{\font\SweD =cmmi6\SweD C}\kern.32mm\lower.1mm\hbox{$_{^{\emptyset,#1}}$}}% \|y\|\sNoreset_i produces \|y\|_{C^{\emptyset,i}}
\def\norSig_#1{|\kern.23mm\lower.87mm\hbox{\font\å=cmr5\å\char'006\kern.15mm{#1}}} % use e.g. as in $|\,x\,\norSig_2$ to get |x|_{\Sigma 2} to denote the Euclidean norm of x in R^n
\def\sbi#1{{_{\kern-0.1mm}}_{#1}} % same as _#1 but a little lower
\def\ais#1_#2{{}_{\font\SweD =cmmi6\lower.15mm\hbox{\kern-.1mm\SweD #1\kern.15mm}}\lower.3mm\hbox{${_{\kern-0.3mm_{#2}}}$}} %
\def\aais#1_#2{\kern.1mm{}_{\font\SweD =cmmi6\lower.25mm\hbox{\kern-.1mm\SweD #1\kern.15mm}}\lower.4mm\hbox{${_{\kern-0.3mm_{#2}}}$}} %
\def\ai#1{{}_{\font\SweD =cmmi6\lower.15mm\hbox{\kern-.1mm\SweD #1\kern.15mm}}} % 6:n pisteen kirjainalaindeksi
\def\yi#1{^{\font\SweD =cmmi6\raise.0mm\hbox{\kern-.1mm\SweD #1\kern.15mm}}} % 6:n pisteen kirjainyl"indeksi
\def\ear#1{{}_{\font\SweD =cmr5\lower.15mm\hbox{\kern.1mm\SweD #1}}} % 5:n pisteen numeroalaindeksi
\def\ar#1{{}_{\font\SweD =cmr6\lower.15mm\hbox{\kern.1mm\SweD #1}}} % 6:n pisteen numeroalaindeksi
\newcommand\aR[1]{_{\lower.15mm\hbox{\font\SweD =cmr6\SweD #1}}} % $Y\aR 1$ is the same as $Y_1$ but `1' in the 6 point roman font and a bit more lowered
\def\aar#1{_{\font\SweD =cmr6\lower.15mm\hbox{\kern.1mm\SweD #1}}} % 6:n pisteen numeroalaindeksi
\def\aars#1_#2{{#1_{\kern-.1mm}}_{#2}} % e.g. use as in $0_{\aars F_8}$ to get $0_{F_8}$ with the 8 a little lower
\newcommand\Yr[2]{\raise1.5mm\hbox{\font\å=cmr5\å#1\kern.#2mm}}%
\def\yr#1{^{\font\SweD =cmr6\raise.0mm\hbox{\kern.3mm\SweD #1}}} % 6:n pisteen numeroyl"indeksi
\def\yrai^#1_#2{^{\kern.4mm\hbox{\font\SweD =cmr6\SweD {#1}}}_{\kern.2mm{#2}}}
\def\upparentes#1{^{\kern.2mm\raise.2mm\hbox{\font\SweD =cmr6\SweD \char'050}\kern.1mm{#1}\kern.1mm\raise.2mm\hbox{\font\SweD =cmr6\SweD \char'051}}} % e.g. x\upparentes l gives x^{(l)} with 6point parentheses
\def\lupar{\kern.2mm\lower1mm\hbox{$^{^(}$}} % pieni vasen yl"sulku
\def\rupar{\lower1mm\hbox{$^{^)}$}\kern-.15mm} % pieni oikea yl"sulku
\def\yyi#1{^{\font\SweD =cmmi6\lower.6mm\hbox{\kern-.25mm\SweD #1\kern-.05mm}}} % pieni kirjainyl"indeksi, tarkoitettu derivaatan merkint""n
\def\yyr#1{^{\font\SweD =cmr6\lower.45mm\hbox{\kern-.25mm\SweD #1\kern-.15mm}}} % pieni numeroyl"indeksi, tarkoitettu derivaatan merkint""n
\def\yplus{\lower1mm\hbox{$^{^+}$}} % + merkki edelliseen
\def\yminus{\lower1mm\hbox{$^{^-}$}} % - merkki edelliseen
\def\aminus{{\kern.15mm\raise.3mm\hbox{$_{_-}$}\kern-.1mm}}%
\def\yvee{\LHB{.9}{^{^{\,\vee}}}\kern-.3mm} % like ^\vee
\def\ywed{\LHB{.9}{^{^{\,\wedge}}}\kern-.3mm} % like ^\wedge
\newcommand\ywedge[1]{\kern.2mm\kern.#1mm\lower.9mm\hbox{$^{^\wedge}$}} %
\newcommand\sprim[1]{\kern.#1mm\raise1.5mm\hbox{\font\Å=cmsy5\Å\char'060}} % use e.g. as in $N\sprim1$ , $A\sprim0$ , $B\sprim3$ , $C\sprim2$ , $x\sprim4$ , to get smaller primes with modified space between the symbol and prime
\newcommand\sprimm[1]{\kern.#1mm\raise1.5mm\hbox{\font\Å=cmsy5\Å\char'060\kern-.15mm\char'060}} % use e.g. as in $N\sprimm1$ , $A\sprimm0$ , $B\sprimm3$ , $C\sprimm2$ , $x\sprimm4$ , to get smaller doubleprimes with modified space between the symbol and primes
\newcommand\myprime{\kern.3mm\raise1.5mm\hbox{\font\SweD =cmsy5\SweD \char'060}} % to get $E'$ with smaller ' write $E\myprime$
\newcommand\myprimes{\kern.3mm\raise1.5mm\hbox{\font\SweD =cmsy5\SweD \char'060\kern-.3mm\char'060}} % to get $E''$ with smaller '' write $E\myprimes$
\def\yplk{{}^{{}_{{}^{'\!}}}} % isoille kirjaimille sopiva yl"pilkku TEXma.tex
\def\adot{\kern.2mm\hbox{\font\SweD =cmb10\SweD \char'056}}%
\def\aadot{\kern.1mm\lower.1mm\hbox{\font\SweD =cmb7\SweD \char'056}}%
\def\ydot{\kern.2mm\raise1.9mm\hbox{\font\SweD =cmb10\SweD \char'056}}% k\ydot = k^. = the integer corresponding to the natural number k
\def\yydot{\kern.2mm\raise1.35mm\hbox{\font\SweD =cmb7\SweD \char'056}\kern.2mm}% the above for 7 point mode
\def\yydott{\kern.2mm\raise1.35mm\hbox{\font\SweD =cmb6\SweD \char'056}\kern.2mm}% the above for 7 point mode
\def\ClTopR#1^#2{\roman{Cl}\kern.6mm_{\hbox{\font\å=eusm6\åT}_{\mathbb R}^{#2}}\kern.#1mm} % use e.g. as in $\ClTopR4^{}\rng\sp y$ to get $Cl_{T_R} rng y$ denoting the closure of rng y in the usual topology of the real line; the "4" puts an extra space of .4 mm before "rng y"
\def\ClT{{\rm Cl}\kern.25mm\lower.4mm\hbox{$_{\Cal T}$}\kern0.2mm} % Cl_Cal T A = closure of A w.r.t. T
\def\IntT{\sp{\rm Int}\kern.2mm\lower.4mm\hbox{$_{\Cal T}$}\kern0.2mm} % Int_Cal T A = interior of A w.r.t. T
\def\Cl_taurd#1{\roman{Cl_{}}_{\kern0.37mm\hbox{\font\SweD =cmmi8\SweD \char'034}\kern-0.15mm{_{}}_{\roman{rd}}\kern0.2mm#1\,}}% Cl_{tau_{rd}#1}S = closure of S in the topology of the topologized vector space E (= #1)
\def\Int_taurd#1{\roman{Int_{}}_{\kern0.37mm\hbox{\font\SweD =cmmi8\SweD \char'034}\kern-0.15mm{_{}}_{\roman{rd}}\kern0.2mm#1\,}}% Int_{tau_{rd}#1}S = interior of S in the topology of the topologized vector space E (= #1)
\def\inc{\subseteq}

\def\exi#1{\exists\,#1\kern.2mm\,;}
\def\all#1{\forall\,#1\kern.2mm\,;}

\def\equivv{\Leftrightarrow}

\def\sP#1 {\kern.#1mm}
\newcommand\sppp{_{\kern0.2mm}} % very very small positive, use e.g. $Y\sppp,B$ instead of $Y,B$ or $Y\spp,B$
\def\spp{\kern0.07mm} % a horizontal very small positive space
\def\sp{\kern0.15mm} % a horizontal small positive space
\def\ssp{\kern0.37mm} % a bigger positive space
\def\snn{\kern-0.2mm} % a very small horizontal negative space
\def\sn{\kern-0.3mm} % a horizontal small negative space
\def\ssn{\kern-0.63mm} % a bigger negative space
\def\nKP#1{$\null$\kern#1mm}
\def\nKN#1\par{$\null$\kern-#1mm} % e.g. write \nKN2.2 at the end of a paragraph to get an indentation 2.2 mm:s smaller than usually in the next paragraph
\newcommand\KPt[1]{\kern.#1mm} % \KPt* is the same as \KP{.*} when * is any of 1,2,3,4,5,6,7,8,9
\def\KP#1{\kern#1mm} % posive kern of #1 mm
\def\KPp#1.#2{\kern#1.#2mm} %
\def\KN#1{\kern-#1mm} % negave kern of #1 mm

\def\text#1{\hbox{\rm#1}}

\def\VBOX/#1/#2/HEREend{\vbox{#2\vskip-#1mm}\vfill\null\eject}
\def\PouN$#1${\hbox{$#1$}} % text math which is not compressed or stretched
\def\œ$#1${\hbox{$#1$}} % text math which is not compressed or stretched, the same as above
\def\"{\"a} \def\"{\"o}
\def\q#1{``\kern0.37mm#1\kern0.37mm"}
\def\newProCla#1\par#2\par{\vskip1.7mm\noindent\bf#1\it#2\vskip1.7mm}
\def\Prooff{{\font\SweD =cmssi10\SweD P\kern.37mmr\kern.37mmo\kern.37mmo\kern.37mmf\kern.37mm. }\rm}

\def\QED{\hfill\hbox{$\ \sqcap\kern-2.45mm\sqcup$}}

\def\noin{\noindent}
\def\Newline{\kern-10mm\newline}
\font\rp=cmr8
\def\eps{\varepsilon}
\def\leu{\raise1.5mm\hbox{\font\SweD =cmmi5\SweD \char'074}\kern.2mm}%
\def\riu{\kern.2mm\raise1.5mm\hbox{\font\SweD =cmmi5\SweD \char'076}}%
\def\Symbol#1Ï{\kern.35mm\hbox{\font\SweD =cmr10\SweD \char'047}\kern.2mm#1\kern.35mm\hbox{\font\SweD =cmr10\SweD \char'047}}%
\def\Symboo#1Ï{\kern.35mm\text{`}\kern.2mm#1\kern.35mm\hbox{\font\SweD =cmr10\SweD \char'047}}%%
\newsymbol\lleftarrow1311
\def\newskline#1 \par{\newline$\null$\kern#1mm}%
\def\vinskip#1#2 \par{\vskip#1mm$\null$\kern-6mm\hspace{#2mm}}
\def\inskipline#1#2 \par{\vskip#1mm$\null$\kern-4.25mm\kern#2mm}
\def\itemb#1_#2 {\item$\null$\kern-1.7mm\lower.75mm\hbox{\font\SweD =cmr6\SweD #2}\kern#1mm\kern3.2mm} % use \itemb2_ax ggg to get \item ggg where _{6 point roman ax} is put immediately after (itemnumber) and where an extra space of 2mm:s is put between (itemnumber)_{ax} and ggg
\newenvironment{myLeftskip}[3]{\leftskip#1mm\parindent0mm\addtolength{\parskip}{#2mm}\addtolength{\baselineskip}{#3mm}\def\newfline \par{\newline$\null\hfill$}}{\par}

\def\RunMyHead#1#2#3#4{%
 \headline{\ifnum\pageno=\firstpage\hfil%
           \else{\ifodd\pageno{\rp#3\phantom\folio\hfil#4\hfil\phantom{#3}\folio}%
                 \else{\rp\folio\phantom{#2}\hfil#1\hfil\phantom\folio#2}%
                 \fi}%
           \fi}%
 \footline{\ifnum\pageno=\firstpage\hfil{\rp[\,\folio\,]}\hfil%
           \else\hfil%
           \fi}%
}% 
\def\bulgin{\noindent$\bullet$ \ \kern.1mm} % may be used at the beginning of new paragraph to indicate the start of a new long step
\def\bulgen{\noindent\kern-1.5mm$\bullet$\kern3.95mm} % may be used at the beginning of new paragraph to indicate the start of a new long step
\def\Fsubhead#1\par{\vskip2mm\noindent\hfill\hbox{\font\Å=cmssi12\Å#1}\hfill$\null$\nopagebreak\vskip3mm\noindent} % use this for `Introduction and ...´ to get it without any numbering
\def\Sfubhead#1 #2\par{\vskip2mm\noindent\kern5mm\hbox{\font\Å=cmssbx11\Å#1\font\Å=cmbx12\Å.}\hfill\hbox{\font\Å=cmssi12\Å#2}\hfill\phantom{\hbox{\font\Å=cmssbx11\Å#1\font\Å=cmbx12\Å.}}\kern5mm$\null$\nopagebreak\vskip3mm\noindent} %
\def\Ssubhead#1 #2\par{\vskip6mm\noindent\kern5mm\hbox{\font\Å=cmssbx11\Å#1\font\Å=cmbx12\Å.}\hfill\hbox{\font\Å=cmssi12\Å#2}\hfill\phantom{\hbox{\font\Å=cmssbx11\Å#1\font\Å=cmbx12\Å.}}\kern5mm$\null$\nopagebreak\vskip3mm\noindent} %
\def\Binsubsubhead#1#2\par{\vskip4mm{\font\SweD=cmb11\SweD#1.}\hskip5mm{\font\SweD =cmss11\SweD #2}\nopagebreak\vskip2mm\nopagebreak\noindent}%
\def\subhead#1\par#2\par{\vskip4mm\smallbreak\null\smallskip\vbox{\noindent\bbf#1\hfill\kern1.5mm#2\hfill\phantom{#1}\vskip2.5mm\nopagebreak}\nopagebreak\noindent}
\def\subheadd#1\par#2\par#3\par{\vskip4mm\smallbreak\null\smallskip\vbox{\noindent\bbf#1\hfill#2\hfill\phantom{#1}\vskip1.5mm\centerline{#3}\vskip2.5mm\nopagebreak}\nopagebreak\noindent}
\def\insubsubhead#1\par{\vskip4mm$\null$\hskip2mm{\font\SweD =cmss10\SweD #1}\nopagebreak\vskip2mm\noindent}%
\def\insssubhead#1\par{\vskip2.5mm$\null$\hskip1mm{\font\SweD =cmss10\SweD #1}\nopagebreak\vskip1.2mm\noindent}%
\def\binSssubhead#1#2#3\par{\vskip2.5mm{\font\Å=cmss8 scaled\magstep1\Å#1\,}{\bf#2}\hskip3.5mm{\font\Å=cmss10\Å#3}\nopagebreak\vskip1.2mm\noindent}%
\def\nsubhead#1#2\par#3\par{\vskip2.5mm{\bf#1}\,\hbox{\font\Å=cmss8\Å#2}\hskip3.5mm{\font\Å=cmss10\Å#3}\nopagebreak\vskip1.2mm\noindent}%
\def\binsssubhead#1#2\par{\vskip2.5mm{\bf#1}\hskip3.5mm{\font\SweD =cmss10\SweD #2}\nopagebreak\vskip1.2mm\noindent}%
\def\Mfsubhead#1 #2\par{{\bf#1.}\hskip5mm{\font\SweD =cmss10\SweD #2}\nopagebreak\vskip2mm\nopagebreak\noindent}%
\def\Mysubhead#1 #2\par{\vskip4mm{\bf#1.}\hskip5mm{\font\SweD =cmss10\SweD #2}\nopagebreak\vskip2mm\nopagebreak\noindent}%
\def\binsubsubhead#1#2\par{\vskip4mm{\bf#1.}\hskip5mm{\font\SweD =cmss10\SweD #2}\nopagebreak\vskip2mm\nopagebreak\noindent}%
\def\subhead#1\par#2\par{\vskip4mm\smallbreak\null\smallskip\vbox{\noindent\bbf#1\hfill\kern1.5mm#2\hfill\phantom{#1}\vskip2.5mm\nopagebreak}\nopagebreak\noindent\rm}
\renewcommand\SS{\raise.15mm\hbox{\font\SweD =cmbsy9\SweD \char'170}\kern.4mm} % write \SS1 to get õ1
\def\wave{\hbox{\font\sweD =cmsy10\sweD \hbox{\char'164}\kern-2.35mm\hbox{\char'165}\kern.55mm}}
\def\wavee{\hbox{\font\sweD =cmsy8\sweD \hbox{\char'164}\kern-2.0mm\hbox{\char'165}\kern.4mm}} % eightpoint wave
\def\barmj{\kern.25mm\bar{\hbox{\font\SweD =cmr10\SweD \char'021}}\kern.4mm}

\def\fssit{\font\SweD =cmssi10\SweD }
\def\sigrd{\sigma\kern-.2mm\lower.7mm\hbox{\font\SweD =cmr6\SweD r\font\SweD =cmr5\SweD d}\kern.6mm}
\def\ssigrd{\sigma\kern-.2mm\lower.7mm\hbox{\font\SweD =cmr6\SweD r\font\SweD =cmr5\SweD d}\kern-1.7mm\raise1.25mm\hbox{\font\SweD =cmr6\SweD 2}\kern1mm}
\def\sssigrd{\sigma\kern-.2mm\lower.7mm\hbox{\font\SweD =cmr6\SweD r\font\SweD =cmr5\SweD d}\kern-1.7mm\raise1.25mm\hbox{\font\SweD =cmr6\SweD 3}\kern1mm}
\def\sigrdu^#1{\sigma\kern-.2mm\lower.7mm\hbox{\font\SweD =cmr6\SweD r\font\SweD =cmr5\SweD d}\kern-1.7mm\raise1.25mm\hbox{\font\SweD =cmr6\SweD #1}\kern1mm}
\def\taurd{\tau\kern-.4mm\lower.7mm\hbox{\font\SweD =cmr6\SweD r\font\SweD =cmr5\SweD d}\kern.6mm}
\def\ttaurd{\tau\kern-.4mm\lower.7mm\hbox{\font\SweD =cmr6\SweD r\font\SweD =cmr5\SweD d}\kern-1.7mm\raise1.25mm\hbox{\font\SweD =cmr6\SweD 2}\kern1mm}
\def\tsigrd{\tau\sigma\kern-.2mm\lower.7mm\hbox{\font\SweD =cmr6\SweD r\font\SweD =cmr5\SweD d}\kern.6mm}
\def\tauRe{\tau{_{\kern-0.6mm}}_{\hbox{\font\SweD =cmmi5\SweD I\!\!R}}} % \tauRe = \tau_{IR} = the natural topology of the real line
\def\bartauRe{\bar\tau{_{\kern-0.6mm}}_{\hbox{\font\SweD =cmmi5\SweD I\!\!R}}} % \bartauRe = \bar\tau_{IR} = the natural topology of the extended real line
\def\tauR#1{\tau_{_{I\!\!R}}\kern-1.5mm^{#1}}
\def\RN{I\!\!R\kern.3mm^{\hbox{\font\SweD =cmmi6\SweD N}}} % Re^N
\def\QTN{Q\kern.1mm_{\lower.2mm\hbox{\font\SweD =cmmi6\SweD T}}^{\kern.2mm\hbox{\font\SweD =cmmi6\SweD N}}} % Q_T^N
\def\lemod_#1(#2){\kern.2mm\raise1.3mm\hbox{\font\Å=cmsy5\Å\char'026}\kern-1.65mm\lower.8mm\hbox{\font\Å=cmr5\Å#1}\kern.4mm(\kern.15mm\boldsymbol{#2}\kern.37mm)} % use e.g. as in $\lemod_{TVS}(K)$ or $\lemod_{l.c.s}(\tfbbR)$
\def\preLCS(#1){\kern.2mm\raise1.3mm\hbox{\font\Å=cmsy5\Å\char'026}\kern-1.65mm\lower.8mm\hbox{\font\Å=cmr5\ÅLCS}\kern.7mm(\kern.15mm\boldsymbol{#1}\kern.37mm)} % use e.g. as in $\preLCS(K)$ or $\preLCS(\tfbbR)$
\def\preLCSv(#1)-{\kern.2mm\raise1.3mm\hbox{\font\Å=cmsy5\Å\char'026}\kern-1.65mm\lower.8mm\hbox{\font\Å=cmr5\ÅLCS}\kern.7mm(\kern.15mm\boldsymbol{#1}\kern.37mm)\text{\kern.85mm-\kern0.2mm}} % use e.g. as in $\leLCSv(K)-\sup\,\Cal E$ , $\leLCSv(\tfbbR)-\sup\,\Cal E$
\def\preTVS(#1){\kern.2mm\raise1.3mm\hbox{\font\Å=cmsy5\Å\char'026}\kern-1.65mm\lower.8mm\hbox{\font\Å=cmr5\ÅTVS}\kern.7mm(\kern.15mm\boldsymbol{#1}\kern.37mm)} % use e.g. as in $\preLCS(K)$ or $\preLCS(\tfbbR)$
\def\preTVSv(#1)-{\kern.2mm\raise1.3mm\hbox{\font\Å=cmsy5\Å\char'026}\kern-1.65mm\lower.8mm\hbox{\font\Å=cmr5\ÅTVS}\kern.7mm(\kern.15mm\boldsymbol{#1}\kern.37mm)\text{\kern.85mm-\kern0.2mm}} % use e.g. as in $\leLCSv(K)-\sup\,\Cal E$ , $\leLCSv(\tfbbR)-\sup\,\Cal E$
\def\leLCSv(#1)-{\kern.2mm\raise1.3mm\hbox{\font\SweD =cmsy5\SweD \char'024}\kern-1.65mm\lower.8mm\hbox{\font\SweD =cmr5\SweD LCS}\kern.4mm(\boldsymbol{#1}\kern.37mm)\text{\,-\kern0.15mm}} % use e.g. as in $\leLCSv(K)-\sup\,\Cal E$
\def\lleLCS(#1)-{\hbox{\font\SweD =cmsy10\SweD \char'024}\kern.3mm\lower.62mm\hbox{\font\SweD =cmr5\SweD LCS}\kern.4mm(\kern0.15mm#1\kern0.37mm)\text{\,-\kern0.15mm}} % use e.g. as in $\lleLCS(\snn\bosy K)-\sup\,\Cal H$ or $\lleLCS(\tfbbR)-\sup\,\Cal H$
\def\leLCSr{\hbox{\font\Å=cmsy10\Å\char'024}\kern.3mm\lower.62mm\hbox{\font\Å=cmr5\ÅLCS}\kern.4mm}
\def\leLCS-{\hbox{\font\Å=cmsy10\Å\char'024}\kern.3mm\lower.62mm\hbox{\font\Å=cmr5\ÅLCS}\kern.6mm\text{-}\kern.35mm}
\def\vspreceq{\kern1.4mm\raise1.7mm\hbox{\font\SweD =cmr5\SweD v}\kern-1.1mm\lower.12mm\hbox{\font\SweD =cmr5\SweD s}\kern-1.5mm\hbox{\font\SweD =cmsy10\SweD \char'026}\kern1mm} % $X\vspreceq Y$ means that Y is a vector substruture of X
\def\tvpreceq{\kern1.4mm\raise1.7mm\hbox{\font\SweD =cmr5\SweD v}\kern-1.2mm\lower.25mm\hbox{\font\SweD =cmr5\SweD t}\kern-1.4mm\hbox{\font\SweD =cmsy10\SweD \char'026}\kern1mm} % E tvpreceq F means the same as the earlier " E le F " for topological vector space E,F
\def\Centerline#1\par#2\par#3{\noindent#1\phantom{#3}\hfill#2\hfill\phantom{#1}#3}

\newcommand\Lone{\hbox{\font\SweD=cmmi14\SweD L}\kern.4mm\raise2.3mm\hbox{\font\SweD=cmr9\SweD 1}\kern.3mm} % gives $L^1$ for the title
\newcommand\Diamo{\kern.25mm\diamond\kern.25mm}
\newcommand\norm[1]{|\kern-.4mm|\kern.5mm#1\kern.5mm|\kern-.4mm|} %
\newcommand\dlty{\lower.7mm\hbox{\font\SweD=cmr5\SweD dl:}\lower.4mm\hbox{\font\SweD=cmr5\SweD y\kern1mm}} % for example $\sigma\dlty b$ = the weak space for a duality b

\newcommand\catmvp{\bold{mvp\kern.9mm}} % $\catmvp(\bosy K\ssp) = $ the category of mv-pairs
\newcommand\catTLG{\kern.5mm\underline{\kern-.5mm\roman{TLG}\kern-.4mm}\kern1mm} % $\catTLG(\ssp\bold K\ssp) = $ the category of Hausdorff topological $\bold K$--linear groups

\def\suptext#1#2{\raise1.65mm\hbox{\font\SweD=cmr5\SweD #1}\kern-.1mm\hskip.#2mm} % \suptext{uv}zX = ^{uv}X with an additional space of .z mm:s before `C'
\def\supbb#1pre#2{\raise1.25mm\hbox{\font\SweD=msbm5\SweD #1}\kern.2mm\raise1.65mm\hbox{\font\SweD=cmr5\SweD pre}\kern-.1mm\hskip.#2mm}
\def\flbb_#1{\kern.3mm\lower.7mm\hbox{\font\SweD=msbm5\SweD #1}\kern.15mm} % tuottaa _{bb #1}
\def\spX#1^#2#3#4#5_#6#7^#8#9{\lower.#1mm\hbox{\raise1.65mm\hbox{\font\SweD=cmr5\SweD #2}\kern-.1mm\hskip.#3mm}#4_{\kern0.15mm\kern0.#5mm\hbox{\font\SweD=cmr5\SweD #6}\kern0.15mm}^{\kern.7mm\kern0.#7mm#8\kern0.#9mm}} % use e.g. as in $\spX0^{vc}2C0_{bd}0^{}0(\vcal Q\sp)$ , $\spX0^{pre}2C0_{H\"}0^\delta0(\sp Q\ssp,\spp\varPi\ssp)$

\def\ssrfnde#1#2#3 #4_#5{[\kern.6mm\kern.#1mm{#4}\kern.#2mm\kern.6mm]\kern.2mm\raise1.85mm\hbox{\font\SweD=cmr5\SweD e}\kern-.9mm\raise.#3mm\hbox{\lower.35mm\hbox{$_{#5}$}}} % use e.g. as in $\ssrfnde000 x_\varXi$ or $\ssrfnde222 f_{\aars\varXi_1}$
\def\rweak#1-{\raise1.5mm\hbox{\font\SweD=cmr5\SweD weak\,}#1\text{\kern.8mm-}} % use e.g. as in $\rweak\varPi-\int_Afd\mu$ , an element of $\varPi'^*_{\beta\sigma}$ when scalarly integrable
\def\rPettis#1-{\raise1.5mm\hbox{\font\SweD=cmr5\SweD Pettis\,}#1\text{\kern.8mm-}} % use e.g. as in $\rPettis\varPi-\int_Afd\mu$
\newcommand\meastss[2]{\raise1.65mm\hbox{\kern.4mm\kern.#1mm\font\Å=cmsy5\Å\char'003\kern-.2mm\font\Å=cmss5\Åm\kern-.2mme\kern-.2mma\kern.#2mm\kern.3mm}} % use e.g. as in $q\meastss33\mu$ to get the pull-back measure of $\mu$ by $q$
\newcommand\otimea[1]{{\otimes}\lower.8mm\hbox{\font\Å=cmss5\Åm\kern-.2mme\kern-.2mma\kern.#1mm\kern.4mm}} % use e.g. as in $\otimea7\bosy\mu$ , $\otimea0(...$ to get the product measure of a family of probability measures
\newcommand\rmMo{\roman M\kern.3mm\lower.8mm\hbox{\font\SweD=cmr7\SweD o}\kern.4mm} % 
\newcommand\meaMsi{\roman M\kern.3mm\lower.73mm\hbox{\font\SweD=cmr6\SweD s}\lower.77mm\hbox{\font\SweD=cmr5\SweD i}\kern.4mm} %
\newcommand\meaMsc{\roman M\kern.3mm\lower.73mm\hbox{\font\SweD=cmr6\SweD sc}\kern.4mm} %
\newcommand\meafaMsi{\raise1.22mm\hbox{\font\SweD=cmr5\SweD f}\kern.2mm\raise1.67mm\hbox{\font\SweD=cmr5\SweD a}\kern.2mm\roman M\kern.3mm\lower.73mm\hbox{\font\SweD=cmr6\SweD s}\lower.77mm\hbox{\font\SweD=cmr5\SweD i}\kern.4mm} %
\newcommand\meafaMsc{\raise1.22mm\hbox{\font\SweD=cmr5\SweD f}\kern.2mm\raise1.67mm\hbox{\font\SweD=cmr5\SweD a}\kern.2mm\roman M\kern.3mm\lower.73mm\hbox{\font\SweD=cmr6\SweD sc}\kern.4mm} %
\newcommand\rmmd{\roman m\lower.75mm\hbox{\font\SweD=cmr5\SweD d}} % the function C owns z mapsto |z|/(1+|z|)
\newcommand\Abs{\hbox{\font\SweD=cmmi8\SweD A}\kern.3mm\lower.75mm\hbox{\font\SweD=cmr5\SweD b}\lower.72mm\hbox{\font\SweD=cmr6\SweD s}} % the absolute value function: C owns z mapsto |z|
\def\aabs#1#2^#3{\hbox{\font\SweD=cmr12\SweD a}\kern.2mm\lower.9mm\hbox{\font\SweD=cmr5\SweD b\font\SweD=cmr6\SweD s}\kern-2mm\raise.#2pt\hbox{$\kern.#1pt{}^{#3}$}} % the function C owns z mapsto |z|^p , use e.g. as in $\aabs99^p\circss01\Nu\circss10 x$ , $\aabs00^1\circss01\Nu\circss10 x$ , $\aabs40^t\circss01\Nu\circss10 x$ , $\aabs29^q\circss01\Nu\circss10 x$ , $\aabs00^{\erm s}\circss01\Nu\circss10 x$
\def\paw^#1{\roman p\lower.75mm\hbox{\font\SweD=cmr5\SweD aw}\kern-3.5mm{_{\hphantom{\hbox{\font\SweD=cmr5\SweD aw}}}^{\kern1.8mm{#1}}}} % the function C owns z mapsto |z|^p
\def\Abrs#1#2^#3{\hbox{\font\Å=cmr12\Åa}\kern.2mm\lower.77mm\hbox{\font\Å=cmss5\Åb}\lower.7mm\hbox{\font\Å=cmss6\Ås}\kern-1.4mm\kern.#2mm\raise.#1mm\hbox{$_{\kern.95mm}^{#3}$}} % the function z mapsto |z|^p , use e.g. as in $\Abrs33^p$ , $\Abrs03^2$

\def\gFil{\raise1.6mm\hbox{\font\SweD=cmr5\SweD g}\mathcal F\!\subtext{il}} % \gFil(eCal F,Omega) = the filter on Omega generated by a filter base eCal F
\def\Fils{\mathcal F\!\raise.2mm\hbox{$_{_{\kern0.15mm\roman{il:s}}}$}\kern.3mm} % \Fils\Omega = the set of filters on Omega

\newcommand\lliminf[1]{\liminf\kern.3mm\kern.#1mm}
\renewcommand\Pows{\lower.14mm\hbox{\font\Å=eusm9 scaled\magstep1\ÅP}\kern-.5mm\lower.7mm\hbox{\font\Å=cmss6\Ås}\kern.8mm} % the power class\renewcommand\Pows{\lower.14mm\hbox{\font\Å=eusm9 scaled\magstep1\ÅP}\kern-.5mm\lower.7mm\hbox{\font\Å=cmss6\Ås}\kern.8mm} % the power class, use e.g. as in $\Pows A$ , $\Pows I$ , $\Pows J$ , $\Pows\sp\Omega$ , $\Pows\vecs E$
\renewcommand\vecs{\upsilon\kern-0.1mm\lower.7mm\hbox{\font\Å=cmssi6\Ås}\kern.8mm} % 
\renewcommand\svecs{\upsilon\kern-.1mm\lower.45mm\hbox{\font\Å=cmssi5\Ås}\kern.6mm} % the above for sub- and superscripts, use e.g. as in $E_{\ssp/\ssp\svecs F}$
\renewcommand\vecss{\hbox{\font\Å=cmitt10\Åv}\kern.1mm\lower.7mm\hbox{\font\Å=cmssi6\Ås}\kern.8mm}  % underlying set of a vector space / vector (= linear) structure, use e.g. as in $\vecss X$ , $\vecss Y$ , $\vecss Z$
\renewcommand\svecss{\hbox{\font\Å=cmitt7\Åv}\lower.45mm\hbox{\font\Å=cmssi5\Ås}\kern.6mm} % the above for sub- and superscripts

\newcommand\upCth{\raise1.3mm\hbox{\font\Å=cmr5\ÅC\kern-.25mmt\kern-.35mm h}} % use e.g. as in $\upCth\mu$ to get the Caratheodory completion of $\mu$
\newcommand\trN[2]{\kern.4mm\kern.#1mm|\ssn|\ssn|\kern.#2mm\kern.4mm} % gives ||| with spaces before and after
\newcommand\subnu{\kern.35mm\lower.86mm\hbox{\font\Å=cmmi7\Å\char'027}} % gives $_\nu$ lowered suitably
\newcommand\trNu[1]{\kern.4mm\kern.#1mm|\ssn|\ssn|\kern.35mm\lower.86mm\hbox{\font\Å=cmmi7\Å\char'027}} % gives |||_\nu with space before
\newcommand\trNun[1]{\kern.4mm\kern.#1mm|\ssn|\ssn|\kern.35mm\lower.6mm\hbox{\font\Å=cmmi7\Å\char'027}\lower1.15mm\hbox{\font\Å=cmr5\Å1}} % gives |||_{\nu_1} with space before
\newcommand\sevib[1]{\hbox{\font\x=cmmib8\x#1}} % 
\newcommand\CltaurdvPidualbeta{\Cl_taurd{(\vPi_{\hbox{\font\Å=cmssi5\Å\char'031}}^{\kern.35mm\prime})\ssp}} % 

\def\vPettis#1int_#2{\kern.5mm\kern.#1mm\text{-}\kern.7mm\raise1.75mm\hbox{\font\=cmr5\Pettis}\kern-.8mm\int_{\,#2}} %

\newcommand\sNorF{\sNor{\fivemath F}}
\newcommand\sNorFp{\sNor{\fivemath F}\LHB{.65}{\sp_{^\prime}}}
\def\Step#1.#2 {\vskip.5mm\noindent$\null$\kern-3mm\hbox{\font\Å=cmss10\ÅStep #1\kern0.15mm\kern-#2pt. }}

\def\tcbbR_#1{\raise1.415mm\hbox{\font\Å=cmss5\Åt}\kern-.15mm\raise1.62mm\hbox{\font\Å=cmss5\Åc}\kern-.15mm\mathbb R\raise.3mm\hbox{$^*$}\kern-1.4mm\raise.52mm\hbox{$_{_{#1}}$}} % 
\newcommand\tcovbbRplus{\raise1.415mm\hbox{\font\Å=cmss5\Åt}\kern-.15mm\raise1.62mm\hbox{\font\Å=cmss5\Åc}\kern-.15mm\ovbbR\kern.2mm\raise.52mm\hbox{$_{_+}$}} % 

\def\BCSps#1(#2){\roman{BCS}\kern0.6mm(\kern.#1pt\boldsymbol#2\kern0.37mm)}
\def\LLw^#1{L\lower.7mm\hbox{\kern-.5mm\font\x=cmr5\x w}\kern-1.5mm\raise.45mm\hbox{$^{\,#1\sp}$}} %

\begin{document}

\title[$\text{\sc Duality of Bochner spaces}$]%
             {Duality of Bochner spaces}

\author[S. Hiltunen]{Seppo\ I\. Hiltunen}
\address{Aalto University                                               \vskip0mm$\hspace{2mm}$
           Department of Mathematics and Systems Analysis               \vskip0mm$\hspace{2mm}$
           P.O.\ Box 11100                                              \vskip0mm$\hspace{2mm}$
           FI-00076 Aalto                                               \vskip0mm
         Finland}
\email{seppo.i.hiltunen\,@\,aalto.fi}



\begin{abstract}                                                     \renewcommand\plusinfty{\lower.82mm\hbox{$^+$}\infty}\def\AbsLrs#1#2^#3{\raise.95mm\hbox{\font\Å=cmtt5\Åm\font\Å=cmssq5\Åv}\kern-.3mm\lower.25mm\hbox{\font\å=cmitt10\åL}\kern0.25mm\kern0.#2mm\lower.#1mm\hbox{\raise.4mm\hbox{$^{#3}$}}\kern0.25mm}

We construct the generalized Lebesgue\,--\,Bochner spaces \math{
\AbsLrs03^p(\ssp\mu\,,\spp\vPi\ssp) } for positive measures \math{\mu} and for 
suitable real or complex topological vector spaces \math{\vPi} so that for \linebreak
                                                                         \œ$
1<p<\plusinfty\ssp$ and Banachable \math{\vPi} with separable topology the 
strong dual of the classical Bochner space \math{
\AbsLrs03^p(\ssp\mu\,,\spp\vPi\ssp) } becomes canonically represented by \mathss37{
\AbsLrs23^p %%% ^{^*}
           \LHB{.2}{\KN{.2}^{^*}}\sn
                     (\ssp\mu\,,\spp\vPi^{\ssp\prime}_\sigma\spp) }. Hence we 
need no separability assumption of the norm topology of the strong dual \math{
\vPi^{\ssp\prime}_  %\beta
                  {%%%\hbox{\font\Å=cmmi5\Å\char'014}
                   \hbox{\font\Å=cmssi5\Å\char'031}}} of \mathss30{\vPi}. For \linebreak
                                                                           \œ$
p=1\ssp$ and for suitably restricted positive measures \math{\mu} we even get 
a similar result without any separability of the norm topology of the target 
space \mathss30{\vPi}. For positive Radon measures on locally compact 
topological spaces these results are essentially contained on pages 
588\,--\,606 in R.\ E.\ Edwards' classical {\sl Functional Analysis\sp}.
  \end{abstract}

\maketitle

% ----------------------------------------------------------------------------
% ----------------------------------------------------------------------------

\Fsubhead            Introduction and some preliminaries

Our main objective in this article is the following

\begin{Atheorem}\label{main Th}

Let \œ$\, 1 \le p < \plusinfty ${\,\rm, }and let \œ$\,q = 
 (\ssp 1 - p\,^{\mminus 1}\sp\big){}^{\ssp\mminus 1} $ if \œ$\, p \not= 1 
${\ssp\rm, }and \œ$\,q=\plusinfty$ if \œ$\,p=1\ssp$. \hfill Further{\sp\rm, }%
let $\,\mu$ be a positive measure on $\,\Omega${\,\rm, }and with \œ$\,\bosy K
 \in\setRC$ let \linebreak
              \œ$\vPi\in\BaSps0(K)$ be such that {\,\rm(1)} or {\,\rm(2)} or 
{\,\rm(3)} or {\,\rm(4)} or {\,\rm(5)} or {\,\rm(6)} below holds when 
{\,\rm(\sp\erm D\sp)} means that $\,\mu$ is almost decomposable. \hfill Also 
let \œ$\, F = \mvLrs03^p(\ssp\mu\,,\spp\vPi\ssp) ${\KP1\rm, }and \œ$\, F\aar 1
 = {\ssn} $ \œ$\mvLrs03^q(\ssp\mu\,,\spp\vPi\dlsigss00\spp) $ if {\,\rm(1)} or 
{\,\rm(2)} or {\,\rm(5)} or {\,\rm(6)} below holds, \hfill otherwise letting \œ$\, 
F\aar 1 = {\ssn} $ \linebreak 
$ \mvsLrs03^q(\ssp\mu\,,\spp\vPi\dlsigss00\spp)
  \KP1 $. \ For \KP1 \vskip.5mm\centerline{$
\Iota = \seq{ \KP{1.2}
\vecs F\snn\times\mathbb C\capss31\{\,(\ssp\smb X\sp,\spp t\ssp) : 
 \aall{x\in\smb X\sp,\sp y\in\smb Y}\,
     t = \int_{\KP{1.1}\Omega\,}y\,.\KPt8 x\rmdss11\mu\KPt9\} : 
          \smb Y\in\vecs F\aar 1\, } $} \inskipline{.5}0

then $\,\Iota \in \Lis(\sp F\aar 1\sp,\spp F\dlbetss10\sp)$ holds. In addition $\,
F\aar 1=\mvLrs03^q(\ssp\mu\,,\spp\vPi\dlbetss01\sp)$ if {\,\rm(1)} or 
{\,\rm(5)} holds. {\rm \inskipline14

(1)} \ $p=1$ and {\,\rm(\sp\erm D\sp)} and $\,\vPi$ is reflexive{\ssp\rm, \inskipline{.5}4

(2)} \ $p=1$ and {\,\rm(\sp\erm D\sp)} and $\,\taurd\vPi$ is a separable 
       topology{\ssp\rm, \inskipline{.5}4

(3)} \ $p=1$ and {\,\rm(\sp\erm D\sp)} and $\,\taurd\mLrs42^1(\ssp\mu\sp)$ is 
       a separable topology{\ssp\rm, \inskipline{.5}4

(4)} \ $p=1$ and {\,\rm(\sp\erm D\sp)} and a choice function $\,c\in
              \Cal L\,(\ssp\mLrs23^\plusinftyy(\ssp\mu\sp)\,,\sp
               \lll^\plusinftyy\sp(\sp\Omega\sp)) $ exists{\ssp\rm, \inskipline{.5}4

(5)} \ $p\not=1$ and $\,\vPi$ is reflexive{\ssp\rm, \inskipline{.5}4

(6)} \ $p\not=1$ and $\,\taurd\vPi$ is a separable topology.
  \end{Atheorem}

The proof is given on pages \pageref{Sec D}\,--\,\pageref{endmpf} below. Here 
we first explain the notation appearing above, mentioning that we generally 
utilize the notational convention explained in \cite[pp.\ 4\,--\,8]{HiDim}\,, 
\cite[pp.\ 4\,--\,9]{SeBGN} and \cite[p.\ 1]{FKBGN}\,, and further to be 
  \q{polished} in \cite{Hif}\,.

Having \math{\bosy K\in\setRC} means that \math{\bosy K} is either the 
topological field of real numbers or that of the complex ones. The underlying 
sets of these fields are \math{\bbR} \nolinebreak and \nolinebreak \mathss38{
\mathbb C}, \linebreak 
            respectively. Then \math{\vPi\in\BaSps0(K)} means that \math{\vPi} 
is a \erm Banach{\sl able\sp}, i.e.\ a complete norm{\sl able\sp} real or 
complex topological vector space. Thus there is a compatible norm \math{\Nu} 
on the underlying vector space \math{\sigrd\vPi} such that \math{
(\ssp\sigrd\vPi\sp,\spp\Nu\ssp) } is a norm{\sl ed\ssp} Banach space. Being 
{\it compatible\ssp} here means that \math{
 \{\KPt8\Nu\invss44\image\spp[\KP{1.1} 0\,,\spp n^{\,\mminus 1}\ssp\big] : 
 n\in\rbb Z^+\ssp\big\} } is a filter base for the filter $\neiBoo\vPi$ of 
zero neighbourhoods. Above \math{\Nu\invss44\image\spp[\KP{1.1} 0\,,\spp 
 n^{\,\mminus 1}\ssp\big] } is the image of the closed interval \math{
[\KP{1.1} 0\,,\spp n^{\,\mminus 1}\ssp\big] } under the relational inverse \math{
\Nu\ssp\inve} of \mathss31{\Nu}. Here \linebreak
                                      $\Nu\ssp$ is a function \math{\vecs\vPi
 \to\lbb R_+} where \math{\vecs\vPi} is the underlying set of vectors of \mathss31{
\vPi}. For \linebreak
           $\xi\in\vecs\vPi\ssp$ one may sometimes write \math{\|\,\xi\,\| } 
for the value \math{\Nu\fvalss11\xi} of \math{\Nu} at \mathss34{\xi}.

For \mathss37{\vPi\aar 1\in\tvsps0(K)}, \,i.e.\ having \math{\vPi\aar 1} a 
real or complex topological vector space with possibly non\ssp-\sp Hausdorff 
topology \mathss34{\taurd\vPi\aar 1}, \,the exact construction of the space 
$E=
\mvLrs03^p(\ssp\mu\,,\spp\vPi\aar 1)$ is given in 
Constructions \ref{defi $L^p$}\,(\ref{simpL^p}) on page \pageref{simpL^p} 
below. Here we informally explain the basic ideas 
under the additional assumption that 
$\vPi\aar 1$ is {\sl suitable\sp} in the sense of 
Definitions \ref{df suit} on page \pageref{df suit} below. Then 
it suffices to consider one fixed 
{\sl dominating\sp} norm $\Nu$ for $\vPi\aar 1\ssp$.

We consider functions 
$x:\Omega\to\vecs\vPi$ such that on every set $A$ of finite measure, i.e.\ 
for $A\in\mu\invss44\image\spp\lbb R_+$ it holds that outside some set $N$ of 
measure zero, i.e.\ with $N\in\mu\invss44\image\snn\{\ssp 0\ssp\}$ we have $x$ 
a pointwise limit of a sequence of simple functions, with convergence in the 
sense of the topology $\taurd\vPi$. In the case $p<\plusinfty$ we then take 
the subset of those $x$ such that the generally nonmeasurable 
function \vskip.3mm\centerline{$
\Abrs33^p\circss00\Nu\circss00 x:\Omega\owns\eta\mapsto
(\ssp\Nu\circ x\fvalue\eta\ssp)\KP1\RHB{.3}{^p}\in\lbb R_+ $} \inskipline{.3}0

is dominated by some integrable function 
$\Alf:\Omega\to[\KP{1.1} 0\,,\plusinfty\KPt9] \,$. With the pointwise 
vector operations from $\sigrd\vPi$ the set of these $x$ becomes a 
vector substructure $X$ of $\sigrd\vPi\expnota^\ssp\Omega\ssp]_{vs} \,$. 
Then we take $E=(\sp X\sp/\vsquotient N\aar 0\,,\spp\scrmt T\,)$ when $
N\aar 0$ is the set of all $x\in\vecss X$ such that for all 
$u\in\Cal L\,(\sp\vPi\sp,\spp\bosy K\ssp)$ and 
$A\in\mu\invss44\image\spp\lbb R_+$ we have 
$\int_{\,A\,}u\circ x\rmdss11\mu=0 \,$. Here 
$\Cal L\,(\sp\vPi\sp,\spp\bosy K\ssp)$ is the set of all 
continuous linear maps $\vPi\to\bosy K\ssp$. Furthermore, we take the 
topology $\scrmt T$ so that 
a filter of zero neighbourhoods is formed by the sets \vskip.3mm\centerline{$
\vecss(\sp X\sp/\vsquotient N\aar 0\sp)
\capss31\{\,\smb X:
\eexi{x\in\smb X}\,
\upint\ssp \Abrs33^p\circss00\Nu\circss00 x\rmdss41\mu < \eps\KPt9\} $} \inskipline{.3}0

for $\eps\in\rbb R^+$. Here we have the upper integral 
of the 
%% possibly 
not\ssp-\sp necessarily measurable function 
$\Abrs33^p\circss00\Nu\circss00 x$ that is defined 
as the infimum of the set of all 
$\int_{\KP{1.1}\Omega\,}\Alf\rmdss11\mu$ with $\Alf$ as above. 
For $p=\plusinfty$ we take the \q{obvious} modification.

The space \math{E\ar 1=
\mvsLrs03^p(\ssp\mu\,,\spp\vPi\aar 1)
} is constructed otherwise similarly except that 
we instead require the functions $x$ to be such that 
$u\circ x\KP1|\KP1(\sp A\sp\setminus N\ssp)$ is measurable, that is, 
for every 
$A\in\mu\invss44\image\spp\lbb R_+$ we require existence of 
some $N\in\mu\invss44\image\snn\{\ssp 0\ssp\}$ such that for all 
$u\in\Cal L\,(\sp\vPi\sp,\spp\bosy K\ssp)$ it holds that 
$u\circ x\KP1|\KP1(\sp A\sp\setminus N\ssp)$ is a measurable real or 
complex valued function on $A\sp\setminus N$. Then every vector of $E$ is 
contained in some vector of \mathss34{E\ar 1}, \,but we need not have \mathss34{
\vecs E\inc\vecs E\ar 1 }. Note above that \mathss39{ \vecs E = 
 \vecss(\sp X\sp/\vsquotient N\aar 0\spp) }, \,and that \mathss03{
X\sp/\vsquotient N\aar 0 } is the quotient vector space structure of \math{X} 
by the linear subspace \mathss36{N\aar 0}.

Having now informally explained the general construction of our generalized 
Bochner spaces, we note that if \math{\vPi} is \erm Banachable, then \math{
\vPi\dualsigma0} is its weak dual space, and that \math{
(\ssp\sigrd\vPi\sp,\spp\Nu\ssp) } is a Banach space for any compatible norm \math{
\Nu} for \mathss31{\vPi}. At least for \rsigma6finite positive measures \math{
\mu} then \math{
(\ssp\sigrd\mvLrs03^p(\ssp\mu\,,\spp\vPi\ssp)\,,\spp\Nu\aR 2\sp) } is a 
classical Bochner space when \math{\Nu\aar 2} is defined by \math{ \smb X
 \mapsto\big(\sp\int_{\KP{1.1}\Omega\,}\Abrs33^p\circss00\Nu\circss00 x
  \rmdss11\mu\ssp)\KP1\RHB{.2}{^p}{^{^{-1}}} } for any \mathss31{x\in\smb X}. 
The appearing \math{\Iota} is the function \math{ \vecs F\aar 1 \owns \smb Y 
 \mapsto\Iota\fvalss02\smb Y}  with \math{\Iota\fvalss02\smb Y } given by \vskip.3mm\centerline{$
\vecs F \owns \smb X \mapsto \int_{\KP{1.1}\Omega\,}y\,.\KPt8 x\rmdss11\mu $} \inskipline{.3}0

for any \math{x\in\smb X} and \mathss30{y\in\smb Y}. Here \math{y\,.\KPt8 x} 
is the function \vskip.3mm\centerline{$
 \Omega\owns\eta\mapsto y\fvalue\eta\fvalue(\ssp x\fvalue\eta\ssp) =
 (\ssp y\fvalue\eta\ssp)\fvalue(\ssp x\fvalue\eta\ssp) \in \vecs\bosy K \in 
   \{\ssbb97 R,\spp\ssbb08 C\} \KP1 $.} \inskipline{.3}0

The message of Theorem \nfss A\,\ref{main Th} is then that \math{ \Iota \in 
 \Lis(\sp F\aar 1\sp,\spp F\dlbetss10\sp) } holds, i.e.\ that \math{\Iota} is 
a linear homeomorphism \math{F\aar 1\to F\dlbetss10} where \math{F\dlbetss10} 
is the normable, hence \erm Banachable strong dual of \math{F} with \mathss38{
\vecs(\sp F\dlbetss10\sp)=\Cal L\,(\sp F\spp,\spp\bosy K\ssp) }. \hfill For 
the spaces appearing in (3) and \linebreak
                                (4) note that we define \mathss39{
\mLrs03^p(\ssp\mu\sp) = \mvLrs03^p(\ssp\mu\,,\sn\tfbbR\ssp) }. Below note that 
in the usual man- ner we have \math{ p\,^* = 
 (\ssp 1 - p\,^{\mminus 1}\sp\big){}^{\ssp\mminus 1} } for \math{
1 < p < \plusinfty} and \math{p\,^*=\plusinfty} for \mathss34{p=1}, \,and also $
p\,^*=1\ssp$ in the case where \math{p=\plusinfty} holds.

The last part in condition (4) means that 
there is a continuous linear map 
$c:\mLrs23^\plusinftyy(\ssp\mu\ssp)\to
               \lll^\plusinftyy\sp(\sp\Omega\sp)$ 
such that $c\sp\fvalue\snn\smb X\in\smb X$ holds for all $
\smb X\in\vecs\mLrs23^\plusinftyy(\ssp\mu\ssp) \, $. 
Continuity here being equivalent to the property that 
for some $\smb M\in\lbb R_+$ it holds that 
for $x\in\smb X\in\vecs
\mLrs23^\plusinftyy(\ssp\mu\ssp)$ and for all $
A\in\mu\invss44\image\spp\lbb R_+
$ there is $N\in\mu\invss44\image\snn\{\ssp 0\ssp\}$ such that 
$
|\KP1 c\sp\fvalue\snn\smb X\fvalue\eta\KP1|
\le\smb M\KP1|\KP1 x\fvalue\eta\KP1|$ holds for all $\eta\in
A\sp\setminus N$, our condition is 
weaker than the requirement (b) in 
\cite[Theorem 8.18.2\ssp, p.\ 588]{Edw} that \math{
\mLrs23^\plusinftyy(\ssp\mu\ssp) } can be \q{lifted}.

\begin{remarks}\label{Rem Rad-deco}

At first sight, it may seem that Theorem \nfss A\,\ref{main Th} is less 
general than the results contained in 
\cite[Theorems 8.18.2\ssp, 8.18.3\ssp, pp.\ 588\ssp, 590]{Edw} when \math{p=1} 
since in Edwards' presentation there is stated no assumption on any kind of 
\q{decompos- ability}. However, one should note that in \cite{Edw} one 
considers only positive measures that are {\sl positive \esl Radonian\sp} in 
the sense of \sp Definitions \ref{df top deco}\,(4) on page \pageref{df pos Radon}
below, and that by Proposition \ref{Propo top-deco} these are \q{automatically} 
almost decomposable. See also \cite[Proposition 4.14.9\ssp, p.\ 229]{Edw}\,.

We also remark the main ideas of the proof of \sp Theorem \nfss A\,\ref{main Th} 
are essentially, at least implicitly, contained in \cite[pp.\ 573\,--\,607]{Edw} 
although it is not quite straight- forward to see the exact details from the 
presentation there. 

Note that in \cite{Edw} positive measures are obtained from positive linear 
functionals in the vector spaces \math{
\sigrd C\sn\sbi{\rm c}\ssp(\ssp\scrmt T\,) } of compactly supported continuous 
functions for locally compact Hausdorff topologies \mathss30{\scrmt T}, \,%
  cf.\ \cite[4.3\ssp, pp.\ 177\,--\,179]{Edw}\,. Furthermore, in \cite{Edw} 
measurability of functions is defined by the Lusin property which is 
  meaningless for general measures.
  \end{remarks}

\begin{remark}

Using Theorem \nfss A\,\ref{main Th} one is able to prove 
\cite[5.22\ssp, p.\ 27]{Am97} in the more general case where only separability 
of the topology \math{\taurd\vPi} is required instead of having \math{
\taurd(\sp\vPi\dlbetss01\sp) } separable. Then for example in the case \math{
\vPi=\LLrs42^1(\ssbb44 I) } the strong dual of the Besov space \math{
\Besovrss600_q^{\emath s\sp,\,p}\sbig(2\yi N\ssbb67 R,\spp\vPi\ssp) } is seen 
to be canonically represented by \linebreak
                               \œ$
\Besovrss300_{q\sast}^{\mminus\emath s\sp,\,p\sast}\ssn \sbig(2\yi N\ssbb67 R,\spp
 \vPi\dlbetss01\sp) \ssp $ when \math{s\in\bbR} and \math{1\le p<\plusinfty } 
and \math{1\le q<\plusinfty} and \math{\smb N\in\bbN} hold. \linebreak
                                                            This is in 
constrast with the case of Bessel potential spaces where the strong dual \linebreak
                                                                         of \math{
\HBsrss606^{\emath s\sp,\,p}\sbig(2\yi N\ssbb67 R,\spp\vPi\ssp) } is 
represented only by \mathss38{\HBsrss300^{\mminus\emath s\sp,\,p\sast}
 \sbig(2\yi N\ssbb67 R,\spp\vPi\dlsigss00\spp) }. We hope to have the 
opportunity to give the details of the proof in a future publication.
  \end{remark}

We shortly review the {\fssit contents\ssp} which is organized according 
to the scheme: \vskip.5mm  {\newcommand\nrm[1]{$\null$\kern2.8mm{\font\Å=cmr9\Å#1\kern.4mm}}\newcommand\ntrm[2]{$\null$\kern2.8mm\ref{#1}\,{\font\Å=cmr8\Å#2\kern.4mm}}%
\parindent6mm \inskipline07

 1. Some special constructions \dotfill \ p.\KP{2.95} \pageref{Ss spec ctrs} \KP4 \inskipline07

 2. Suitable locally convex spaces \dotfill \ p.\ \pageref{Ss suit lcs} \KP4

A \ Measurability and integration \dotfill \ p.\ \pageref{Sec A} \KP4 \inskipline07

 1. Measurability of measure\ssp-\sp vector maps \dotfill \ p.\ \pageref{Ss C1} \KP4 \inskipline07

 2. Decomposable positive measures \dotfill \ p.\ \pageref{Ss decos} \KP4 \inskipline07

 3. Integration of scalar functions \dotfill \ p.\ \pageref{Ss int scal} \KP4 \inskipline07

 4. Pettis integration of vector functions \dotfill \ p.\ \pageref{Ss Pettis} \KP4

B \ Generalized Bochner spaces \dotfill \ p.\ \pageref{Sec B} \KP4

C \ Lifting and integral representations \dotfill \ p.\ \pageref{Sec C} \KP4 \inskipline07

 1. Dunford\,--\,Pettis property of $\,\mLrs42^1(\ssp\mu\ssp)$ \dotfill \ p.\ \pageref{Ss Dun-Pet} \KP4 \inskipline07

 2. Absolutely continuous vector measures \dotfill \ p.\ \pageref{Ss abs conti} \KP4

D \ Duality of Bochner spaces \dotfill \ p.\ \pageref{Sec D} \KP4

E \ Examples and open problems \dotfill \ p.\ \pageref{Sec E} \KP4 \inskipline0{9.2}

   References \dotfill \ p.\ \pageref{Sec Bib} \KP4

%%%% This ENDs the scope of "\nrm" !!!
    \par} \vskip.5mm

In subsection 1 of this introductory section we give some special 
constructions in order to be able to express certain matters concisely and 
precisely at the same time. In 2 we give the basic definitions associated with 
suitable spaces. We also establish some lemmas that are needed in the sequel.

In section A we present our approach to measurability and integration of 
scalar and vector valued functions, or put more precisely, 
{\sl mv\ssp-\sp map\ssp}s. These are triplets \math{
(\ssp x\,;\spp\mu\,,\spp\vPi\ssp) } where \math{\vPi} is a real or complex 
topological vector space and \math{\mu} is a positive measure on some set \math{
\Omega} and \math{x:\Omega\to\vecs\vPi } is a function.

In section B we first give the formal construction of our generalized 
Lebesgue\,--\,Bochner spaces of equivalence classes \math{\smb X} of 
measurable functions \math{x:\Omega\to\vecs\vPi } when a positive measure \math{
\mu} on \math{\Omega} is given. Then we prove several results associated with 
these spaces that are needed in the proof of our main theorem.

Section C contains several auxiliary results that are needed to prove that a 
given continuous linear functional \math{ \smb U \sn : 
 \mvLrs03^p(\ssp\mu\,,\spp\vPi\ssp) = F \to \bosy K } can be represented by 
some vector \math{\smb Y} of the space \math{F\aar 1} in the sense that for \math{
x\in\smb X\in\vecs F} and \math{y\in\smb Y} we have the equality \mathss36{
\smb U\sp\fvalue\snn\smb X = \int_{\KPp1.1\Omega\,}y\,.\KPt8 x\rmdss11\mu}.

At the beginning of section D we note how other assertions 
of \sp Theorem \nfss A\,\ref{main Th} except surjectivity of \math{\Iota} 
follow from results that have already been established in section B\ssp. Then 
we prove the surjectivity in Lemmas \nfss A\,\ref{LeA(1)}$\,,\ldots\KPt8
$\nfss A\,\ref{final lemma} separately in the cases (1)$\,,\ldots\KPt8$(6) 
with (5) and (6) being treated together in \nfss A\,\ref{final lemma}\ssp.

In section E we have collected examples to make more concrete some points of 
the general theory. We also present some related open problems. \vskip.5mm

Above we already indicated that \math{\tvsps0(K)} is the class of all 
topological vector spaces over \math{\bosy K} when \math{\bosy K} is a 
topological field. We put \vskip.3mm\centerline{$
\TVSps0(K) = \tvsps0(K)\capss31\{\, E : 
              \taurd E\ssp\text{ is a Hausdorff topology } \} \KP1$,} \inskipline{.3}0 

and we let \math{\LCSps0(K)} be the subclass of \math{\TVSps0(K)} formed by 
the locally convex spaces. For \math{E\in\tvsps0(K)} we have \math{\bouSet E} 
the set of all bounded sets in \mathss35{E}, also cal- led the {\sl von Neumann 
bornology\sp} of \mathss35{E}. For \math{E\,,\sp F\in\tvsps0(K) } we let \math{
E\tvpreceq F} mean \linebreak 
                   that the identity \math{\idv F} is a continuous linear map \mathss35{
  F\to E}.

If \math{E} is a real or complex topological vector space, then \math{
\Cal S\sbi{\sp\emath r\,}E } and \math{\BSnorm E} and \linebreak
                                                    \œ$\Bqnorm E\ssp$ are the 
sets of continuous \mathss35{r}--\,seminorms, bounded seminorms and bounded 
quasi\ssp-\ssp seminorms, respectively, the formal constructions being given 
in (1)$\,,\ldots\,$(3) below. We also put \math{ \SemiNor E = 
 \Cal S\LHB{.3}{_{\,1\,}} E} thus getting the set of continuous seminorms. 
Note the implication \math{ \Nu \in \Cal S\sbi{\sp\emath r\,}E \impss33 
 0 < r \le 1 } and that a quasi\ssp-\ssp seminorm \math{\Nu} being 
{\sl bounded\ssp} means that \math{ \sup \KP1 (\ssp\Nu\ssp\image\snn B\ssp) < 
 \plusinfty } holds for every \mathss35{B\in\bouSet E}. We generally have \mathss35{
\SemiNor E\inc\BSnorm E}, \,and the converse inclusion holds if \math{E} is 
  normable.

\begin{enumerate}\begin{myLeftskip}{-2}{.5}{.3}

\item \ $\Cal S\sbi{\sp\emath r\,}E = \sp^{\svecs E}\KP1\lbb R_+\snn\cap\ssp 
          \{\KPt8\Nu : 0 < r \le 1\ssp\text{ and }\ssp                    \label{df r-semin E}
           \aall{t\ssp,\sp x\ssp,\sp y\ssp,\sp z}\, $ \newskline{22}

  $[\KP{1.4}(\ssp t\ssp,\spp x\ssp,\spp y\ssp)\in\tsigrd E\impss33
   \Nu\fvalss30 y = |\,t\,|\KP1(\ssp\Nu\fvalss30 x\ssp)\KP{1.4}] \ssp$ and \newskline{22}

  $[\KP{1.4}(\ssp x\ssp,\spp y\ssp,\spp z\ssp)\in\ssigrd E\impss33
    (\ssp\Nu\fvalss30 z\ssp)\RHB{.2}{\KPt8^{\emath r}} \sn \le \sp 
    (\ssp\Nu\fvalss30 x\ssp)\RHB{.2}{\KPt8^{\emath r}\sn} + 
    (\ssp\Nu\fvalss30 y\ssp)\RHB{.2}{\KPt8^{\emath r}} \KPp1.4\big]$ \newfline

  and \math{\Nu} is continuous \mathss39{\taurd E\to\nsTbb_R \, \} }, \KP{10}

\item \ $\Bqnorm E = \sp^{\svecs E}\KP1\lbb R_+\snn\cap\ssp \{\KPt8\Nu : \label{defi bqnor E}
          \eexi{\smb A}\,\aall{t\ssp,\sp x\ssp,\sp y\ssp,\sp z}\,
      \smb A\in\rbb R^+\sp$ and \newskline{22}

  $[\KP{1.4}(\ssp t\ssp,\spp x\ssp,\spp y\ssp)\in\tsigrd E\impss33
   \Nu\fvalss30 y = |\,t\,|\KP1(\ssp\Nu\fvalss30 x\ssp)\KP{1.4}] \ssp$ and \newskline{22}

  $[\KP{1.4}(\ssp x\ssp,\spp y\ssp,\spp z\ssp)\in\ssigrd E\impss33       \label{defi bqnor E p}
   \Nu\fvalss30 z \le \smb A\KP1(\ssp\Nu\fvalss30 x + \Nu\fvalss30 y\ssp) 
    \KP{1.4}]$ \newfline

  and $\ssp \Nu\KPt7\images\spp\bouSet E\inc\bouSet\tfbbR \KP1 \} \KP1 $, \KP{10}

\item \ $\BSnorm E = \Bqnorm E\capss30\{\KPt8\Nu:\aall{x\ssp,\sp y\ssp,\sp z}\,   \label{defi bsnor E}
         (\ssp x\ssp,\spp y\ssp,\spp z\ssp)\in\ssigrd E$ \newfline

     $\impss03 \Nu\fvalss30 z \le \Nu\fvalss30 x + \Nu\fvalss30 y\KPt9\}\KP1$. \KP{10}

  \end{myLeftskip}\end{enumerate}

% ¤¤¤¤¤¤¤¤¤¤¤¤¤¤¤¤¤¤¤¤¤¤¤¤¤¤¤¤¤¤¤¤¤¤¤¤¤¤¤¤¤¤¤¤¤¤¤¤¤¤¤¤¤¤¤¤¤¤¤¤¤¤¤¤¤¤¤¤¤¤¤¤¤¤¤¤

\insubsubhead             Some special constructions                      \label{Ss spec ctrs}

We are working within a Kelley\,--\,Morse\,--\,G\"odel\,--\,Bernays\,--\,von 
Neumann type approach to set theory, like for example the one introduced in 
\cite[pp.\ 250\,--\,281]{Ky}\,. Then with \math{ x\ssp\yplus =
 x\cupss21\{\ssp x\ssp\} } putting \mathss38{ \bbNo = \ssp
 \bigcap\KPt8\{\,N\sn:\emptyset\in N\ssp\text{ and }\ssp\aall{k\in N}\,
 k\ssp\yplus\in N\KPt9\} }, \,we may call \math{\bbNo} the set of {\sl natural 
number\sp}s. It equals the set of finite cardinals, as well as the set of 
finite ordinals. Let \math{\infty=\bbNo} and \mathss38{ \bbN = 
 \bbNo\sn\setminus\{\ssp\emptyset\ssp\} }.

We assume that the set \math{\mathbb H} of {\sl quaternion\sp}s is constructed 
in a certain manner so that we have \math{\mathbb H\inc
 \ovbbR\ar 1\sn\times\ovbbR\ar 1\sn\times(\ssp\ovbbR\ar 1\sn\times\ovbbR\ar 1) } 
for some set \math{\ovbbR\ar 1} with \vskip.3mm\centerline{$
\ovbbR\ar 1 \inc \Pows(\ssp\Pows(\ssp\bbNo\snn\times\bbN\ssp)\times
                           \Pows(\ssp\bbNo\snn\times\bbN\ssp)) $} \inskipline{.3}0

where the {\sl power class\sp} \math{\Pows A} of \math{A} is defined in 
Definitions \ref{misc defs}\,(14) below. Then for some set \math{ 0\ar 1 \in 
 \ovbbR\ar 1} we have \math{\bbR \inc \ovbbR\ar 1\timesn\{\,0\ar 1\snn\}\snn
 \times\sbig(2\{\,0\ar 1\snn\}\timesn\{\,0\ar 1\snn\}\sp\sbig)0 } and \inskipline{.2}{41.7}

$\mathbb C \inc \ovbbR\ar 1\sn\times \ovbbR\ar 1\sn\times\sbig(2
               \{\,0\ar 1\snn\}\timesn\{\,0\ar 1\snn\}\sp\sbig)0 \KP1 $. \inskipline{.2}0

The definitions of the sets \math{\mathbb Z} and \math{\lbb Z_+} of 
{\sl integers\sp} and {\sl nonnegative\sp} integers, respectively, being given 
in \ref{misc defs}\,(5) and \ref{misc defs}\,(7) below, we have a bijection \math{
\bbNo\to\lbb Z_+} given by \mathss03{i\mapsto n=i\sp\ydot} with inverse \math{
n\mapsto i=n\sp\adot} and now for example \mathss30{ i\ssp\yplus\sp\ydot =
 (\ssp i + 1\sp\adot\spp)\spp\ydot = n + 1 } and \math{
 i\ssp\yplus\snn\yplus\sp\ydot = (\ssp i + 2\sp\adot\spp)\spp\ydot=n + 2 } and \mathss37{
i\ssp\yplus\sp\ydot\ssp^{\mminus 1} = (\ssp n + 1\ssp)\,^{\mminus 1} }. Also \math{
\emptyset = 0\sp\adot} holds.

Having \mathss38{ \ovbbR = [\sp\minusinfty\,,\plusinfty\KP1] = \{\,t : 
 \minusinfty\le t\le\plusinfty\KPt9\} }, \,we assume the formal definitions 
having been arranged so that for all \math{u\ssp,\sp v} we have \math{u\le v} 
if{}f \math{\minusinfty\le u\le v\le\plusinfty} or 
\math{u\ssp,\sp v} are functions with 
\math{u\cupss22 v\inc\dom v\times\ovbbR } and \mathss38{\dom u\inc
\{\,\eta:u\fvalue\eta\le v\fvalue\eta\KPt8\} }. Hence if \math{u} and \math{v} 
are extended real valued functions, then \math{u\le v} means that 
we have \math{\dom u\inc\dom v} and that 
\math{u\fvalue\eta\le v\fvalue\eta} holds for all 
\mathss34{\eta\in\dom u}. 
Furthermore \math{\emptyset\le v} is equivalent to having \math{v} a 
function with \mathss38{\rng v\inc\ovbbR}.

In order to specify some set theoretic notation already utilized above that 
also has largely been explained in \cite[pp.\ 4\,--\,8]{HiDim} and 
\cite[pp.\ 4\,--\,9]{SeBGN}\,, ending on page \pageref{END set th extract} 
below, we next present an extract from \cite{Hif}\,.

We assume that the {\sl intuitive class\sp} of all {\sl variable symbol\sp}s 
of our set theory is implicitly intuitively well\ssp-\ssp ordered so that it 
makes sense to speak of the first variable (\sp symbol\ssp) not possessing 
  some property.

\begin{def:al schemas}[set notation]\label{defi {F:...:P}}

Let \math{\mfrk F} be any term and \math{\mfrk P} a formula and $\afr x_1\sp,
\ldots\,\mfrk x\ssp\ai k\ssp,\sp\afr y_1\sp,\ldots\,\mfrk y\ssp\ai l\ssp$
distinct variable symbols such that \math{\afr x_1\sp,\ldots\,\mfrk x\ssp\ai k}
are precisely the variable symbols which have a free occurrence both in \math{
\mfrk F} and \math{\mfrak P} and are not in the list \math{\afr y_1\sp,\ldots\,
\mfrk y\ssp\ai l}. Also let \math{\mfrk x} be the first variable symbol not
occurring free in \math{\mfrk F} or \math{\mfrk P}. Then we let \math{
\{\,\mfrk F:\afr y_1\sp,\ldots\,\mfrk y\ssp\ai l\ssn:\mfrk P\,\} =
 \{\,\mfrk x:\exi{\afr x_1\sp,\ldots\,\mfrk x\ssp\ai k}\,\mfrk x=\mfrk F\ssp$
 and $\ssp\mfrk P\,\}\yxbtext{15}1b \sp}. \hfill In the case where \math{
\afr y_1\sp,\ldots\,\mfrk y\ssp\ai l} is an empty list, we further let \math{
\{\,\mfrk F:\mfrk P\,\} = \{\,\mfrk F:\ :\mfrk P\,\} \sp}.

The variable symbols which are free in the term \math{
\{\,\mfrk F:\afr y_1\sp,\ldots\,\mfrk y\ssp\ai l\ssn:\mfrk P\,\}} are (\ssp by
re- cursive definition\sp) exactly those which are free either in \math{
\mfrk F} or \math{\mfrk P}, and are not in the list \math{\afr x_1\sp,\ldots\,
 \mfrk x\ssp\ai k}. The free variables of \math{\{\,\mfrk F:\mfrk P\,\}} are
precisely those which are free in \math{\mfrk F} or \math{\mfrk P} but not in
both of them.
  \end{def:al schemas}

The above schemata, which we introduced to overcome the notational problem
presented in \cite[4 Notes, pp.\ 5\,--\,6]{Ky}\ssp, only provide reduction of
\PouN$\ssp\{\KPt8\mfrk F:\afr y_1\sp,\ldots\,\mfrak y\ssp\ai l\ssn:\mfrak P\,\}$
and \math{\{\KPt8\mfrk F:\mfrak P\,\} } to \math{
\{\KPt8\mfrk x:\mfrk Q\,\}\yxbtext{15}1b \sp}. In order to be able to prove
something nontrivial about $\{\KPt8\mfrk x:\mfrk Q\,\}\yxbtext{15}1b \,$, we
need some {\fssit axioms\ssp}. As such, we accept all the formulas

\begin{enumerate}\begin{myLeftskip}{-4}{.3}{.1}

\itemb0_ax $u=v\equivss22\aall x\,x\in u\equivss22 x\in v\,$, \label{ax of extent}

\itemb0_ax $u\in v\impss22\eexi{w\ssp,z}\,w\in z\ssp$ and $\,
           \aall x\,x\inc u\impss22 x\in w\ssp$, \label{ax of subsets}

\itemb0_ax $x\in u\ssp$ and $\ssp y\in v\impss22\eexi w\,
            x\inc w\ssp$ and $\ssp y\inc w\ssp$, \label{ax of union}

\itemb0_ax $u\in z\ssp$ and $\,[\KP{1.4}\aall{x\ssp,y\ssp,z}\,
           (\ssp x\ssp,y\sp)\ssp,(\ssp x\ssp,z\sp)\in f \impss22
            y=z \KP{1.4} ] \impss22 \eexi{v\ssp,w}$ \newfline

  $v\in w\ssp$ and $\,\aall y\,y\in v\equivss22\eexi x\,x\in u\ssp$
  and $\ssp(\ssp x\ssp,y\sp)\in f\sp$, \KP{17.9} \label{ax of substi}

\itemb0_ax $z\in w\impss22\eexi{u\ssp,v}\,u\in v\ssp$ and $\,\aall x\,
            x\in u\equivss22\eexi y\,x\in y\in z\,$, \label{ax of amalg}

\itemb0_ax $v\in u\impss22\eexi x\,x\in u\ssp$ and not $\eexi z\,
            z\in x\ssp$ and $\ssp z\in u\,$, \label{ax of regularity}

\itemb0_ax $\eexi{e\ssp,\spp N\sp,\spp S}\,e\in N\in S\ssp$ and \,[\ not $
            \eexi x\,x\in e \KP{1.4} ]\ssp$ and $\,\aall{n\ssp,\spp m}$ \newfline

  $n\in N\ssp$ and $\,[\KP{1.4}\aall x\,x\in m\equivss22
   x\in n\ssp$ or $\ssp x=n \KP{1.4} ]\impss22 m\in N\ssp$, \KP{17.9} \label{ax of infin}

\itemb0_ax $\eexi C\,[ \KP{1.4} \aall{x\ssp,z\ssp,u}\,(\sp u\ssp,x\sp)\ssp,
           (\sp u\ssp,z\sp)\in C\impss22 x=z\in u \KP{1.4} ]\ssp$ and \newfline

  $\aall{z\ssp,u\ssp,w}\,z\in u\in w\impss22\eexi x\,
           (\sp u\ssp,x\sp)\in C\ssp$, \KP{17.9} \label{ax of choice}

  \end{myLeftskip}\end{enumerate}

\noin and also all the formulas (\sp s\sp) given in the next

\begin{axiom schema}[classification]\label{class axi}

Let \math{\mfrk x} be any variable symbol and \math{\mfrk P} any formula. Let
\math{\mfrk y} be the first variable symbol distinct from \math{\mfrk x} and
not occurring free in \math{\mfrk P}. Then we accept as an axiom the formula
(\sp s\sp) \ $\mfrk x\in\{\KPt8\mfrk x:\mfrk P\,\}\yxbtext{15}1b\equivss22
           \eexi{\mfrk y}\,\mfrk x\in\mfrk y\ssp$ and $\ssp\mfrk P\,$.
  \end{axiom schema}

Above (\ref{ax of choice})$\ar{ax}$ is the {\sl global axiom of choice\sp} and 
(\ref{ax of infin})$\ar{ax}$ is the {\sl axiom of infinity\sp}.

\begin{remark}\label{rem about class sch}

Among others, we accept as logical axioms the formulas \inskipline{.3}3

(1)$\ar{az}\,$ \ $\mfrk P\impss22\aall{\mfrk x}\,\mfrk P\,$, \KP{20}
(2)$\ar{az}\,$ \ $[ \KP{1.4} \aall{\mfrk x}\,\mfrk Q \KP{1.4} ]\impss22
                 \mfrk Q\,(\ssp\mfrk x\sn\lleftarrow\sn\mfrak F\ssp) \,$, \inskipline{.3}0

when \math{\mfrk x} is any variable symbol and \math{\mfrk F} is any term and
\math{\mfrk P\ssp,\mfrk Q} are any formulas such that for any variable symbol \math{
\mfrk y} having a free occurrence in \math{\mfrk F} the bound (\ssp i.e.\
non\ssp-\ssp free\sp) occurrences of \math{\mfrk y} in \math{\mfrk Q} and \math{
\mfrk Q\,(\ssp\mfrk x\sn\leftarrow\sn\mfrak y\ssp)} are the same. Having these
logical axioms, we could give Axiom schema \ref{class axi} above a simpler 
formulation than has the corresponding \cite[\erm{II}\ssp, p.\ 253]{Ky} which 
in our notation would (\sp as already a bit corrected\ssp) read as follows. 
For any variable symbols \math{\mfrk x\ssp,\sp\mfrk y\ssp,\sp\mfrk z} and for 
any formula \mathss30{\mfrk P} such that \math{\mfrk y} is the first one 
distinct from \math{\mfrk x} and \mathss34{\mfrk z}, \,and not occurring free 
in \mathss34{\mfrk P}, we accept as an axiom the formula \inskipline{.3}3

(t) \ $\aall{\mfrk z}\,\mfrk z\in\{\,\mfrk x:\mfrk P\,\}\yxbtext{15}1b\equivss22
    \eexi{\mfrk y}\,\mfrk z\in\mfrk y\ssp$ and $\ssp
    \mfrk P\,(\ssp\mfrk x\sn\leftarrow\sn\mfrak z\ssp)\,$. \inskipline{.3}0

However, this would make the system contradictory as shown in Example \ref{exa about Kelley's class sch} 
below. One should put the additional restriction that the bound occurrences of \math{
\mfrk z} in $\ssp\mfrk P$ \linebreak 
                          and \math{
\mfrk P\,(\ssp\mfrk x\sn\leftarrow\sn\mfrk z\ssp) } are the same.
  \end{remark}

\begin{example}\label{exa about Kelley's class sch}

It follows from (\ref{ax of infin})$\ar{ax}$ and (\ref{ax of extent})$\ar{ax}$
and Proposition \ref{Pro basic}\,(18) below that there are \math{a\ssp,\sp b
 \ssp,\sp c} with \math{b\not=a} and \mathss36{a\ssp,\sp b\in c}. For \mathss36{ 
A = \{\,x:\eexi y\,x=y\ssp\text{ and }\ssp y=a\,\}\yxbtext{15}1b }, we then 
get from Remark \ref{rem about class sch}\,(t) and 
Proposition \ref{Pro basic}\,(17) that for all \math{x\ssp,\sp y} we have \vskip.4mm \centerline{$
 x\in A\equivss22[ \KP{1.4} x\ssp$ set and $\ssp\eexi y\, x = y \ssp$ and $\ssp
 y = a \KP{1.4} ]\equivss22 x = a \,$,} \inskipline{.2}0

and $\KP{13.7} y \in A \equivss22[ \KP{1.4} y \ssp$ set and $\ssp \eexi y\,
 y = y\ssp $ and $\ssp y = a \KP{1.4} ]\equivss22 y\ssp$ set\ssp, \inskipline{.4}0

whence taking \math{x=y=b\sp}, we obtain \mathss38{[ \KPp1.4 b \ssp\text{ set }
 \impss03 b\in A\impss33 b = a \KPp1.4 ] }, \,a {\sl contradiction\sp}. The 
formula \q{\math{\eexi y\, y = y\ssp\text{ and }\ssp y = a }} contains four 
occurrences of \math{\Symboo yÏ}. They are all bound and the second of them is 
not present in \q{\math{\eexi y\, x = y\ssp$ and $\ssp y = a }}.
  \end{example}

When we write a formula \mathss34{\mfrak P}, \,for example \q{\mathss00{ x = 
 \int_{\,\ssmb A}^{\KPt8\ssmb B}f\fvalss21 t\rmdss11 t}}, associated with the 
writing appearance of \math{\mfrak P} we assume that there is an implicitly 
understood well\sp-\ssp order between the occurring variable symbols so that 
e.g.\ it makes sense to refer to the first variable symbol occurring free in 
the {\sl writing appearance\sp} of \mathss34{\mfrak P}. This has nothing to do 
with the intuitive \q{overall} well\sp-\ssp order of all variable symbols of 
our set theoretic language.

For example in the above formula the variable symbols \math{\Symboo xÏ,
 \Symboo\ssn\smb A\snÏ,\Symboo\snn\smb B\snnÏ,\Symboo\ssn fÏ} occur free, and \math{
\Symboo tÏ} has two bound occurrences. We may assume that the order of the 
free variable\ssp( symbol\ssp)\ssp s is precisely the one given above, 
although it may not be perfectly clear which one of \math{
\Symboo\ssn\smb A\snÏ} and \math{\Symboo\snn\smb B\snnÏ} is before the other. 
To avoid confusion, in such vague cases we refrain from referring to that 
\q{implicit order}. In the above case we may then say that \math{\Symboo xÏ} 
is the first one, whereas in the case of the formula \q{\ssn\mathss01{
\int_{\,\ssmb A}^{\KPt8\ssmb B}f\fvalss21 t\rmdss11 t = x }} we would not 
speak of the free variable symbol that is in the first place in the writing 
  appearance.

Having the above preparative explanation, in order to have available a 
convenient means of specifying functions, we give the following

\begin{def:al schema}\label{<T:F>}

Let \math{\mfrak T} be a term and \math{\mfrak F} a formula and \math{
\mfrak x\ssp,\afr x_1\sp,\ldots\,\mfrak x\ssp\ai k\ssp,\sp\afr y_1\sp,\ldots\,
\mfrak y\ssp\ai l} distinct variable symbols such that \math{\afr x_1\sp,
\ldots\,\mfrak x\ssp\ai k} are precisely the variable symbols which have a
free occurrence both in \math{\mfrak T} and \math{\mfrak F} and are distinct
from any of $\,\mfrak x\,,\sp\afr y_1\sp,$ $\ldots\,\mfrak y\ssp\ai l\ssp$.
Also assume that \math{\mfrak F} is of the form \math{
\leu\,\mfrak p\sp\,\mfrak x\ssp\riu\sp\,\mfrak E} or \math{
\leu\,\mfrak k\sp\,\mfrak p\sp\,\mfrak x\ssp\riu\sp\,\mfrak E} where \math{
\mfrak p} is some predicate symbol and \math{\mfrak k} is a connective such
that in the writing appearance of $\ssp\mfrak F$ we have \math{\mfrak x} in
the first place. Then we let \vskip.2mm\centerline{$
\seq{\,\sp\mfrak T\sn:\afr y_1\sp,\ldots\,\mfrak y\ssp\ai l\ssn:
\mfrak F\sp\,}=\{\,\mfrak z:\exi{\mfrak x\ssp,\afr x_1\sp,\ldots\,
\mfrak x\ssp\ai k}\,\mfrak z=(\ssp\mfrak x\,,\mfrak T\ssp)\ssp$ and $\ssp
\mfrak F\,\}$} \vskip.2mm

\noin where \math{\mfrak z} is the first variable symbol not occurring free in
\math{\mfrak T} or \math{\mfrak F\sp}.

We also put \math{\seq{\,\sp\mfrak T\sn:\mfrak F\sp\,}=\seq{\,\sp\mfrak T\sn:
\ :\mfrak F\sp\,}} in the case where \math{\afr y_1\sp,\ldots\,
\mfrak y\ssp\ai l} is an empty list, and further \math{\seq{\,\sp\mfrak T\sn:
\mfrak x\in\mfrak U\sp\,}\subtext{old}=\{\,\mfrak z:\exi{\mfrak x}\,\mfrak z =
(\ssp\mfrak x\,,\mfrak T\ssp)\ssp$ and $\ssp\mfrak x\in\mfrak U\,\}} when \math{
\mfrak U} is any term not containing a free occurrence of \math{\mfrak x\sp},
and \math{\mfrak z} is the first variable symbol distinct from \math{\mfrak x}
and not occurring free in \math{\mfrak T} or \math{\mfrak U\ssp}.
  \end{def:al schema}

\begin{def:al schema}\label{uniqset}

We let \math{\uniqset\mfrak x:\mfrak P=\bigcap\,\{\,\mfrak z:\all{\mfrak x}\,
\mfrak P\sp\equivv\sp\mfrak x=\mfrak z\,\}\sp}, when \math{\mfrak x\ssp,
\mfrak z} are any distinct variable symbols and \math{\mfrak P} is any formula
where \math{\mfrak z} does not occur free. To get $\mfrak z$ uniquely chosen,
we may take as \math{\mfrak z} the first admissible w.r.t\ the intuitive 
well\ssp-\ssp ordering of the variable symbols of our set theoretic language.
  \end{def:al schema}

Under the agreement of unique choice of $\mfrak z$ above, for any formula $
\mfrak P$ and any distinct variable symbols $\mfrak x\ssp,\mfrak z$ with $
\mfrak z$ not occurring free in $\mfrak P\ssp$, now the formula $
\uniqset\mfrak x:\mfrak P=\bigcap\,\{\,\mfrak z:\all{\mfrak x}\,\mfrak P\sp
\equivv\sp\mfrak x=\mfrak z\,\}\,$ is a theorem. \vskip.2mm

One quickly deduces that if a unique {\it set\ssp} \math{\mfrak x} exists with
\math{\mfrak P}, then \math{\uniqset\mfrak x:\mfrak P=\mfrak x\sp}. In all
other cases, i.e.\ when there is no \math{\mfrak x\in\Univ} with \math{
\mfrak P}, or if (with $\mfrak y$ being a variable symbol not occurring in \math{
\mfrak P}) there are \math{\mfrak x\ssp,\mfrak y\in\Univ} with \math{
\mfrak x\not=\mfrak y} and \math{\mfrak P} and \math{
\mfrak P\,(\ssp\mfrak x\sn\leftarrow\sn\mfrak y\ssp)\sp}, substitution in
places of free occurrence, then \math{                                    \label{END set th extract}
      \uniqset\mfrak x:\mfrak P = \bigcap\ssp\emptyset = \Univ\ssp}. \vskip.4mm

Below in Definitions \ref{misc defs}\,(9) we have \math{\cinfty = 
 (\ssp 0\ar 1\sp,\plusinfty\ar 1\ssp;\sp 0\ar 1\sp,\spp 0\ar 1) } the 
{\sl complex infinity\sp} for some \math{\plusinfty\ar 1\in\ovbbR\ar 1} whose 
exact construction we here omit. Also omitting the precise definition, note 
that \math{|\,\zeta\,|\suba} is the standard \erm Euclidean absolute value of 
any quaternion \math{\zeta} and that we below usually have \math{\zeta} a real 
or complex number.

\begin{definitions}\label{misc defs}

(1) \ $\bbI=[\KPp1.1 0\,,\spp 1\KPt9]  \KP1 $, \KP{%%25.3
                                                   8.3}
(2) \ $\mathbb J=\openIval{\KPt5 0\,,\spp 1\sp} \KP1 $, \inskipline{.5}2

(3) \ $\rbb R^+=\bbR\capss41\{\,t:0<t\KPt9\} \KP1 $, \KP{7.5}
(4) \ $\lbb R_+=\bbR\capss41\{\,t:0\le t\KPt9\} \KP1 $, \inskipline{.5}2

(5) \ $
\mathbb Z=
\bigcap\KPt8\{\,N\sn:
0\in N\inc\bbR\text{ and }\ssp
\aall{n\in N}\,
\{\,n-1\ssp,\sp n+1\,\}\inc N\KP1\}
  \KP1 $, \inskipline{.5}2

(6) \ $\rbb Z^+=\mathbb Z\capss41\{\,n:0<n\KPt9\} \KP1 $, \KP{6.5}
(7) \ $\lbb Z_+=\mathbb Z\capss41\{\,n:0\le n\KPt9\} \KP1 $, \inskipline{.5}2

(8) \ $p\,^* = \uniqset t:[\KPp1.4 1 < p < \plusinfty\ssp$ and $\ssp t =
    (\ssp 1 - p\,^{\mminus 1}\sp\big){}^{\ssp\mminus 1}\KP1\big]\ssp$ or \inskipline0{31.6}

   $[\KPp1.4 p=1\ssp$ and $\ssp t=\plusinfty\KPp1.4]\ssp$ or $\ssp
    [\KPp1.4 p=\plusinfty\ssp$ and $\ssp t=1\KPp1.4] \KP1 $, \inskipline{.5}2

(9) \ $\Abrs33^p = \uniqset\chi:p\in\rbb R^+\sp$ and \inskipline{.2}{28.65}

  $\chi \ssp = \ssp \{\sp\minusinfty\ssp,\plusinfty\ssp,\cinfty\,\}\timesn
                    \{\sp\plusinfty\,\} \cupss21 \big\langle\KP{1.2}
  |\,\zeta\,|\suba\RHB{.25}{^p} \sn : \zeta\in\mathbb H\KPp1.2 \rangle \KP1 $, \inskipline{.5}2

(10) \ $\scrb8 T\,$ is a {\it topology\ssp} 
$\equivss33 \emptyset\not=\scrb8 T\in\Univ\ssp$ and \inskipline{.2}{10.5}

$\aall{\Cal A}\,\Cal A\inc\scrb8 T\impss33
\bigcup\,\Cal A\in\scrb8 T\ssp\text{ and }\ssp
[\KP{1.2} \Cal A\not=\emptyset\ssp\text{ and }\ssp
\Cal A\ssp\text{ is finite }\Rightarrow\ssp\bigcap\,\Cal A\in
\scrb8 T \KP{1.5} ] \KP1$, \inskipline{.5}2

(11) \ $\scrb8 T\,$ is a {\it separable\ssp} topology $\equivss33
\scrb8 T\,$ is a topology and \inskipline{.2}{10.5}

$\eexi D\,D\ssp$ is countable and 
$\ssp\aall U\,U\in\scrb8 T\impss33 D\capss33 U\not=\emptyset\ssp\text{ or }\ssp
U=\emptyset\KPt8$, \inskipline{.5}2

(12) \ $\scrb8 T\,$ is a {\it compact\ssp} topology $\equivss33
\scrb8 T\,$ is a topology and \inskipline{.2}{10.5}

$\aall{\Cal A}\,\eexi{\Cal B}\,
\Cal A\inc\scrb8 T\impss33
\Cal B\inc\Cal A\ssp$ and $\ssp\Cal B\ssp$ is finite and \inskipline{.2}{10.5}

$\ssp[\KP{1.6}
\bigcup\,\Cal A\inc\bigcup\,\Cal B\ssp$ or $\ssp
\bigcup\,\Cal A\not=\bigcup\,\scrb8 T \KP{1.5} ] \KP1$, \inskipline{.5}2

(13) \ $\scrmt A\ssp$ is {\it disjoint\ssp} $\equivss33\aall{A\,,\sp B}\,
 A\,,\sp B\in\scrmt A\impss33 A=B\ssp$ or $\ssp A\capss32 B=\emptyset \,$, \inskipline{.5}2

(14) \ $\Pows A=\{\,B:B\inc A\KP1\} \KP1 $, \KP{7.2}
(15) \ $\scrmt A\leiss42 B=\{\,A\capss32 B : A\in\scrmt A\KP1\} \KP1 $, \inskipline{.5}2

(16) \ $\nsTbb_R=\big\{\ssp\bigcup\,\scrmt A:\scrmt A\inc\big\{\,
\openIval{\ssp s\ssp,\spp t\ssp}:\minusinfty < s < t < \plusinfty\KPt9\}
\sp\} \KP1 $, \inskipline{.5}2

(17) \ $\barscTbb_R=
\Pows[\ssp\minusinfty\,,\plusinfty\KP1]\capss51\{\KPt8 U\sn:
U\cap\ssbb40 R\in\nsTbb_R\ssp$ and $\ssp\eexi{s\ssp,\sp r\in\bbR}\,$
\inskipline{.2}{21}

$
[\KPp1.4 \plusinfty\in U\impss33 
{\ssp]}\KP1 s\ssp,\plusinfty\KPt9]
\inc U\KPp1.4]\ssp$ and $\ssp[\KPp1.4 \minusinfty\in U\impss33 
[\ssp\minusinfty\,,\spp r\KP1{[\ssp}
\inc U\KPp1.4]\KP1\big\} \KP1 $, \inskipline{.5}2

(18) \ $f\fvalss20 x=
\bigcap\KPt8\{\,y:\aall z\,
(\ssp x\ssp,\spp z\ssp)\in f\equivss33 y=z\,\} \KP1 $, \inskipline{.5}2

(19) \ $f\ssp\image\ssn A=\{\,y:
\eexi{x\in A}\,(\ssp x\ssp,\spp y\ssp)\in f\KP1\} \KP1 $, \KP4
(20) \ $f\KPt8[\KPt8 A\KPt9]=f\ssp\image\ssn A \KPt8 $, \inskipline{.5}2

(21) \ $f\,\images\sn\scrmt A=
\{\,f\ssp\image\ssn A:A\in\scrmt A\KP1\} \KP1 $, \KP{18.2}
(22) \ $f\invss40=\{\,(\ssp y\ssp,\spp x\ssp):
(\ssp x\ssp,\spp y\ssp)\in f\KP1\} \KP1 $, \inskipline{.5}2

(23) \ $\dom f=\{\,x:\eexi y\,(\ssp x\ssp,\spp y\ssp)\in f\KP1\} \KP1 $, \KP8
(24) \ $\rng f=\{\,y:\eexi x\,(\ssp x\ssp,\spp y\ssp)\in f\KP1\} \KP1 $, \inskipline{.5}2

(25) \ $E\Reit3=\uniqset F:\eexi{a\ssp,\sp c\,,\sp\scrmt S}\, E = 
          (\ssp a\ssp,\sp c\,,\spp\scrmt S\ssp)\ssp$ and \mathss38{F =
(\ssp a\ssp,\sp c\KPp1.1|\KP1(\ssbb40 R\times\Univ\ssp)\,,\spp\scrmt S\ssp) }, \inskipline{.5}2

(26) \ $^A\,B = A\times B\capss31\{\,f:f\ssp\text{ is a function and }\ssp 
          A\inc\dom f\KP1\} \KP1 $, \inskipline{.5}2

(27) \ $\prodc\sp\bmii8 A \sp = \sp ^{\dom\sn\bmii6 A}\,\Univ\capss31\{\,
          x:\aall{i\ssp,\sp\xi}\,(\ssp i\ssp,\spp\xi\ssp)\in x\impss33\xi\in
 \bmii8 A\fvalue\sp i\KPt9\} \KP1 $, \inskipline{.5}2

(28) \ $\bosy x\to x\ssp$ in top \math{\scrmt T\equivss33\bosy x \in \sp
          ^\sbbNo\,\bigcup\,\scrmt T } and \math{x\in\bigcup\,\scrmt T } and \inskipline0{24.2}

 $\aall{\ssp U}\,\eexi{\smb N}\,x\in\sp U\sn\in\scrmt T \impss33 \smb N \in 
 \bbNo \ssp$ and \mathss30{\bosy x\KP1[\KPp1.1\bbNo\sn\setminus\smb N\KP1]
 \inc \sp U}, \inskipline{.5}2

(29) \ $E\subsigrs04=\uniqset F:\eexi{\bosy K}\,\bosy K\ssp$ is a 
          topological division ring and \math{\Bnull_{\bosy K} = 0 } and \inskipline0{14}

 $E\in\tvsps0(K) \ssp $ and \math{\aall{\Iota\ssp,\sp I\sp,\sp\scrmt T}\,
 I = \Cal L\,(\sp E\ssp,\spp\bosy K\ssp) } and \inskipline0{31.5}

 $\Iota=\seqss03{\sn\seqss33{u\fvalue x:u\in I}:x\in\vecs E} \ssp$ and \mathss39{
 \scrmt T=\Iota\invss44\images\sp(\ssp
  \taurd\bosy K\expnota^\sp I\sp]_{ti}\big) } \inskipline0{90.5}

 $\impss03 F = (\ssp\sigrd E\ssp,\spp\scrmt T\,) \KP1 $.
  \end{definitions}

About the {\sl weakening\sp} \math{E\subsigrs04} of \math{E} in 
Definitions \ref{misc defs}\,(29) above we note the following. If \math{E} is 
a topological vector space over a topological division ring \mathss32{\bosy K
}, \,there may exist another topological division ring \math{\bosy K\sn\ar 1} 
with \mathss38{E\in\tvsps9(K\sn\ar 1\sn) }. In every case then \math{
\vecs\bosy K=\domm\tsigrd E=\vecs\bosy K\aR 1} holds, but \math{\bosy K} and \math{
\bosy K\sn\ar 1} may possess different zero elements if \math{ \vecs E = 
 \{\,\Bnull_E\} } holds. Then the condition \math{\Bnull_{\bosy K}=0=
 \Bnull_{\sp\aars{\bosy K\snn}_1} } excludes this possibility. If \math{
\vecs E=\{\,\Bnull_E\} } holds, for \math{ I = 
 \Cal L\,(\sp E\ssp,\spp\bosy K\ssp) } and \mathss30{ n = 
 \{\,\vecs E\times\snn\{\ssp 0\ssp\}\sp\} } then necessarily \math{ I = 
 \{\ssp n\ssp\} } holds, and we get \math{ \Iota = \{\KPt8(\,\Bnull_E\ssp,\sp
 \{\,(\ssp n\ssp,\spp 0\ssp)\,\}\ssp\sbig)2\ssp\big\} } and further \mathss06{
\scrmt T=\Pows\vecs E }. Hence in this case \math{\scrmt T} is uniquely 
determined although \mathss30{\taurd\bosy K\not=\taurd\bosy K\sn\ar 1 } may 
hold. If \math{I\not=\{\ssp n\ssp\} } holds, then one deduces from the 
postulates in the definition of a topological vector space that we necessarily 
have \math{\bosy K=\bosy K\sn\ar 1 } and consequently again \math{
(\ssp\sigrd E\ssp,\spp\scrmt T\,) } is uniquely determined.

Thus the above definition of \math{E\subsigrs04} is meaningful for precisely 
those topological vector spaces \math{E} that are \q{over} some topological 
division ring whose zero element is the same as that of the quaternionic one. 
For more general cases one has to use a more complicated notation e.g.\ from \math{
E\subsigrs04\ssp\langle\,\bosy K=E\subsigrs04\ssp(\sp I\ssp) } for \math{I} as 
above, once the appropriate additional definition is specified.

In \ref{misc defs}\,(18) above \math{f\fvalss20 x} is the function value of \math{
f} at \math{x} which usually is written in a more complicated manner 
\q{$f\ssp(x)$}, and possibly having a different formal definition as for 
example in \cite[Definition 68\ssp, p.\ 261]{Ky}\,. We further state some 
basic definitions and their simple consequences without proofs in the 
  following

\begin{proposition}\label{Pro basic}

$\null$ {\rm \inskipline{.5}2

(1) \ }$\emptyset=\{\,x:x\not=x\,\} \KP1 ${\rm, \KP{24} 
(2) \ }$\Univ=\{\,x:x=x\,\} \KP1 ${\rm, \inskipline{.5}2

(3) \ $ \roman{pr}\ar 1 = 
    \{\,(\ssp x\ssp,\spp y\ssp,\spp x\ssp):x\ssp,\sp y\in\Univ\KP1\} \KP1 $, \KP5 
(4) \ }$\roman{pr}\ar 2 =
    \{\,(\ssp x\ssp,\spp y\ssp,\spp y\ssp):x\ssp,\sp y\in\Univ\KP1\} \KP1 
       ${\rm, \inskipline{.5}2

(5) \ }$\roman{ev}=\{\,(\ssp x\ssp,\spp u\ssp,\spp y\ssp):u\ssp\text{ is a 
    function  and }\ssp(\ssp x\ssp,\spp y\ssp)\in u\,\} \KP1 ${\rm, \inskipline{.5}2

(6) \ }$\roman{ev}\sbi{\sp\emath x} = 
        \seqss33{u\fvalue x:u\ssp\text{ is a function}}$ \inskipline{.2}{14.45}

  ${} = \{\,(\ssp u\,,\spp y\ssp):u\ssp\text{ is a function and }\ssp
            (\ssp  x\ssp,\spp y\ssp)\in u\KPt8\} \KP1 ${\rm, \inskipline{.5}2

(7) \ }$\scrmt A\,,\sp\scrmt B$ disjoint iff $\,\scrmt A$ and $\,\scrmt B$ 
      disjoint iff $\,\scrmt A$ and $\,\scrmt B$ are disjoint iff \inskipline0{52.2}

 $\scrmt A$ is disjoint and $\,\scrmt B$ is disjoint{\sp\rm, \inskipline{.5}2

(8) \ }$                                                               \newcommand\opair[2]{\hbox{\kern.#1mm\kern-.2mm\font\Å=cmtt10\Å,\kern-.2mm\kern.#2mm}} % notation for ordered pair; use e.g. as in $x \opair14 y$ , $x\ar 1 \opair00 \tfbbR$
        x \opair14 y = (\sp x\ssp,\spp y\ssp) =
        \{\sp\{\,x\ssp,\sp y\,\}\,,\sp\{\ssp y\ssp\}\sp\} \KP1 ${\rm, \KP{3.9}
(9) \ }$(\sp x\ssp,\spp y\ssp,\spp z\ssp) = 
        ((\ssp x\ssp,\spp y\ssp)\ssp,\spp z\ssp) \KP1 ${\rm, \inskipline{.5}2

(10) \ }$(\ssp x\,;\spp y\ssp,\spp z\ssp) = 
         (\ssp x\ssp,\spp(\ssp y\ssp,\spp z\ssp)) \KP1 ${\rm, \KP{15}
(11) \ }$(\sp x\ssp,\spp y\,;\spp u\ssp,\spp v\ssp) = 
         (\ssp x\ssp,\spp y\ssp,\spp(\ssp u\ssp,\spp v\ssp)) \KP1 ${\rm, \inskipline{.5}2

(12) \ }$\sigrd z = \bigcup\bigcup\sp z\setminus\bigcup\bigcap\sp z
                                    \sp\cup\sp  \bigcap\bigcup\sp z ${\,\rm, \KP{3.8}
(13) \ }$\taurd z=\bigcap\bigcap\sp z ${\,\rm, \inskipline{.5}2

(14) \ }$\ssigrd z = \sigrd(\ssp\sigrd z\ssp) ${\KP1\rm, \KP{22.8}
(15) \ }$\tsigrd z = \taurd(\ssp\sigrd z\ssp) ${\KP1\rm, \inskipline{.5}2

(16) \ }$z = (\ssp x\ssp,\spp y\ssp)\in\Univ\impss33 x = \sigrd z$ and $\ssp 
         y = \taurd z ${\,\rm, \inskipline{.5}2

(17) \ }$x\ssp$ is a set $\equivss22 x\ssp$ a set $\equivss22
         x\ssp$      set $\equivss22 \eexi y\,x\in y ${\ssp\rm, \inskipline{.5}2

(18) \ }$x \not= y\equivss22$ not $[\KP{1.4} x=y \KP{1.4} ] \KP1 $.
  \end{proposition}

Observe for example that if \math{E=(\ssp a\ssp,\sp c\,,\spp\scrmt S\ssp)\not=
 \Univ } with \math{c} a function \mathss34{R\times S\to S}, \,then \ 
$\domm\tsigrd E=\dom(\sp\dom(\ssp\taurd(\ssp\sigrd E\ssp)))=
\dom(\sp\dom(\ssp\taurd(\ssp a\ssp,\sp c\ssp)))$ \inskipline0{28}

${}=\dom(\sp\dom c\ssp)=\dom(\sp R\times S\ssp)=R\ssp$ {\sl if\sp} \math{
S \not= \emptyset } holds. \vskip.5mm

To see that the above given convention of \q{\mathss00{u\le v}} having a 
meaning for both extended real numbers and extended real number valued 
functions \math{u\ssp,\sp v} does not create any contradiction in our logical 
system, we need the following

\begin{lemma}

For every function $\,u$ with $\,\rng u\inc\ovbbR$ it holds that $\,u\not\in
\ovbbR \, $.
  \end{lemma}

\begin{proof} The {\sl regularity axiom\sp} (\ref{ax of regularity})$\ar{ax}$ 
on page \pageref{ax of regularity} above, cf.\ \cite[\erm{VII}\sp, p.\ 266]{Ky} 
or \cite[\erm{ZF\,}9\ssp, p.\ 401]{Du}\,, has the simple consequence that 
there {\sl do not exist\sp} any \mathss30{x\ar 0\,,\sp x\ar 1\sp,\sp x\ar 2} 
such that \math{x\ar 0\in x\ar 1\in x\ar 2\in x\ar 0} holds. We show that this 
will be contradicted if there exists a function \math{u} with \math{\rng u\inc
 \ovbbR} and \mathss37{u\in\ovbbR}. Indeed, then there is \math{r} with \inskipline{.42}{17.72}

$u = (\ssp r\sp,\spp 0\ar 1\sp;\spp 0\ar 1\sp,\spp 0\ar 1)
   = (\ssp r\sp,\spp 0\ar 1\sp,\sp(\ssp 0\ar 1\sp,\spp 0\ar 1))$ \inskipline{.2}{19.8}

${}= ((\ssp r\sp,\spp 0\ar 1)\,,\sp(\ssp 0\ar 1\sp,\spp 0\ar 1))
   = \{\sp\{\,(\ssp r\sp,\spp 0\ar 1)\,,\sp(\ssp 0\ar 1\sp,\spp 0\ar 1)\,\}\,,\sp
          \{\,(\ssp 0\ar 1\sp,\spp 0\ar 1)\,\}\sp\} \KP1 $. \inskipline{.4}0

Since \math{u} is a function with \math{\rng u\inc\ovbbR} there are \math{
x\ssp,\sp s} with \inskipline{.4}{22}

$ \{\spp\{\spp\{\,0\ar 1\}\spp\}\spp\}
= \{\,(\ssp 0\ar 1\sp,\spp 0\ar 1)\,\}
= (\ssp x\ssp,\sp(\ssp s\ssp,\spp 0\ar 1\sp;\sp 0\ar 1\sp,\spp 0\ar 1))$ \inskipline{.2}{36.8}

${}= \{\sp\{\,x\ssp,\sp(\ssp s\ssp,\spp 0\ar 1\sp;\sp 0\ar 1\sp,\spp 0\ar 1)\,
          \}\,,\sp\{\,(\ssp s\ssp,\spp 0\ar 1\sp;\sp 0\ar 1\sp,\spp 0\ar 1)\,
                  \}\sp\} \KP1 $, \inskipline{.4}0

and hence \mathss38{
 \{\,x\ssp,\sp(\ssp s\ssp,\spp 0\ar 1\sp;\sp 0\ar 1\sp,\spp 0\ar 1)\,\} = 
 \{\spp\{\,0\ar 1\}\spp\} = \{\,(\ssp s\ssp,\spp 0\ar 1\sp;\sp 0\ar 1\sp,\spp 
 0\ar 1) \, \} }, \,whence further \mathss30{\{\,0\ar 1\} } \mathss08{{\KN{.99}}
= (\ssp s\ssp,\spp 0\ar 1\sp;\sp 0\ar 1\sp,\spp 0\ar 1)
= \{\sp\{\,(\ssp s\ssp,\spp 0\ar 1)\,,\sp(\ssp 0\ar 1\sp,\spp 0\ar 1)\,\}\,,\sp
  \{\,(\ssp 0\ar 1\sp,\spp 0\ar 1)\,\}\sp\} }. Then we get \inskipline{.2}{7.7}

$0\ar 1\in\{\,0\ar 1\} \in \{\spp\{\,0\ar 1\}\spp\} \in 
 \{\spp\{\spp\{\,0\ar 1\}\spp\}\spp\} = \{\,(\ssp 0\ar 1\sp,\spp 0\ar 1)\,\} = 
 0\ar 1 \ssp $, \,a {\sl contradiction\sp}.
  \end{proof}

For \math{0<q<\plusinfty} we assume that \math{ \plusinfty\RHB{.25}{\KPt8^q} 
 = \plusinfty} in the following

\begin{constructions}[of Lebesgue quasi\ssp-\ssp norms]\label{Ctr |x|_lL^p} $\null$ \inskipline{.7}2

(1) \ $\|\,x\,\|\lllnor_p=\uniqset s : [\KPp1.4 0<p<\plusinfty\ssp$ and $\ssp 
      s=\big(\sp\sum_{\KPt8 i\ssp\in\ssp\dom\snn\emath x} \ssp | \KP1 
   x\fvalss01 i\KP1|\suba\RHB{.25}{^p}\ssp\sbig)0\,^{p^{-1}} \KPp 1.4 \big ] \ssp $ \inskipline{.2}{40}

  or $\ssp [\KPp1.4 p = \plusinfty\ssp$ and $\ssp s = \sup\sp\big\{\,
  |\,t\,|\suba\sn:t\in\rng x\KPt8\} \KPp 1.4 \big ] \KP1 $, \inskipline{.5}2

(2) \ $\|\,x\,\|\Lnorss33^p_\mu=\uniqset s : [\KPp1.4 0<p<\plusinfty\ssp$ and $\ssp \label{ctr L^p-norm}
       \aall\Omega\,\Omega = \bigcup\,\dom\mu\impss30{}$ \inskipline{.2}{35.8}

  $s = \inf\sp\big\{\,\big(\sp\int_{\KPp1.1\Omega\,}\varphi\rmdss11\mu\ssp)\KP1
   ^{p^{-1}}\ssn:\varphi\in\sp^\Omega\KP1[\KPp1.1 0\,,\plusinfty\KPt9]\ssp$ 
   and \inskipline{.2}{23.4}

  $\varphi\invss44\images\spp\barscTbb_R\inc\dom\mu\ssp$ and $\ssp
    \aall{\eta\,,\sp t}\,(\ssp\eta\ssp,\spp t\ssp)\in x\impss33
    |\,t\,|\suba\RHB{.25}{^p} \le \varphi\fvalue\eta\KP1\} \KPp 1.4 \big ] $ \inskipline{.2}{7.5}

  or $\ssp[\KPp1.4 p=\plusinfty\ssp$ and $\ssp s=\inf\,\{\,\smb M:\smb M\in
  \rbb R^+\sp$ and $\ssp\aall A\,\eexi N\,A\in\mu\invss44\image\spp\lbb R_+$ \inskipline{.2}{23.5}

  ${}\impss03 N\in\mu\invss33\image\snn\{\ssp 0\ssp\}\ssp$ and $\ssp\sup\sp\big\{\,
  |\,t\,|\suba\sn:t\in x\sp\image(\sp A\setminus N\ssp)\KPt8\}\le\smb M\KP1\}
    \KPp 1.4 \big ] \KP1 $.
  \end{constructions}

For completeness' sake, in Constructions \ref{defi $L^p$} below of the 
generalized Lebesgue\,--\,Bochner spaces we have included items (7) and (11) 
where we define \math{\suptext{vc}0\Lrs03^p(\vcal Q\sp) } and \mathss38{
\LLrs03^p(\ssp Q\ssp,\spp\vPi\ssp) }. There we utilize the concepts of 
{\sl quasi\ssp-\sp\esl Euclidean vector column\sp} and 
{\sl quasi\ssp-\sp usual space\sp}. To make matters precise, we give the 
  following

\begin{definitions}

(1) \ Say that \math{\vcolQ} is a {\it quasi\ssp-\sp\eit Euclidean \mathss36{
    \bosy K}--\,vector column\ssp} if{}f there are $Q\ssp,\sp\Yps\spp,\sp\vPi\ssp$ 
with \math{\vcolQ = (\ssp Q\ssp,\sp\Yps\sppp,\spp\vPi\ssp) } and such that \math{
Q\inc\vecs\Yps} and \math{\Yps\in\LCSps5(\tfbbR) } and $\vPi\in\tvsps0(K)$ 
hold with \math{\dimHa\Yps\in\bbNo} and \math{\bosy K\in\setRC} and for every 
\math{\xi\in\vecs\vPi%%\ssp\setminus\{\KPt8\Bnull_\vPi\}
 } there is \math{
u\in\Cal L\,(\sp\vPi\sp,\spp\bosy K\ssp) } with \math{\xi=\Bnull_\vPi} or \mathss06{
u\fvalss02\xi \not= 0 }. \inskipline{.5}2

(2) \ Say that \math{\ebit F} {\it usualizes} \math{F} over \math{\bosy K} 
    if{}f \math{\bosy K\in\setRC} and there is \math{k\in\bbNo} with \math{
(\ssp\emptyset\,,\spp\bosy K\ssp)\,,\sp(\ssp k\ssp,\spp F\ssp)\in
 \ebit F\in\sp^{k\ssp +\ssp 1.}\,\Univ } and for every \math{i\in k} there are \math{
i\ar 1\ssp,\sp i\ar 2\in i\ssp\yplus } and \math{l\in\bbN} and \math{
\ebit E\in\sp^l\,(\ssp\ebit F\KPt8\image i\ssp\yplus)} with \mathss38{
\ebit F\fvalss61 i\ssp\yplus \in \{\KPt8\bmii8 F\fvalss61 i\ar 1\sqcap\sp
 (\spp\ebit F\fvalss61 i\ar 2\spp)\,,\spp\vecs\bosy K\text{\,-\sp}
   \prodsubtext{tvs}\ebit E\KP1\} }. \inskipline{.5}2

(3) \ Say that \math{F} is {\it quasi\ssp-\sp usual\,} over \math{\bosy K} 
    if{}f \inskipline0{18.8}

there is \math{\ebit F} such that \math{\ebit F} usualizes \math{F} over \math{
  \bosy K}.
  \end{definitions}

A quasi\ssp-\sp usual space necessarily has finite nonzero dimension. For
example \linebreak the space \math{ F =
\bosy R\sp\sqcap(\spp\bosy R\sp\sqcap\bosy R\ssp)\expnota^\ssmb N]_{tvs}} is
quasi\ssp-\sp usual over \math{\bosy R} when \math{ \bosy R \in \setRC} and $
\smb N\in\bbN\sp\,$, with for example \math{
 \seqss44{\bosy R\,,\sp\bosy R\sp\sqcap\bosy R\,,
      (\spp\bosy R\sp\sqcap\bosy R\ssp)\expnota^\ssmb N]_{tvs}\spp,\sp F} }
            usualizing \mathss31{F}.

\begin{lemma}\label{Le for q-usu} % from FunSps.tex

For \PouN$\ssp\sbi{\iota\ssp=\ssp\sixroman{1\sp,\ssp 2}}\,,$ let $\,
\varUpsilon\ssn\sbi\iota\sp$ be quasi\ssp-\sp usual over $\ssp
\bosy K\snn\sbi\iota$ with $\ssp Q\inc\vecs\varUpsilon\ssn\sbi\iota\ssp$. If
also $\,[\ Q\not=\emptyset\sp$ and $\,\bosy K\sn\ar 1=\bosy K\snn\ar 2\ ]$ or
$\,\Int_taurd{\varUpsilon_\iota} Q\not=\emptyset\,,$ \,then $\,
\varUpsilon\aar 1=\varUpsilon\aar 2\,$.
  \end{lemma}

Thus for example quasi\ssp-\sp usual spaces \math{\Yps} over \math{\tfbbR} are 
such that every single point \math{\eta\in\vecs\Yps} uniquely determines the 
whole algebraic and topological structure \mathss30{\Yps}. The proof of 
Lemma \ref{Le for q-usu} is given in \cite{Hif}\,. It is quite long and 
requires delving in the set theoretic formal construction of the complex 
number system starting from the set \math{\bbNo} of natural numbers, and so we 
  omit it here.

For \math{Q\inc\vecs\Yps} this allows us to define a structured vector space 
\math{\roman S\,(\ssp Q\ssp,\spp\vPi\ssp) } based on a set of functions \math{
Q\to\vecs\vPi} without explicit reference to the structure \math{\Yps} by 
putting \math{\roman S\,(\ssp Q\ssp,\spp\vPi\ssp) = 
 \roman S\,(\ssp Q\,\sbi\Yps\sp,\spp\vPi\ssp) } when the latter is already 
defined. So we just get a bit simpler notation for the same space.

% ¤¤¤¤¤¤¤¤¤¤¤¤¤¤¤¤¤¤¤¤¤¤¤¤¤¤¤¤¤¤¤¤¤¤¤¤¤¤¤¤¤¤¤¤¤¤¤¤¤¤¤¤¤¤¤¤¤¤¤¤¤¤¤¤¤¤¤¤¤¤¤¤¤¤¤¤

\insubsubhead           Suitable locally convex spaces                    \label{Ss suit lcs}

Suitable locally convex spaces are those that are obtained from some 
\erm Banachable space by weakening the topology so that we {\sl do not\sp} 
get more bounded sets. Our basic important examples of suitable spaces are the 
weak$^*$ duals \math{E\dlsigss22} of \erm Banachable spaces \mathss35{E}. We 
put the following

\begin{definitions}\label{df suit}

Say that\inskipline{.5}2

(1) \ $\Nu\ssp$ is a {\it dominating norm\ssp} for \math{E} if{}f 
    there is \math{\bosy K\in\setRC} with \inskipline09

 $E\in\tvsps0(K)\ssp$ and \math{\Nu} a norm on \math{\sigrd E} with \inskipline09

 $\bouSet E = \vecs E\capss31\{\,B:\eexi{n\in\rbb Z^+}\,B \inc
  \Nu\invss44\image\sp[\KP{1.1} 0\,,\spp n\KPt9]\KP1\big\} \KP1 $, \inskipline{.5}2

(2) \ $E\ssp$ is {\it almost suitable\ssp} over \math{\bosy K} if{}f \math{
      \bosy K\in\setRC} and \math{E\in\LCSps0(K)} and \inskipline09

 there is a normable \math{F\in\LCSps0(K)} with \math{E\tvpreceq F} and $\ssp
  \bouSet E\inc\bouSet F\sp$, \inskipline{.5}2

(3) \ $E\ssp$ is {\it suitable\ssp} over \math{\bosy K} if{}f \math{ \bosy K 
      \in\setRC} and \math{E\in\LCSps0(K)} and \inskipline09

 there is \math{F\in\BaSps0(K)} with \math{E\tvpreceq F} and $\ssp
  \bouSet E\inc\bouSet F\sp$. \inskipline{.5}2

For \q{almost suitable} or \q{suitable} in place of \sp X also say that \inskipline0{38.6}

 $E\ssp$ is X if{}f \math{E} is X over \math{\bosy K} for some \mathss31{
    \bosy K}.
  \end{definitions}

If \math{ \bosy K\in\setRC} and \math{E\in\LCSps0(K)} and \math{F\in\BaSps0(K)} 
with \math{E\tvpreceq F} and \œ$\ssp\bouSet E$ \linebreak
                                   \œ${\ssn}\inc\bouSet F\sp$, then \math{
\bouSet E=\bouSet F} holds since from \math{E\tvpreceq F} we get \mathss35{
\bouSet F\inc\bouSet E}. If we \linebreak
                             also have \math{E\tvpreceq G\in\BaSps0(K)} and \mathss32{
\bouSet E\inc\bouSet G}, \,then \math{F=G} holds. This is seen by noting that 
\erm Banachable spaces are bornological, and hence have the strongest locally 
convex topology with the same bounded sets. Thus \math{F} is the unique 
\erm Banachable space from which \math{E} is obtained by weakening the 
topology. The dominating norms \linebreak
                               $\Nu\ssp$ for \math{E} are precisely the 
compatible norms for \mathss30{F}, \,and then \math{
(\ssp\sigrd E\ssp,\spp\Nu\ssp) } is a corresponding (\sp norm{\sl ed\ssp}) 
  Banach space.

One should observe that the bornology of a suitable space {\sl does not\sp} 
determine the dual, i.e.\ there exist suitable spaces obtained by weakening 
the same \erm Banachable space but with different duals. This is seen by 
considering \math{\ell\KPt8^1\sp(\ssp\bbNo\spp)\subw0 } and \œ$\ssp
\ell\KPt8^1\sp(\ssp\bbNo\spp)\subsigrs03$ \linebreak
                                          which both are obtained by weakening \mathss38{
\ell\KPt8^1\sp(\ssp\bbNo\spp) }. The former has the initial topolo- gical 
vector structure from \math{(\ssp\roman I\,A\,,\sn\tvbbR4^A\sp\big) } for \math{
A=\vecs\co(\ssp\bbNo\spp) } and the latter for \mathss30{ A = \sn} \mathss03{
\vecs\lll^\plusinftyy\sp(\ssp\bbNo\spp) } when we let \mathss39{ \roman I\,A =
 \big\langle\big\langle\ssp\sum\KP1(\ssp x\cdot y\ssp) : y\in A\KP1\rangle : 
 x\in\vecs\ell\KPt8^1\sp(\ssp\bbNo\spp)\KP1\rangle }.

\begin{lemma}\label{Le suit dom}

Let $\,E$ be almost suitable with $\,\Nu$ a dominating norm. \hfill Then for 
every $\Nu\aR 1\in\Bqnorm E$ there is $\,\smb M\in\rbb R^+$ with $\,
\Nu\aR 1\sn\fvalue x\le\smb M\KP1(\ssp\Nu\fvalss10 x\ssp)$ for all $\,
  x\in\vecs E\,$.
  \end{lemma}

\begin{proof} We have \mathss35{\Nu\invss44\image\ssbb 20 I\in\bouSet E
}, \,and letting \math{\Nu\aR 1\in\Bqnorm E} then \math{
\Nu\aR 1\,[\KP{1.1}\Nu\invss44\image\ssbb 26 I\,] \in {\ssn}} \mathss03{
\bouSet\tfbbR } and hence for \math{ \smb M = 
 \sup\KPt8(\ssp\Nu\aR 1\,[\KP{1.1}\Nu\invss44\image\ssbb 26 I\,]\ssp\sbig)0 } 
we have \mathss35{\smb M < \plusinfty }. Considering \œ$\ssp x\in {\ssn}$ $
\vecs E\,$, if \math{x=\Bnull_E} holds, we trivially have \mathss38{
\Nu\aR 1\sn\fvalue x=0\le 0=\smb M\KP1(\ssp\Nu\fvalss10 x\ssp) }. Otherwise 
taking \math{\smb A=\Nu\fvalss10 x} we have \mathss34{
(\sp\smb A\,^{\mminus 1\,}x\ssp)\svs E \in \Nu\invss44\image\ssbb 20 I}, \,and 
hence \mathss30{\Nu\aR 1\sn\fvalue(\sp\smb A\,^{\mminus 1\,}x\ssp)\svs E \le 
 \smb M} and further \mathss38{\Nu\aR 1\sn\fvalue x \le \smb M\,\smb A = 
 \smb M\KP1(\ssp\Nu\fvalss10 x\ssp) }.
  \end{proof}

\begin{lemma}\label{Le E'' adher}

With \œ$\,\bosy K\in\setRC$ let \œ$\,E\in\LCSps0(K)$ and \œ$\, w \in 
 \Cal L\,(\sp E\dlbetss12\ssp,\spp\bosy K\ssp) \KP1 $. Then there is some \œ$\,
B\in\bouSet E$ such that for every finite \œ$\, A \inc 
 \Cal L\,(\sp E\ssp,\spp\bosy K\ssp) $ there is \œ$\, x \in B $ with $\,
 |\KP1 u\fvalue x - w\fvalue u\KP1|\suba\le 1 $ for all $\,u\in A\,$.
  \end{lemma}

\begin{proof} Putting \mathss38{D\ar 1=\vecs\bosy K\capss21\{\,t:|\,t\,|\suba
 \le 1 \KPt8\} }, \,from \math{ w \in 
 \Cal L\,(\sp E\dlbetss12\ssp,\spp\bosy K\ssp) } we first get existence of 
some nonempty absolutely convex bounded set \math{B} in \math{E} such that for \mathss03{ 
U = \Cal L\,(\sp E\ssp,\spp\bosy K\ssp) \capss31 \{\, u : u\sp\image\snn B 
 \inc D\ar 1\ssp\} } we have \mathss34{w\spp\image\spp U \inc D\ar 1 }. Then 
for the canonical evaluation \math{ \Iota = \seqss33{
 \roman{ev}\sbi{\sp\emath x}\,|\KP1\Cal L\,(\sp E\ssp,\spp\bosy K\ssp) : 
 x\in\vecs E} } and for \math{\scrmt T=\taurd(\sp E\dlbetss12{\ssn}\dlsigss02) } 
from the {\sl bipolar theorem\sp} \cite[3.3.1\sp, p.\ 192]{Ho} or 
\cite[8.2.2\ssp, p.\ 149]{Jr} we see \math{ w \in 
 \roman{Cl}\sbi{\KPt8\scrm7 T\KP1}(\ssp\Iota\image\snn B\ssp) } to hold whence 
the assertion follows.
  \end{proof}

The content of \cite[Lemma 8.17.8 \erm B\ssp, p.\ 585]{Edw} is in the 
  following

\begin{lemma}\label{Le 8.17.8 B}

With \œ$\,\bosy K\in\setRC$ let \œ$\,E\in\LCSps0(K)$ be normable{\ssp\rm, }and 
let \œ$\,F=E\dlbetss12\,$. Also let $\,S\ar 1$ be a linear subspace in $\,
\sigrd F$ such that \œ$\,\taurd F\leiss22 S\ar 1$ is a separable 
topology{\ssp\rm, }and let \œ$\, w \in 
 \Cal L\,(\sp F_{\sp/\ssp\aars S_1},\spp\bosy K\ssp) \,$. Then there is \œ$\,
\bosy x\in\sp^\sbbNo\,\vecs E$ with \œ$\,\rng\bosy x\in\bouSet E$ and such that $\,
w\fvalue u = \lim\,(\ssp u\circss01\bosy x\ssp)$ holds for every $\,u\in 
 S\ar 1\ssp$.
  \end{lemma}

\begin{proof} By Hahn\,--\,Banach there is \math{\bar w\in
 \Cal L\,(\sp F\spp,\spp\bosy K\ssp) } with \mathss35{w\inc\bar w}. 
Furthermore, for \linebreak
        \mathss03{\Iota\ar 2} the canonical embedding \mathss36{E\to 
 E\dlbetss12{\!}\dlbetss02}, \,i.e.\ for \math{ \Iota\ar 2 = 
 \seqss33{\roman{ev}\sbi{\sp\emath x}\,|\KP1\vecs F:x\in\vecs E} } and for \mathss03{
\scrmt T=\taurd(\sp E\dlbetss12{\!}\dlsigss04) } by \sp Lemma \ref{Le E'' adher} 
above, there is some \math{B\in\bouSet E} such that \mathss03{\bar w\in B\ar 2} 
holds for \mathss38{B\ar 2 =
 \roman{Cl}\sbi{\KPt8\scrm7 T\KPt8}(\ssp\Iota\ar 2\sn\image\sn B\ssp) }. Now 
letting \math{\scrmt T\aR 1} be the initial topology from \mathss30{
\taurd((\sp \sp F_{\sp/\ssp\aars S_1}\sbig)0\dlsigss12\spp) } under \math{
\seqss33{z\KP1|\KP1 S\ar 1\sn:z\in\Cal L\,(\sp F\spp,\spp\bosy K\ssp)} } we 
have \math{\scrmt T\aR 1\leiss02 B\ar 2\inc \scrmt T\leiss32 B\ar 2 } with \math{
\scrmt T\aR 1\leiss02 B\ar 2 } semimet- rizable. Hence there is \math{\bosy x
 \in\sp^\sbbNo\,B} with \mathss30{w\fvalue u = \bar w\fvalue u = 
 \lim\,(\ssp u\circss01\bosy x\ssp) } for \mathss34{u\in S\ar 1}.
  \end{proof}

\begin{lemma}\label{Le qtvs}

With \œ$\,\bosy K\in\setRC$ let \œ$\,E\in\tvsps0(K)$ and let $\,S$ be a vector 
subspace in $\,\sigrd E\,$. \hfill Also let \œ$\,F=E\,/\tvsquotient S$ and \œ$\,
\Cal V = \{\,\vecs F\capss21\{\,\smb X:U\capss02\smb X\not=\emptyset\KP1\} : 
 U\sn\in\Cal U\KP1\}$ \linebreak
                      where $\,\Cal U$ is a filter base for $\,\neiBoo E\,$. 
Then $\,\Cal V$ is a filter base for $\,\neiBoo F\sp$.
  \end{lemma}

\begin{proof} With \mathss38{ \tweq = \vecs E\times\vecs F \capss31 \{\,
 (\ssp x\ssp,\spp\smb X\sp):x\in\smb X\,\} } we know from the discussion in 
\cite[p.\ 104]{Ho} that \math{\tweq} is continuous and open \mathss30{
\taurd E\to\taurd F}, \,and consequently \linebreak
                                       \œ$\Cal V=\tweq\,\images\sp\Cal U\inc
 \neiBoo F\ssp$ holds. Moreover, for every \math{V\in\neiBoo F} we have \œ$\ssp
\tweq\invss64\image\sp V\in\neiBoo E$ \linebreak
                                      and hence there is \math{U\in\Cal U} 
with $\ssp U\inc\tweq\invss64\image\sp V\spp$, but then \math{
\tweq\,\image U\inc V} holds.
  \end{proof}

% ----------------------------------------------------------------------------

\Ssubhead A             Measurability and integration                     \label{Sec A}

In this section, we first explain what it means for functions \math{ x : 
 \Omega\to\vecs\vPi} to be measurable when \math{\vPi} is a topological vector 
space and \math{\Omega} is a set equipped with a positive measure \mathss36{
\mu}. In the next section, for \math{0\le p\le\plusinfty} we construct the 
spaces $\mvLrs02^p(\ssp\mu\,,\spp\vPi\ssp)\ssp$ and \math{
\mvsLrs02^p(\ssp\mu\,,\spp\vPi\ssp) } of certain equivalence classes \math{
\smb X} of such \mathss34{x}.

By saying that \math{\mu} is a {\it positive measure\ssp} on \math{\Omega} we 
mean that \math{\mu} is a function with \math{\bigcup\,\dom\mu=\Omega} and \math{
\rng\mu\inc[\KP1 0\,,\plusinfty\KPt9] } and \math{\dom\mu} a 
\rsigma3algebra and such that \linebreak 
                            \œ$ \sum\,(\ssp\mu\KP1|\KP1\scrmt A\ssp) = 
 \mu\fvalue\big(\sp\bigcup\,\scrmt A\ssp) \ssp$ holds for any countable 
disjoint \mathss36{\scrmt A\inc\dom\mu}. Here we assume that the definitions 
associated with sum conventions are arranged so that \linebreak 
                                                   \œ$ \sum\,\emptyset = 0\ssp$ 
holds. Further, by a \rsigma0{\it algebra\ssp}, usually written 
\q{$\sigma\,$-\,algebra}, we mean any \linebreak 
                                      $\scrmt A\ssp$ such that \math{
\bigcup\,\scrmt A\sp\setminus\ssp\bigcup\,\scrmt B\in\scrmt A} holds for any 
countable \mathss36{\scrmt B\inc\scrmt A}. A positive measure $\mu\ssp$ is 
\rsigma1{\it finite\ssp} if{}f \math{
\bigcup\,\dom\mu\inc\bigcup\,\scrmt A } holds for some countable \mathss35{
 \scrmt A\inc\mu\invss44\image\sp\lbb R_+ }.

To compact language in some, quite rare cases, we introduce the concept of 
measure space as follows. Say that \math{P} is a {\it measure space\ssp} if{}f 
there are \math{\mu} and \math{\Omega} such \linebreak 
                                            that \math{\mu} is a positive 
measure on \math{\Omega} with \mathss38{P = (\ssp\Omega\ssp,\spp\mu\ssp) }. We 
also say that \math{\mu} is a positive measure if{}f \math{\mu} is a positive 
measure on \mathss36{\bigcup\,\dom\mu}, \,and a measure space \math{P} we \linebreak 
                                                                          say 
to be \rsigma6finite in the case where \math{\taurd P} is such. \vskip.5mm

% ¤¤¤¤¤¤¤¤¤¤¤¤¤¤¤¤¤¤¤¤¤¤¤¤¤¤¤¤¤¤¤¤¤¤¤¤¤¤¤¤¤¤¤¤¤¤¤¤¤¤¤¤¤¤¤¤¤¤¤¤¤¤¤¤¤¤¤¤¤¤¤¤¤¤¤¤

\insubsubhead    Measurability of measure\ssp-\sp vector maps             \label{Ss C1}

We consider {\it mv\ssp-\sp map\,}s, short for \q{measure\ssp-\sp vector}, 
which are triplets, i.e.\ ordered pairs \math{ \tilde x = 
 (\ssp x\ssp,\spp\varXi\ssp) = (\ssp x\,;\sp\mu\,,\spp\vPi\ssp) } where in 
turn \math{\varXi = (\ssp\mu\,,\spp\vPi\ssp) } is an {\it mv\ssp-\sp pair\sp}. 
This means that \math{\vPi} is a real or complex topological vector space and \math{
\mu} is a positive measure on some \math{\Omega} and \math{ x : \Omega \to
 \vecs\vPi} is a function. In order to introduce some concepts of 
measurability for such mv\ssp-\sp maps we first put the following

\begin{definitions}\label{df simple}

(1) \ Say that \math{\sigma} is {\it simple\ssp} in \math{\varXi} if{}f 
there are \math{\bosy K\sp,\sp\mu\,,\sp\Omega\,,\sp\vPi} with \math{\mu} 
a po- sitive 
measure on \math{\Omega} and \math{\bosy K\in\setRC} and 
\math{\vPi\in\tvsps0(K)} and 
\math{\varXi=(\ssp\mu\,,\spp\vPi\ssp) } 
and \mathss03{\sigma\in\sp^\Omega\,\vecs\vPi} and 
\math{\rng\sigma} finite and 
\math{\{\,\sigma\invss44\image\snn\{\ssp\xi\ssp\}:\xi\in
\rng\sigma\spp\setminus\{\,\Bnull_\vPi\}\sp\}\inc
\mu\invss44\image\spp\lbb R_+}, \inskipline{.5}2

(2) \ Say that \math{\bosy\sigma} is a {\it simple sequence\ssp} in \math{\varXi} 
if{}f \math{\bosy\sigma\in\sp^\sbbNo\,\Univ} and \inskipline0{41.5}

$\sigma\ssp$ is simple in \math{\varXi} for all \math{\sigma\in
\rng\bosy\sigma}.
  \end{definitions}

Let \math{\text{\efss R\KPt8}\tilde x\,A=
\uniqset\tilde z:\eexi{x\ssp,\sp\mu\,,\sp\vPi}\,
\tilde x=(\ssp x\,;\sp\mu\,,\spp\vPi\ssp) } and \mathss38{
\tilde z=(\ssp x\KP1|\KP1 A\,;\sp\mu\KP1|\KP1\Pows A\,,\spp\vPi\ssp) }.

\begin{def:al schemas}\label{df meas}

For any mv\ssp-\sp map \math{\tilde x=
(\ssp x\ssp,\spp\varXi\ssp)=(\ssp x\,;\sp\mu\,,\spp\vPi\ssp) } with \math{
\Omega=\dom x} 
%%putting \œ$ \roman R\KP1\tilde x\,A =  (\ssp x\KP1|\KP1 A\,;\sp\mu\KP1|\KP1\Pows A\,,\spp\vPi\ssp) \ssp $ , 
 assuming $\vPi\in\tvsps0(K)$ , first say that \inskipline12

(1) \ $\tilde x\ssp$ is {\it measurable\ssp} if{}f \math{
    \{\,x\invss44\image\spp U\sn:U\in\taurd\vPi\KP1\}\inc\dom\mu} holds, \inskipline{.5}2

(2) \ $\tilde x\ssp$ is {\it simply measurable\ssp} if{}f there is \math{
    \bosy\sigma} with \math{\bosy\sigma} a simple sequence in \math{\varXi} \inskipline09

 and \math{\bosy\sigma\to x} in top \mathss35{
 \taurd\vPi\expnota^\ssp\Omega\ssp]_{ti} },   $\hfill \bosy\sigma\in\sp^\sbbNo\ssp\big(\,^\Omega\,\vecs\vPi\ssp) $ \KP6 \inskipline{.5}2

(3) \ $\tilde x\ssp$ is {\it scalarly measurable\ssp} if{}f 
$(\ssp u\circ x\,;\sp\mu\,,\spp\bosy K\ssp)$ is measurable for all 
$u\in\Cal L\,(\sp\vPi\sp,\spp\bosy K\ssp) \,$.\KN9 \inskipline10

Then for any of \q{measurable}, \q{simply measurable} or 
\q{scalarly measurable} in place of \,X\ssp, say that \inskipline12

(4) \ $\tilde x\ssp$ is {\it almost\ssp} X if{}f 
\math{\text{\efss R\KPt8}\tilde x\KP1(\ssp\Omega\sp\setminus\sp N\ssp) } is X 
for some \mathss39{N\in\mu\invss44\image\snn\{\ssp 0\ssp\} }, \inskipline{.5}2

(5) \ $\tilde x\ssp$ is {\it finitely\ssp} X if{}f 
\math{\text{\efss R\KPt8}\tilde x\,A} is X for 
every \mathss36{A\in\mu\invss44\image\sp\lbb R_+}, \inskipline{.5}2

(6) \ $\tilde x\ssp$ is {\it finitely almost\ssp} X if{}f 
\math{\text{\efss R\KPt8}\tilde x\,A} is almost X for 
every \mathss36{A\in\mu\invss44\image\sp\lbb R_+}.
  \end{def:al schemas}

In loose speach, we may express the content of \ref{df meas}\,(2) by saying 
that \math{\bosy\sigma} is a {\sl sequence of simple functions converging 
pointwise\sp} to \mathss34{x}. Then for the \math{\sigma\in\bosy\sigma} there \linebreak
we may also say that \math{(\ssp\sigma\,;\sp\mu\,,\spp\vPi\ssp) } is a 
{\it simple\ssp} mv\ssp-\sp map, and we may loosely say that \linebreak
                                                             $\sigma\ssp$ is a 
{\sl simple function\sp}.

Note that by our definitions above we may also say e.g.\ that \math{\tilde x} 
is {\sl measurable\sp} if{}f there are \math{\mu\,,\sp\Omega\,,\sp\vPi\sp,\sp 
 x} with \math{\Omega\times\snn\{\,\Bnull_\vPi\} } simple in \math{
(\ssp\mu\,,\spp\vPi\ssp) } and \mathss03{ \tilde x = 
 (\ssp x\,;\spp\mu\,,\spp\vPi\ssp) } and \math{x\in\sp^\Omega\,\vecs\vPi } and \mathss36{
x\invss44\images\sp\taurd\vPi\inc\dom\mu}.

\begin{proposition}\label{pro-mea-equ}

Let \PouN$\,\bosy K\in\setRC${\,\rm, }and let \PouN$\,\vPi\in\LCSps0(K)$ be 
normable with $\,\taurd\vPi$ a separable topology. Also let $\,\mu$ be a \sp
\rsigma1finite positive measure on $\,\Omega\,$. If in addition $ \tilde x =
(\ssp x\,;\sp\mu\,,\spp\vPi\dlsigss00\spp)$ with $\, x \in\sp^\Omega\,
\Cal L\,(\spp\vPi\sp,\spp\bosy K\ssp) \KP1 ${\rm, \,}then $\ 
(1) \equivss22 (2) \equivss22 (3) $ \,where {\rm \inskipline{.9}4

(1)} \ $(\ssp\roman{ev}\KPt2\sbi\xi\snn\circ\spp x\,;\sp\mu\,,\spp\bosy K\ssp)$ is 
       measurable for all $\,\xi\in\vecs\vPi${\sp\rm, \inskipline{.4}4

(2)} \ $\tilde x$ is simply measurable{\ssp\rm, \inskipline{.4}4

(3)} \ $\tilde x$ is measurable\ssp.                       \end{proposition}

\begin{proof} Since \math{ \{\KP1 \roman{ev}\sp\sbi\xi \, | \KP1
 \Cal L\,(\spp\vPi\sp,\spp\bosy K\ssp) : \xi\in\vecs\vPi\KP1\} \inc
 \Cal L\,(\spp\vPi\dualsigma0\ssp,\spp\bosy K\ssp) } holds, we trivially have 
the implication \mathss36{(3)\impss22(1)}. Likewise, we trivially have \mathss36{
(2)\impss22(1)}. It now suffices to prove that the implications \math{
(1)\impss22(3)} and \math{(3)\impss22(2)} hold.

For \mathss36{(1)\impss22(3)}, letting \math{\Nu} be some compatible norm for \mathss33{
\vPi}, \,let \math{\Nu\aar 1} be the corresponding dual norm, i.e.\ put \mathss39{
\Nu\aR 1=\seqss44{
\sup\KPt8(\ssp\Abrs00^1\circ\sp u\circss01\Nu\invss44\image\ssbb15 I)
:u\in\Cal L\,(\sp\vPi\sp,\spp\bosy K\ssp)} }. 
\newline$\null\hfill$
For \mathss36{
i\in\bbNo}, \,then put \mathss38{ \roman B\,i =
\Nu\aar 1\!\inve\ssp\image\sp[\KP1 0\,,\sp i\,\yplus\sp\ydot\,] }. Assuming 
(1)\,, if 
\linebreak
now \mathss37{U\in\taurd(\spp\vPi\dlsigss00\spp) }, \,for any 
fixed \math{i\in\bbNo} and for \mathss35{\roman U\,i=U\capss14\roman B\,i }, \,
it suffices to prove that \math{x\invss33\image\sp\roman U\,i\in\dom\mu} 
holds. For \mathss35{ \scriptm9 T\ar 1 =
\taurd(\spp\vPi\dlsigss00\spp)\leiss33\roman B\,i }, \,now \math{
\scriptm9 T\aR 1} is the uniform $\ssp\scriptm{10}U\,$-- \linebreak 
topology where \math{\scriptm{10}U } is the uniformity generated by \PouN$\ssp
\{\KP1\roman V\,\xi\,n:\xi\in D\ssp\text{ and }\ssp n\in\rbb Z^+\,\}$ where \math{
\roman V\,\xi\,n = \roman B\,i\times\roman B\,i \capss31 \{\,
 (\ssp u\ssp,\spp v\ssp) :
  |\KP1 u\fvalss01\xi - v\fvalss01\xi\KP1| < n^{\ssp\mminus 1}\,\} } and \math{
D} is any countable $
\taurd\vPi\leiss22(\,\Nu\invss43\image\sp[\KPt9 0\,,1\KPt8]\ssp)\,$--\,dense 
set. By Alaoglu's theorem, we have \math{\scriptm9 T\aR 1} a compact topology. 
Since \math{\scriptm{10}U } is generated by a countable set, we see that there 
is some \mathss33{\scriptm9 T\aR 1}-- dense and countable set \mathss34{
D\aar 1}. Using (1) and noting that \ 
$x\invss33\image(\ssp\roman V\,\xi\,n\sp\image\{\sp u\sp\}\sp) = {} $ \vskip.4mm\centerline{$
\bigcap\KP1\{\,\{\KP1 t:|\KP1 x\fvalss01 t\fvalss21\xi -
u\fvalss01\xi\KP1| < n^{\ssp\mminus 1}
\ssp\text{ and }\,
|\KP1 x\fvalss01 t\fvalss21\zeta\KP1|\le i\,\yplus\sp\ydot\,\} :
\zeta\in D\KP1\}\KP1$,} \inskipline{.4}0

we see that \math{
x\invss33\image(\ssp\roman V\,\xi\,n\sp\image\{\sp u\sp\}\sp)\in\dom\mu} holds 
when \mathss31{ (\ssp\xi\,,\spp u\ssp,\spp n\ssp) \in
D\times D\ar 1\sn\times\sp\rbb Z^+}. It is left as an exercise to the reader 
to show that \math{\roman U\,i} can be expressed as a union of finite 
intersections of the sets \mathss37{\roman V\,\xi\,n\sp\image\{\sp u\sp\} }. 
Then \math{x\invss33\image\sp\roman U\,i\in\dom\mu} follows.

For \mathss36{(3)\impss22(2)}, we need the assumption that \math{\mu} be 
\rsigma6finite. It is an easy exercise to show that if the implication to be 
established holds for bounded measures, then it holds also for \rsigma6finite 
ones. So we assume that \mathss36{\mu\fvalss01\Omega < \plusinfty}. Further, 
if we can show that for any fixed \mathss37{i\in\bbNo}, \,and for \math{\bar x
=x\capss33(\ssp\Univ\times\roman B\,i\ssp)} with \mathss34{B=\dom\bar x}, \,
the required implication holds for \math{
(\ssp\bar x\,;\sp\mu\KP1|\KP1\Pows B\ssp,\spp\vPi\dlsigss00\spp)} in place of \mathss34{
\tilde x}, \,then it easily follows also for \mathss34{\tilde x}, observing 
that \mathss36{ B = \bigcap \KP1 \{\, \{ \KP1 t :
 |\KP1 x\fvalss01 t\fvalss21\zeta\KP1|\le i\,\yplus\sp\ydot\,\} :
   \zeta\in D\KP1\}\in\dom\mu}.

As just explained, for \mathss36{(3)\impss22(2)}, assuming (3) and making the 
additional assumptions that \math{\mu\fvalss01\Omega < \plusinfty} and \math{
\rng x\inc B\ar 1=\roman B\,i\ar 0 } for some fixed \mathss36{i\ar 0\in\bbNo}, 
we should establish (2)\ssp. For this, we construct \math{ \bosy s \in
 \sp^\sbbNo\ssp\big(\,^\Omega\ssp B\ar 1\spp) } with \math{
(\ssp\smb S\,;\sp\mu\,,\spp\vPi\dlsigss00\spp) } a simple mv\ssp-\ssp map for 
every \mathss36{\smb S\in\rng\bosy s}, \,and such that \math{\bosy s\to x} in 
top \math{\scriptm9 T\expnota^\KPt8\Omega\sp]_{ti} } when we take \linebreak \PouN$
\scriptm9 T=\taurd(\spp\vPi\dlsigss00\spp)\leiss33 B\ar 1\,$. Since \math{
\scriptm9 T} is a compact topology, and also the uniform $\ssp\scriptm{10}U\,
$-- topology with \math{\scriptm{10}U} being countably generated, we may first 
choose some decreasing \linebreak \PouN$
\bosy w:\bbNo\to\scriptm{10}U\ssp$ with \math{\rng\bosy w} a symmetric base 
for \mathss36{\scriptm{10}U}, \,and then some \math{\bosy u:\bbNo\to\Univ} 
with the following property. For every \math{i\in\bbNo} there is \math{
k\in\bbNo} with \math{\bosy u\fvalss01 i\in B\ar 1\sp^k } and $ B\ar 1 \inc
\bigcup\KP1\{\,\bosy w\fvalss01 i\,\image\{\,\bosy u\fvalss01 i\fvalss21 j\,\}
:j\in k\,\}\KPt8$. For short writing \inskipline{.8}{6.75}

$\roman U\,i\,j = 
 \bosy w\fvalss01 i\,\image\{\,\bosy u\fvalss01 i\fvalss21 j\,\} 
 \sp\setminus\, \bigcup \KPt8 \{\,
 \bosy w\fvalss01 i\,\image\{\,\bosy u\fvalss01 i\fvalss21 l\,\}
 : l\in j\,\} $ \ and \inskipline{.4}{6.75}

$\roman A\,i\,j = x\invss33\image\ssp\roman U\,i\,j\,$, \,and taking \vskip.4mm

$ \null\hfill
 \bosy s = \big\langle\, \bigcup \KPt8 \{\,
 \roman A\,i\,j\times\{\,\bosy u\fvalss01 i\fvalss21 j\,\} :
 \bosy u\ssp,\spp i : j\in\dom(\ssp\bosy u\fvalss01 i\ssp)\,\} :
 i\in\bbNo \,\big\rangle \KP1$, \,we are done, \inskipline{.5}0

leaving the required straightforward verifications as exercises to the reader.
  \end{proof}

\begin{example}

Without separability of the topology \mathss30{\taurd\vPi}, \,the implication \mathss30{
(1)\impss11(2)} \linebreak
                in Proposition \ref{pro-mea-equ} need not hold. Indeed, with \math{
1<p<\plusinfty} and \math{\vPi=\lll^p(\ssbb44 I) } and letting \math{\mu} be 
the Lebesgue measure defined for all Lebesgue measurable sets \mathss35{A\inc
 \bbI}, \,taking \math{\tilde x=(\ssp x\,;\spp\mu\,,\spp\vPi\dlsigss00\spp) } 
where \mathss38{x = \seqss34{\roman{ev\sp}\sbi t\,|\KP1\vecs\vPi\snn:t\in\bbI} 
}, \,we trivially have (1) since \linebreak
                        \mathss02{\roman{ev}\KPt2\sbi\xi\snn\circ\spp x\fvalss01 t = 
 x\fvalss01 t\fvalss21\xi = \xi\fvalss11 t }, \,and so \math{
\roman{ev}\KPt2\sbi\xi\snn\circ\spp x\fvalss01 t \not = 0 } only for countably many \math{
t\in\bbI} for \linebreak
              each fixed \mathss30{\xi\in\vecs\vPi }. However \math{\tilde x} 
cannot be simply measurable since otherwise there would exist some countable 
set \math{N\aar 0\inc\bbI} such that \math{x\fvalss01 t\fvalss21\xi\not=0 } 
holds only for vec- \linebreak
                    tors \math{\xi\in\vecs\vPi} with \math{
\xi\invss46[\KP{1.2}\Univ\sp\setminus\snn\{\ssp 0\ssp\}\KP1]\capss42 N\aar 0
 \not=\emptyset} when \mathss34{t\in\bbI}.
  \end{example}

\begin{proposition}\label{Pro rfx si mea}

With \œ$\,\bosy K\in\setRC$ let \œ$\,\vPi\in\BaSps0(K)$ be reflexive{\sp\rm, }%
and let \inskipline06

$(\ssp x\,;\spp\mu\,,\spp\vPi\subsigrs03\spp)$ be simply measurable. Then $\,
(\ssp x\,;\spp\mu\,,\spp\vPi\ssp)$ is simply measurable.
  \end{proposition}

\begin{proof} Putting \mathss35{\Omega=\bigcup\ssp\dom\mu}, \,let \math{
\bosy\sigma} be some simple sequence in \math{
(\ssp\mu\,,\spp\vPi\subsigrs03\spp) } with \math{ \bosy\sigma\to x} in top \mathss35{
 \taurd(\sp\vPi\subsigrs03\spp)\expnota^\ssp\Omega\ssp]_{ti} }, \,and let \math{
S} be the closed linear span of \math{\rng\bigcup\ssp\rng\bosy\sigma} in \mathss31{
\vPi}. Then trivially \math{(\ssp x\,;\spp\mu\,,\spp\vPi\ssp) } is scalarly 
measurable, and by Hahn\,--\,Banach also \math{\taurd\vPi\leiss22 S} is a 
separable topology with \mathss34{\rng x\inc S}. Consequently by 
Pettis' theorem the assertion follows.
  \end{proof}

\begin{lemma}\label{Le Nu_1 ci y meas}

With \œ$\,\bosy K\in\setRC$ let \œ$\,\vPi\in\LCSps0(K)$ be normable with $\,
\Nu$ a compatible norm and $\,\taurd\vPi$ a separable topology. Also let \vskip.5mm\centerline{$
\Nu\aR 1 = \seqss44{ \sup \KPt8(\ssp\Abrs00^1\circ\sp u\circss01
 \Nu\invss44\image\ssbb15 I) : u\in\Cal L\,(\sp\vPi\sp,\spp\bosy K\ssp)} $} \inskipline{.5}0

and let $\,(\KPt5 y\,;\spp\mu\,,\spp\vPi\dlsigss00\spp)$ be scalarly 
measurable. Then $\,
(\ssp\Nu\aR 1\sn\circ\sp y\,;\spp\mu\,,\sn\tfbbR\ssp)$ is measurable.
  \end{lemma}

\begin{proof} Putting \math{\Omega=\sp\bigcup\ssp\dom\mu} and 
taking any countable \math{D} such that \math{D} 

is 
\mathss37{\taurd\vPi\leiss22(\ssp\Nu\invss44\image\ssbb15 I)}--\,dense, 
for every \math{\eta\in\Omega} we 

have \mathss38{\Nu\aR 1\sn\circ\sp y\fvalue\eta=                          \label{Nu_1 = sup ...}
\sup \KPt8(\ssp\Abrs00^1\circ\sp(\ssp y\fvalue\eta\ssp)\spp\image\sn D\ssp)
=\sup\KPt8\{\KPt8\Abrs00^1\circ\KPt2
\roman{ev}\KPt2\sbi\xi\snn\circ\sp y\fvalue\eta:\xi\in D\KP1\} }. Noting that 
by our assumption for every fixed \math{\xi\in D} we have 

that 
\math{
(\ssp\Abrs00^1\circ\KPt2
\roman{ev}\KPt2\sbi\xi\snn\circ\sp y\,;\spp\mu\,,\sn\tfbbR\ssp)
} is 
measurable, the assertion immediately follows.
  \end{proof}

\begin{lemma}\label{Le Nu_1 = sup ...}

With \œ$\,\bosy K\in\setRC$ let \œ$\,\vPi\in\BaSps0(K)$ be reflexive with $\,
\Nu$ a compatible norm and $\,S$ a separable linear subspace in 
$\,\vPi\dlbetss01\,$. Also let \vskip.5mm\centerline{$
\Nu\aR 1 = \seqss44{ \sup \KPt8(\ssp\Abrs00^1\circ\sp u\circss01\Nu\invss44
 \image\ssbb15 I) : u\in\Cal L\,(\sp\vPi\sp,\spp\bosy K\ssp)} \KP1 $.} \inskipline{.5}0

Then there is a countable $\,D\inc\Nu\invss44\image\ssbb10 I$ with 
$\,\Nu\aR 1\ssn\fvalue u=
\sup \KPt8(\ssp\Abrs00^1\circ\sp u\spp\image\sn D\ssp)$ for $\,u\in S\,$.
  \end{lemma}

\begin{proof} Putting \math{E=\vPi\dlbetss01\ssp_{/\,S} } let \math{A} be 
countable and \mathss37{\taurd E}--\,dense, and let \math{\scrmt R} be the set 
of all pairs \math{(\ssp u\ssp,\spp\xi\ssp)\in A\times\vecs\vPi} with 
\math{\Nu\fvalss11\xi\le 1} and \mathss35{u\fvalss01\xi=\Nu\aR 1\ssn\fvalue u
}. By reflexivity and Hahn\,--\,Banach we then have \math{A\inc\dom\scrmt R} 
and hence by {\sp countable choice\sp} there is a function \math{\scrmt P\inc
 \scrmt R} with \mathss31{A\inc\dom\scrmt P}. Let \mathss31{D=\rng\scrmt P}.

Now for all \math{u\in S} we trivially have \mathss38{\sup \KPt8(\ssp
\Abrs00^1\circ\sp u\spp\image\sn D\ssp)\le\Nu\aR 1\ssn\fvalue u }, \,and hence 
assuming that \math{\sup \KPt8(\ssp\Abrs00^1\circ\sp u\spp\image\sn D\ssp)
 < \Nu\aR 1\ssn\fvalue u } holds for some \mathss34{u\in S}, \,it suffices to 
get a contradiction. Taking \math{ \eps = \frac 12\KP1(\ssp\Nu\aR 1\ssn\fvalue 
 u - \sup \KPt8(\ssp\Abrs00^1\circ\sp u\spp\image\sn D\ssp)) } we first find 
some \math{v\in A} with \mathss34{\Nu\aR 1\ssn\fvalue(\ssp u - v\ssp) < \eps}. 
Then for \math{\xi=\scrmt P\fvalss30 v} we have \math{\Nu\fvalss11\xi\le 1} 
and \mathss34{v\fvalss01\xi=\Nu\aR 1\ssn\fvalue v}, \,and hence \ $
  \sup \KPt8(\ssp\Abrs00^1\circ\sp u\spp\image\sn D\ssp)
= \Nu\aR 1\ssn\fvalue u - 2\KPt8\eps 
< \Nu\aR 1\ssn\fvalue u - \Nu\aR 1\ssn\fvalue(\ssp u - v\ssp) - \eps$ \inskipline{.25}{35.7}

${}\le\Nu\aR 1\ssn\fvalue v - \eps = v\fvalss01\xi - \eps = 
      (\ssp v - u\ssp)\fvalss01\xi + u\fvalss01\xi - \eps $ \inskipline{.25}{35.7}

${}\le|\KP1 u\fvalss01\xi\KP1| + \Nu\aR 1\ssn\fvalue(\ssp u - v\ssp) -\eps
   < |\KP1 u\fvalss01\xi\KP1| \KP1 $, \,a contradiction.
  \end{proof}

% ¤¤¤¤¤¤¤¤¤¤¤¤¤¤¤¤¤¤¤¤¤¤¤¤¤¤¤¤¤¤¤¤¤¤¤¤¤¤¤¤¤¤¤¤¤¤¤¤¤¤¤¤¤¤¤¤¤¤¤¤¤¤¤¤¤¤¤¤¤¤¤¤¤¤¤¤

\insubsubhead          Decomposable positive measures                    \label{Ss decos}

Decomposability, as well as being {\sl almost decomposable\sp}, is a property 
for a positive measure $\mu$ that is weaker than the usual 
\q{$\sigma\,$-\,finiteness} which we call \rsigma6finiteness, and that is 
sufficiently strong still to have \mathss38{\mLrs23^\plusinftyy(\ssp\mu\sp) } 
canonically represent the strong dual of \mathss38{\mLrs42^1(\ssp\mu\sp) }. 
For example Haar measures of suitably \q{large} locally compact topological 
groups are almost decomposable but not \rsigma6finite. See 
Example \ref{Exa Haar} on page \pageref{Exa Haar} below for some details 
concerning this assertion.

\begin{definitions}\label{df decomp}

(1) \ Say that \math{N\sprim1} is \mathss37{\mu}{\it--\,negligible\ssp} if{}f \math{
    \mu} is a positive measure with \inskipline09

$N\sprim1\inc \sp\bigcup\,\dom\mu\ssp$ and \mathss38{
 \mu\invss44\image\sp\rbb R^+\ssn\lei\sp N\sprim1 \inc \ssp
 \bigcup\KPt8\{\KPt8\Pows N\sn:N\in\mu\invss44\image\snn\{\ssp 0\ssp\}\sp\} }, \inskipline{.5}2

(2) \ For a positive measure \math{\mu} on \math{\Omega} say that \math{\mu} 
    is {\it almost decomposable\ssp} if{}f there are \mathss30{\scrmt A \inc \label{decos A}
 \mu\invss44\image\sp\rbb R^+} and \math{N\sprim1} with \math{
\scrmt A\cupss42\{\ssp N\sprim1\sp\} } disjoint and \math{ \Omega = 
 \bigcup\,\scrmt A\cupss42 N\sprim1 } and such that \math{N\sprim1} is \mathss37{
\mu}--\,negligible, and such that also \math{N\sprimm1} is \mathss37{\mu
}--\,negligible whenever \œ$\ssp N\sprimm1\inc\Omega$ \linebreak
                                                      is such that for every \math{
A\in\scrmt A} there is \math{N} with \mathss38{A \capss31 N\sprimm1 \inc N \in 
 \mu\invss44\image\snn\{\ssp 0\ssp\} }, \inskipline{.5}2

(3) \ For a positive measure \math{\mu} on \math{\Omega} say that \math{\mu} 
    is {\it decomposable\ssp} if{}f there is some disjoint \math{\scrmt A \inc 
 \mu\invss44\image\sp\lbb R_+} with \math{\Omega=\bigcup\,\scrmt A } and such 
that every \math{N\sprim1\inc\Omega} is \mathss37{\mu}--\,negligible whenever \math{
\scrmt A\leiss42 N\sprim1 \inc \ssp \bigcup\KPt8\{\KPt8\Pows N \sn : 
 N\in\mu\invss44\image\snn\{\ssp 0\ssp\}\sp\} } holds, \inskipline{.5}2

(4) \ For a positive measure \math{\mu} on \math{\Omega} say that \math{\mu} 
    is {\it truly decomposable\ssp} if{}f there is some disjoint \math{
\scrmt A\inc\mu\invss44\image\sp\lbb R_+} with \math{\Omega=\bigcup\,\scrmt A } 
and such that \math{N\in\mu\invss44\image\snn\{\ssp 0\ssp\} } holds for every 
\math{N\inc\Omega} with \mathss38{\scrmt A\leiss42 N \inc 
 \mu\invss44\image\snn\{\ssp 0\ssp\} }.
  \end{definitions}

Trivially \rsigma6finite positive measures are truly decomposable, and these 
in turn are decomposable by 
Proposition \ref{Propo top-deco} below. If 
\math{\mu} is a positive measure on \math{\Omega} such that \mathss03{
\mu\invss44\image\sp\rbb R^+=\emptyset} holds, \,then trivially every \math{
A\inc\Omega} is \mathss37{\mu}--\,negligible, and hence \math{\mu} is 
almost decomposable. A positive measure \math{\mu} on \math{ \Omega = 
 \bbR\timesn\bbR} that is decomposable but not truly decomposable is given in 
Example \ref{Exa not trul deco} on page \pageref{Exa not trul deco} below. It 
seems to be quite difficult to find positive measures that {\sl are not\sp} 
almost decomposable. \hfill See Problem \ref{Prblm z-z mea} on page \pageref{Prblm z-z mea} 
as well as the subsequent examples and problems.

\begin{definitions}\label{df top deco}

(1) \ Say that \math{\scrmt T} {\it positively almost \eit Radonizes\ssp} \math{
    \mu} if{}f there is \math{\Omega} with \math{\mu} a \linebreak 
                                                        positive measure on \math{
\Omega} and \math{(\ssp\Omega\,,\spp\scrmt T\,) } a locally compact Hausdorff 
topological space such that for \math{\scrmt K=\{\,K:K\ssp\text{ is \mathss37{
 \scrmt T}--\,compact } \} } it holds that \math{ \scrmt K \inc 
 \mu\invss44\image\spp\lbb R_+ } and also for all \math{ A \in 
 \mu\invss44\image\spp\lbb R_+ } it holds that \mathss38{ \mu\fvalue\ssn A =
 \sup\KPt8\{\,\mu\fvalue K:K\in\scrmt K\capss22\Pows A\KP1\} }, \inskipline{.5}2

(2) \ Say that \math{\scrmt T} {\it positively \eit Radonizes\ssp} \math{\mu \label{df pos rdz} } 
    if{}f there is \math{\Omega} with \math{\mu} a positive measure on \math{
\Omega} and \math{(\ssp\Omega\,,\spp\scrmt T\,) } a locally compact Hausdorff 
topological space such that for \œ$\ssp\scrmt K={\ssn}$ \linebreak 
                                                \œ$\{\,K:K\ssp\text{ is \mathss37{
 \scrmt T}--\,compact } \} \ssp$ it holds that \math{\scrmt T\inc\dom\mu } and \math{
\scrmt K\inc\mu\invss44\image\spp\lbb R_+ } and also for all \math{ U \sn\in 
 \scrmt T} it holds that \mathss38{\mu\fvalss01 U = \sup\KPt8\{\,\mu\fvalue K:
 K\in\scrmt K\capss22\Pows U\KP1\} } and for all \œ$\ssp A \in 
 \mu\invss44\image\spp\lbb R_+\ssn$ \linebreak
                                    it holds that \mathss38{ \mu\fvalue\ssn A 
 = \inf\,\{\,\mu\fvalss01 U:A\inc U\in\scrmt T\KP1\} }, \inskipline{.5}2

(3) \ Say that \math{\mu} is {\it positive almost \eit Radonian\ssp} if{}f \inskipline0{11}

    there is \math{\scrmt T} such that \math{\scrmt T} positively almost 
\erm Radonizes \mathss35{\mu}, \inskipline{.5}2

(4) \ Say that \math{\mu} is {\it positive \eit Radonian\ssp} if{}f \label{df pos Radon} \inskipline0{11}

    there is \math{\scrmt T} such that \math{\scrmt T} positively 
\erm Radonizes \mathss35{\mu}, \inskipline{.5}2

(5) \ Say that \math{\mu} is {\it topologically almost decomposable\ssp} if{}f 
    there are \math{\scrmt A\,,\sp\scrmt T\spp,\sp N\sprim1} such that \math{
\scrmt T} positively almost \erm Radonizes \mathss35{\mu} and for \math{
\scrmt K = \{\,K:K\ssp\text{ is \mathss37{\scrmt T}--\,compact } \} } it holds 
that \math{\scrmt A\inc\scrmt K\sp\setminus 1\spp\adot } and \math{
\scrmt K\leiss32 N\sprim1 \inc \ssp \bigcup\KPt8\{\, \Pows N : N \in
 \mu\invss44\image\snn\{\ssp 0\ssp\}\sp\} } and \math{
\scrmt A\cupss42\{\,N\sprim1\sp\} } is disjoint with \math{\bigcup\,\dom\mu = 
 \bigcup\,\scrmt A\cupss42 N\sprim1 } and \math{\scrmt A\leiss42 K} is 
countable for all \math{K\in\scrmt K} and also \math{ A\capss33 U \in 
 \mu\invss44\image\spp\rbb R^+\sn\cup\sp 1\spp\adot } holds for all \math{
A\in\scrmt A} and \mathss30{U\sn\in\scrmt T}, \inskipline{.5}2

(6) \ Say that \math{\mu} is {\it topologically decomposable\ssp} if{}f there 
    are \math{\scrmt A\,,\sp\scrmt T\spp,\sp N\sprim1} such that \math{
\scrmt T} positively \erm Radonizes \mathss35{\mu} and such that for 
\math{\scrmt K=
\{\,K:K\ssp\text{ is \mathss37{\scrmt T}--\,compact } \} } it 
holds that $\scrmt A\inc\scrmt K\sp\setminus 1\spp\adot$ and 
$\scrmt K\leiss32 N\sprim1\inc\mu\invss44\image\snn\{\ssp 0\ssp\}$ and 
$\scrmt A\cupss42\{\,N\sprim1\sp\}$ is disjoint with 

$\bigcup\,\dom\mu=\bigcup\,\scrmt A\cupss42 N\sprim1$ 
and \math{\scrmt A\leiss42 K} is countable for all \math{K\in\scrmt K} 

and \math{
A\capss33 U\in\mu\invss44\image\spp\rbb R^+\sn\cup\sp 1\spp\adot } holds for 
all \math{A\in\scrmt A} and \mathss30{U\sn\in\scrmt T}.
  \end{definitions}

Note that the condition \math{ A\capss33 U \in 
 \mu\invss44\image\spp\rbb R^+\sn\cup\sp 1\spp\adot } in (5) and (6) of 
Definitions \ref{df top deco} above means that we have \math{ A\capss33 U \in 
 \mu\invss44\image\spp\rbb R^+} or \math{A\capss33 U \in 1\spp\adot = 
 \{\ssp\emptyset\ssp\} } which in turn is equivalent to having \math{
0 < \mu\fvalue(\sp A\capss33 U\ssp) < \plusinfty} or \mathss35{ A\capss33 U = 
 \emptyset}. Note also the impli- cations \mathss35{
\mu\fvalue(\sp A\capss33 U\ssp) \in \rbb R^+ \impss13 
 \mu\fvalue(\sp A\capss33 U\ssp) \not= \Univ \impss33 
 A\capss33 U\in\dom\mu}. That positive (\sp almost\sp) \erm Radonian measures 
are topologically (\sp almost\sp) decomposable, and that these in turn are 
almost decomposable is seen from the next

\begin{proposition}\label{Propo top-deco}

For the properties given below the implications \œ$\,(5)\impss11(6)\impss11(7)$ 
and $\,(3)\impss11(4)\impss11(7)$ and $\,(1)\impss11[\KPp1.4(2)\text{ and }(3)\KPp1.4
 ]$ and $\,(2)\impss11(4)$ hold. {\rm\inskipline{.6}4

(1) \ }$\mu$ is positive \eit Radonian{\sp\rm,\inskipline{.2}4

(2) \ }$\mu$ is positive almost \eit Radonian{\sp\rm,\inskipline{.2}4

(3) \ }$\mu$ is topologically decomposable{\sp\rm,\inskipline{.2}4

(4) \ }$\mu$ is topologically almost decomposable{\sp\rm,\inskipline{.2}4

(5) \ }$\mu$ is truly decomposable{\sp\rm,\inskipline{.2}4

(6) \ }$\mu$ is decomposable{\sp\rm,\inskipline{.2}4

(7) \ }$\mu$ is almost decomposable.
  \end{proposition}

\begin{proof} For \math{(1)\impss11(2)} letting \math{\scrmt K\,,\spp\scrmt T\sp
,\sp\mu} be as in Definitions \ref{df top deco}\,(2) above, we need to verify 
that for \math{ A \in \mu\invss44\image\spp\lbb R_+ } we have \mathss38{ 
\mu\fvalue\ssn A = \sup\KPt8\{\,\mu\fvalue K:K\in\scrmt K\capss22\Pows A\KP1\} 
}. Thus for given \math{\eps\in\rbb R^+} it suffices to find some \math{K\aR 1
 \in\scrmt K\capss22\Pows A} with \mathss34{\mu\fvalue\ssn A - \eps < 
 \mu\fvalue K\aR 1}. Now, we first find some set \math{U\sn\in\scrmt T} with \math{
A\inc U} and \mathss35{\mu\fvalss01 U < \mu\fvalue\ssn A + \frac 13\KP1\eps}. 
Then we find some \math{V\sn\in\scrmt T} with \math{U\ssp\setminus A\inc V} 
and \mathss35{\mu\fvalss01 V < 
 \mu\fvalue(\ssp U\ssp\setminus A\ssp) + \frac 13\KP1\eps}. We further find 
some \math{K\in\scrmt K} with \math{K\inc U} and \mathss31{
\mu\fvalss01 U - \frac 13\KP1\eps < \mu\fvalue K}, \,and taking \math{K\aR 1 = 
 K\ssp\setminus V} we now see  that \math{K\aR 1\in\scrmt K} holds with \mathss36{
K\aR 1\inc A}. Furthermore, we have \inskipline{.5}{14.6}

$\mu\fvalss01 V < 
 \mu\fvalue(\ssp U\ssp\setminus A\ssp) + \frac 13\KP1\eps
=\mu\fvalss01 U - \mu\fvalue\ssn A + \frac 13\KP1\eps < \frac 23\KP1\eps
$ \ and hence \vskip.2mm\centerline{$
    \mu\fvalue\ssn A - \eps \le \mu\fvalss01 U - \eps
=   \mu\fvalss01 U - \frac 13\KP1\eps - \frac 23\KP1\eps
<   \mu\fvalue K - \mu\fvalss01 V \le \mu\fvalue K\aR 1\,$.} \vskip.5mm

Having the above, for the proofs of \math{(1)\impss11(3)} and \math{
(2)\impss11(4)} letting \math{\scrmt K\,,\spp\scrmt T\sp,\sp\mu} be as in 
Definitions \ref{df top deco}\,(5) above, it suffices to show existence of \math{
\scrmt A} and \math{N\sprim1} such that we have \math{ \scrmt A \inc
 \scrmt K\sp\setminus 1\spp\adot } and \math{\scrmt K\leiss32 N\sprim1 \inc 
 \mu\invss44\image\snn\{\ssp 0\ssp\} } and such that \math{
\scrmt A\cupss42\{\,N\sprim1\sp\} } is disjoint with \math{ \bigcup\,\dom\mu = 
 \bigcup\,\scrmt A\cupss42 N\sprim1 } and for all \math{K\in\scrmt K} we have 
that \math{\scrmt A\leiss42 K} is countable and also \math{
A\capss33 U\in\mu\invss44\image\spp\rbb R^+\sn\cup\sp 1\spp\adot } holds for 
all \math{A\in\scrmt A} and \mathss30{U\sn\in\scrmt T}.

To get such \math{\scrmt A\,,\sp N\sprim1} we let \math{\Cal K} be the set of 
all disjoint \math{\scrmt A\inc\scrmt K\sp\setminus 1\spp\adot } with the 
property that \math{ A\capss33 U \in 
 \mu\invss44\image\spp\rbb R^+\sn\cup\sp 1\spp\adot } holds for all \math{A\in
 \scrmt A} and \mathss30{U\sn\in\scrmt T}. Since trivi- ally \math{\emptyset 
 \in\Cal K} holds, by {\sl Zorn's lemma\sp} there is some \math{\scrmt A} that 
is maximal in \mathss34{\Cal K}. Then we take \mathss36{ N\sprim1 = 
 \bigcup\,\dom\mu\spp\setminus\sp\bigcup\,\scrmt A }.

We first show that \math{\scrmt A\leiss42 K} is countable for all \mathss35{
K\in\scrmt K}. So we fix \math{K} and using local compactness of \math{
\scrmt T} find a relatively \mathss37{\scrmt T}--\,compact \math{ U \sn \in
 \scrmt T} with \mathss30{K\inc U}. Then we have \math{
    \sum_{\,A\ssp\in\ssp\scrm7 A\,}\mu\fvalue(\sp A\capss33 U\ssp) 
\le \mu\fvalue U < \plusinfty } from which it follows that the set \vskip.3mm

$\scrmt A\capss31\{\,A:\mu\fvalue(\sp A\capss33 U\ssp)
\not=0\KP1\} \ssp$ is countable. Since we have 

an injection $\ssp \scrmt A\leiss42 K\ssp\setminus 1\spp\adot
\to
\scrmt A\capss31\{\,A:\mu\fvalue(\sp A\capss33 U\ssp)
\not=0\KP1\} $ 

\noin
given by \mathss35{A\capss31 K\mapsto A}, \,the assertion follows.

To establish \math{ \scrmt K\leiss32 N\sprim1 \inc 
 \mu\invss44\image\snn\{\ssp 0\ssp\} } indirectly, suppose that we have some \math{
K\aar 0\in\scrmt K} with \mathss38{K\aar 0\capss02 N\sprim1 \not \in
 \mu\invss44\image\snn\{\ssp 0\ssp\} }. Since \math{\scrmt A\leiss42 K\aar 0} 
is countable, we first see 

that \math{K\aar 0\capss02 N\sprim1\in\dom\mu} holds, and hence we have 
\mathss30{K\aar 0\capss02 N\sprim1\in
\mu\invss44\image\spp\rbb R^+}. Then 

by \math{ 
\mu\fvalue(\sp K\aar 0\capss02 N\sprim1\sp) = 
\sup\KPt8\{\,\mu\fvalue K:K\in\scrmt K\capss22
\Pows(\sp K\aar 0\capss02 N\sprim1\sp)\,\} } we find 

some \math{K\aR 1\in
\mu\invss44\image\spp\rbb R^+\sn\cap\ssp\scrmt K} with 
\mathss30{K\aR 1\inc K\aar 0\capss02 N\sprim1}, \,and we take 

$A = K\aR 1\ssn\setminus\sp\bigcup\KP1(\ssp
 \mu\invss44\image\snn\{\ssp 0\ssp\}\capss22(\ssp\scrmt T\leiss31 K\aR 1)) $ 

having now $A\in\scrmt K$ with 
$A\capss33 U\in\mu\invss44\image\spp\rbb R^+\sn\cup\sp 1\spp\adot$ 
for all $U\sn\in\scrmt T$.

Indeed, with \math{U\sn\in\scrmt T} supposing that \math{ \emptyset \not = 
 A\capss33 U\not\in\mu\invss44\image\spp\rbb R^+ } holds, we then have 
$A\capss33 U\in\mu\invss44\image\snn\{\ssp 0\ssp\}$ , and 
{\sl for the moment supposing\sp} (\sp$*$\sp) that also 

$K\aR 1\ssn\setminus A\capss33 U\in\mu\invss44\image\snn\{\ssp 0\ssp\}$ 
holds, we 

obtain \math{\emptyset\not=A\capss33 U\inc 
K\aR 1\sn\cap\ssp U\inc\bigcup\KP1(\ssp
 \mu\invss44\image\snn\{\ssp 0\ssp\}\capss22(\ssp\scrmt T\leiss31 K\aR 1))
 } and 

hence \mathss36{A\capss34\bigcup\KP1(\ssp
 \mu\invss44\image\snn\{\ssp 0\ssp\}\capss22(\ssp\scrmt T\leiss31 K\aR 1))
\not=\emptyset }, \,a {\sl contradiction\sp}. 

Since now \math{A\capss35\bigcup\,\scrmt A\inc K\aR 1\capss05\bigcup\,\scrmt A
 \inc N\sprim1\capss05\bigcup\,\scrmt A = \emptyset} holds, by the maximality 
of \linebreak 
   $\scrmt A\ssp$ we have \math{A=\emptyset} and hence \mathss38{ K\aR 1 = 
 \bigcup\KP1(\ssp
 \mu\invss44\image\snn\{\ssp 0\ssp\}\capss22(\ssp\scrmt T\leiss31 K\aR 1)) }. 
From this we see existence of some set \math{V} with \math{K\aR 1\inc \sp
V\sn\in\scrmt T} and \mathss36{ \mu\fvalue K\aR 1 = 
 \mu\fvalue(\ssp V\capss01 K\aR 1)=0 }. Hence we obtain \mathss38{ K\aR 1 \in
 \mu\invss44\image\spp\rbb R^+\sn\cap\ssp(\ssp
 \mu\invss44\image\snn\{\ssp 0\ssp\}\sp\sbig)0 }, \,a {\sl contradiction\sp}.

So, to finish the indirect proof of 
\mathss39{\scrmt K\leiss32 N\sprim1\inc\mu\invss44\image\snn\{\ssp 0\ssp\} 
}, \,we must show that (\sp$*$\sp) above holds. Indeed, 
in the contrary case we have \math{
K\aR 1\ssn\setminus A\capss33 U\in\mu\invss44\image\spp\rbb R^+} 
and then, as above, we find 
some \math{K\aar 2\in
\mu\invss44\image\spp\rbb R^+\sn\cap\ssp\scrmt K} 

with 
\mathss38{K\aar 2\inc 
K\aR 1\ssn\setminus A\capss33 U
\inc K\aR 1\ssn\setminus A\inc
\bigcup\KP1(\ssp
 \mu\invss44\image\snn\{\ssp 0\ssp\}\capss22(\ssp\scrmt T\leiss31 K\aR 1))
}, \,and further 
some 
\math{V\sn\in\scrmt T} with 
\math{K\aar 2\inc V\capss01 K\aR 1} and 
\mathss36{\mu\fvalue(\ssp V\capss01 K\aR 1)=0}. 
Hence we obtain 
\mathss38{K\aar 2\in\mu\invss44\image\snn\{\ssp 0\ssp\} }, \,a 
{\sl contradiction\sp}.

For \math{(5)\impss11(6)} letting \math{\mu\,,\sp\Omega\,,\sp N\sprim1} be as 
in Definitions \ref{df decomp}\,(3) above and letting $\ssp\scrmt A$ \linebreak 
                                                                     be as in 
(4) there, using the {\sl axiom of choice\sp} we find \math{N} with \math{
N\sprim1\inc N\inc\Omega } and such that \math{\scrmt A\leiss42 N \inc 
 \mu\invss44\image\snn\{\ssp 0\ssp\} } holds. Now it easily follows that \math{
N\sprim1} is \mathss37{\mu}--\,negligible.

For \math{(6)\impss11(7)} letting \math{\mu\,,\sp\Omega} be as 
in Definitions \ref{df decomp}\,(2) and letting \math{\scrmt A\sprim0} stand 
for the \math{\scrmt A} in (3) there, we take \math{ \scrmt A = 
 \mu\invss44\image\sp\rbb R^+\snn\cap\KPt5\scrmt A\sprim0 } and \mathss38{
N\sprim1=\ssp\bigcup\KP1(\ssp\mu\invss44\image\snn\{\ssp 0\ssp\}\capss22
 \scrmt A\sprim0\sp) }. It is now a simple exercise to show that \math{
N\sprim1} is \mathss37{\mu}--\,negligible, and that also the condition 
concerning \math{N\sprimm1} there holds.

By \math{(1)\impss11(2)} we trivially have \mathss37{(3)\impss11(4)}, \,and 
for \math{(4)\impss11(7)} letting \œ$\ssp\scrmt A\,,\sp\scrmt K\,,\sp\scrmt T\spp
,$ \linebreak
 \œ$N\sprim1,\sp\mu\ssp$ be as in Definitions \ref{df top deco}\,(5) it 
suffices to show that \math{\scrmt A \inc \mu\invss44\image\sp\rbb R^+ } holds 
and that \math{N\sprim1} is \mathss37{
\mu}--\,negligible, and that also \math{N\sprimm1} is \mathss37{\mu
}--\,negligible whenever \mathss30{N\sprimm1\inc\Omega \sn} \mathss03{
\sn=\bigcup\ssp\dom\mu } is 
such that for \math{
A\spp\ar 1\in\scrmt A} there is \math{N} with \mathss38{
A\spp\ar 1\snn\cap\sp N\sprimm1 \inc N \in 
 \mu\invss44\image\snn\{\ssp 0\ssp\} }.

Now for \math{\scrmt A \inc \mu\invss44\image\sp\rbb R^+ } taking \math{ A \in 
 \scrmt A} we have \math{A\not=\emptyset } and hence \mathss30{ A = 
 A\capss32\Omega\in\sn} \mathss03{
                   \mu\invss44\image\spp\rbb R^+\sn\cup\sp 1\spp\adot } whence 
\mathss30{A\in\mu\invss44\image\spp\rbb R^+ \sn}. To show that \math{N\sprim1} 
is \mathss37{\mu}--\,negligible, given any \mathss03{ A \in 
 \mu\invss44\image\spp\rbb R^+ } we must find some \math{N} with \mathss38{
A\capss31 N\sprim1\inc N\in\mu\invss44\image\snn\{\ssp 0\ssp\} }. Using 
{\sl countable choice\sp}, we first find an increasing \math{ \bmii8 K \in \sp
 ^\sbbNo\,(\ssp\scrmt K\capss22\Pows A\ssp) } with \mathss36{ \mu\circ\bmii8 K 
 \to\mu\fvalue\ssn A }. Then we find \math{ \bmii8 N \in \sp 
 ^\sbbNo\,(\ssp\mu\invss44\image\snn\{\ssp 0\ssp\}\sp\sbig)0 } with \math{
\bmii8 K\fvalss51 i\capss42 N\sprim1\inc\bmii8 N\fvalss41 i } for all \mathss36{
i\in\bbNo}. Now it suffices to take \mathss32{ N = 
 A\spp\setminus\bigcup\ssp\rng\bmii8 K\cupss54\bigcup\ssp\rng\bmii8 N }.

To get the assertion concerning \math{N\sprimm1} given \math{A} we first take \math{
\bmii8 K} as above. Then noting that \math{
\scrmt A\leiss32(\ssp\bmii8 K\fvalss51 i\ssp) } is countable for every \math{
i\in\bbNo } by {\sl countable choice\sp} we get \mathss03{
\bmii8 K\aR 1\in\sp^\sbbNo\,(\ssp\scrmt K\capss22\Pows A\ssp) } with \math{
\bigcup\ssp\rng\bmii8 K\ssp\setminus N\sprim1\inc
\bigcup\ssp\rng\bmii8 K\aR 1 } and such that for every \mathss30{i\in\bbNo } 
there is some \math{A\spp\ar 1\in\scrmt A } with \mathss34{
\bmii8 K\aR 1\sn\fvalue\sp i\inc A\spp\ar 1 }. Then again by {\sl countable 
choice\sp} we get \mathss03{ \bmii8 N \in \sp
 ^\sbbNo\,(\ssp\mu\invss44\image\snn\{\ssp 0\ssp\}\sp\sbig)0 } with \math{
\bmii8 K\aR 1\sn\fvalue\sp i\capss42 N\sprimm1\inc\bmii8 N\fvalss41 i } for 
all \mathss36{i\in\bbNo }. Since we already know that \math{N\sprim1} is \mathss37{
\mu}--\,negligible, we find some \math{N\aar 1} with \mathss38{
A\capss31 N\sprim1\inc N\aar 1\in\mu\invss44\image\snn\{\ssp 0\ssp\} }. Now it 
suffices to take \mathss32{ N = A\spp\setminus\bigcup\ssp\rng\bmii8 K\cupss54
 \bigcup\ssp\rng\bmii8 N\cupss52 N\aar 1 } to get some \math{N} such that we 
have \mathss38{ A\capss32 N\sprimm1 \inc N \in 
 \mu\invss44\image\snn\{\ssp 0\ssp\} }.
  \end{proof}

The idea for the proofs of \math{(1)\impss11(3)} and \math{(2)\impss11(4)} 
above is taken from \cite[Proposition 4.14.9\ssp, p.\ 229]{Edw}\,. Note that 
the logical structure of these proofs of the implication \math{(\ssp i\ssp)
 \impss33(\,i\ssp\yplus\sn\yplus\spp) } is basically the following: \ 
{\sl Axioms\sp} \ $\vdash\KN{1.64}\RHB{.35}-$ \inskipline{.5}5

$[\KPp1.4(\ssp i\ssp)\impss33\eexi{\scrmt Z}\,\mfrak P\ssp$ and \math{[\KPp1.4
 \mfrak Q\ssp\text{ or }\neg\KP1\mfrak Q\KPp1.4]\KP1]} and \inskipline{.3}5

$[\KP1[\KPp1.4\eexi{\scrmt Z}\,\mfrak P\ssp\text{ and }\ssp\mfrak Q\KPp1.4]
 \impss33(\,i\ssp\yplus\sn\yplus\spp)\KPp1.4] \ssp$ and \math{[\KP1[\KPp1.4
 \aall{\scrmt Z}\,\mfrak P\ssp\text{ and }\ssp\neg\KP1 \mfrak Q\KPp1.4]\impss33
 \mfrak R\ssp\text{ or }\neg\KP1\mfrak R\KPp1.4]} \inskipline0{26}

and \math{ [\KPp1.4 \mfrak R \impss33 \mfrak S\ar 0\ssp\text{ and }\ssp \neg\KP1
 \mfrak S\ar 0 \KP1 ] } and \math{ [\KP1 \neg\KP1\mfrak R \impss33 \mfrak R^*\sp\text{ 
 or }\ssp \neg\KP1\mfrak R^* \KP1 ] } \inskipline0{26}

and \math{ [\KPp1.4 \mfrak R^* \impss33 \mfrak S\ar 1\sp\text{ and }\ssp \neg\KP1
 \mfrak S\ar 1 \, ] } and \mathss39{ [\KP1 \neg\KP1 \mfrak R^* \impss33 
 \mfrak S\ar 2\ssp\text{ and }\ssp \neg\KP1\mfrak S\ar 2 \KP1 ] }.

\begin{lemma}[schema]\label{Le deco meas}

Let $\,\mu$ be a positive measure on $\,\Omega${\,\rm, }and with \œ$\,\bosy K
 \in\setRC$ let \œ$\,\vPi\in\tvsps0(K)$ hold. Also let \œ$\, x \in \sp
 ^\Omega\,\vecs\vPi$ and let {\,\rm X} stand for any of {\rm\,\q{almost}} or 
{\rm\,\q{almost scalarly}} or {\rm\,\q{almost simply}}. Further{\sp\rm, }let $\,
\scrmt A\,,\sp N\sprim1$ be as in {\,\rm Definitions \ref{df decomp}\,(2)} on 
page {\,\rm\pageref{decos A}} above. If \œ$\,
(\ssp x\KP1|\KP1 A\KPt8;\sp\mu\KP1|\KP1\Pows A\,,\spp\vPi\ssp)$ is {\,\rm X} 
measurable for all \œ$\,A\in\scrmt A${\,\rm, }then $\,
(\ssp x\,;\spp\mu\,,\spp\vPi\ssp)$ is finitely {\,\rm X} measurable.
  \end{lemma}

\begin{proof} Noting that we can write \q{X measurable} in the form \q{almost 
Z}, assuming that \math{
(\ssp x\KP1|\KP1 A\KPt8;\sp\mu\KP1|\KP1\Pows A\,,\spp\vPi\ssp) } is X 
measurable for every \mathss36{A\in\scrmt A}, \,for given \linebreak
                                                 \mathss03{
A\sp\ar 0\in\mu\invss44\image\spp\rbb R^+} it suffices to find some \math{
N\in\mu\invss44\image\snn\{\ssp 0\ssp\} } such that for \math{ B = 
 A\sp\ar 0\ssn\setminus N } and \math{\mu\ar 0=\mu\KP1|\KP1\Pows B } and \math{
x\ar 0=x\KP1|\KP1 B} it holds that \math{
(\ssp x\ar 0\,;\spp\mu\ar 0\,,\spp\vPi\ssp)} is Z\ssp. For this letting \math{
\scrmt N} be the set of all pairs \math{(\sp A\,,\spp N\aar 1)} with \math{
A\in\scrmt A} and \math{N\aar 1\in\mu\invss44\image\snn\{\ssp 0\ssp\} } and 
such that for \math{A\spp\ar 1=A\spp\setminus N\aar 1 } it holds that \math{
(\ssp x\KP1|\KP1 A\spp\ar 1\ssp;\sp\mu\KP1|\KP1\Pows A\spp\ar 1\sp,\spp
 \vPi\ssp) } is Z\ssp, by our assumption \math{\scrmt A\inc\dom\scrmt N} 
holds, and hence by the {\sl axiom of choice\sp} there is a function \math{
\scrmt N\spp\ar 0\inc\scrmt N } with \mathss36{ \scrmt A \inc 
 \dom\scrmt N\spp\ar 0}. Then taking \math{ \scrmt A\sp\ar 0 = 
 \scrmt A\capss41\{\,A : A\capss31 A\sp\ar 0 \in 
 \mu\invss44\image\spp\rbb R^+\sp\} } we have \math{\scrmt A\sp\ar 0} 
countable, and for \math{N\sprimm1 = N\sprim1 \cupss04 \bigcup \KP1 (\ssp
 \scrmt N\spp\ar 0\sn\image\sn\scrmt A\sp\ar 0\spp) \cupss24 \bigcup\KP1 (\ssp
 \scrmt A\setminus\scrmt A\sp\ar 0\spp) \capss21 A\sp\ar 0} it holds that \math{
N\sprimm1} is \mathss37{\mu}--\,negligible. Since \math{N\sprim1\inc A\sp\ar 0} 
holds, for some \math{N} we have \mathss38{N\sprim1 \inc N \in 
 \mu\invss44\image\snn\{\ssp 0\ssp\} }, \,and taking \math{B} as above, it is 
straightforward to verify that \math{
(\ssp x\ar 0\,;\spp\mu\ar 0\,,\spp\vPi\ssp)} is Z\ssp.
  \end{proof}

% ¤¤¤¤¤¤¤¤¤¤¤¤¤¤¤¤¤¤¤¤¤¤¤¤¤¤¤¤¤¤¤¤¤¤¤¤¤¤¤¤¤¤¤¤¤¤¤¤¤¤¤¤¤¤¤¤¤¤¤¤¤¤¤¤¤¤¤¤¤¤¤¤¤¤¤¤

\insubsubhead          Integration of scalar functions                    \label{Ss int scal}

Since for arbitrary functions \math{ u \inc 
 \Omega\times[\KPp1.1 0\,,\plusinfty\KPt9] } we need to consider upper 
inte- grals \math{\upint u\rmdss11\mu } where \math{\mu} is a positive measure 
on \mathss36{\Omega}, \,we here shortly give the asso- ciated formal 
definitions in order to make things precise. Note that in the definition of 
the Lebesgue quasi\ssp-\ssp norm \math{\|\,x\,\|\Lnorss33^p_\mu } in 
Constructions \ref{Ctr |x|_lL^p}\,(2) on page \pageref{ctr L^p-norm} above we 
already implicitly used the concept of upper integral.

\begin{constructions}[\sp positive, real and pseudo\ssp-\sp usual integrals]\label{defi re scal int} $\null$ \vskip.5mm

\begin{enumerate}\begin{myLeftskip}{-4}{.6}{.6}

\item \ $\loint u\rmdss20\mu = \uniqset t: \mu \in \sp 
      ^{\dom\sn\mu}\KP1[\KP{1.1} 0\,,\plusinfty\KP1] \ssp$ and \newskline{32}

 $ u \in\sp ^{\dom\snn\eightmath u}\KP1[\KP{1.1} 0\,,\plusinfty\KP1] \ssp$ and $\ssp
     \dom u\inc\bigcup\ssp\dom\mu \ssp$ and \newskline{28}

 $ t = \sup\sp\big\{\ssp\sum\,\seq{\KPt8
   t\cdotn(\ssp\mu\spp\fvalue(\sp\sigma\invss33\image\snn\{\sp t\ssp\}\sp))
   : \sigma : t\in\rng\sigma\KPt8} : \mu : $ \newskline{25.5}

 $\sigma\in\sp^{\dom\snn\eightmath\sigmaa}\KPt8\lbb R_+ \ssp$ and $\ssp
  \dom\sigma\inc\{\,\eta:\sigma\spp\fvalue\eta\le u\fvalue\eta\,\}$ \newfline

 and $\ssp\rng\sigma\ssp$ is finite and $\ssp \{\,\sigma\invss33\image\snn\{\ssp
     t\ssp\} : t \in \Univ\KP1\} \inc \dom\mu \,\big\} \KP1 $, \KP9

\item \ $\upint u\rmdss20\mu = \uniqset t : \mu \in \sp     \label{ctr upint}
      ^{\dom\sn\mu}\KP1[\KP{1.1} 0\,,\plusinfty\KP1]  \ssp$ and \newskline{31.9}

 $u \in\sp ^{\dom\snn\eightmath u}\KP1[\KP{1.1} 0\,,\plusinfty\KP1] \ssp$ and $\ssp
    \dom u\inc\bigcup\ssp\dom\mu \ssp$ and \newskline{11.5}

 $ t = \inf\sp\big\{\sp\loint v\rmdss20\mu:\mu:v\in\sp 
   ^{\dom\snn\eightmath v}\KP1[\KP{1.1} 0\,,\plusinfty\KP1]\ssp$ and \newfline

 $v\invss44\images\spp\barscTbb_R\inc\dom\mu \ssp $ and $\ssp 
 \dom u\inc\{\KPt8\eta:u\fvalue\eta\le v\fvalue\eta \KPt9 \}\,\big\} \KP1 $, \KP9

\item \ $\plusint x\rmdss20\mu = \uniqset t : t \label{def +int}
         = \loint x\rmdss20\mu = \upint x\rmdss20\mu\,$,

\item \ $\Reint x\rmdss10\mu=\uniqset\smb I\sn\ar 1\sn:x\ssp$ a function \label{defi Reint}
        and $\ssp\rng x\inc[\minusinfty\,,\plusinfty\,]\ssp$ and $\,
        \aall{\smb I\sp,\spp\smb J}$ \newskline{25.7}

 $\smb I = \plusint\,\seq{\KP1\sup\,\{\,0\,,x\fvalue\eta\,\}:\eta=\eta\KP1}
            \rmdss20\mu\ssp$ and \newfline

 $\smb J=\plusint\,\seq{\KP1\sup\,\{\,0\,,\minus(\ssp x\fvalue\eta\ssp)\,\} :
          \eta=\eta\KP1}\rmdss20\mu\impss22\smb I\sn\ar 1=\smb I-\smb J\,$, \KP9

\item \ $\int_{\,A\,}x\rmdss10\mu=\uniqset\smb I:A\inc\bigcup\sp\dom\mu\ssp$ \label{defi ps-usual int}
        and $\,\eexi{\vPi\sp,\spp S}\,x \in \sp
         ^{\dom\snn\eightmath x}\ssp S \ssp $ and \newskline{23}

 $\vPi\ssp$ is complex pseudo\ssp-\sp usual and $\ssp\vPi\not=\tfbbC\ssp$
 and $\,\big[\ \big[\KP{1.5} S = $ \newskline{9.5}

 $\mathbb C\cupss31\{\minusinfty\,,\plusinfty\,\}\ssp$ and $\ssp
 \smb I = \Reint\,(\sp\fRe x\,|\,A\sp)\rmdss10\mu\sp + \ssp\roman i\,
          \Reint\,(\sp\fIm x\,|\,A\sp)\rmdss10\mu $ \newskline{13}

 $\in S \KP{1.5}\big]\ssp$ or $\ssp\big[\KP{1.5} x \not= \emptyset\ssp$ and $\ssp
   \smb I \in S = \vecs\vPi\ssp$ and $\, \aall \ell\,
    \ell\in\Cal L\,(\spp\vPi\Reit2\ssp,\tfbbR\ssp)$ \newfline

 $\impss02 \ell\ssp\fvalue\smb I = \Reint\,
  (\ssp\ell\sp\circ\spp x\KP1|\KP1 A\ssp)\rmdss10\mu\KP{1.5}\big]\ \big]\ $.

  \end{myLeftskip}\end{enumerate}
  \end{constructions}

From Constructions \ref{defi re scal int}\,(1) we get the {\it lower integral\ssp} 
of a\q{positive} valued function \math{u} with respect to a positive measure \mathss35{
\mu}, \,and \ref{defi re scal int}\,(2) and \ref{defi re scal int}\,(3) give 
the corresponding {\it upper\ssp} and {\it positive\ssp} integral. The 
{\it real\ssp} integral of an extended real valued function w.r.t.\ a positive 
measure is given in \ref{defi re scal int}\,(\ref{defi Reint})\,, and item (5) 
defines the {\it pseudo\ssp-\sp usual\,} integral. Without delving in the 
relevant formal definition given in \cite{Hif} we shortly remark that 
pseudo\ssp-\sp usual spaces \math{E} are such structured vector spaces over 
some subfield \math{\bold K} of the complex field \math{\fbbC} that e.g.\ we 
have unambiguously \math{(\ssp x + y\ssp)\svs E = x + y } and \math{
(\ssp t\,x\ssp)\svs E = t\,y } for all \math{x\ssp,\sp y\in\vecs E } and \mathss36{
t\in\vecs\bold K}. If \linebreak
                    \œ$\vPi\ssp$ is pseudo\ssp-\sp usual and \math{I} is any 
set with \math{I\in\{\,1\spp\adot\ssp,\spp 2\sp\adot\sp\} } or \mathss31{
3\sp\adot\inc\Card\sp I}, \,then \œ$\ssp(\sp X\sp,\spp S\ssp)$ \linebreak
                                                               is 
pseudo\ssp-\sp usual for any set \math{S} and any vector substructure \math{X} 
of \math{\sigrd\vPi\expnota^\sp I\sp]_{vs} }.

\begin{definitions}

(1) \ Say that \math{u} is {\it positive \mathss37{\mu}--\,measurable\ssp} 
    if{}f there is \math{\Omega} such that \math{\mu} is a positive measure on \math{
\Omega} and \math{u} is a function with \math{ u \inc
 \Omega\times[\KPp1.1 0\,,\plusinfty\KPt9] } and such that \math{
u\invss44\image\sp[\KP1 r\sp,\plusinfty\KPt9]\in\dom\mu} holds for all \mathss30{
r\in\rbb R^+}, \inskipline{.5}2

(2) \ Say that \math{u} is {\it fully positive \mathss37{\mu}--\,measurable\ssp} 
    if{}f \inskipline0{23.6}

$u\ssp$ is positive \mathss37{\mu}--\,measurable with \mathss36{ \dom u = \sp
 \bigcup\,\dom\mu}, \inskipline{.5}2

(3) \ Say that \math{\sigma} is {\it positive \mathss37{\mu}--\,simple\ssp} 
    if{}f 
there is \math{\Omega} such that \math{\mu} is a positive measure on \math{
\Omega} and \math{\sigma} is a function with \math{\sigma\inc\Omega\times
 \lbb R_+ } and such that \math{\rng\sigma} is finite and also \math{
\{\,\sigma\invss33\image\snn\{\ssp t\ssp\} : t \in \Univ\KP1\} \inc \dom\mu } 
  holds.
  \end{definitions}

Thus in the case where \math{\mu} is a positive measure, in 
Constructions \ref{defi re scal int}\,(\ref{ctr upint}) above we have \math{t} 
the infimum of the set of lower integrals of all positive \mathss37{\mu
}--\,measurable functions \math{v} dominating \math{u} in the sence that \math{
u\fvalue\eta\le v\fvalue\eta} holds for all \mathss34{\eta\in\dom u}.

\begin{lemma}\label{Le +int}

Let $\,\mu$ be a positive measure. Then for all $\,x$ the equivalences \vskip.5mm\centerline{$
\plusint x\rmdss20\mu \not= \Univ \equivss33 0\le\plusint x\rmdss20\mu =
\loint x\rmdss20\mu = \upint x\rmdss20\mu\le\plusinfty $} \inskipline{.5}{20}

and $ \KPp18.6 0\le\plusint x\rmdss20\mu < \plusinfty \equivss33 (\sp*\sp) 
\null\hfill $ hold when $\,(\sp*\sp)$ \inskipline{.5}0

means that there exist positive $\,\mu\,$--\,measurable functions $\,u\ssp,\sp 
v$ with \œ$\,u\le x\le v$ and $\,\loint u\rmdss20\mu \not= \plusinfty$ and $\,
v\invss44\image\spp\rbb R^+\snn\setminus\sn\dom u\cupss31\{\,\eta : 
 v\fvalue\eta\not=u\fvalue\eta\in\Univ\KP1\} \in 
 \mu\invss44\image\snn\{\ssp 0\ssp\} \KP1 $.
  \end{lemma}

\begin{proof} Assuming \math{\plusint x\rmdss20\mu \not= \Univ } we have \math{
\plusint x\rmdss20\mu =\loint x\rmdss20\mu = \upint x\rmdss20\mu = t } for 
some \mathss36{t\in\Univ}. Then taking \math{\sigma = \emptyset } in 
Constructions \ref{defi re scal int}\,(1) we see that \math{t = \sup\ssp A } 
for some \math{A} with \math{0\in A\inc[\KPp1.1 0\,,\plusinfty\KPt9] } whence \math{
0\le t\le\plusinfty } follows. Conversely, if we have \mathss35{ 0 \le 
 \plusint x\rmdss20\mu\le\plusinfty }, \,then \math{ \plusint x\rmdss20\mu \in 
 \Univ } and hence \mathss35{\plusint x\rmdss20\mu\not=\Univ }.

Assuming \math{0\le\plusint x\rmdss20\mu < \plusinfty } from (1) and 
(\ref{ctr upint}) in Constructions \ref{defi re scal int} we see existence of 
sequences \math{\bosy u} of positive \mathss37{\mu}--\,simple and \math{
\bosy v} of positive \mathss37{\mu}--\,measurable functions with \math{
\lim\sbi{\sp i\ssp\to\ssp\infty}\sp\loint\bosy w\fvalss01 i\rmdss11\mu=0 } for \math{
\bosy w=\seq{\seqss33{
\bosy v\fvalss01 i\fvalss10\eta - \bosy u\fvalss01 i\fvalss10\eta:
\eta=\eta}:i\in\bbNo\,} } and 
\math{\loint\bosy v\fvalss01\emptyset\rmdss11\mu < \plusinfty} and such that 
\math{\bosy u\fvalss01 i\le\bosy u\fvalss01 i\ssp\yplus\le x\le
\bosy v\fvalss01 i\ssp\yplus\le\bosy v\fvalss01 i} holds for all 
\mathss36{i\in\bbNo}. Taking then 
\math{u=\seqss43{\sup\KPt8\{\KPt8\bosy u\fvalss01 i\fvalss10\eta:
i\in\bbNo\ssp\}:\bosy u:\eta\in\bigcup\ssp\rng\bosy u} } 

and 
\mathss39{v=\langle\KP1\inf\,\{\KPt8\bosy v\fvalss01 i\fvalss10\eta:
i\in\bbNo\ssp\}:\bosy v:\eta\in\bigcap\KPt8\{\,\dom v:
v\in\rng\bosy v\KPt8\} \,\big\rangle }, \,now 

$u$ and $v$ are positive $\mu\,$--\,measurable and hence 

$\{\KPt8
v\invss44\image\spp\rbb R^+\snn\setminus\sn\dom u\ssp,\sp\{\,\eta : 
 v\fvalue\eta\not=u\fvalue\eta\in\Univ\KP1\}\sp\}
\inc\dom\mu$ holds. If we 

have \mathss38{
v\invss44\image\spp\rbb R^+\snn\setminus\sn\dom u
\not\in\mu\invss44\image\snn\{\ssp 0\ssp\} }, \,we see that 
$\loint u\rmdss11\mu < \loint v\rmdss11\mu$ holds, leading to 
a contradiction. Similarly we see that \math{\{\,\eta : 
 v\fvalue\eta\not=u\fvalue\eta\in\Univ\KP1\}
\not\in\mu\invss44\image\snn\{\ssp 0\ssp\} } is impossible. So we have 
\mathss38{v\invss44\image\spp\rbb R^+\snn\setminus\sn\dom u\cupss31\{\,\eta : 
 v\fvalue\eta\not=u\fvalue\eta\in\Univ\KP1\} \in 
 \mu\invss44\image\snn\{\ssp 0\ssp\} }.

The implication \math{(\sp*\sp)\impss33 0\le\plusint x\rmdss20\mu < \plusinfty } 
is straightforward.
  \end{proof}

Assuming that \math{\mu} is a positive measure on \math{\Omega} and that \math{
x} is a function \mathss36{\Omega\to\mathbb C}, \,from Lemma \ref{Le +int} via 
inspection of items (\ref{defi Reint}) and (\ref{defi ps-usual int}) in 
Constructions \ref{defi re scal int} above we see that \math{ \smb I = 
 \int_{\KPp1.1\Omega}\ssp x\rmdss11\mu\not=\Univ } implies that \math{\smb I
 \in\mathbb C} holds together with \math{(\ssp x\,;\spp\mu\,,\sn\tfbbC\ssp) \label{int fin a.e. meas} } 
being finitely almost \mathss37{\mu}--\,measurable. Thus in the case of an 
incomplete probability measure an {\sl integrable function need not be    \label{int not meas}
measurable\sp} according to our conventions.

\begin{proposition}\label{Pro upint}

Let $\,p\in\rbb R^+$ and 
let $\,\mu$ be a positive measure on $\,\Omega\,$. Also let 

$\,w=
\seqss43{u\fvalue\eta + v\fvalue\eta:\eta=\eta}$ 

where $\,u$ and $\,v$ are any functions with 
$\,u\cupss22 v\inc\Omega\times[\KPp1.1 0\,,\plusinfty\KPt9] \KP1 $. 

Then $\,\|\,w\,\|\Lnorss33^p_\mu
\le
\sup\KPt8\{\,1\ssp,\spp 2\KP1^{p^{-1}-\ssp 1\ssp}\big\}\KP1\big(\ssp
\|\,u\,\|\Lnorss33^p_\mu + \|\,v\,\|\Lnorss33^p_\mu\sp\sbig)0$ holds.
  \end{proposition}

\begin{proof} Let $\,\roman M\,x=
\sp^\Omega\KP1[\KPp1.1 0\,,\plusinfty\KPt9]
\capss51\{\,\varphi:\varphi\invss44\images\spp\barscTbb_R\inc\dom\mu$ \inskipline{.2}{56}

$\text{ and }\ssp\aall{\eta\,,\sp t}\,
(\ssp\eta\,,\spp t\ssp)\in x\impss33 
t\le\varphi\fvalue\eta
\KP1\} $ . \inskipline{.4}0

From \cite[pp.\ 49\,--\,50]{Jr} we know that the assertion 
holds under the additional restriction that we have 
\math{u\in\roman M\,u} and \math{v\in\roman M\,v} with 
\math{\rng(\ssp u\cupss22 v\ssp)\inc\ssbb05 R}, \,noting that 
it is trivial if \math{
\|\,u\,\|\Lnorss33^p_\mu + \|\,v\,\|\Lnorss33^p_\mu=\plusinfty} holds. From 
this one easily extends the result to the case 
where the restriction $\plusinfty\not\in\rng(\ssp u\cupss22 v\ssp)$ is 
removed.

Now putting 
\math{\smb A=\sup\KPt8\{\,1\ssp,\spp 2\KP1^{p^{-1}-\ssp 1\ssp}\big\} } for 
the general case, to proceed indirectly, 
suppose that we have 
$\smb A\,\big(\ssp
\|\,u\,\|\Lnorss33^p_\mu + \|\,v\,\|\Lnorss33^p_\mu\sp\sbig)0
<\|\,w\,\|\Lnorss33^p_\mu$ and take any $\eps\in\rbb R^+$ with 
$2\KPt8\smb A\KPt8\eps < \|\,w\,\|\Lnorss33^p_\mu - 
\smb A\,\big(\ssp
\|\,u\,\|\Lnorss33^p_\mu + \|\,v\,\|\Lnorss33^p_\mu\sp\sbig)0$ . 
Then there are $\varphi\in\roman M\,u$ and $\psi\in\roman M\,v$ with 
$
\|\,\varphi\,\|\Lnorss33^p_\mu
< \|\,u\,\|\Lnorss33^p_\mu + \eps \,$ and $\,
\|\,\psi\,\|\Lnorss33^p_\mu < \|\,v\,\|\Lnorss33^p_\mu + \eps \,$, 
whence with 
$\chi=\seqss43{\varphi\fvalue\eta + \psi\fvalue\eta:\eta=\eta}$ 
we obtain $\chi\in\roman M\,w$ and consequently \inskipline{.7}{8.5}

$\|\,\chi\,\|\Lnorss33^p_\mu\le\smb A\,\big(\ssp
\|\,\varphi\,\|\Lnorss33^p_\mu + \|\,\psi\,\|\Lnorss33^p_\mu\sp\sbig)0
$ \inskipline{.4}{21}

${}<\smb A\,\big(\ssp
\|\,u\,\|\Lnorss33^p_\mu + \|\,v\,\|\Lnorss33^p_\mu\sp\sbig)0
+2\KPt8\smb A\KPt8\eps
 < \|\,w\,\|\Lnorss33^p_\mu \,$, \,a contradiction.
  \end{proof}

\begin{remark}\label{Rem +*}

According to our updated definitional conventions in \cite{Hif} concerning 
sums and products of elements in a {\sl pseudo\ssp-\sp usual algebroid\sp}, if 
in Proposition \ref{Pro upint} for \mathss03{ \scrmt A = 
 \{\,\dom u\ssp,\dom v\,\} } we have \math{\bigcap\,\scrmt A \not= \emptyset } 
or \mathss36{\bigcap\,\scrmt A = \emptyset = \bigcup\,\scrmt  A }, \,then also \mathss30{w = 
 u + v } holds. Hence under this additional assumption we could have written 
the expression for \math{w} a bit more simply. However, if \math{
\bigcap\,\scrmt A = \emptyset \not= \bigcup\,\scrmt  A } holds and we also 
have \mathss06{\|\,u\,\|\Lnorss33^p_\mu \not= \plusinfty \not= 
 \|\,v\,\|\Lnorss33^p_\mu }, \,then \math{u + v = \Univ } and for this in 
place of \math{w} we would \linebreak \vskip-3.3mm\noindent get $ \KP{3.3}
  \plusinfty = \inf\ssp\emptyset = \|\KPt8\Univ\KPt8\|\Lnorss33^p_\mu
= \|\KP1 u + v\KP1\|\Lnorss33^p_\mu = \|\,w\,\|\Lnorss33^p_\mu$ \inskipline{.2}{15}

${} \le \sup\KPt8\{\,1\ssp,\spp 2\KP1^{p^{-1}-\ssp 1\ssp}\big\}\KP1\big(\ssp
         \|\,u\,\|\Lnorss33^p_\mu + \|\,v\,\|\Lnorss33^p_\mu\sp\sbig)0
    < \plusinfty \,$, \,a contradiction. \inskipline{.6}0

A similar remark applies to \math{\varphi\,,\sp\psi\ssp,\sp\chi} in the proof, 
thus having \math{\chi = \varphi + \psi } the function given by \math{\Omega
 \owns\eta\mapsto\varphi\fvalue\eta + \psi\fvalue\eta } since here \math{
\Omega=\dom\varphi=\dom\psi } holds.

We also suggest the reader to see \cite{A+Mo} for another kind of treatment of 
the notational \q{plus\ssp-\sp times} problem referred to above.
  \end{remark}

We also extend H\"older's inequality to upper integrals in the next

\begin{proposition}\label{Pro Hölder}

Let $\,1\le p < \plusinfty$ and let $\,\mu$ be a positive measure on $\,\Omega
\,$. Also let 

$\,w=\seqss43{u\fvalue\eta\cdot(\ssp v\fvalue\eta\ssp):\eta=\eta}$ where $\,u$ 
and $\,v$ are any functions with 

$\,u\cupss22 v\inc\Omega\times[\KPp1.1 0\,,\plusinfty\KPt9] \KP1 $. 
Then $\,\upint w\rmdss11\mu
\le \|\,u\,\|\Lnorss33^p_\mu \KP1 \|\,v\,\|\Lnorss40^{p\sast}_\mu$ holds.
  \end{proposition}

\begin{proof} Let \math{\roman M\,x} be as in the proof of Proposition \ref{Pro upint} 
above. For the indirect verification, suppose now that we have \mathss37{
\|\,u\,\|\Lnorss33^p_\mu \KP1 \|\,v\,\|\Lnorss40^{p\sast}_\mu < 
 \upint w\rmdss11\mu}, \,and with \mathss03{ A = 
 \big\{\,\|\,u\,\|\Lnorss33^p_\mu \,,\sp \|\,v\,\|\Lnorss40^{p\sast}_\mu\ssp\} } 
then put \mathss36{\smb M=\sup\,A}. We cannot have \math{0\in A} since by a 
simple exercise this would force \mathss36{ \upint w\rmdss11\mu = 0
}, \,contradicting our assumption. It follows that \math{\smb M < \plusinfty} 
holds, and we then take any \math{\eps} with \vskip.5mm\centerline{$
0 < \eps < \inf\,\{\,1\ssp,\sp(\ssp 2\KPt8\smb M + 1\ssp)\,^{\mminus 1}\ssp
 \big(\sp\upint w\rmdss11\mu - \|\,u\,\|\Lnorss33^p_\mu \KP1 
 \|\,v\,\|\Lnorss40^{p\sast}_\mu \sp \sbig) 0 \,\big\} \KP1 $.} \inskipline{.5}0

Now there are functions \math{\varphi\in\roman M\,u} and \math{ \psi \in 
 \roman M\,v} with \math{ \|\,\varphi\,\|\Lnorss33^p_\mu < 
 \|\,u\,\|\Lnorss33^p_\mu + \eps } and \mathss05{
\|\,\psi\,\|\Lnorss40^{p\sast}_\mu < \|\,v\,\|\Lnorss40^{p\sast}_\mu + \eps 
}, \,whence taking \math{ \chi = 
 \seqss43{\varphi\fvalue\eta \cdot(\ssp\psi\fvalue\eta\ssp):\eta=\eta} } we 
then have \math{\chi\in\roman M\,w} and consequently by the usual H\"older's 
inequality extended to measurable functions with values in \math{
[\KPp1.1 0\,,\plusinfty\KPt9] } we obtain \inskipline{.7}{7.78}

$   \upint w\rmdss11\mu \le \int_{\KPp1.1\Omega\,}\chi\rmdss11\mu
\le \|\,\varphi\,\|\Lnorss33^p_\mu \KP1 \|\,\psi\,\|\Lnorss40^{p\sast}_\mu $ \inskipline{.4}{20}

${}< \big(\ssp\|\,u\,\|\Lnorss33^p_\mu + \eps\ssp)\KP1
     \big(\ssp\|\,v\,\|\Lnorss40^{p\sast}_\mu + \eps\ssp) $ \inskipline{.4}{20}

${}= \|\,u\,\|\Lnorss33^p_\mu\KP1\|\,v\,\|\Lnorss40^{p\sast}_\mu + \big(\ssp
     \|\,u\,\|\Lnorss33^p_\mu + \|\,v\,\|\Lnorss40^{p\sast}_\mu\ssp)\KP1\eps 
     + \eps\KPt8^2 $ \inskipline{.4}{20}

${}\le \|\,u\,\|\Lnorss33^p_\mu\KP1\|\,v\,\|\Lnorss40^{p\sast}_\mu +
       (\ssp 2\KPt8\smb M + 1\ssp)\KP1\eps 
 < \upint w\rmdss11\mu \,$, \,a contradiction.
  \end{proof}

\begin{constructions}[standard Lebesgue measures]\label{defi Leb mea} $\null$ \vskip.5mm

\begin{enumerate}\begin{myLeftskip}{-4}{.6}{.4}

\item \ $\upCth\mu = \mu\KP1|\KP1\{\,A:\aall{B\in\dom\mu}\,               \label{ctr Carath}
       \mu\fvalue(\sp A\capss21 B\ssp) +
       \mu\fvalue(\sp B\sp\setminus A\ssp) \le \mu\fvalue\snn B\KP1\} \KP1 $,

\item \ $\mu\meatimes\nu = \uniqset\mu\ar 1\sn:\aall{\mu\ar 2\ssp,\sp\mu\ar 3}$ \newskline8

 $\mu\ar 2 = \{\,(\spp A\times B\ssp,\sp s\cdotn t\,):
  (\spp A\,  ,\spp s\ssp)\in\mu    \ssp$ and $\ssp
  (\spp B\ssp,\spp t\ssp)\in\nu\,\}\ssp$ and \vskip-.3mm

 $\nKP{12.7} \mu\ar 3 = 
   \big\langle\,\inf\sp\big\{\,\sum\,(\ssp\mu\ar 2\circ\ebit B\ssp)
  : \mu\ar 2\snn : \ebit B\in\sp^\sbbNo\,(\sp\dom\mu\ar 2)\ssp$ and \newfline

 $ A\inc\bigcup\ssp\rng\ebit B\KPt8\} : \mu\ar 2\snn :
   A\inc\bigcup\ssp\dom\mu\ar 2\,\big\rangle \impss22
    \mu\ar 1 = \upCth\mu\ar 3 \KPt8 $, \KP{8.5}

\item \ $\Lebmef^{\ssp\ssmb N} = \uniqset\mu\ar 1\sn:\smb N\in\bbNo\ssp$ and $\,   \label{ctr Lebmea}
      [\ [\KP{1.3}\smb N = \emptyset\ssp$ and $\ssp
                \mu\ar 1 = \seq{\,0\,,1\,} \KP{1.3}]\ssp$ or \newskline{15}

 $[\KP{1.3} \smb N \not= \emptyset\ssp$ and $\,\aall{
 \Cal B\ssp,\sp\Cal J\spp,\spp\mu\,,\spp\nu\ssp,\spp\nu\ar 1}\,
 \Cal J = \{\,\openIval{\smb A\ssp,\smb B} :
  \smb A\ssp,\spp\smb B\in\ssbb09 R\}\ssp$ and \newskline{6.5}

 $\nu\ar 1 = \seq{\,\smb B-\smb A : J = \openIval{\smb A\ssp,\smb B}
   \in\Cal J\ssp$ and $\ssp\smb A\le\smb B\,}\ssp$ and \newskline{8}

 $\nu = \Seq{\,\prod\,(\ssp\nu\ar 1\snn\circ\spp\ebit I\ssp) : B =
  \prodc\ebit I\ssp$ and $\ssp\ebit I\in\sp\yi N\sp\Cal J\KP1}\ssp$ and $\ssp
   \Cal B=\dom\nu$

 $\nKP{8} \mu = \Seq{\,\inf\sp\big\{\,\sum\,(\ssp\nu\circ\ebit B\ssp) :
  \ebit B\in\sp^\sbbNo\,\Cal B\ssp$ and $\ssp A\inc\bigcup\sp\rng\ebit B\,\} :
                                            A\inc\sp\yi N\ssbb69 R}$ \newfline

 $\impss02 \mu\ar 1 = \upCth\mu \KP{1.3}]\ ]\ $, \KP{8.5}

\item \ $\Lebmef^{} = \Lebmef^{\sp 1.} \circ\sp
        \big\langle\KP1^{1.}A:A\inc\ssbb09 R\sp\rangle \KP1 $, \label{def Lebm on R} \vskip.3mm

\item \ $\int_{\,\ssmb A}^{\,\ssmb B\,}x = \uniqset\smb I :  \label{df Leb int}
    \eexi{A\,,\spp\sigma}\,\big[\ \big[\KP{1.5} \sigma = 1\ssp$ and $\ssp
                  \smb A < \smb B\ssp$ and $\ssp
          A = \openIval{\smb A\ssp,\spp\smb B} \KP{1.5} \big]$ \inskipline{-.1}{43.5}

 or $\ssp\big[\KP{1.5} \sigma = \minus 1\ssp$ and $\ssp \smb B \le \smb A\ssp$
 and $\ssp A = \openIval{\smb B\ssp,\spp\smb A} \KP{1.5} \big]\ \big]$ \inskipline{-.1}{44.33}

 and $\, \smb I = \sigma\spp\int_{\,A}\ssp x\rmdss10\Lebmef^{} \KP1 $.

  \end{myLeftskip}\end{enumerate}
  \end{constructions}

\newcommand\sPows{{\lower.22mm\hbox{\font\Å=eusm8\ÅP}\kern-.3mm\lower.7mm\hbox{\font\Å=cmss5\Ås}\kern.65mm}} % the power class\renewcommand\Pows{\lower.14mm\hbox{\font\Å=eusm9 scaled\magstep1\ÅP}\kern-.5mm\lower.7mm\hbox{\font\Å=cmss6\Ås}\kern.8mm}} % \Pows for sub- and superscripts
Saying that \math{\mu} is an {\it outer measure\ssp} on \math{\Omega} if{}f \math{
\mu\in\sp^{\sPows\Omega}\KP1[\KPp1.1 0\,,\plusinfty\KPt9] } and \mathss30{
\mu\fvalue\ssn A\le\mu\fvalue\snn B } and \math{
\mu\spp\fvalue\snn\bigcup\,\scrmt A\le\sum\KP1(\ssp\mu\KP1|\KP1\scrmt A\ssp) } 
hold whenever we have \math{A\inc B\inc\Omega } and \math{ \scrmt A \inc 
 \Pows\Omega } with \linebreak
                    $\scrmt A\ssp$ countable, essentially from 
\cite[Lemma 3.1.8\ssp, Proposition 3.1.9\ssp, pp.\ 67\,--\,68]{Du} we get the 
proof of the following

\begin{proposition}\label{Pro Cth}

If $\,\mu$ is an outer measure on $\,\Omega${\,\rm, }then $\,\upCth\mu$ is a 
complete positive measure on $\,\Omega\,$.
  \end{proposition}

Thus by Proposition \ref{Pro Cth} in \ref{defi Leb mea}\,(\ref{ctr Carath}) we 
have the standard {\sl Carath\'eodory construc- tion\sp} associating a 
complete positive measure with any outer measure. For \mathss36{\smb N\in\bbN
}, \,the function \math{\Lebmef^{\ssp\ssmb N}} is the standard complete 
Lebesgue measure on \math{\yi N\ssbb60 R} defined on the class of Lebesgue 
measurable subsets. The corresponding measure on \math{\bbR} is \math{
\Lebmef^{}\sp}. Note that if we had not separately defined \math{
\Lebmef^{\sp 0.} = \seqss20{0\,,\spp 1} = \{\,(\ssp\emptyset\,,\spp 0\ssp)\,,\sp
 (\ssp 1\adot\ssp,\spp 1\ssp)\,\} } by inserting \q{\mathss03{ \smb N = 
 \emptyset\ssp\text{ and }\ssp\mu\ar 1=\seqss20{0\,,\spp 1} }} in 
\ref{defi Leb mea}\,(\ref{ctr Lebmea})\,, then it would have given \mathss30{
\Lebmef^{\sp 0.}=\sn} \mathss08{2\adot\timesn\{\sp\plusinfty\,\} }. We also 
put \math{ \LeBmef^{\ssp\ssmb N} = 
 \Lebmef^{\ssp\ssmb N}\,|\KP1\sigmAlg3\nsTbb_R\ssn\yi N } and \math{\LeBmef^{}
 = \Lebmef^{}\,|\KP1\sigmAlg3\nsTbb_R } get- ting the restrictions of the 
Lebesgue measures to the standard Borel \rsigma3algebras.

% ¤¤¤¤¤¤¤¤¤¤¤¤¤¤¤¤¤¤¤¤¤¤¤¤¤¤¤¤¤¤¤¤¤¤¤¤¤¤¤¤¤¤¤¤¤¤¤¤¤¤¤¤¤¤¤¤¤¤¤¤¤¤¤¤¤¤¤¤¤¤¤¤¤¤¤¤

\insubsubhead       Pettis integration of vector functions                \label{Ss Pettis}

In some places of the proof of Theorem \nfss A\,\ref{main Th} we refer to 
something being {\sl Pettis\sp}. In order to make the meaning of this 
explicit, we give the following

\begin{definitions}[for Pettis integration]\label{df Petti}

(1) \ Say that \math{\tilde c} is {\it scalar integrable\ssp} to \math{x} 
    if{}f \math{\tilde c} is an mv\ssp-\sp map and for all \math{\bosy K\ssp,\sp 
E\ssp,\sp\mu\,,\sp\Omega\ssp,\sp x\ssp,\sp u} from \math{ \tilde c = 
 (\,c\KPt8;\sp\mu\,,\spp E\ssp) } and \math{\Omega=\bigcup\,\dom\mu} and \math{
\bosy K\in\setRC} and \math{\domm\tsigrd E=\vecs\bosy K } and \math{ u \in 
 \Cal L\,(\sp E\ssp,\spp\bosy K\ssp) } it follows that \newline
$u\fvalue x=\int_{\KPp1.1\Omega}\,u\circss01 c\rmdss21\mu\not=\Univ$ holds, \inskipline{.5}2

(2) \ $E\vPettis3int_A\,c\rmdss21\mu=\uniqset x:c\ssp$ a function and \math{
    \mu} is a positive measure and \inskipline{.2}{8.5}

$A\inc\dom c\capss43\bigcup\,\dom\mu\ssp$ and \math{
(\,c\KP{1.2}|\KP1 A\,;\sp\mu\KP1|\KP1\Pows A\,,\spp E\ssp)} is scalar 
  integrable to \mathss34{x}, \inskipline{.5}2

(3) \ Say that \math{\tilde c} is {\it Pettis\ssp} if{}f 
\math{\tilde c} is an mv\ssp-\sp map and for all 
\math{A\,,\sp E\ssp,\sp\mu\,,\sp c} from 
$\ssp\tilde c={}$ \inskipline{.2}{8.5}

$(\,c\KPt8;\sp\mu\,,\spp E\ssp) \ssp$ and 
\math{A\in\dom\mu} it follows that \math{
E\vPettis3int_A\,c\rmdss21\mu\not=\Univ } holds.
  \end{definitions}

Because of the manner we have put the definitions, from the discussion after 
the proof of Lemma \ref{Le +int} on page \pageref{int not meas} above, it 
follows that \math{\tilde c} being Pettis implies it being {\sl finitely 
almost scalarly measurable\sp}. Also from \math{(\,c\KPt8;\sp\mu\,,\spp E\ssp) } 
being Pettis with \mathss03{\rng\mu\capss34\rbb R^+\not=\emptyset} it follows 
that \math{\Card\sp\vecs E\not=1\spp\adot=\Card\sp\vecs(\sp E\dlsigss11\sp) } 
cannot hold. That is, if \math{\mu} and \math{E} are nontrivial, then also the 
dual of \math{E} must be such. For example \mathss03{
(\,c\KPt8;\Lebmef^{}\ssp,\spp\LLrs42^{\frac 12}(\ssbb44 I)) } cannot be 
Pettis, whereas \math{
(\,c\KPt8;\Lebmef^{}\ssp,\sp\ell\KPt8^{\frac 12\ssp}(\ssp\bbNo\spp)) } can. \vskip.3mm

To have at our disposal also some partially weaker and more general notions of 
integrability of mv\ssp-\sp maps, we put the following

\begin{definitions}\label{df sc+Gel-int}

(1) \ Say that \math{\tilde c} is {\it scalarly integrable\ssp} if{}f \math{
    \tilde c} is an mv\ssp-\sp map and for all \mathss03{c\,,\sp\mu\,,\sp
\bosy K\sp,\sp\vPi} from \math{\tilde c=(\,c\KPt8;\sp\mu\,,\spp\vPi\ssp) } and \math{
\bosy K\in\setRC} and \math{\vecs\bosy K=\domm\tsigrd\vPi } it follows that 
for all \math{\xi\in\vecs\vPi\sp\setminus\{\,\Bnull_\vPi\} } there is \math{
u\in\Cal L\,(\sp\vPi\sp,\spp\bosy K\ssp) } with \math{u\fvalss01\xi\not=0 } 
and for all \math{u\in\Cal L\,(\sp\vPi\sp,\spp\bosy K\ssp) } and \math{ A \in 
 \dom\mu} it holds that \mathss36{
\int_{\,A\,}u\circss01 c\rmdss21\mu \not= \Univ }, \inskipline{.5}2

(2) \ Say that \math{\tilde c} is {\it finitely scalarly integrable\ssp} if{}f \math{
    \tilde c} is an mv\ssp-\sp map and for all \mathss30{c\,,\sp\mu\,,} $
\vPi\ssp$ from \math{\tilde c=(\,c\KPt8;\sp\mu\,,\spp\vPi\ssp) } and \math{
A\in\mu\invss44\image\spp\lbb R_+} it follows that \math{
(\,c\KPp1.2|\KP1 A\,;\sp\mu\KP1|\KP1\Pows A\,,\spp\vPi\ssp) } is scalarly 
integrable, \inskipline{.5}2

(3) \ Say that \math{\tilde c} is {\it Gelfand\,} if{}f \math{\tilde c} is  
    scalarly integrable and for all \mathss30{c\,,\sp\mu\,,\sp A\,,\sp\bosy K\sp
 ,\sp\vPi} from \math{\tilde c=(\,c\KPt8;\sp\mu\,,\spp\vPi\ssp) } and \math{
\bosy K\in\setRC} and \math{\vecs\bosy K=\domm\tsigrd\vPi } and \mathss30{A\in 
 \dom\mu} it follows that \math{
\big\langle\sp\int_{\,A\,}u\circss01 c\rmdss21\mu : u\in
 \Cal L\,(\sp\vPi\sp,\spp\bosy K\ssp)\KP1\rangle } is continuous \mathss32{
\taurd(\sp\vPi\dlbetss01\sp)\to\taurd\bosy K}.
  \end{definitions}

A simple example of a Banach space valued mv\ssp-\sp map that is Gelfand but 
not Pettis is given in the following

\begin{example}\label{Exa Gel /= Pet}

Let \math{\tilde x=(\ssp x\,;\spp\mu\,,\spp\vPi\ssp) } where \math{\mu=
 \seqss30{\card3 A:A\inc\bbNo\snn} }  and \math{\vPi=\co(\ssp\bbNo\spp) } 
and \mathss08{x=\seqss30{(\ssp\bbNo\sn\setminus\{\ssp i\ssp\}\sp\sbig)0\times\snn
 \{\ssp 0\ssp\}\cupss22\{\,(\ssp i\ssp,\spp 1\ssp)\,\}:i\in\bbNo\snn} }. For 
every \math{A\inc\bbNo } and \mathss30{ \zeta \in 
 \vecs\ell\KPt8^1\spp(\ssp\bbNo\spp) } then \mathss38{
\int_{\,A\,}x\fvalue\eta\cdot\zeta\rmdss21\mu\,(\sp\eta\sp) = 
 \sum\KP1(\ssp\zeta\KP1|\KP1 A\ssp) }, \,and hence \math{\tilde x} is scalarly 
integrable. It is also Gelfand since for \math{ \lambda = 
 \bbNo\timesn\{\ssp 1\ssp\} } we have \math{ \lambda \in 
 \vecs\lll^\plusinftyy\spp(\ssp\bbNo\spp) } with \mathss30{
\sum\KP1(\ssp\lambda\cdot\zeta\KP1|\KP1 A\ssp) = \sn} \mathss03{
 \sum\KP1(\ssp\zeta\KP1|\KP1 A\ssp) = 
  \int_{\,A\,}x\fvalue\eta\cdot\zeta\rmdss21\mu\,(\sp\eta\sp) } for all \math{
A\inc\bbNo} and \mathss38{\zeta\in\vecs\ell\KPt8^1\spp(\ssp\bbNo\spp) }. Since 
we here have \mathss31{\lambda\not\in\vecs\vPi}, \,we see that \math{\tilde x} 
  is not Pettis.
  \end{example}

% ----------------------------------------------------------------------------

\Ssubhead B               Generalized Bochner spaces                      \label{Sec B}

In this section, we first give the formal construction of the generalized 
Lebesgue\,--\,Bochner spaces  spaces \math{ F = 
 \mvLrs03^p(\ssp\mu\,,\spp\vPi\ssp) } and \math{ F\aar 1 = 
 \mvsLrs03^p(\ssp\mu\,,\spp\vPi\ssp) } of equivalence classes of order \math{p} 
integrable functions \math{x:\Omega\to\vecs\vPi} when \math{ 0 \le p \le 
 \plusinfty} and \math{\mu} is a positive measure on some set \math{\Omega} 
and \math{\vPi} is a real or complex topological vector space. Then we 
establish the basic relevant properties of these spaces under the additional 
assumption that the space \math{\vPi} is suitable.

For the construction of the space $\ssp F$, the functions \math{x} are 
required to be such that \math{\tilde x=(\ssp x\,;\spp\mu\,,\spp\vPi\ssp)} is 
{\sl finitely almost simply measurable\sp} in the sense of 
Definitional schemata \ref{df meas} on page \pageref{df meas} above. For \math{
F\aar 1} we instead require \math{\tilde x} to be only {\sl finitely almost 
scalarly measurable\sp}. 

The integrability condition is formulated so that in the case \math{p\not=0} 
for any bounded quasi\ssp-\ssp seminorm \math{\Nu} in \math{\vPi} we should 
have \math{\|\KP1\Nu\circss01 x\KP1\|\Lnorss33^p_\mu<\plusinfty } which in the 
case \math{p\in\rbb R^+} is equivalent to the function \math{
\aabs99^p\circss01\Nu\circss10 x : \Omega \owns \eta \mapsto
 (\ssp\Nu\circss10 x\fvalss11\eta\ssp)\,^p } \nolinebreak pos- sessing a 
dominating \mathss37{
\mu}--\,integrable function \mathss38{ \varphi : \Omega \to 
 [\KPp1.1 0\,,\plusinfty\KPt9] }. Then 

$\null$
\math{
\upint\ssp\aabs99^p\circss01\Nu\circss10 x\rmdss11\mu<\plusinfty} holds, and 
this determines

by $x\mapsto\|\KP1\Nu\circss01 x\KP1\|\Lnorss33^p_\mu$ a corresponding 
quasi\ssp-\ssp seminorm%%, or a \mathss37{p}--\,quasi\ssp-\ssp seminorm in the case where $0<p<1$ holds
. 

For \math{p=0} no integrability is required, and in this case the topology is 
determined by the quasi\ssp-\ssp semimetrics \math{ \roman d\,A\KP1\Nu : 
 (\sp x\ssp,\sp y\ssp)\mapsto\upint\ssp\rmmd\circ\spp\Nu\circss10 z\rmdss11\mu} 
where with given \math{A\in\mu\invss44\image\sp\lbb R_+} we have \mathss38{
z = (\ssp\Omega\spp\setminus A\ssp)\times\snn\{\,\Bnull_\vPi\}\cupss22
 \seqss33{(\ssp x\fvalue\eta-y\fvalue\eta\ssp)\svs\vPi\snn:\eta\in A} }, \,and \math{
\rmmd = \seqss30{ (\ssp 1 + t\ssp)^{\,\mminus 1}\,t : t \in \lbb R_+ \sn } } 
hence \math{\lbb R_+\to[\KPp1.1 0\,,\spp 1\KPt9{[\,} } given by \mathss34{
t \mapsto (\ssp 1 + t\ssp)^{\ssp\mminus 1}\,t }.

To get Hausdorff topologies for the spaces \math{F} and \mathss32{F\aar 1
}, \,we finally take the quotient space by the vector subspace \math{N\aar 0} 
of functions \math{x} with \math{\int_{\,A}\sp u\circss01 x\rmdss11\mu=0} for 
all \linebreak 
    $A\in\mu\invss44\image\spp\lbb R_+\sp$ and for all \mathss38{
u\in\Cal L\,(\spp\vPi\sp,\spp\bosy K\ssp) }.

Observe that if with \math{\bosy K=\tfbbR} we have for example \mathss38{\vPi=
 \LLrs42^{\frac 12}(\ssbb44 I) }, \,then the dual set \math{
\Cal L\,(\spp\vPi\sp,\spp\bosy K\ssp) = 
 \{\KPt8\vecs\vPi\timesn\{\ssp 0\ssp\}\sp\} } and hence the spaces \math{F} 
and \math{F\aar 1 \label{triv L^p exa}} become trivial. For $ \vPi = 
 \ell\KPt8^{\frac 12}\sp(\ssp\bbNo\spp) \ssp $ the situation is different 
since then \math{\vPi} has nontrivial dual.

\begin{constructions}[of generalized Lebesgue\,--\,Bochner spaces]\label{defi $L^p$} $\null$ \vskip.5mm

\begin{enumerate}\begin{myLeftskip}{-4}{.6}{.6}

\item \ $\roman{Leb}\sbi{\sixroman n\fiveroman{bh}}\!\RHB{.3}{^p}\ssp
         \varXi\sbi M=\uniqset\Cal V:0\le p\le\plusinfty\ssp$ and \math{  \label{L^p nbhs}
 \eexi{\bosy K\ssp,\sp\mu\,,\sp\Omega\,,\sp\vPi} } \newskline{20}

 $\bosy K\in\setRC \ssp$ and $\ssp
  \mu\ssp$ is a positive measure on $\ssp\Omega \ssp$ and \newfline

 $ \vPi\in\tvsps0(K)\ssp$ and $\ssp \varXi=(\ssp\mu\,,\spp\vPi\ssp) \ssp$
 and $ [ \KP1 [ \KP{1.4} p = 0 \ssp$ and \KP{11.1} \newskline6     

 $\Cal V=\{\KPt8 V\ssn:\eexi{\Nu\ssp,\spp A\,,\spp\eps}\,\Nu\in\Bqnorm\vPi \ssp$
 and $\ssp A\in\mu\invss23\image\sp\lbb R_+\ssp$
 and $\ssp \eps\in\rbb R^+ \ssp$ and \newfline

 $V=M\capss21\{\,x:
  \upint\ssp\rmmd\snn\circ\sp\Nu\circss00 x\KP1|\KP1 A\rmdss41\mu<\eps\,\}
   \sp\} \KP{1.4} ] \ssp$ or $\ssp [ \KP{1.4} p\in\rbb R^+\ssp$ and \KP{8.3} \newskline6

 $\Cal V=\{\KPt8 V\ssn:\eexi{\Nu\in\Bqnorm\vPi}\,V = M\capss21\{\, x :
 \upint\ssp \Abrs33^p\circss00\Nu\circss00 x\rmdss41\mu<1\,\}\sp\} \KP{1.4} ] $ \newfline

 or $\ssp [ \KP{1.4} p = \plusinfty\ssp$ and \KP7 \newskline6

 $\Cal V=\{\KPt8 V\ssn:\eexi{\Nu\in\Bqnorm\vPi}\, V = M\capss21\{\,x:
  \aall{A\in\mu\invss23\image\sp\lbb R_+}$ \newskline{16}

 $\eexi{N\in\mu\invss44\image\snn\{\ssp 0\ssp\} }\, 
  \sup\,(\ssp\Nu\circss00 x\KP1 [\KP1 A\setminus N\KP1]\ssp) < 1 \KPt8\}\,\} 
   \KP{1.4} ] \KP{1.2} ] \KP{1.4} $,

\item \ $\raise1.7mm\hbox{\font\Å=cmssi5\Åpr}\kern-.3mm
         \LLrs02^p\varXi\sp\sbi{M\ssp\aars N_0}=\uniqset F:               \label{preL^p_{MN_0}}
         \roman{Leb}\sbi{\sixroman n\fiveroman{bh}}\!\RHB{.3}{^p}\ssp
         \varXi\sbi M\not=\Univ\ssp$ and \newskline6

 $
 \aall{\mu\,,\sp\Omega\,,\sp\vPi\sp,\sp S\ssp,\sp\scrmt T\sp,\sp\Cal V\sp,
 \sp\scrmt V\aar 0\,,\sp X\sp,\sp Y}\,
 \varXi=(\ssp\mu\,,\spp\vPi\ssp) \ssp$ and $\ssp\Omega=\bigcup\,\dom\mu\ssp$ 
 and \newskline6

 $X=\sigrd\vPi\expnota^\,\Omega\ssp]_{vs}\sp$ and $\ssp\Cal V=
 \roman{Leb}\sbi{\sixroman n\fiveroman{bh}}\!\RHB{.3}{^p}\ssp
         \varXi\sbi M\sp$ and \newskline6

 $S = \ssp
  \bigcap\ssp\big\{\KP{1.1}[\KP{1.2}\mathbb Z\KP{1.1}V\KP{1.1}]\vvs X\sn:
  V\in\Cal V\KP1\} \ssp$ and $\ssp Y=
 X_{\ssp|\,S}\,/\vsquotient N\aar 0\ssp$ and \newskline6

 $\scrmt V\aar 0 = \{\,\vecss Y\capss01\{\,\smb X : \smb X\capss02 V\aar 1 
   \not= \emptyset \KP1 \} : V\aar 1\in\Cal V\KP1\} \ssp $ and \newskline6

 $\scrmt T=
 \{\KPt8 U\sn:\aall{\smb X\in U}\,\eexi{V\in\scrmt V\aar 0}\,
 [\KP{1.1}\{\ssp\smb X\ssp\} + V\KP{1.2}]\vvs Y\inc U\inc\vecs Y\KP1\}$ \newskline6

 $
 \impss33 M\ssp$ is a vector subspace in $\ssp X\ssp$ and \newskline6
 
 $N\aar 0\ssp$ is a vector subspace in $\ssp X_{\ssp|\,S}\ssp$ and 
 $\ssp F=
 (\ssp Y\sppp,\spp\scrmt T\,) \KPt9 $,

\item \ $\mvLrs02^p(\ssp\mu\,,\spp\vPi\ssp)=\uniqset F: \eexi{\bosy K}\,  \label{simpL^p}
         \bosy K\in\setRC\ssp$ and $\ssp\vPi\in\tvsps0(K)\ssp$ and \newskline{11}
 
 $\aall{M\sp,\sp N\aar 0\,,\sp\Omega\,,\sp S\ssp,\sp 
 \Cal V\sp,\sp X}\,\Omega=\bigcup\,\dom\mu\ssp$ and $\ssp 
 X = \sigrd\vPi\expnota^\,\Omega\ssp]_{vs}\sp$ and \newskline{14}

 $ M = \vecs X\capss01\{\,x : (\ssp x\,;\spp\mu\,,\spp\vPi\ssp) \ssp $ is 
        finitely \newskline{50.5}

                 almost simply measurable $ \} \ssp $ and \newskline6

 $\Cal V=
 \roman{Leb}\sbi{\sixroman n\fiveroman{bh}}\!\RHB{.3}{^p}\ssp
         (\ssp\mu\,,\spp\vPi\ssp)\ssp\sbi M\sp$ and $\ssp S = \ssp
  \bigcap\ssp\big\{\KP{1.1}[\KP{1.2}\mathbb Z\KP{1.1}V\KP{1.1}]\vvs X\sn:
  V\in\Cal V\KP1\}\ssp$ and \newskline6

 $N\aar 0=S\capss21\{\,x:\aall{A\,,\sp u}\,\eexi N\,A\in
 \mu\invss44\image\spp\lbb R_+\sp$ and $\ssp u\in
 \Cal L\,(\sp\vPi\sp,\spp\bosy K\ssp)$ \newskline{30}

 $\impss03
 N\in\mu\invss44\image\snn\{\ssp 0\ssp\}\ssp$ and $\ssp
 u\circ x\image(\sp A\setminus N\ssp)\inc\{\ssp 0\ssp\}\,\} $ \newfline

 $\impss03 \Cal V\not=\Univ\ssp$ and $\ssp F=
 \raise1.8mm\hbox{\font\Å=cmssi5\Åpr}\kern-.3mm
         \LLrs02^p(\ssp\mu\,,\spp\vPi\ssp)\ssp\sbi{M\ssp\aars N_0} \,$, \KP9

\item \ $\mvsLrs02^p(\ssp\mu\,,\spp\vPi\ssp)=\uniqset F: \eexi{\bosy K}\,  \label{ctr mvL_s^p}
         \bosy K\in\setRC\ssp$ and $\ssp\vPi\in\tvsps0(K)\ssp$ and \newskline{11}
 
 $\aall{M\sp,\sp N\aar 0\,,\sp\Omega\,,\sp S\ssp,\sp 
 \Cal V\sp,\sp X}\,\Omega=\bigcup\,\dom\mu\ssp$ and $\ssp 
 X = \sigrd\vPi\expnota^\,\Omega\ssp]_{vs}\sp$ and \newskline{14}

 $ M = \vecs X\capss01\{\,x : (\ssp x\,;\spp\mu\,,\spp\vPi\ssp) \ssp $ is 
        finitely \newskline{49}

                 almost scalarly measurable $ \} \ssp $ and \newskline6

 $\Cal V=
 \roman{Leb}\sbi{\sixroman n\fiveroman{bh}}\!\RHB{.3}{^p}\ssp
         (\ssp\mu\,,\spp\vPi\ssp)\ssp\sbi M\sp$ and $\ssp S = \ssp
  \bigcap\ssp\big\{\KP{1.1}[\KP{1.2}\mathbb Z\KP{1.1}V\KP{1.1}]\vvs X\sn:
  V\in\Cal V\KP1\}\ssp$ and \newskline6

 $N\aar 0=S\capss21\{\,x:\aall{A\,,\sp u}\,\eexi N\,A\in
 \mu\invss44\image\spp\lbb R_+\sp$ and $\ssp u\in
 \Cal L\,(\sp\vPi\sp,\spp\bosy K\ssp)$ \newskline{30}

 $\impss03
 N\in\mu\invss44\image\snn\{\ssp 0\ssp\}\ssp$ and $\ssp
 u\circ x\image(\sp A\setminus N\ssp)\inc\{\ssp 0\ssp\}\,\} $ \newfline

 $\impss03 \Cal V\not=\Univ\ssp$ and $\ssp F=
 \raise1.8mm\hbox{\font\Å=cmssi5\Åpr}\kern-.3mm
         \LLrs02^p(\ssp\mu\,,\spp\vPi\ssp)\ssp\sbi{M\ssp\aars N_0} \,$, \KP9

\item \ $\mLrs03^p(\ssp\mu\sp) =
       \mvLrs03^p(\ssp\mu\,,\tfbbR\ssp) \KP1 $, \hfill
(6) \ $\mLrs03^p(\ssp\mu\sp)\lfbb_C =
       \mvLrs02^p(\,\mu\,,(\sp\tfbbC\ssp)\Reit1) \KP1 $, \KP{13} \inskipline0{-2}

(7) \ $\suptext{vc}0\Lrs03^p(\vcal Q\sp) = \uniqset F:\eexi{\bosy K}\,\vcal Q\ssp$
      is a quasi\ssp-\ssp\erm Euclidean \mathss37{\bosy K}--\,vector column \newskline{17}

  and $\,\aall{\ell\,,\sp\smb N\sp,\sp\mu\,,Q\ssp,\sp\Yps\spp,\spp\vPi}\,
  \vcal Q = (\sp Q\ssp,\Yps\sp,\spp\vPi\ssp) \ssp$ and $\ssp
  \smb N \in \bbNo \ssp$ and \newfline

  $\ell\in\Lis(\,\Yps\Reit0\ssp,\tvbbR5^{\ssmb N}\ssp\big)\ssp$ and $\ssp
  \mu = \seq{\KP1 r : A \inc Q \ssp$ and $\ssp
    r = \Lebmef^{\ssp\ssmb N}\!\fvalue(\ssp\ell\,\image\ssn A\,) \KP1 } $ \KP{8.95} \newfline

  $\impss02 \mu\fvalue Q \not = \Univ \ssp$ and $\ssp
    F = \mvLrs02^p(\ssp\mu\,,\spp\vPi\ssp) \KP1 $, \KP9 \inskipline0{-2}

(8) \ $\LLrs03^p(\sp Q\,\sbi\Yps\ssp,\spp\vPi\ssp) =
       \suptext{vc}0\Lrs03^p(\sn(\sp Q\ssp,\Yps\sp,\spp\vPi\ssp)\sn) \KP1 $, \inskipline0{-2}

(9) \ $\LLrs03^p(\sp Q\,\sbi\Yps) =
       \LLrs03^p(\sp Q\,\sbi\Yps\ssp,\tfbbR\ssp) \KP1 $, \hfill
(10) \ $\LLrs03^p(\sp Q\,\sbi\Yps)\lfbb_C =
       \LLrs03^p(\ssp Q\,\sbi\Yps\ssp,(\sp\tfbbC\ssp)\Reit1) \KP1 $, \KP{11.6} \inskipline0{-3.75}

(11) \ $\LLrs03^p(\sp Q\ssp,\spp\vPi\ssp) = \uniqset F:\eexi\Yps\,\Yps\ssp$ is
       quasi\ssp-\sp usual over $\ssp\tfbbR\ssp$ \newfline

       and $\ssp Q\inc\vecs\Yps\ssp$ and $\ssp
       F = \LLrs03^p(\sp Q\,\sbi\Yps\ssp,\spp\vPi\ssp) \KP1 $, \KP{17.4} \inskipline0{-3.75}

(12) \ $\LLrs03^p(\spp Q\sp) =
        \LLrs03^p(\sp Q\ssp,\tfbbR\ssp) \KP1 $, \hfill
(13) \ $\LLrs03^p(\spp Q\sp)\lfbb_C = \label{df $L^p(Q)_C$}
        \LLrs03^p(\ssp Q\ssp,(\sp\tfbbC\ssp)\Reit1) \KP1 $. \KP{16.9}

  \end{myLeftskip}\end{enumerate}
  \end{constructions}

\begin{theorem}\label{L^p in TVS}

Let \œ$\,0\le p\le\plusinfty$ and let $\,\mu$ be a positive measure. With \œ$\,
\bosy K \in\setRC $ also let \œ$\,\vPi\in\tvsps0(K)$ and either \œ$\, F = 
 \mvLrs02^p(\ssp\mu\,,\spp\vPi\ssp) $ or \œ$ F = 
 \mvsLrs02^p(\ssp\mu\,,\spp\vPi\ssp) \KPt8 $. Then \œ$\,F\in{\ssn}$ $
\TVSps0(K)$ holds. If in addition \œ$\,1\le p$ and $\,\vPi$ is almost 
suitable{\sp\rm, }then \œ$\,F\in\LCSps0(K)$ holds with $\,F$ normable. 
Furthermore{\sp\rm, }for $\,\Nu$ any dominating norm for $\,\vPi$ it holds 
that $\,\seqss43{\inf\sp\big\{\KPt8\|\KP1\Nu\circss01 x\KP1\|\Lnorss33^p_\mu\snn 
 : x\in\smb X\,\} : \smb X\in\vecs F} $ is a compatible norm for $\,F\sp$.
  \end{theorem}

\begin{proof} Let \math{\Omega=\bigcup\,\dom\mu} and \math{ X = 
 \sigrd\vPi\expnota^\,\Omega\ssp]_{vs} } and \vskip.5mm\centerline{$
M=
\vecs X\capss01\{\,x : (\ssp x\,;\spp\mu\,,\spp\vPi\ssp) \ssp \text{ is 
 finitely almost S measurable } \} $} \inskipline{.5}0

where S stands for either \q{simply} or \q{scalarly}. Then \math{X} is a 
vector structure over \mathss32{\sigrd\bosy K}, \,and it is a straightforward 
standard exercise (\sp to the reader\sp) to verify that \math{M} is a vector 
subspace in \mathss31{X}. So \math{X_{\ssp|\,M}} is a vector structure over \mathss32{
\sigrd\bosy K}. Now for \math{\Cal V=
 \roman{Leb}\sbi{\sixroman n\fiveroman{bh}}\!\RHB{.3}{^p}\ssp
         (\ssp\mu\,,\spp\vPi\ssp)\ssp\sbi M} and \math{S = \ssp
  \bigcap\ssp\big\{\KP{1.1}[\KP{1.2}\mathbb Z\KP{1.1}V\KP{1.1}]\vvs X\sn:
  V\in\Cal V\KP1\} } we first see that \œ$\ssp S\inc{\ssn}$ \linebreak
                                                            $M\ssp$ holds and 
that \math{S} is a vector subspace in \mathss31{X}. Hence \math{X_{\ssp|\,S}} 
is a vector structure over \mathss32{\sigrd\bosy K}. For the set \math{
\Cal V\leiss42 S} in \math{X_{\ssp|\,S}} one verifies that the properties 
(\sp\erm{NB}\,1\sp) and (\sp\erm{NB}\,2\sp) given in \cite[p.\ 33]{Jr} hold. 
Indeed, for given \math{\Nu\in\Bqnorm\vPi} utilizing the short- \linebreak
                                                                hands \math{
\|\,x\,\|\subnu=\|\KP1\Nu\circss01 x\KP1\|\Lnorss33^p_\mu } and \mathss38{
 \trN03\smb X\trNu2=\inf\sp\big\{\KPt8\|\,x\,\|\subnu\snn:x\in\smb X\,\} 
}, \,in the case \œ$\ssp p\not=0$ \linebreak
                                  from Proposition \ref{Pro upint} on page \pageref{Pro upint} 
above, putting \math{ \smb M = \smb A \sn \cdot 
 \sup\KPt8\{\,1\ssp,\spp 2\KP1^{p^{-1}-\ssp 1\ssp}\big\} } 

$\null\hfill$ where \math{\smb A} is as on line 3 in (\ref{defi bqnor E}) on 
page \pageref{defi bqnor E p} above, we first see \linebreak
                                                  that \math{
\|\KP1(\ssp x + y\ssp)\svs X\,\|\subnu \sp \le \ssp
 \smb M\KPt8\big(\ssp\|\,x\,\|\subnu + \|\,y\,\|\subnu\sp) } holds for all \mathss31{
x\ssp,\sp y\in\vecss X}. This gives (\sp\erm{NB}\,2\sp) and 
(\sp\erm{NB}\,1\sp) follows trivially from the property given on line 2 in 
(\ref{defi bqnor E}) above. In the case \math{p=0} again (\sp\erm{NB}\,1\sp) 
is trivial, and (\sp\erm{NB}\,2\sp) is seen by observing that we have 
\math{1\le\smb A} and hence for all 
\math{x\ssp,\sp y\in S} and \math{\eta\in\Omega} it holds that \inskipline{.5}{19}

$   \rmmd\snn\circ\sp\Nu\circss00(\ssp x + y\ssp)\vvs X\sn\fvalue\eta
\le \rmmd\sn\fvalue(\ssp\smb A\KPt8(\ssp
     \Nu\circss00 x\fvalue\eta + \Nu\circss00 y\fvalue\eta\ssp)) $ \inskipline{.2}{53}

${} \le \smb A\KPt8(\ssp\rmmd\snn\circ\sp\Nu\circss00 x\fvalue\eta + 
                        \rmmd\snn\circ\sp\Nu\circss00 y\fvalue\eta\ssp) $ \inskipline{.5}0

whence further $\,\upint\ssp\rmmd\snn\circ\sp\Nu\circss00
                  (\ssp x + y\ssp)\vvs X\,|\KP1 A\rmdss41\mu$ \inskipline{.4}{25}

${} \le \smb A\,\big(\ssp
 \upint\ssp\rmmd\snn\circ\sp\Nu\circss00 x\KP1|\KP1 A\rmdss41\mu\sp + \sn
 \upint\ssp\rmmd\snn\circ\sp\Nu\circss00 y\KP1|\KP1 A\rmdss41\mu\ssp) \KP1 $. \inskipline{.5}0

Consequently, we see that there is a unique vector topology \math{
\scrmt T\aR 1} for \math{X_{\ssp|\,S}} such that with \math{E =
 (\sp X_{\ssp|\,S}\ssp,\spp\scrmt T\aR 1)} we have \math{\Cal V\leiss42 S} a 
filter base for \mathss35{\neiBoo E}.

Now letting \math{N\aar 0} be as on lines 6--7 in 
Constructions \ref{defi $L^p$}\,(\ref{simpL^p}) or (\ref{ctr mvL_s^p})\,, it 
is a simple matter to verify that \math{N\aar 0} is a vector subspace in \mathss37{
X_{\ssp|\,S}}. So for \math{Y= X_{\ssp|\,S}\,/\vsquotient N\aar 0} and \linebreak
\œ$F\aar 1=E\,/\tvsquotient N\aar 0\ssp$ we have \math{F\aar 1} a topological 
vector space over \math{\bosy K} with \œ$\ssp\sigrd F\aar 1=Y\sppp$. Let- ting \math{
\scrmt V\aar 0} and \math{\scrmt T} be as on lines 5--6 in 
Constructions \ref{defi $L^p$}\,(\ref{preL^p_{MN_0}})\,, we have \math{ F = 
 (\ssp Y\sppp,\spp\scrmt T\,)} and from Lemma \ref{Le qtvs} on 
page \pageref{Le qtvs} above we see that \math{\scrmt V\aar 0} is a filter 
base for \mathss34{\neiBoo F\aar 1}, \,and hence \math{ F = F\aar 1 \in 
 \tvsps0(K) } holds.

To prove that \math{F\in\TVSps0(K)} holds, we need to show that \math{\taurd F} 
is a Hausdorff topology. For this, 
arbitrarily fixing \mathss38{\smb X\in\vecs F\ssp\setminus\{\,\Bnull_F\} 
}, \,in the case \math{p\not=0} it 
suffices to show existence of some \math{\Nu\in\Bqnorm\vPi} such that \math{
\trN03\smb X\trNu2\not=0 } holds.

To proceed, fixing any \mathss30{x\ar 0\in\smb X}, \,there are some 
\math{u\in\Cal L\,(\sp\vPi\sp,\spp\bosy K\ssp) } and 
\math{A\in\mu\invss44\image\spp\lbb R_+ } 

with $\int_{\,A}\ssp u\circss00 x\ar 0\rmdss01\mu\not=0$ and 
hence also 
$\int_{\,A\,}|\KP1 u\circss00 x\ar 0\sn\fvalue\eta\KP1|\rmdss11\mu\not=0$ . 

Consequently for \math{A\sp\ar 1=
A\capss31\{\,\eta:
|\KP1 u\circss00 x\ar 0\sn\fvalue\eta\KP1| \not= 0 \KP1\} } 

and \math{\Nu\sp=\ssp
\big\langle\KP1|\KP1 u\fvalss02\xi\KP1|:\xi\in\vecs\vPi\KP1\rangle } 
now \math{\mu\fvalue\ssn A\sp\ar 1 > 0} and 
\math{\Nu\in\SemiNor\vPi\inc\Bqnorm\vPi } hold.

For every \math{x\in\smb X} and 
\math{B\in\mu\invss44\image\spp\lbb R_+} we have 
$\int_{\KPt8 B\,}u\circss00 x\rmdss11\mu
=\int_{\KPt8 B\,}u\circss00 x\ar 0\rmdss01\mu$

and hence there is some $N
\in\mu\invss44\image\snn\{\ssp 0\ssp\}$ such that 

$u\circss00 x\fvalue\eta
=u\circss00 x\ar 0\sn\fvalue\eta$ and hence also 
$\Nu\circss00 x\fvalue\eta
=\Nu\circss00 x\ar 0\sn\fvalue\eta$ 
holds for all $\eta\in A\sp\ar 1\ssn\setminus N$ . 

In the case 
\math{p\not=0} we hence 

get \math{
0<\|\KP1\Nu\circss01 x\ar 0\,|\KP1 A\sp\ar 1\,\|\Lnorss33^p_\mu
 =\|\KP1\Nu\circss01 x\KP1|\KP1 A\sp\ar 1\,\|\Lnorss33^p_\mu
\le\|\,x\,\|\subnu } . 

Since this holds for arbitrarily 
given \math{x\in\smb X} we 
consequently 

obtain 
\mathss38{0<\|\KP1\Nu\circss01 x\ar 0\,|\KP1 A\sp\ar 1\,\|\Lnorss33^p_\mu
\le\inf\sp\big\{\KPt8\|\,x\,\|\subnu\snn:x\in\smb X\,\}=
\trN03\smb X\trNu2 }. 

In the case \math{p=0} the above deduction gives \vskip.5mm\centerline{$
0 < \int_{\,\aars A_1\sn}\rmmd\circ\spp\Nu\circss00 x\ar 0\rmdss01\mu
  = \int_{\,\aars A_1\sn}\rmmd\circ\spp\Nu\circss00 x\rmdss11\mu
  = \upint\ssp\rmmd\snn\circ\sp\Nu\circss00 x\KP1|\KP1 A\sp\ar 1\rmdss01\mu $} \inskipline{.3}0

for all \math{x\in\smb X} and hence taking \math{ \eps = 
 \int_{\,\aars A_1\sn}\rmmd\circ\spp\Nu\circss00 x\ar 0\rmdss01\mu } and \inskipline{.2}{7.3}

$V\aar 1 = M\capss21\{\,x : 
 \upint\ssp\rmmd\snn\circ\sp\Nu\circss00 x\KP1|\KP1 A\sp\ar 1\rmdss01\mu
  < \eps \KPt8 \} $ \inskipline{.4}0

and \math{V\aar 0 = \vecs F\capss21\{\,\smb X : \smb X\capss02 V\aar 1 
 \not= \emptyset \KP1 \} } we have \mathss36{ \smb X \sn\not\in\sp V\aar 0 
 \in \scrmt V\aar 0 }. \vskip.3mm

Finally assuming that also \math{1\le p} holds and that \math{\vPi} is almost 
suitable, we fix any dominating norm \math{\Nu} for \mathss31{\vPi}. Then by 
Lemma \ref{Le suit dom} on page \pageref{Le suit dom} above, we see that the 
set \math{ \{\,\vecs F\capss21\{\,\smb X\sn:n\KPt8\trN03\smb X\trNu2 < 1\KPt8
 \} : n\in\rbb Z^+\ssp\big\} } a filter base for \mathss31{\neiBoo F}. 
Consequently $\ssp F$ \linebreak
                      is locally convex and normable with a compatible norm as 
asserted. Note that we get the triangle inequality \math{
    \trN04(\ssp\smb X + \smb Y\,)\vvs Y\trNu2
\le \trN03\smb X\trNu2 + \trN03\smb Y\trNu4 } for \math{\smb X\sp,\sp\smb Y\sp
 \in\vecs F } from \vskip.4mm\centerline{$
    \inf\,\{\,\|\,z\,\|\subnu\snn:z\in(\ssp\smb X+\smb Y\,)\vvs Y\,\}
\le \inf\,\{\,\|\,x\,\|\subnu\snn:x\in\smb X\,\} +
    \inf\,\{\,\|\,y\,\|\subnu\snn:y\in\smb Y\KP1\} \KP1 $,} \inskipline{.4}0

and that the implication \math{\trN03\smb X\trNu2=0\impss33 \smb X=\Bnull_F} 
holds for all \math{\smb X\in\vecs F} since we already know that \math{
\taurd F} is a Hausdorff topology.
  \end{proof}

\begin{lemma}\label{Le 0_{L^p}}

Let \œ$\,0 \le p \le\plusinfty$ and let $\,\mu$ be a positive measure. With \œ$\,
\bosy K\in\setRC$ also let \œ$\,\vPi\in\LCSps0(K)$ be normable{\sp\rm, }and 
let \œ$\,F=\mvLrs02^p(\ssp\mu\,,\spp\vPi\ssp)$ and \œ$\,x\in\smb X\in\vecs F$ 
and \œ$\,y\in\bigcup\ssp\vecs F\sp$. Then \œ$\,y\in\smb X$ holds if and only 
if for every $\,A\in\mu\invss44\image\spp\lbb R_+$ there is some $\,
N\in\mu\invss44\image\snn\{\ssp 0\ssp\}$ with $\,
x\KP1|\KP1(\sp A\spp\setminus N\ssp)\inc y \, $.
  \end{lemma}

\begin{proof} The asserted sufficiency being trivial, we only verify 
necessity. So letting \math{y\in\smb X} and \math{A\in
\mu\invss44\image\spp\lbb R_+ } we need to get some \math{N\in
\mu\invss44\image\snn\{\ssp 0\ssp\} } with \mathss34{
x\KP1|\KP1(\sp A\spp\setminus N\ssp)\inc y }. Now we first find some \math{
N\aar 1\in\mu\invss44\image\snn\{\ssp 0\ssp\} } and simple sequences 
\math{\bosy\sigma\aR 1} and \math{\bosy\sigma\aR 2} in 
 \newline
\math{(\ssp\mu\KP1|\KP1\Pows(\sp A\spp\setminus N\aar 1)\,,\spp\vPi\ssp) } 
with \math{
\roman{ev}\sbi\eta\snn\circ\sp\bosy\sigma\aR 1\to x\fvalue\eta } and \math{
\roman{ev}\sbi\eta\snn\circ\sp\bosy\sigma\aR 2\to y\fvalue\eta } in top \mathss34{
\taurd\vPi} for all \math{\eta\in A\spp\setminus N\aar 1}. Then 
letting \math{S} be the linear 
\mathss37{\sigrd\vPi}--\,span of \math{
\bigcup\ssp\rng(\ssp\bosy\sigma\aR 1\snn\cup\ssp\bosy\sigma\aR 2\spp) } we 
have 
\math{\taurd\vPi\leiss22 S } a separable topology, and for \math{B\ar 1} 
the closed unit dual ball corresponding to some fixed compatible norm for 
\math{\vPi} and for \math{\scrmt T=
\taurd((\spp\vPi_{\ssp/\,S}\spp)\dlsigss00\spp)\leiss22 B\ar 1 } hence 
by \cite[Proposition 8.5.3\ssp, p.\ 157]{Jr} we see that 
\math{\scrmt T} is a 
metrizable and separable topology. Let then 
\math{D} be countable and \mathss37{\scrmt T}--\,dense. Now 
by {\sl Hahn\,--\,Banach\sp} for 
every fixed \math{u\in D} and for all \math{
B\in\dom\mu\capss22\Pows(\sp A\spp\setminus N\aar 1) } we have \math{
\int_{\KP1 B\,}u\circss00 x\rmdss11\mu = 
\int_{\KP1 B\,}u\circss00 y\rmdss11\mu } and hence there is some 
\math{N\sprim1\in\mu\invss44\image\snn\{\ssp 0\ssp\} } with \mathss34{
u\circss00 x\KP1|\KP1(\sp A\spp\setminus N\sprim1\sp)\inc 
u\circss00 y }. By {\sl countable choice\sp} taking the union of these 
\math{N\sprim1} for 
\math{u\in D} we obtain  \math{N} with 
\math{N\aar 1\inc N\in\mu\invss44\image\snn\{\ssp 0\ssp\} } and \mathss34{
u\circss00 x\KP1|\KP1(\sp A\spp\setminus N\ssp)\inc 
u\circss00 y } for all \math{u\in D}. Then to get \math{
x\KP1|\KP1(\sp A\spp\setminus N\ssp)\inc y } arbitrarily fixing \math{\eta\in
A\spp\setminus N } and \math{v\in\Cal L\,(\sp\vPi\sp,\spp\bosy K\ssp) } 
by {\sl Hahn\,--\,Banach\sp} it suffices to have 
\mathss34{v\circss00 x\fvalue\eta=v\circss00 y\fvalue\eta}. Now we find some 
\math{\bosy u\in\sp^\sbbNo\,D } with 
\math{\roman{ev}\KPt2\sbi\xi\snn\circ\sp\bosy u \to v\fvalss01\xi } 
for all 
 \newline
\math{\xi\in\{\,x\fvalue\eta\ssp,\sp y\fvalue\eta\KPt8\} } and then we get 
\inskipline{.2}{21}

$ v\circss00 x\fvalue\eta
= \lim\,(\ssp\roman{ev}\sp\sbi{\emath x\ffvalue\eta}\snn\circ\sp\bosy u\ssp)
= \lim\,(\ssp\roman{ev}\sp\sbi{y\ffvalue\eta}\snn\circ\sp\bosy u\ssp)
= v\circss00 y\fvalue\eta \, $.
  \end{proof}

From Lemma \ref{Le 0_{L^p}} we see in particular that in the case where \math{
\mu} is \rsigma5finite, elements \mathss03{x\ssp,\sp y\in\bigcup\ssp\vecs F} 
represent the same vector of \math{F} if and only if they are equal almost 
everywhere in the classical sense. In the case \math{p\not=0} even without 
\rsigma5finiteness we also see that corresponding to any given compatible norm \math{
\Nu} for \math{\vPi} we have the equality \math{
\inf\sp\big\{\KPt8\|\KP1\Nu\circss01 z\KP1\|\Lnorss33^p_\mu \snn :        \label{discus inf N = N}
 z\in\smb X\,\} = \|\KP1\Nu\circss01 x\KP1\|\Lnorss33^p_\mu } for \mathss31{
x\in\smb X\in\vecs F}. \vskip.3mm

On page \pageref{triv L^p exa} above we noted that for \math{ 0 \le p \le 
 \plusinfty } and e.g.\ for \œ$\ssp F=\mvLrs03^p(\ssp\mu\,,\spp\vPi\ssp) $ \linebreak
with \math{\vPi=\LLrs42^{\frac 12}(\ssbb44 I) } we have \math{F} trivial in 
the sense that \math{\vecs F=\{\,\Bnull_F\} } holds. However, in 
Constructions \ref{defi $L^p$}\,(\ref{preL^p_{MN_0}}) taking \math{ M = 
 \bigcup\ssp\vecs F=\Bnull_F} and \vskip.5mm\centerline{$
N\aar 0 = M\capss21\{\,x:\aall{\Nu\in\Bqnorm\vPi}\,
  \upint\ssp \Abrs33^p\circss00\Nu\circss00 x\rmdss41\mu = 0\KP1\} $} \inskipline{.5}0

we generally get a nontrivial space \math{ E = 
\raise1.7mm\hbox{\font\Å=cmssi5\Åpr}\kern-.3mm
        \LLrs03^p(\ssp\mu\,,\spp\vPi\ssp)\sp\sbi{M\ssp\aars N_0} } such that 
e.g.\ for \œ$\ssp p=\frac 12$ \linebreak
                              and \math{\mu=\Lebmef^{}\,|\KP1\Pows\bbI } the 
spaces \math{E} and \math{\LLrs03^p(\ssbb40 I\times\ssbb04 I) } become 
naturally linearly homeomorphic. We leave the proof as an exercise to the 
  reader.

\begin{theorem}\label{Th L_s^p Ba}

Let \œ$\,1\le p\le\plusinfty$ and with \œ$\,\bosy K\in\setRC$ let \œ$\,\vPi\in
 \LCSps0(K)$ be suitable. Let $\,\mu$ be a positive measure such that in the 
case \œ$\,p=\plusinfty$ it holds that $\,\mu$ is almost decomposable. Then $\,
 \mvsLrs02^p(\ssp\mu\,,\spp\vPi\ssp) \in \BaSps0(K) $ holds. If in addition $\,
\vPi$ is \eit Ba- nachable{\sp\rm, }then also $\,
\mvLrs02^p(\ssp\mu\,,\spp\vPi\ssp) \in \BaSps0(K) $ holds.
  \end{theorem}

\begin{proof} We give the proof for \math{\mvsLrs02^p(\ssp\mu\,,\spp\vPi\ssp) } 
and leave it as an exercise to the reader to make the slight modifications 
that are needed to get the assertion related to \math{
\mvLrs02^p(\ssp\mu\,,\spp\vPi\ssp) } that is classical in the cases where \math{
p\not=\plusinfty} holds or \math{\mu} is \rsigma5finite. For hint we only 
mention that \cite[Theorem 4.2.2\ssp, p.\ 95]{Du} and 
\cite[Corollary 4.2.7, p.\ 97]{Du} together can be utilized to deduce that for 
the obtained \math{y} it then holds that \mathss03{
(\KPt5 y\,;\spp\mu\,,\spp\vPi\ssp) } is finitely almost simply measurable.

Now we put \math{\Omega=\bigcup\,\dom\mu} and \math{ X = 
 \sigrd\vPi\expnota^\ssp\Omega\ssp]_{vs} } and \mathss38{ F = 
 \mvsLrs02^p(\ssp\mu\,,\spp\vPi\ssp)}, \,and let \linebreak
                                                 $\Nu\ssp$ be any dominating 
norm for \mathss31{\vPi}. By Theorem \ref{L^p in TVS} only completeness of \math{
F} has to be verified. For this, it suffices to show that 
for any \PouN$\sevib X \in\sp^\sbbNo\,\vecs F$ with %% $\sum_{\,i\ssp\in\ssp\sbbNo\sp}\|\KPt9\sevib X\fvalss32 i\KP1\| < 1$ \ TAI/VAI \ 
\PouN$\trN05\sevib X\fvalss32 i\trNu6 < 
 4^{\,\mminus(\sp i\sp\yydot\spp)}$ for all $i\in\bbNo\,$, the sequence 
\math{\sevib Y=
\seqss33{\ssigrd F\text{\,-\sp}\sum_{\,k\ssp\in\ssp i^+\ssp}(\sp
\sevib X\fvalss32 k\ssp):i\in\bbNo} } converges in the 
topology $\taurd F$. In order to get this, we first take some $\bosy x \in
\prodc\sp\sevib X$ with 
$\|\KP1\bosy x\fvalss12 i\KP1\|\subnu < 
 4^{\,\mminus(\sp i\sp\yydot\spp)}$ for all $i\in\bbNo\,$, %% $\sum_{\,i\ssp\in\ssp\sbbNo\sp}\|\KP1\bosy x\fvalss12 i\KP1\| < 1\ssp$.
and put $\bosy y=
\seqss33{\ssigrd\vPi\text{\,-\sp}\sum_{\,k\ssp\in\ssp i^+\ssp}(\ssp
\bosy x\fvalss12 k\ssp):i\in\bbNo}\,$.

First considering the case \mathss35{p=\plusinfty}, \,letting \math{\scrmt A} 
and \math{N\sprim1} be as in Definitions \ref{df decomp}\,(2) on page \pageref{decos A} 
above, let \math{\scrmt N\ar 1} be the set of all \math{(\sp A\,,\spp N\aar 1) } 
with \math{A\in\scrmt A} and \œ$\ssp N\aar 1 \in 
 \mu\invss44\image\snn\{\ssp 0\ssp\} $ \linebreak
                                       and such that \math{
\Nu\circ(\ssp\bosy x\fvalss01 i\ssp)\fvalue\eta 
 < 4^{\,\mminus(\sp i\sp\yydot\spp)} } holds for all \math{i\in\bbNo} and \math{
\eta\in A\sp\setminus N\aar 1}. Then we have \mathss32{ \scrmt A \inc 
 \dom\scrmt N\ar 1}, \,and hence by the {\sl axiom of choice\sp} there is a 
function \œ$\ssp\scrmt N\inc\scrmt N\ar 1$ \linebreak
                                           with \mathss35{ \scrmt A \inc 
 \dom\scrmt N}. Now taking \mathss35{N\sprimm1 = 
 N\sprim1\cupss24\bigcup\,\rng\scrmt N }, \,we see that \math{N\sprimm1} is \mathss37{
\mu}--\,negli- gible and such that \math{
 \Nu\circ(\ssp\bosy x\fvalss01 i\ssp)\fvalue\eta
 < 4^{\,\mminus(\sp i\sp\yydot\spp)} } holds for \math{i\in\bbNo} and \mathss32{
\eta\in\Omega\spp\setminus N\sprimm1}.

Letting \math{\vPi\ar 0} be the Banach{\sl able\sp} space determined by the 
norm \math{\Nu} for \mathss30{\sigrd\vPi}, \,from the above we see that for 
every fixed \math{\eta\in\Omega\spp\setminus N\sprimm1 } the sequence \math{
\roman{ev}\sbi\eta\circ\spp\bosy y } converges in the topology \math{
\taurd\vPi\aar 0} and hence also in the weaker topology \mathss30{\taurd\vPi}. 
Taking \vskip.3mm\centerline{$
y = N\sprimm1\sn\times\snn\{\,\Bnull_\vPi\} \cupss22 \{ \,
 (\ssp\eta\ssp,\spp\xi\ssp):\eta\in\Omega\spp\setminus N\sprimm1\sp\text{ and }\ssp
   \roman{ev}\sbi\eta\circ\spp\bosy y\to\xi\ssp\text{ in top }\ssp
    \taurd\vPi\aar 0\,\} \KP1 $,} \inskipline{.3}0

by Lemma \ref{Le deco meas} on page \pageref{Le deco meas} above \math{
(\KPt5 y\,;\sp\mu\,,\spp\vPi\ssp) } is finitely almost scalarly measurable. It 
is also a simple exercise to see that \math{y\in\bigcup\ssp\vecs F} holds, and 
that for the unique class \math{\smb Y} with \math{y\in\smb Y\in\vecs F} we 
indeed have \math{\ebit Y\to\smb Y} in top \mathss30{\taurd F}.

Next, for the case \mathss35{p<\plusinfty}, \,we choose a sequence \math{
\bosy u} of fully positive \mathss37{\mu}--\,mea- surable functions such that 
for all \math{i\in\bbNo} and \math{\eta\in\Omega} we have

$\Nu\circ(\ssp\bosy x\fvalss01 i\ssp)\fvalue\sp\eta\le
\bosy u\fvalss01 i\fvalss10\eta$ and $
\int_{\KP1\Omega}\ssp\aabs99^p\circss01(\ssp
  \bosy u\fvalss01 i\ssp)\rmdss11\mu 
< 4^{\,\mminus(\sp i\sp\yydot\spp)\,p}\,$. Putting

$\roman A\,i=\{\,\eta:
\bosy u\fvalss01 i\fvalss10\eta
\ge 2^{\,\mminus(\sp i\sp\yydot\spp)}\,\big\}$ and 
$\roman B\,i=
\bigcup\,\{\,\roman A\,j:i\inc j\in\bbNo\,\}\,$, for $i\in\bbNo$ we 

have $
\mu\fvalue\roman A\,i\cdot 2^{\,\mminus(\sp i\sp\yydot\spp)\,p}
\le\int_{\KP1\roman A\ssp i}\ssp\aabs99^p\circss01(\ssp
  \bosy u\fvalss01 i\ssp)\rmdss11\mu
\le\int_{\KP1\Omega}\ssp\aabs99^p\circss01(\ssp
  \bosy u\fvalss01 i\ssp)\rmdss11\mu 
< 4^{\,\mminus(\sp i\sp\yydot\spp)\,p}\,$ 

and hence $
\mu\fvalue\roman A\,i < 2^{\,\mminus(\sp i\sp\yydot\spp)\,p}\,$, whence 
further $
\mu\fvalss13\roman B\,i < 2\KP1^{(\sp 1\ssp-\ssp(\sp i\sp\yydot\spp))\,p}\,$, 
and 

consequently for 

$N=\bigcap\,\{\KP1\roman B\,i:i\in\bbNo\,\}\,$, we get $
N\in\mu\invss34\image\snn\{\sp 0\sp\}\,$. For each fixed $\eta\in
\Omega\spp\setminus N$ there 

is $i\ar 0\in\bbNo$ with $\eta\not\in
\roman A\,i$ for all $i\in\bbNo\sn\setminus\spp i\ar 0\,$. Hence for $i\in
\bbNo\sn\setminus\spp i\ar 0\,$, we have 

$\Nu\circ(\ssp\bosy x\fvalss01 i\ssp)\fvalue\sp\eta\le
\bosy u\fvalss01 i\fvalss10\eta < 
2^{\,\mminus(\sp i\sp\yydot\spp)}\sp$, and consequently the sequence $
\,\roman{ev}\sbi\eta\circ\spp\bosy y$ 

converges in the topology $\taurd\vPi\aar 0\,$. It follows that 

there is a function $y:\Omega\to\vecs\vPi$ with $y\fvalue\eta=\Bnull_\vPi$ for $
\eta\in N$ and 

$\roman{ev}\sbi\eta\circ\spp\bosy y\to y\fvalue\eta$ in top 
$\taurd\vPi$ for $\eta\in\Omega\spp\setminus N\sp$. This immediately 
gives that

$(\ssp u\circss01 y\,;\spp\mu\,,\spp\bosy K\ssp)$ is finitely 
measurable for every $u\in\Cal L\,(\spp\vPi\sp,\spp\bosy K\ssp)\,$.

To show that $y\in\bigcup\ssp\vecs F\sp$, we must verify that
$\upint\ssp\aabs99^p\circss01\Nu\circss10 y\rmdss11\mu < \plusinfty$
holds.

% $\int|y|^p$

For each fixed $\eta\in
\Omega\spp\setminus N$ we have

\vskip1mm

$\Nu\circ\sp y\fvalue\eta
=   \lim\sp\sbi{i\ssp\to\ssp\infty\,}(\ssp\Nu\sp\fvalue(\ssp
    \sigrd\vPi\text{\KP1-}\sum\KP1(\ssp
    \roman{ev}\spp\sbi\eta\snn\circ\sp\bosy x\KP1|\KP1 i\ssp))) $ \inskipline{.7}{15.5}

${}
=   \liminf\sbi{i\ssp\to\ssp\infty\,}(\ssp\Nu\sp\fvalue(\ssp
    \sigrd\vPi\text{\KP1-}\sum\KP1(\ssp
    \roman{ev}\spp\sbi\eta\snn\circ\sp\bosy x\KP1|\KP1 i\ssp)))$ \inskipline{.7}{15.5}

${}
\le \liminf\sbi{i\ssp\to\ssp\infty\,}\sum\KP1(\ssp
    \Nu\circ\sp\roman{ev}\spp\sbi\eta\snn\circ\sp\bosy x\KP1|\KP1 i\ssp)
\le \liminf\sbi{i\ssp\to\ssp\infty\,}\sum\KP1(\ssp
               \roman{ev}\spp\sbi\eta\snn\circ\sp\bosy u\KP1|\KP1 i\ssp)$

\vskip1mm
\noin
and hence by Fatou's lemma we get \vskip1mm

$  \upint\ssp\aabs99^p\circss01\Nu\circss10 y\rmdss11\mu
\le
\int_{\KP{1.1}\Omega\,}(\,
\liminf\sbi{i\ssp\to\ssp\infty\,}\sum\KP1(\ssp
               \roman{ev}\spp\sbi\eta\snn\circ\sp\bosy u\KP1|\KP1 i\ssp)
\sbig)3\RHB{.3}{^{\,p}}\rmdss01\mu\,(\sp\eta\sp)$ \inskipline{.7}{30.2}

${}=\int_{\KP{1.1}\Omega\,}
\liminf\sbi{i\ssp\to\ssp\infty\,}\big(\sp\sum\KP1(\ssp
               \roman{ev}\spp\sbi\eta\snn\circ\sp\bosy u\KP1|\KP1 i\ssp)
\sbig)3\RHB{.3}{^{\,p}}\rmdss01\mu\,(\sp\eta\sp)$ \inskipline{.7}{30.2}

${}\le\liminf\sbi{i\ssp\to\ssp\infty}\int_{\KP{1.1}\Omega\sp}
\big(\sp\sum\KP1(\ssp
               \roman{ev}\spp\sbi\eta\snn\circ\sp\bosy u\KP1|\KP1 i\ssp)
\sbig)3\RHB{.3}{^{\,p}}\rmdss01\mu\,(\sp\eta\sp)$ \inskipline{.7}{30.2}

${}\le
\liminf\sbi{i\ssp\to\ssp\infty\,}\big(\sp
      \sum_{\,k\ssp\in\ssp i\sp}
\big(\sp\int_{\KP{1.1}\Omega\,}
(\ssp\bosy u\fvalss01 k\fvalss10\eta\ssp)\RHB{.3}{\KP1^p}
\rmdss01\mu\,(\sp\eta\sp))\KP1^{p^{-1}}\big)
\RHB{.7}{\,^p}$ \inskipline{.7}{30.2}

${} \le \liminf\sbi{i\ssp\to\ssp\infty\,}\big(\sp
      \sum_{\,k\ssp\in\ssp i\,}4^{\,\mminus(\sp k\spp\yydot)}\big)
       \RHB{.7}{\,^p}
    = \big(\sp\frac43\sp\big)\RHB{.7}{\,^p}<\plusinfty$ . \inskipline10

So we have $y\in\bigcup\ssp\vecs F\sp$, and hence there is 
$\smb Y$ with $y\in\smb Y\in\vecs F$. It remains to show that 
$\sevib Y\to\smb Y$ in top $\taurd F$. For this, similarly as above, we 
compute \vskip1mm

$   \upint\ssp\aabs99^p\circss01\Nu\circss10 
    (\ssp\bosy y\fvalss01 i - y\ssp)\vvs X\rmdss11\mu
\le \int_{\KP1\Omega}\,(\,\liminf_{\,j\ssp\to\ssp\infty}\sum_{\KPt8 
     k\ssp\in\ssp j\ssp\setminus\ssp i\sp^+\,}\bosy u\fvalss01 k\fvalss10\eta
      \,)\RHB{.3}{\KPt8^p}\rmdss01\mu\,(\sp\eta\sp) $ \inskipline1{30}

${}=\int_{\KP1\Omega}\ssp\liminf_{\,j\ssp\to\ssp\infty}\big(\sum_{\KPt8 
   k\ssp\in\ssp j\ssp\setminus\ssp i\sp^+\,}\bosy u\fvalss01 k\fvalss10\eta\,)
    \RHB{.3}{\KPt8^p}\rmdss01\mu\,(\sp\eta\sp)$ \inskipline1{15}

${}\le\liminf_{\,j\ssp\to\ssp\infty}\int_{\KP1\Omega}\ssp\big(\sum_{\KPt8 
   k\ssp\in\ssp j\ssp\setminus\ssp i\sp^+\,}
 \bosy u\fvalss01 k\fvalss10\eta\,)\RHB{.3}{\KPt8^p}\rmdss01\mu\,(\sp\eta\sp)$ \inskipline1{15}

${}\le\liminf_{\,j\ssp\to\ssp\infty}\big(\sum_{\KPt8 
   k\ssp\in\ssp j\ssp\setminus\ssp i\sp^+}\big(\int_{\KP1\Omega}\ssp(\ssp
 \bosy u\fvalss01 k\fvalss10\eta\,)\RHB{.3}{\KPt8^p}\rmdss01\mu\,(\sp\eta\sp)
 )\RHB{.3}{\KPt8^{p^{-1}}}\sp\big)^{\,p}$ \inskipline1{15}

${}\le\liminf_{\,j\ssp\to\ssp\infty}\big(\sum_{\KPt8 
   k\ssp\in\ssp j\ssp\setminus\ssp i\sp^+\,} 
    4^{\,\mminus(\sp k\yydot)}\sp\big)\RHB{.7}{\,^p}
= \lim_{\,j\ssp\to\ssp\infty}\big(\sum_{\KPt8 
   k\ssp\in\ssp j\ssp\setminus\ssp i\sp^+\,} 
    4^{\,\mminus(\sp k\yydot)}\sp\big)\RHB{.7}{\,^p}$ \inskipline1{15}

${}=(\ssp 3^{\,\mminus 1}\,4^{\,\mminus(\sp i\yydot)}\sp\big)\RHB{.7}{\,^p}
\to 0 \,$ as $\, i\to\infty \KPt7 $, \,whence the assertion.
  \end{proof}

\begin{corollary}\label{Cor L^p Ban} % OLD = \label{Cor L^p(P_s') Ba}

Let \œ$\,1\le p\le\plusinfty$ and \œ$\,\vPi\in\BaSps0(K)$ with \œ$\,\bosy K\in
 \setRC\KPt7$. Let $\,\mu$ be a positive measure such that in the case \œ$\,p=
 \plusinfty$ it holds that $\,\mu$ is decomposable. Also let \œ$\, F = 
 \mvLrs02^p(\ssp\mu\,,\spp\vPi\ssp) ${\KP1\rm, }or let \œ$\, F = 
 \mvLrs02^p(\ssp\mu\,,\spp\vPi\dlsigss00\spp)$ with $\,\vPi$ reflexive or $\,
\taurd\vPi$ a separable topology. Then \œ$\,F\in\BaSps0(K)$ holds. Furthermore \œ$\,
\mvLrs03^p(\ssp\mu\,,\spp\vPi\dlsigss00\spp) =
 \mvLrs03^p(\ssp\mu\,,\spp\vPi\dlbetss01\sp)$ holds when the space $\,\vPi$ is 
  reflexive.
  \end{corollary}

\begin{proof} The first alternative is immediate. For the second in the 
separable case we note that by Proposition \ref{pro-mea-equ} on page \pageref{pro-mea-equ} 
for \math{\vPi\aar 1=\vPi\dualsigma0} we have \mathss37{ F = 
 \mvsLrs02^p(\ssp\mu\,,\spp\vPi\aar 1\spp) }. Since the conditions of Theorem \ref{Th L_s^p Ba} 
for \math{\vPi\aar 1} in place of \math{\vPi} hold true, consequently the 
assertion follows.

For the reflexive case putting \mathss37{ F\aar 0 = 
 \LLrs02^p(\ssp\mu\,,\spp\vPi\dlbetss01\sp) }, \,it suffices to verify that \math{
F=F\aar 0} holds, and this in turn follows if \math{\vecs F=\vecs F\aar 0} can 
be established. Trivially every \math{\smb Y\in\vecs F\aar 0} is contained in 
some \mathss30{\smb Y\aR 1\in\vecs F}. For the converse, letting \mathss30{ y 
\in \bigcup\ssp\vecs F} and \mathss30{A\in\mu\invss44\image\spp\rbb R^+
}, \,there is \math{N\in\mu\invss44\image\snn\{\ssp 0\ssp\} } such that \math{
(\ssp y\KP1|\KP1 B\,;\sp\mu\KP1|\KP1\Pows B\ssp,\spp\vPi\dlsigss00\spp) } is 
simply measurable for \mathss30{B=A\spp\setminus N}. Hence by Proposition \ref{Pro rfx si mea} 
on page \pageref{Pro rfx si mea} above also \mathss03{
(\ssp y\KP1|\KP1 B\,;\sp\mu\KP1|\KP1\Pows B\ssp,\spp\vPi\dlbetss01\sp) } is 
simply measurable, and so \math{
(\ssp y\KP1|\KP1 A\KPt8;\sp\mu\KP1|\KP1\Pows A\,,\spp\vPi\dlbetss01\sp) } is 
almost simply measurable. Having here \math{A} arbitrary, consequently \math{
y\in\bigcup\ssp\vecs F\aar 0} holds, and we are done.
  \end{proof}

Note that by Banach\,--\,Steinhaus for the second alternative in \label{non Ban but Bai} 
Corollary \ref{Cor L^p Ban} we could weaken the assumption that \math{
\vPi\in\BaSps0(K)} hold to requiring \œ$\ssp\vPi\in\LCSps0(K)$ \linebreak 
                                                               with \math{\vPi} 
normable and barrelled. For an example of an incomplete normable barrelled 
space, see e.g.\ \cite[5.7.\sp\erm B\ssp, p.\ 97]{Jr}\,. \vskip3mm

Since in \cite[10.7, p.\ 214]{Jr} the term {\sl quasi\ssp-\sp normable\sp} is 
reserved for a different meaning, for Proposition \ref{Pro simp Lp dense} 
below we here agree to say that a real or complex topological vector space \math{
E} is {\it pseudonormable\ssp} if{}f there is some \math{\Nu\in\Bqnorm E} with \mathss03{
\{\KPt8\Nu\invss44\image\sp[\KPp1.1 0\,,\spp n^{\,\mminus 1}\sp{\big[\sp} : 
 n\in\rbb Z^+\sp\big\} } a filter base for \mathss34{\neiBoo E}. 

Now the Hausdorff quotients of pseudonormable spaces correspond to the locally 
bounded spaces in the following sense. If \math{E} is pseudonormable, then for \linebreak \mathss03{
F=E\,/\tvsquotient\sn\bigcap\KPt8\neiBoo E} we have that \math{F} is locally 
bounded, and from \cite[Theorem 6.8.3\ssp, p.\ 114]{Jr} it follows existence 
of \math{r\sp,\sp\Nu\aR 0} with \math{0 < r \le 1} and \math{ \Nu\aR 0 \in 
 \Cal S\sbi{\sp\emath r\,}F } with \linebreak
                          \mathss03{
\Nu\aR 0\sn\inve\ssp\image\snn\{\ssp 0\ssp\}\inc\{\,\Bnull_F\} } and \math{
\{\KPt8\Nu\aR 0\sn\inve\ssp\image\sp[\KPp1.1 0\,,\spp n^{\,\mminus 1}\sp{\big[\sp} 
 : n\in\rbb Z^+\sp\big\} } a filter base for \mathss30{\neiBoo F}. Now with \math{
\tweq=\vecs E\times\vecs F\capss21\{\,(\ssp x\ssp,\spp\smb X\sp) : 
 x\in\smb X\,\} } taking \math{\Nu=\Nu\aR 0\circ\sp\tweq} we see that \mathss30{
\Nu\in\Cal S\sbi{\sp\emath r\,}E } \linebreak
                                   with also \math{
\{\KPt8\Nu\invss44\image\sp[\KPp1.1 0\,,\spp n^{\,\mminus 1}\sp{\big[\sp} : 
 n\in\rbb Z^+\sp\big\} } a filter base for \mathss34{\neiBoo E}. Thus the 
{\sl zero neighbourhoods of a pseudonormable space \math{E} are given by a 
single continuous \mathss35{r}--\,seminorm\sp} which is an \mathss35{r
}--\,norm if \math{\taurd E} is a Hausdorff topogy. In particular, the 
Hausdorff pseudonormable spaces are precisely the locally bounded ones.

\begin{proposition}\label{Pro simp Lp dense}

Let $\,p\in\rbb R^+$ and let $\,\mu$ be a positive measure{\sp\rm, }and with 
$\,\bosy K\in{}$ 
\newline
$\setRC$ let $\,\vPi\in\tvsps0(K)$ be pseudonormable. Also let 
$\,F=\mvLrs03^p(\ssp\mu\,,\spp\vPi\ssp)$ and 
\newline
$\,D=
{\ssn}$ $ \vecs F\capss21\{\,\smb X:
\eexi{x\in\smb X}\,x\ssp\text{ is simple in }\ssp
(\ssp\mu\,,\spp\vPi\ssp)\KPt8\} \KPt8 $. Then $\,D$ is $\,\taurd F\,
$--\,dense.
  \end{proposition}

\begin{proof} Put \math{\Omega=\sp\bigcup\ssp\dom\mu} and let \mathss31{x\in
 \smb X\in\vecs F}. Let \math{\Nu\in\Cal S\sbi{\sp\emath r\,}\vPi} be such 
that \math{
\{\KPt8\Nu\invss44\image\sp[\KPp1.1 0\,,\spp n^{\,\mminus 1}\sp{\big[\sp} : 
 n\in\rbb Z^+\sp\big\} } is a filter base for \mathss31{\neiBoo\vPi}. Then 
we take 
some fully positive 
\mathss37{\mu}--\,measurable \math{\Alf} with \math{
\Abrs33^p\snn\circ\KPt2\Nu\circss00 x\le\Alf} 
and \mathss36{       
%\upint\Abrs33^p\circss00\Nu\circss00 x\rmdss11\mu
\int_{\KPp1.1\Omega\,}
%%\Abrs33^p\snn\circ\sp
\Alf\rmdss11\mu < 
 \plusinfty}. Putting 
\newline
\math{\roman A\,n=
\Alf\invss44\image\spp[\KP1 2\KP1^{\emath n\ssp - \ssp 1}\spp,\spp 
 2\KPt9^{\emath n}\ssp{[\sp} } we have \math{\{\,\roman A\,n:n\in\ssbb03 Z\,\}
\inc
\mu\invss44\image\spp\lbb R_+}, \,and by 
{\sl countable choice\sp} we find \math{
N\in\mu\invss44\image\snn\{\ssp 0\ssp\} } and 
\math{
\bmii8 S\in\sp^{\ssbb05 Z}\big(\KPt6^\sbbNo\,\Univ\ssp) } such that for 
every \math{n\in\mathbb Z} we have \math{\bmii8 S\fvalss20 n} a 
simple sequence in 
\math{(\ssp\mu\KP1|\KP1\Pows\roman A\,n\ssp,\spp\vPi\ssp) } with 
\math{\roman{ev}\sbi\eta\snn\circ\spp(\sp\bmii8 S\fvalss20 n\ssp)
\to x\fvalue\eta } in top \math{\taurd\vPi} 
for all \mathss30{\eta\in\roman A\,n\setminus N}. Then let 
\math{\bmii8 S\ar 1\in\sp^{\ssbb05 Z}\big(\KPt6^\sbbNo\,\Univ\ssp) } be 
the unique one such that for all 
\math{n\in\mathbb Z} and \math{i\in\bbNo} and \math{
\sigma=\bmii8 S\ar 1\ssn\fvalue n\fvalss01 i} we have \math{
\sigma\in\sp^{\roman A\,\emath n}\,\vecs\vPi} and such that 
for all \math{\eta\in
\roman A\,n} and \math{\xi\ar 1=
\sigma\fvalue\eta } and \math{\xi=
\bmii8 S\fvalss20 n\fvalss01 i\fvalss10\eta } we have 
\math{\xi\ar 1=\xi} if \math{\Abrs33^p\snn\circ\KPt2\Nu\fvalss11\xi < 
2\KPt9^{\emath n} } holds, otherwise having \mathss32{
\xi\ar 1=\Bnull_\vPi}. Then for 
every \math{n\in\mathbb Z} we have \math{\bmii8 S\ar 1\ssn\fvalue n} a 
simple sequence in 
\math{(\ssp\mu\KP1|\KP1\Pows\roman A\,n\ssp,\spp\vPi\ssp) } with 
\math{\roman{ev}\sbi\eta\snn\circ\spp(\sp\bmii8 S\ar 1\ssn\fvalue n\ssp)
\to x\fvalue\eta } in top \math{\taurd\vPi} for 
all \mathss30{\eta\in\roman A\,n\setminus N}, \,and in addition \math{
\Abrs33^p\snn\circ\KPt2\Nu\circss00(\sp
 \bmii8 S\ar 1\ssn\fvalue n\fvalss01 i\ssp)
\le 2\KP1\Alf } holds for all \math{n\in\mathbb Z} and \mathss36{i\in\bbNo}.

Letting \math{\roman E\,\sigma=
(\ssp\Omega\spp\setminus\sn\dom\sigma\ssp)\times\snn\{\,\Bnull_\vPi\}\cupss22
 \sigma } we now take 
 \newline
\math{\bosy\sigma=
\seqss40{\roman E\KPt8\bigcup\KPt8\{\,
\bmii8 S\ar 1\ssn\fvalue n\fvalss01 i:
i:n\in\mathbb Z\ssp\text{ and }\ssp|\,n\,|\suba\le i\sp\ydot\ssp\}:
i\in\bbNo\snn} } thus obtaining a 
simple sequence \math{\bosy\sigma} in 
\math{(\ssp\mu\,,\spp\vPi\ssp) } with 
\math{\roman{ev}\sbi\eta\snn\circ\spp\bosy\sigma
\to x\fvalue\eta } in top \math{\taurd\vPi} for all \mathss30{\eta\in
\Omega\spp\setminus N} with 
\mathss36{\Alf\fvalss10\eta\not=\plusinfty}, \,and such that also 
\math{
\Abrs33^p\snn\circ\KPt2\Nu\circss00(\ssp\bosy\sigma\fvalss01 i\ssp)\le
2\KP1\Alf } holds for all \mathss36{i\in\bbNo}. Noting 
\math{\int_{\KPp1.1\Omega\,}\Alf\rmdss11\mu < 
 \plusinfty } and that 
from \math{\roman{ev}\sbi\eta\snn\circ\spp\bosy\sigma
\to x\fvalue\eta } in top \math{\taurd\vPi} we get 
\newline
\math{\lim_{\KPt8 i\ssp\to\ssp\infty\,}(\ssp
\Abrs33^p\snn\circ\KPt2\Nu\fvalss10(
\ssp\bosy\sigma\fvalss01 i\fvalss10\eta - x\fvalue\eta\ssp)\svs\vPi\spp
) = 0 } it now follows from the dominated convergence theorem 
that \math{
\lim_{\KPt8 i\ssp\to\ssp\infty}\sp
\int_{\KPp1.1\Omega\,}\Abrs33^p\snn\circ\KPt2\Nu\fvalss10(
\ssp\bosy\sigma\fvalss01 i\fvalss10\eta - x\fvalue\eta\ssp)\svs\vPi
\rmdss11\mu\,(\sp\eta\sp) = 0 } holds, giving the conclusion.
  \end{proof}

\begin{proposition}\label{Pro LpLp* dual}

Let \œ$\,1\le p\le\plusinfty$ and let $\,\mu$ be a positive measure on $\,
\Omega${\,\rm, }and with $\,\bosy K\in\setRC$ let $\,\vPi\in\BaSps0(K)$ and $\,
F=\mvLrs03^p(\ssp\mu\,,\spp\vPi\ssp) \KPt8 $. Also let \inskipline0{23.4}

$\,F\aar 1=\mvLrs14^{p\sast}\ssn(\ssp\mu\,,\spp\vPi\dlsigss00\spp)$ or $\,
   F\aar 1=\mvsLrs14^{p\sast}\ssn(\ssp\mu\,,\spp\vPi\dlsigss00\spp) \, $. For \vskip.5mm\centerline{$
\beta=\vecs F\times\vecs F\aar 1\sn\times\sp\ssbb00 C\capss41\{\,
 (\sp\smb X\sp,\sp\smb Y\sppp,\spp t\ssp):\aall{x\in\smb X\sp,\sp y\in\smb Y}\,
 t=\int_{\KP1\Omega}\,y\,.\KPt8 x\rmdss11\mu\KPt9\} \KP1 $,} \inskipline{.5}0

then $\,\beta$ is a continuous bilinear map $\,F\ssp\sqcap\sp F\aar 1\to
 \bosy K$ \inskipline0{17}

with $\,\seqss40{\beta\,(\,\cdot\,,\smb Y\,) : \smb Y\sp\in\vecs F\aar 1\sn}$ 
  an injection.
  \end{proposition}

\begin{proof} First we note that \math{\beta} is trivially a function since 
the vectors of \math{F} and $\ssp F\aar 1$ \linebreak 
                                           are nonempty sets. Further, if we 
know ($\sp*\sp$) that \math{\vecs F\times\vecs F\aar 1\inc\dom\beta } holds, 
then bilinearity is readily seen. So we only need to prove ($\sp*\sp$) 
together with continuity and the last nondegeneracy assertion. For short let \mathss36{
\roman I\KP1 x\KPt8 y = \int_{\KP1\Omega}\,y\,.\KPt8 x\rmdss11\mu}.

For ($\sp*\sp$) arbitrarily given \math{x\ssp,\sp x\ar 1\in\smb X\in\vecs F} 
and \mathss31{y\ssp,\sp y\ar 1\in\smb Y\in\vecs F\aar 1}, \,we need to verify 
that \math{\roman I\KP1 x\KPt8 y=\roman I\KP1 x\ar 1\,y\ar 1 \in \mathbb C} 
holds. For this we first note that \math{\roman I\KP1 x\KPt8 y \in\mathbb C } 
under the additional assumption that \math{(\ssp x\,;\spp\mu\,,\spp\vPi\ssp) } 
is a simple mv\ssp-\sp map. Indeed, in this case \mathss03{
\roman I\KP1 x\KPt8 y} a finite sum of expressions of the type \math{
\int_{\,A}\ssp\roman{ev}\KPt2\sbi\xi\snn\circ\sp y\rmdss11\mu } where \math{\xi
 \in\vecs\vPi} and \mathss04{A\in\mu\invss44\image\sp\lbb R_+}. Noting that we 
  here have \vskip.5mm\centerline{$
\roman{ev}\KPt2\sbi\xi\snn\circ\sp y\KP1|\KP1 A \in
\bigcup\ssp\vecs\mvLrs24^{p\sast}\ssn(\ssp\mu\KP1|\KP1\Pows A\,,\spp
  \bosy K\ssp) \inc
\bigcup\ssp\vecs\mvLrs42^1(\ssp\mu\KP1|\KP1\Pows A\,,\spp
  \bosy K\ssp) \KP1 $,} \inskipline{.5}0

the assertion follows. Directly from the definition we then see that \mathss30{
\roman I\KP1 x\KPt8 y=\roman I\KP1 x\KPt8 y\ar 1} holds. Then considering the 
general \math{x} first with \math{p\not=\plusinfty} and taking a compatible 
norm \math{\Nu} for \math{\vPi} and letting \math{\Nu\aR 1} be the 
corresponding dual norm, similarly as in the proof of Proposition \ref{Pro simp Lp dense} 
above we find 
a 
simple sequence \math{\bosy\sigma} in 
\math{(\ssp\mu\,,\spp\vPi\ssp) } and 
some \math{N\in\mu\invss44\image\snn\{\ssp 0\ssp\} } 
and a positive \mathss37{\mu}--\,measurable 
\math{\Alf} with \math{\int_{\KPp1.1\Omega\,}\Alf\rmdss11\mu < 
 \plusinfty } and such that 
\math{
\Abrs33^p\snn\circ\KPt2\Nu\circss00(\ssp\bosy\sigma\fvalss01 i\ssp)\le
2\KP1\Alf } holds for all \mathss36{i\in\bbNo} and also 
\newline
\math{\roman{ev}\sbi\eta\snn\circ\spp\bosy\sigma
\to x\fvalue\eta } in top \math{\taurd\vPi} for all \mathss30{\eta\in
\Omega\spp\setminus N}. Then with 
 \newline
\math{A=\Alf\invss46[\KPp1.1\Univ\spp\setminus\{\ssp 0\ssp\}\KP1]
\ssp\setminus N } we take 
a positive \mathss37{\mu}--\,measurable \math{\Alf\ar 1} with 
 \newline
\math{\|\KPt8\Alf\ar 1\ssp\|\Lnorss33^{\sp p\sast\sn}_\mu < \plusinfty } 
and \math{\Nu\aR 1\sn\circ\sp z\KP1|\KP1 A \le \Alf\ar 1} for 
\mathss38{z\in\{\KPt8 y\ssp,\sp y\ar 1\sp\} }. 
For 
 \newline
\math{\Alf\ar 2=2\KP1\RHB{.3}{^p}\LHB{.2}{\sp^{^{-1}}\ssp}
\Abrs33^p\LHB{.2}{\sp^{^{-1}}}\KN{.5}\circ\KPt2\Alf\cdot\Alf\ar 1} now 
\math{
\seqss40{z\,.\KPt8(\ssp\bosy\sigma\fvalss01 i\ssp)
\KP1|\KP1 A
:i\in\bbNo} } converges 
pointwise to 
 \newline
\math{z\,.\KPt8 x\KP1|\KP1 A } and is dominated by 
\math{\Alf\ar 2} 
for which H\"older's inequality gives 
 \newline
\mathss36{
\|\KP1\Alf\ar 2\,\|\Lnorss33^1_\mu
\le
(\,2\sp
\int_{\KPp1.1\Omega\,}\Alf\rmdss11\mu
\ssp)\KP1\RHB{.3}{^p}\LHB{.2}{\sp^{^{-1}}\ssp}
\|\KPt8\Alf\ar 1\ssp\|\Lnorss33^{\sp p\sast\sn}_\mu < \plusinfty
}. Consequently, by the dominated convergence theorem we obtain \inskipline{.5}{14.8}

$\roman I\KP1 x\KPt8 y=
\int_{\,A}\ssp y\,.\KPt8 x\rmdss11\mu=
\lim_{\KPt8 i\ssp\to\ssp\infty\sp}
 \int_{\,A}\ssp y\,.\KPt8(\ssp\bosy\sigma\fvalss01 i\ssp)\rmdss11\mu$ \inskipline{.5}{21.75}

${}=
\lim_{\KPt8 i\ssp\to\ssp\infty\sp}
 \int_{\,A}\ssp y\ar 1\ssp.\KPt8(\ssp\bosy\sigma\fvalss01 i\ssp)\rmdss11\mu
=
\int_{\,A}\ssp y\ar 1\ssp.\KPt8 x\rmdss11\mu
=
\roman I\KP1 x\KPt8 y\ar 1\in\mathbb C \KPt9 $. \vskip.5mm

In the case \math{p=\plusinfty} we modify the above deduction as follow. 
Indeed, now we have \math{p^{\,*}=1} and taking 
a positive \mathss37{\mu}--\,measurable \math{\Alf\ar 1} with 
\math{\|\KPt8\Alf\ar 1\ssp\|\Lnorss33^1_\mu < \plusinfty } 
and \math{\Nu\aR 1\sn\circ\sp z \le \Alf\ar 1} for 
\mathss38{z\in\{\KPt8 y\ssp,\sp y\ar 1\sp\} } we let \mathss38{\scrmt A=
\{\KP1\Alf\ar 1\sn\inve\ssp\image
\spp[\KP1 2\KP1^{\emath n\ssp - \ssp 1}\spp,\spp 
 2\KPt9^{\emath n}\ssp{[\sp}:n\in\ssbb03 Z\,\} }. Then we find 
\math{\smb M\in\lbb R_+} and \math{N\in\mu\invss44\image\snn\{\ssp 0\ssp\} } 
such that for \math{B=\bigcup\,\scrmt A\spp\setminus N} we have 
 \newline
\mathss38{\Nu\circss00 x\KP1|\KP1 B\le B\times\snn\{\ssp\smb M\,\} }. Now 
with the notation \math{\roman E\KP1 x\,A=
(\ssp\Omega\spp\setminus A\ssp)\times\snn\{\KPt8\Bnull_\vPi\}\cupss22
(\ssp x\KP1|\KP1 A\ssp) } a slight modification of 
the above deduction gives us 
\math{\roman I\KP1\roman E\KP1 x\,A\KP1 y=
      \roman I\KP1\roman E\KP1 x\,A\KP1 y\ar 1 \in\mathbb C } 
for all \mathss36{A\in\scrmt A}. Then again 
by dominated convergence we obtain \inskipline{.4}{22}

$\roman I\KP1 x\KPt8 y = \roman I\KP1\roman E\KP1 x\,B\KP1 y = 
     \sum_{\,A\ssp\in\ssp\scrm7 A\,}\roman I\KP1\roman E\KP1 x\,A\KP1 y$ \inskipline{.2}{29}

${}= \sum_{\,A\ssp\in\ssp\scrm7 A\,}\roman I\KP1\roman E\KP1 x\,A\KP1 y\ar 1 
   = \roman I\KP1\roman E\KP1 x\,B\KP1 y\ar 1
   = \roman I\KP1 x\KPt8 y\ar 1 \in \mathbb C \KPt9 $. \inskipline{.5}0

Now for the general case by the above we get \math{ \roman I\KP1 x\KPt8 y = 
 \roman I\KP1 x\KP1 y\ar 1 = \roman I\KP1 x\ar 1\,y\ar 1} by noting that for 
some \math{N\in\mu\invss44\image\snn\{\ssp 0\ssp\} } we have \math{
y\,.\KPt8 x\fvalue\eta=y\,.\KPt8 x\ar 1\ssn\fvalue\eta } for all \mathss30{
\eta\in\Omega\spp\setminus N}. \vskip.3mm

For continuity putting \math{ \trN03\smb X\trNu2 = \inf\sp \big\{ \KPt8
 \|\KP1\Nu\circss01 x\KP1\|\Lnorss33^p_\mu\snn:x\in\smb X\,\} } and \inskipline{.25}{40.1}

$\trN03\smb Y\trNun5 = \inf\sp\big\{\KPt8 \|\KP1 \Nu\aR 1\sn\circ\sp y \KP1
 \|\Lnorss40^{p\sast}_\mu\snn : y\in\smb Y\KP1\} \KP1 $, \inskipline{.25}0

by Theorem \ref{L^p in TVS} on page \pageref{L^p in TVS} above it suffices 
that we have \vskip.3mm\centerline{$
   |\KPp1.2\beta\fvalss10(\sp\smb X\spp,\smb Y\,)\KP1|
\le\trN03\smb X\trNu2\,\trN03\smb Y\trNun5 $} \inskipline{.3}0

for \math{\smb X\in\vecs F} and \math{\smb Y\in\vecs F\aar 1}. By Proposition \ref{Pro Hölder} 
on page \pageref{Pro Hölder} for \math{x\in\smb X} and \math{y\in\smb Y} we 
have \math{ |\KPp1.2\beta\fvalss10(\sp\smb X\spp,\smb Y\,)\KP1| \le
 \|\KP1\Nu\circss01 x\KP1\|\Lnorss33^p_\mu\,
 \|\KP1\Nu\aR 1\sn\circ\sp y\KP1\|\Lnorss40^{p\sast}_\mu} trivially giving the 
  result.

Finally, letting \math{\Bnull_{\aars F_1}\not=\smb Y\sp\in\vecs F\aar 1} we 
need to show existence of some \math{\smb X\in\vecs F} with \linebreak \mathss03{
\beta\fvalss10(\sp\smb X\spp,\smb Y\,)\not=0}. Now, by \math{ \smb Y \not =
 \Bnull_{\aars F_1}} there are \math{y\in\smb Y} and \math{ A \in 
 \mu\invss44\image\spp\rbb R^+} and \mathss30{\xi\in\vecs\vPi} with \mathss36{
\int_{\,A}\ssp\roman{ev}\KPt2\sbi\xi\snn\circ\sp y\rmdss11\mu\not=0}. Then taking \math{
x=(\ssp\Omega\sp\setminus A\ssp)\times\snn\{\,\Bnull_\vPi\}\cupss22
 (\sp A\times\snn\{\ssp\xi\ssp\}\sp\sbig)0 } there is \linebreak \mathss03{
\smb X} with \mathss30{x\in\smb X\in\vecs F}, \,and we now have \mathss36{
\beta\fvalss10(\sp\smb X\spp,\smb Y\,) = 
 \int_{\,A}\ssp\roman{ev}\KPt2\sbi\xi\snn\circ\sp y\rmdss11\mu \not = 0 }.
  \end{proof}

\begin{corollary}\label{Coro Io inj etc}

Let \œ$\,1\le p\le\plusinfty$ and let $\,\mu$ be a positive 
measure on $\,\Omega$ and with 
\newline
$\,\bosy K
 \in\setRC$ let $\vPi\in\BaSps0(K) \KP1 $. Also let 
$\,F=\mvLrs03^p(\ssp\mu\,,\spp\vPi\ssp) $ and let 
\newline
$\,F\aar 1\in\{\,
\mvLrs14^{p\sast}\ssn(\ssp\mu\,,\spp\vPi\dlbetss01\sp)\,,
\mvLrs14^{p\sast}\ssn(\ssp\mu\,,\spp\vPi\dlsigss00\spp)\,,
\mvsLrs14^{p\sast}\ssn(\ssp\mu\,,\spp\vPi\dlsigss00\spp)\KPt8\}$ and 

$
\Iota = \vecs F\aar 1\sn\times\Univ\capss31\{\,(\ssp\smb Y\sppp,\spp\smb U
\sp):\aall{y\in\smb Y}\,$ \inskipline{0}{37.7}

$\smb U=
\vecs F\snn\times\mathbb C\capss31\{\,(\ssp\smb X\sp,\spp t\ssp) : 
 \aall{x\in\smb X}\,
     t = \int_{\KP{1.1}\Omega\,}y\,.\KPt8 x\rmdss11\mu\KPt9\}\sp\}
\KP1 $. 

\noin
Then 
$\,\Iota\in\Cal L\,(\sp F\aar 1\sp,\spp F\dlbetss10\sp)$ holds with 
$\,\Iota$ an injection.
  \end{corollary}

\begin{proof} Note that although written differently, the \math{\Iota} above 
is precisely the same as in Theorem \nfss A\,\ref{main Th} above. Now we first 
see that the assertion directly follows from Proposition \ref{Pro LpLp* dual} 
above in the cases where \math{ F\aar 1 = 
 \mvLrs14^{p\sast}\ssn(\ssp\mu\,,\spp\vPi\dlsigss00\spp) } or \mathss30{
F\aar 1 = \mvsLrs14^{p\sast}\ssn(\ssp\mu\,,\spp\vPi\dlsigss00\spp) } holds. 
For the case \math{ F\aar 1 = 
 \mvLrs14^{p\sast}\ssn(\ssp\mu\,,\spp\vPi\dlbetss01\sp) } putting \math{
F\aar 0=\mvLrs14^{p\sast}\ssn(\ssp\mu\,,\spp\vPi\dlsigss00\spp) } and 
letting \math{\Iota\ar 0} be the corresponding \math{\Iota} 
in the corollary, taking 
% \newline
\math{\Iota\ar 1=
\vecs F\aar 1\sn\times F\aar 0\capss01\{\,(\sp\smb X\spp,\spp\smb Z\sp):
\smb X\inc\smb Z\,\} } we then have 
\mathss34{\Iota=\Iota\ar 0\circ\spp\Iota\ar 1 }. Trivially having 
\math{\Iota\ar 1\in\Cal L\,(\sp F\aar 1\sp,\spp F\aar 0\spp) } we get 
\math{\Iota\in\Cal L\,(\sp F\aar 1\sp,\spp F\dlbetss10\sp) } and we only 
need to show that \math{\Iota\ar 1} is injective. Indeed, supposing that we 
have \math{x\in\smb X} and \mathss34{
(\sp\smb X\spp,\spp\Bnull_{\aars F_0}\sbig)0\in
 \Iota\ar 1 }, \,for arbitrarily given \math{A\in\mu\invss44\image\spp\rbb R^+} 
and \math{w\in\Cal L\,(\sp\vPi\dlbetss01\KPt2,\spp\bosy K\ssp) } we then must 
show that \math{
\int_{\,A}\ssp w\circss00 x\rmdss11\mu = 0 } holds. In order to get this, we 
first note that there are some \math{
N\in\mu\invss44\image\snn\{\ssp 0\ssp\} } and a 
separable linear subspace 
\math{S\ar 1} in \math{\vPi\dlbetss01} with \mathss34{
x\KP1[\KP1 A\setminus N\KP1]\inc S\ar 1 }. Then from 
Lemma \ref{Le 8.17.8 B} on page \pageref{Le 8.17.8 B} above we get 
existence of some \math{\bosy\xi\in\sp^\sbbNo\,\vecs\vPi } with 
\math{\rng\bosy\xi\in\bouSet\vPi } and such that \math{
w\fvalue u = \lim\,(\ssp u\circss11\bosy\xi\ssp) } holds for every 
\mathss34{u\in  S\ar 1}. Now for all 
\math{\eta\in A\setminus N} we have \vskip.25mm\centerline{$
w\circss00 x\fvalue\eta=\lim\,(\ssp x\fvalue\eta\circss11\bosy\xi\ssp)
=
\lim_{\KPt8 i\ssp\to\ssp\infty\,}(\KPt5
\roman{ev}\KPt2\sbi{\bosy\xi\ffvalue i}\circ\sp x\fvalue\eta\ssp) \KP1 $.} \vskip.25mm

Since \math{x\in\smb X\inc\Bnull_{\aars F_0} } holds, for any fixed \math{ \xi
 \in\vecs\vPi} we have \mathss36{\int_{\,A\ssp\setminus\ssp N\,}
  \roman{ev}\KPt2\sbi\xi\snn\circ\spp x\rmdss11\mu = 0 }. From \math{ x \in 
 \bigcup\ssp\vecs F\aar 1} we see that \math{ x\KP1|\KP1 A \in 
 \bigcup\ssp\vecs\mvLrs42^1(\ssp\mu\KP1|\KP1\Pows A\,,\spp\vPi\dlbetss01\sp) } 
holds, and taking into account \math{\rng\bosy\xi\in\bouSet\vPi } we get 
existence of some positive \mathss37{\mu}--\,measurable \mathss30{\Alf} with \math{
\|\KPt8\Alf\,\|\Lnorss33^1_\mu < \plusinfty } and such that \math{
\Abrs03^1\snn\circ\sp\roman{ev}\KPt2\sbi{\bosy\xi\ffvalue i}\circ\sp x
 \KP1|\KP1 A \le \Alf } holds for all \mathss36{i\in\bbNo}. Then by dominated 
convergence we obtain \inskipline{.5}{11}

$ \int_{\,A}\ssp w\circss00 x\rmdss11\mu
= \int_{\,A\ssp\setminus\ssp N\,}w\circss00 x\rmdss11\mu
= \lim_{\KPt8 i\ssp\to\ssp\infty\sp}\int_{\,A\ssp\setminus\ssp N\,}
                \roman{ev}\KPt2\sbi{\bosy\xi\ffvalue i}\circ\sp x\rmdss11\mu 
= 0 \KPt8 $.
  \end{proof}

In the next lemma we utilize the formal definitions \inskipline{.6}{18}

$^{\Omega\sp,\sp\vPi}\ssp\xi\sp\sbi A = 
 (\ssp\Omega\sp\setminus A\ssp)\times\snn\{\,\Bnull_\vPi\} \cupss22
 (\sp A\times\snn\{\ssp\xi\ssp\}\spp\sbig)0 \hfill $ and \KP{18} \inskipline{.4}{18}

$\lfloor\,^{p\sp,\ssp\mu\sp,\ssp\vPi}\ssp\xi\sp\sbi A = \uniqset\smb X : {}
 ^{\bigcup\sp\dom\snn\mu\sp,\KPt3\vPi}\ssp\xi\sp\sbi A \in \smb X \in 
 \vecs\mvLrs03^p(\ssp\mu\,,\spp\vPi\ssp) \KP1 $. \vskip.5mm

If \math{\mu} is a positive measure on \mathss36{\Omega}, \,for all \math{
A\in\mu\invss44\image\spp\lbb R_+} and \math{\xi\in\vecs\vPi} thus \mathss30{
^{\Omega\sp,\sp\vPi}\ssp\xi\sp\sbi A } is the simple function \math{\Omega\to
 \vecs\vPi} that has the value \math{\xi} at points \math{\eta\in A} and \math{
\Bnull_\vPi} else- where. Then \math{
\lfloor\,^{p\sp,\ssp\mu\sp,\ssp\vPi}\ssp\xi\sp\sbi A } is the unique vector of \math{
\mvLrs03^p(\ssp\mu\,,\spp\vPi\ssp) } having \math{
^{\Omega\sp,\sp\vPi}\ssp\xi\sp\sbi A } as one of its representatives.

\begin{lemma}\label{Le-first}

Let $\,1\le p < \plusinfty$ and let $\,\mu$ be a positive measure on 
$\,\Omega\,$. Also with $\,\bosy K\in\setRC$ let 
$\,\vPi\in\BaSps0(K)$ with $\,\Nu$ a compatible norm and 

$\,\Nu\aR 1=\seqss44{
\sup\KPt8(\ssp\Abrs00^1\circ\sp u\circss01\Nu\invss44\image\ssbb15 I)
:u\in\Cal L\,(\sp\vPi\sp,\spp\bosy K\ssp)}$ 

\noin
and $\,F=\mvLrs03^p(\ssp\mu\,,\spp\vPi\ssp)$ and 
$\,\smb U\in\Cal L\,(\sp F\sp,\spp\bosy K\ssp) ${\KP1\rm, }and let 
$\,(\KPt5 y\,;\spp\mu\,,\spp\vPi\dlsigss00\spp)$ be finitely almost 

scalarly measurable with $\,
\|\KP1\Nu\ar 1\snn\circ\sp y\KP1 \|\Lnorss50^{p^*}_\mu < \plusinfty$ 
and such that 

$\,
\smb U\fvalss11\lfloor\,^{p\sp,\ssp\mu\sp,\ssp\vPi}\ssp\xi\sp\sbi A
=\int_{\,A}\ssp\roman{ev}\KPt2\sbi\xi\snn\circ\sp y\rmdss11\mu
$ holds for all $\,A\in
\mu\invss44\image\spp\rbb R^+$ and 
$\,\xi\in\vecs\vPi\sp$. 

Then $\,y\in\bigcup\ssp\vecs
\mvsLrs23^{p\sast}\ssn(\ssp\mu\,,\spp\vPi\dlsigss00\spp)$ holds with 

$\,\smb U=\vecs F\snn\times\mathbb C\capss31\{\,(\ssp\smb X\sp,\spp t\ssp) : 
 \aall{x\in\smb X}\,
  t = \int_{\KPp1.1\Omega\,}y\,.\KPt8 x\rmdss11\mu\KPt9\} \KP1 $.
  \end{lemma}

\begin{proof} We get \math{y\in\bigcup\ssp\vecs
\mvsLrs23^{p\sast}\ssn(\ssp\mu\,,\spp\vPi\dlsigss00\spp) } directly from the 
definition, and hence 
only the last formula has to be verified. To get this, we note that for \math{
\smb X} and \math{x} with \math{\rng x} finite and \mathss30{
x \in \smb X\in\vecs F}, \,i.e.\ for some finite function \math{ \scrmt S 
 \inc (\ssp\mu\invss44\image\spp\rbb R^+\spp\sbig)0\times\vecs\vPi } with \math{
\dom\scrmt S} disjoint and \math{x=
\sigrd\vPi\expnota^\ssp\Omega\sp]_{vs}\,\text{-}\sum_{\,A\ssp\in\ssp
\dom\snn\scrm7 S\,}
{}^{\Omega\sp,\sp\vPi}\ssp
(\ssp\scrmt S\fvalue\ssn A\ssp)\sp\sbi A } we trivially have \inskipline{.5}{20}

$ \smb U\sp\fvalue\snn\smb X
= \sum_{\,A\ssp\in\ssp\dom\snn\scrm7 S}\sp
   \int_{\,A}\ssp y\fvalue\eta\fvalue(\ssp\scrmt S\fvalue\ssn A\ssp)
    \rmdss11\mu\,(\sp\eta\sp)
= \int_{\KP{1.1}\Omega\,}y\,.\KPt8 x\rmdss11\mu $ \inskipline{.7}0

and by Proposition \ref{Pro Hölder} with \math{ \Alf = \seqss33{
 \Nu\aR 1\sn\circ\sp y\fvalue\eta\cdot(\ssp\Nu\circss11 x\fvalue\eta\ssp) : 
  \eta = \eta} } we get \inskipline1{19.1}

$ \big|\sp\int_{\KP{1.1}\Omega\,}y\,.\KPt8 x\rmdss11\mu\KP1|
= \big|\,\sum_{\,A\ssp\in\ssp\dom\snn\scrm7 S}\sp
   \int_{\,A}\ssp y\fvalue\eta\fvalue(\ssp\scrmt S\fvalue\ssn A\ssp)
    \rmdss11\mu\,(\sp\eta\sp)\KP1|$ \inskipline{.7}{41}

${}\le\sum_{\,A\ssp\in\ssp\dom\snn\scrm7 S}\sp\int_{\,A\,}|\KP{1.2}
    y\fvalue\eta\fvalue(\ssp\scrmt S\fvalue\ssn A\ssp) \KP{1.1} |
     \rmdss21\mu\,(\sp\eta\sp)$ \inskipline{.7}{41}

${}\le\sum_{\,A\ssp\in\ssp\dom\snn\scrm7 S}\sp\upint\ssp
       \Nu\fvalss10(\ssp\scrmt S\fvalue\ssn A\ssp)\KP1(\ssp
       \Nu\aR 1\sn\circ\sp y\KP1|\KP1 A\ssp)\rmdss11\mu$ \inskipline{.7}{41}

${}\le \upint\sp\Alf\rmdss11\mu
   \le \|\KP1\Nu\aR 1\sn\circ\sp y\KP1\|\Lnorss40^{p\sast}_\mu\,
        \|\KP1\Nu\circss11 x\KP1\|\Lnorss33^p_\mu \KP1 $. \inskipline10

Since by Proposition \ref{Pro simp Lp dense} on page \pageref{Pro simp Lp dense} 
above the set of vectors with simple representatives is \mathss35{\taurd F
}--\,dense, from Corollary \ref{Coro Io inj etc} it follows that \math{
\smb U\sp\fvalue\snn\smb X = \int_{\KP{1.1}\Omega\,}y\,.\KPt8 x\rmdss11\mu } 
holds for all \math{x\ssp,\sp\smb X} with \math{x\in\smb X\in\vecs F}, \,and 
this is precisely what we needed.
  \end{proof}

\begin{proposition}\label{Pro L^1'=L^i}

Let $\,\mu$ be an almost decomposable positive measure on $\,\Omega${\,\rm, }%
and with $\,\bosy K\in\setRC$ let $\,F=\mvLrs42^1(\ssp\mu\,,\spp\bosy K\ssp)$ 
and $\,F\aar 1=\mvLrs23^\plusinftyy(\ssp\mu\,,\spp\bosy K\ssp)$ and \vskip.5mm\centerline{$
 \Iota\ar 1=\seq{ \KP{1.2} \vecs F\snn\times\mathbb C\capss31\{\,
 (\ssp\smb X\sp,\spp t\ssp) : \aall{x\in\smb X\sp,\sp y\in\smb Y}\,
  t = \int_{\KP{1.1}\Omega\,}x\cdot y\rmdss11\mu\KPt9\} : 
       \smb Y\in\vecs F\aar 1\, } \KP1 $.} \inskipline{.5}0

Then $\,\Iota\ar 1\in\Lis(\sp F\aar 1\sp,\spp F\dlbetss10\sp)$ holds.
  \end{proposition}

\begin{proof} Taking \math{p=1} and \math{\vPi=\bosy K} in Corollary \ref{Coro Io inj etc} 
above, we see that $\Iota\ar 1$ is a continuous linear injection \mathss35{
F\aar 1\to F\dualbeta}. Since \math{F} is normable by Theorem \ref{L^p in TVS} 
above, by Corollary \ref{Cor L^p Ban} the spaces \math{F\aar 1} and \math{
F\dlbetss10} are \erm Banachable, and so by the open mapping theorem we only 
need to prove that \math{ \Cal L\,(\sp F\spp,\spp\bosy K\ssp) \inc 
 \rng\Iota\ar 1 } holds. To establish this, arbitrarily fixing \mathss37{
\smb U\in\Cal L\,(\sp F\spp,\spp\bosy K\ssp) }, \,let 

$\smb M=\sup\sp\big\{\KPt8|
\KP{1.1}\smb U\sp\fvalue\snn\smb X\KPt9|:
\smb X\in\vecs F\ssp\text{ and }\ssp\aall{x\in\smb X}\,
\int_{\KP{1.1}\Omega\,}
|\KP1 x\fvalue\eta\KP1|\rmdss11\mu\,(\spp\eta\spp)
\le 1\KPt9\}$ , 

and let $\scrmt A$ and $N\sprim1$ be as in 
Definitions \ref{df decomp}\,(2) on page \pageref{decos A} above. 
Then $\smb M\in\lbb R_+$ holds, and we let $\scrmt Y\ar 1$ be the set of all 
pairs $(\sp A\ar 1\sp,\spp y\ar 1)$ with $A\ar 1\in\scrmt A$ and 

$y\ar 1\in\bigcup\ssp\vecs
\mvLrs23^\plusinftyy(\ssp\mu
\KP1|\KP1\Pows A\ar 1\sp,\spp\bosy K\ssp) \,$ and $\,
\sup\ssp\rng(\ssp\Abrs00^1\circ\sp y\ar 1) \le \smb M \,$ and

such that 
$\smb U\sp\fvalue\snn\smb X=
\int_{\,\aars A_1}x\cdot y\ar 1\rmdss01\mu$ 

holds for all $x\ssp,\sp\smb X$ with $x\in\smb X\in\vecs F$ and 
$x\invss44\image\sp[\,\ssbb52 C\setminus\{\ssp 0\ssp\}\KP{1.1}]\inc A\ar 1
\ssp$. 

Then from \cite[Theorem 6.4.1\sp, p.\ 162]{Du} we know that 
$\scrmt A\inc\dom\scrmt Y\ar 1$ holds, and hence by the {\sl axiom of choice\sp} 
there is a function $\scrmt Y\inc\scrmt Y\ar 1$ with 
$\scrmt A\inc\dom\scrmt Y\ar 1\inc\dom\scrmt Y\,$. Taking \math{y=
N\sprim1\sn\times\snn\{\ssp 0\ssp\}\cupss24\bigcup\ssp\rng\scrmt Y}, \,by 
Lemma \ref{Le deco meas} on page \pageref{Le deco meas} 
above \math{(\KPt5 y\,;\spp\mu\,,\spp\bosy K\ssp) } now 
\linebreak 
is finitely almost measurable, and hence 
\math{y\in\smb Y} holds for some \mathss31{\smb Y\in\vecs F\aar 1}.

Then for given \math{x\in\smb X\in\vecs F} letting 
\mathss38{\scrmt A\sp\ar 0=\scrmt A\capss31\{\,A:
\int_{\,A\ssp}\Abrs00^1\circ\sp x\rmdss11\mu\not=0\KP1\} }, \,we have 
\math{\scrmt A\sp\ar 0} countable. If \math{\scrmt A\sp\ar 0} is infinite, 
we take any bijection \mathss36{\ebit A:\bbNo\to\scrmt A\sp\ar 0}, \,and if 
it is finite, for some 
\math{\smb N\in\bbNo} we first take a 
bijection \mathss36{\ebit A\ar 0\snn:\smb N\to\scrmt A\sp\ar 0} and then 
put \math{\ebit A=(\ssp\bbNo\sn\setminus\smb N\ssp)\times\snn\{\ssp
\emptyset\ssp\}\cupss21\ebit A\ar 0}. 
Let now 
\math{\ebit B=\sp\big\langle\,\bigcup\KP1(
\sp\ebit A\KPt8|\KP1 i\ssp):i\in\bbNo\,\rangle} and 

$\roman x\,i=(\ssp\Omega\spp\setminus(\ssp\ebit B\fvalss01 i\ssp))\times\snn
\{\ssp 0\ssp\}\cupss22
(\ssp\ebit B\fvalss01 i
\times\snn\{\ssp 1\ssp\}\sp\sbig)0 \,$ and

$\bosy x=
\seqss30{(\ssp\Omega\spp\setminus(\sp\ebit A\sn\fvalue\sp i\ssp))\times\snn
\{\ssp 0\ssp\}\cupss22(\ssp x\KP1|\KP1(\sp\ebit A\sn\fvalue\sp i\ssp)):
i\in\bbNo}$

and \math{\ebit X=
\seqss30{\uniqset\smb X\sn:\bosy x\fvalss01 i\in\smb X\in\vecs F\sn:
i\in\bbNo} } 

and 
\mathss39{\ebit Y\ssp=
\seqss33{\sigrd F\KPt8\text{-\sp}\sum\KP1(\sp\ebit X\KPt8|\KP1 i\ssp)
:i\in\bbNo} }.

Then we have 
\math{\ebit Y\ssp\to\smb X} in top \math{\taurd F} and hence also 
\math{\smb U\circ\ebit Y\to
\smb U\sp\fvalue\snn\smb X}. Consequently, by 
dominated convergence we obtain \inskipline1{11}

$\smb U\sp\fvalue\snn\smb X
= \lim\sp\sbi{i\ssp\to\ssp\infty}\,(\ssp\smb U\circ\ebit Y\fvalss81 i\ssp)
=
 \lim\sp\sbi{i\ssp\to\ssp\infty}\sum_{\KPt8 k\ssp\in\ssp i\,}
(\ssp\smb U\circ\ebit X\fvalss21 k\ssp) $ \inskipline1{18}

${}
= \lim\sp\sbi{i\ssp\to\ssp\infty}\sum_{\KPt8 k\ssp\in\ssp i\sp}
   \int_{\,\bmii6 A\ffvalue k\,}x\cdot y\rmdss11\mu 
= \lim\sp\sbi{i\ssp\to\ssp\infty}
   \int_{\KPt8\bmii6 B\sp\ffvalue i\,}x\cdot y\rmdss11\mu $ \inskipline1{18}

${}
= \lim\sp\sbi{i\ssp\to\ssp\infty}
   \int_{\KPp1.1\Omega\,}x\cdot y\cdot\roman x\,i\rmdss11\mu
= \int_{\KPp1.1\Omega\,}x\cdot y\rmdss11\mu
= \Iota\ar 1\ssn\fvalue\ssp\smb Y\KPt8\fvalue\smb X\ssp$.
  \end{proof}

For a topology \math{\scrmt T} we say that that \math{\scrmt T} is 
{\it separably metrizable\ssp} if{}f \math{\scrmt T} is a metrizable topology 
and there is a countable \math{D\inc\bigcup\,\scrmt T} with \mathss34{
\bigcup\,\scrmt T\inc\roman{Cl\KPt8}\sbi{\scrm7 T\KPt8}D }. In particular then \math{
D} is \mathss37{\scrmt T}--\,dense. Now, for the purpose of Lemma \ref{Le |int| < M imp ...} 
below we put the following

\begin{definitions}\label{df sep cnv metr}

(1) \ Say that \math{C} is {\it separably uniform metrizable\ssp} in \math{E} 
    if{}f \math{E} is a real or complex topological vector space and there are 
some nonempty countable sets \mathss03{D\ssp,\sp\scrmt U} with \math{D\inc C
 \inc\vecs E} and \math{\scrmt U\inc\neiBoo E} and such that \math{D} is \mathss37{
\taurd E\leiss33 C}--\,dense and for every \math{x\in C} it holds that \math{
\big\{\,[\KP1\{\ssp x\ssp\} +\sp U\KPp1.1]\svs E\capss13 C : \sp U \sn \in 
 \scrmt U\KP1\} } is a filter base for \mathss38{
\Nbh(\ssp x\ssp,\spp\taurd E\leiss33 C\ssp) }, \inskipline{.5}2

(2) \ Say that \math{E} is {\it countably separably convex metrizable\ssp} 
    if{}f \math{E} is a real or complex Hausdorff locally convex space and 
there is a countable \math{\scrmt C} with \math{ \vecs E = \bigcup\,\scrmt C} 
and such that \math{C} is separably uniform metrizable in \mathss03{E} for 
every \mathss34{C\in\scrmt C}.
  \end{definitions}

Examples of countably separably convex metrizable spaces are all locally 
convex spaces \math{E} with \math{\taurd E} separably metrizable as well as 
countable strict inductive limits of such spaces. In particular, for example \math{
\mathscr D\,(\ssbb43 R) } and \math{\Cinfty(\ssbb43 R) } are countably 
separably convex metrizable. Also \math{\vPi\dlsigss00\spp} is countably 
separably convex metrizable when \math{\vPi} is normable with \math{
\taurd\vPi} a separable topology.

Note that by the metrization theorem \cite[6.13\ssp, p.\ 186]{Ky} the 
\q{uniform} filter base condition in Definitions \ref{df sep cnv metr}\,(1) 
implies that \math{\taurd E\leiss33 C} is a metrizable topology. We leave it 
as an {\sl open problem\sp} whether we would have obtained an equivalent 
definition if in \ref{df sep cnv metr}\,(1) instead of that uniformity 
condition we had just required \math{\taurd E\leiss33 C} to be a metrizable 
topology. We also remark that the definition given above is precisely what we 
  need in the next

\begin{lemma}\label{Le |int| < M imp ...}

Let $\,\vPi$ be countably separably convex metrizable{\sp\rm, }and let $\,C$ 
be closed and convex in $\,\vPi\sp$. Also let $\,
(\ssp x\,;\spp\mu\,,\spp\vPi\ssp)$ be finitely scalarly integrable and such 
that \vskip.4mm\centerline{$
\int_{\,A\,}u\circss00 x\rmdss11\mu\in\{\KPt8\mu\fvalue\ssn A\cdot t : 
 t\in u\sp\image\sp C\KP1\} $} \inskipline{.4}0

for $\, A \in \mu\invss44\image\spp\rbb R^+$ and $\, u \in 
 \Cal L\,(\sp\vPi\Reit2\sp,\sn\tfbbR\ssp) \KPt8 $. Then  $\,
x\invss46[\KP1\vecs\vPi\sp\setminus C\KP1]$ is $\,\mu\,$--\,negligible.
  \end{lemma}

\begin{proof} Let \math{\scrmt C} be as in 
Definitions \ref{df sep cnv metr}\,(2) above when in place of \math{E} we have 
taken the \math{\vPi} in the lemma. Then taking into account 
(\sp\erm{NB}\,2\sp) in \cite[p.\ 33]{Jr} by {\sl dependent choice\sp} we find 
countable sets \math{D\inc\vecs\vPi} and \math{\scrmt P\inc
 \scrmt C\times(\ssp\taurd\vPi\capss12\neiBoo\vPi\ssp)} with 
\math{\scrmt C\inc\dom\scrmt P } and such that for \math{
(\ssp C\ar 1\sp,\sp U\ssp)\in\scrmt P } and 
\math{\scrmt U=\scrmt P\,\image\snn\{\,C\aar 1\} } it holds that 
\math{U} is absolutely \mathss37{\sigrd\vPi}--\,convex and there is 
\math{V\in\scrmt U} with 
\math{[\KPp1.1 V\sn + \sp V\KPp1.1]\svs\vPi\inc U} 
and 
also \math{C\ar 1\snn\cap\KPt3 D} is \mathss37{
\taurd\vPi\leiss33 C\ar 1}--\,dense and 
                                 for every \math{\xi\in C\ar 1} it holds that 
 \newline
\math{
\big\{\,[\KP1\{\ssp\xi\ssp\} +\sp V\KPp1.1]\svs\vPi\capss13 C\ar 1\sn : 
\sp V \sn \in 
 \scrmt U\KP1\} } is a filter base for \mathss38{
\Nbh(\,\xi\,,\spp\taurd\vPi\leiss33 C\ar 1\spp) }.

Now we let \math{\scrmt R} be the countable set of all triplets \math{
(\ssp C\ar 1\sp,\spp\xi\,,\spp U\aar 1) } such that there is 
\math{U} with \math{
(\ssp C\ar 1\sp,\sp U\ssp)\in\scrmt P } and 
\math{\xi\in C\ar 1\snn\cap\KPt3 D } and 
\math{U\aar 1=[\KPp1.1\{\ssp\xi\ssp\} + \sp U\KPp1.1]\svs\vPi } and 
\mathss36{C\capss23 U\aar 1=\emptyset}. Then by 
{\sl Hahn\,--\,Banach\sp} \cite[7.3.2\ssp, p.\ 130]{Jr} in conjunction 
with {\sl countable choice\sp} we get existence of a function 
\math{\scrmt R\to\Cal L\,(\sp\vPi\Reit2\sp,\sn\tfbbR\ssp) } with the 
property that 
\newline
\math{\sup\,(\ssp u\image C\ssp) < u\fvalss01\xi\ar 1} 
holds for 
\math{(\ssp C\ar 1\sp,\spp\xi\,,\spp U\aar 1\sp,\spp u\ssp)\in
\scrmt S} and \mathss32{\xi\ar 1\in\sp U\aar 1}.

Now taking \mathss38{ \scrmt O = \{\,u\invss44\image\sp\openIval{\sup\,(\ssp 
 u\image C\ssp)\,,\plusinfty\sp}:u\in\rng\scrmt S\KPt8\} }, \,we have \mathss30{
\vecs\vPi\sp\setminus C={\ssn}} \mathss04{\bigcup\,\scrmt O}. Indeed, 
trivially \math{\bigcup\,\scrmt O\inc\vecs\vPi\sp\setminus C } holds, and for 
the converse inclusion arbitrarily fixing \math{ \xi\ar 0 \in 
 \vecs\vPi\sp\setminus C} we first find some \math{C\ar 1} 
with \mathss34{\xi\ar 0\in C\ar 1\in \scrmt C}. Then we find 
\math{U\in\scrmt P\,\image\snn\{\,C\aar 1\} } such that for 
\math{U\aar 0=[\KPp1.1\{\,\xi\ar 0\sp\} + \sp U\KPp1.1]\svs\vPi } 
we have \mathss36{C\capss23 U\aar 0=\emptyset}. We further find 
\math{V\in\scrmt P\,\image\snn\{\,C\aar 1\} } with 
\math{[\KPp1.1 V\sn + \sp V\KPp1.1]\svs\vPi\inc U} 
and then there is some \math{\xi\in C\ar 1\snn\cap\KPt3 D } 
with \mathss30{(\,\xi - \xi\ar 0\spp)\svs\vPi\in\sp V}. Now 
putting \math{U\aar 1=[\KPp1.1\{\ssp\xi\ssp\} + \sp V\KPp1.1]\svs\vPi } 
we have \math{(\ssp C\ar 1\sp,\spp\xi\,,\spp U\aar 1)
 \in\scrmt R=\dom\scrmt S } and hence there is 
\math{u} with \mathss34{
(\ssp C\ar 1\sp,\spp\xi\,,\spp U\aar 1\sp,\spp u\ssp)\in\scrmt S }. Noting 
that now \math{\xi\ar 0\in U\aar 1\inc
u\invss44\image\sp\openIval{\sup\,(\ssp 
 u\image C\ssp)\,,\plusinfty\sp}
} holds, 
%%consequently 
we obtain \mathss34{\xi\ar 0\in\bigcup\,\scrmt O }.

Now, to prove that \math{x\invss46[\KP1\vecs\vPi\sp\setminus C\KP1] } is \mathss37{
\mu}--\,negligible, arbitrarily fixing 
 \newline
\mathss30{A\sp\ar 0\in
 \mu\invss44\image\spp\rbb R^+}, \,by 
{\sl countable choice\sp} and the discussion after 
the proof of Lemma \ref{Le +int} on page \pageref{int not meas} we find 
\math{N\aar 1\in\mu\invss44\image\snn\{\ssp 0\ssp\} } and 
a function \math{\varPhi:\rng\scrmt S\to
\sp^{\aars A_0}\,\bbR} such that \math{
(\ssp\varphi\,;\spp\mu\KP1|\KP1\Pows A\sp\ar 0\,,\sn\tfbbR\ssp)
} is measurable and such that 
\math{u\circss00 x\fvalue\eta=\varphi\fvalue\eta} holds for 
\math{(\ssp u\ssp,\spp\varphi\ssp)\in\varPhi} and 
\mathss32{\eta\in A\sp\ar 0\sn\setminus N\aar 1}. For 
\math{A\in\dom\mu\capss22\Pows A\sp\ar 0} we then also have 
 \newline
\mathss38{
\int_{\,A}\ssp\varphi\rmdss21\mu=
\int_{\,A\,}u\circss00 x\rmdss11\mu\in
\{\KPt8
\mu\fvalue\ssn A\cdot t:t\in u\sp\image\sp C\KP1\} }. Now 
\math{u\sp\image\sp C} is a real interval and hence for 
\math{N\sprim1=\varphi\invss44
[\,\ssbb42 R\setminus u\sp\image\sp C\KPp1.1] } we have 
\mathss36{N\sprim1\in
\dom\mu\capss22\Pows A\sp\ar 0}. Since \math{0 < \mu\fvalue\snn N\sprim1} 
would 
trivially give a contradiction, we in fact have 
\mathss38{N\sprim1\in\mu\invss44\image\snn\{\ssp 0\ssp\} }. 
Then {\sl countable choice\sp} gives us existence of 
\math{N} with 
\math{N\aar 1\inc N\in\mu\invss44\image\snn\{\ssp 0\ssp\} } and such that 
 \newline
\math{\varphi\KPp1.1[\KP1 A\sp\ar 0\sn\setminus N\KP1]\inc 
u\sp\image\sp C} holds for 
\mathss34{(\ssp u\ssp,\spp\varphi\ssp)\in\varPhi}. It being a trivial 
exercise to check that now \math{
x\invss46[\KP1\vecs\vPi\sp\setminus C\KP1]\capss41 A\sp\ar 0
\inc N} holds, we are done.
  \end{proof}

From Lemma \ref{Le |int| < M imp ...} we obtain the following immediate

\begin{corollary}\label{Coro |f|<M}

Let $\,\mu$ be a positive measure on $\,\Omega$ with \œ$\,\mu\fvalss01\Omega < 
 \plusinfty${\,\rm, }and with \œ$\,\smb M\in\lbb R_+$ 
%\linebreak
and 
\œ$\,\varphi\in\bigcup\ssp\vecs\mLrs42^1(\ssp\mu\,,\snn\tfbbC\ssp) $ let \œ$\,
\big|\sp\int_{\,A}\ssp\varphi\rmdss11\mu\KP1|\le\smb M
\KP1(\ssp\mu\fvalue\ssn A\ssp)$ hold for all 
\œ$\,A\in\dom\mu\,$. Then there is 
$\,N\in\mu\invss44\image\snn\{\ssp 0\ssp\}$ such that $\,
|\KP1\varphi\fvalue\eta\KP1|\le\smb M$ 
holds for all $\,\eta\in\Omega\spp\setminus N\spp$.
  \end{corollary}

\begin{lemma}\label{Le L^1_si-compa}

Let 
$\,\bosy K\in\setRC$ and let $\,\mu$ be a positive measure with 

$\,\mu\fvalue\snn\bigcup\ssp\dom\mu < \plusinfty\,$. Also let 
$\,K\in
\bouSet
\mvLrs42^1(\ssp\mu\,,\spp\bosy K\ssp) \KP1$. Then $\,K$ is relatively 

$\,
\taurd\sbig(3\mvLrs42^1(\ssp\mu\,,\spp\bosy K\ssp)\subsigma\spp)\,
$--\,compact if and only if 
for every $\,\eps\in\rbb R^+$ there is $\,\delta\in\rbb R^+$ 

such that $\,\|\KP1\varphi\KP1|\KP1 A\KP1\|\Lnorss33^1_\mu
<\eps$ holds for all 
$\,\varphi\in\bigcup\,K$ and $\,
A\in\mu\invss44\image\spp[\KPp1.1 0\,,\spp\delta\KP1{[}
\KP1$.
  \end{lemma}

\begin{proof}                                                          \newcommand\sFsigmaprime{F^{\kern.2mm\prime}_{\kern-.2mm\sigma}}%
              The assertion is already in 
\cite[Theorem 3.2.1\sp, p.\ 376]{Du-pe}\,, although one should note that 
\q{weakly compact} there means \q{relatively weakly sequentially compact}. To 
get a proper proof, suitably adapt the proof of 
\cite[Theorem 4.21.2\ssp, pp.\ 274\,--\,275]{Edw}\,. Since we shall below need 
the \q{if\sp} part, we here give an explicit proof of it. Indeed, letting 
(\sp$*$\sp) denote the asserted sufficient condition, and putting \mathss03{E=
 \mvLrs42^1(\ssp\mu\,,\spp\bosy K\ssp) } and \mathss38{F=
 \mvLrs23^\plusinftyy(\ssp\mu\,,\spp\bosy K\ssp) }, \,let \mathss30{ \twEps = 
 \seqss43{\roman{ev}\ssp\sbi{\ssmb\Phii}\KPt8|\KP1\Cal L\,(\sp E\ssp,\spp
 \bosy K\ssp):\smb\Phii\in\vecs E} } \linebreak
                                     and \math{\Iota=\Iota\ar 2\circ\sp\twEps} 
where \math{\Iota\ar 2:E\dlbetss12\ssn\dlbetss01\to F\dlbetss10} is the 
transpose of \math{\Iota\ar 1\snn:F\to E\dlbetss11} when \math{\Iota\ar 1} is 
as in Proposition \ref{Pro L^1'=L^i} on page \pageref{Pro L^1'=L^i} above. 
Then \math{\Iota} is a strict morphism \math{E\subsigrs04\to F\dlsigss10} in 
the sense of \cite[Definition 2.5.1\ssp, p.\ 100]{Ho}\,. Now assuming that 
(\sp$*$\sp) holds, since by Alaoglu's theorem from \math{K\in\bouSet E} we 
know that \math{\Cl_taurd{(\sFsigmaprime)}(\ssp\Iota\sp\image\sn K\ssp) } is \mathss37{
\taurd(\sp F\dlsigss00\sp)}--\,compact, it suffices to prove that \math{
\Cl_taurd{(\sFsigmaprime)}(\ssp\Iota\sp\image\sn K\ssp)\inc\rng\Iota } holds. 

Thus arbitrarily given \math{w\in
\Cl_taurd{(\sFsigmaprime)}(\ssp\Iota\sp\image\sn K\ssp) } with \math{\Omega=
\bigcup\ssp\dom\mu} and \vskip.3mm\centerline{$
\roman x\,A=(\ssp\Omega\spp\setminus A\ssp)\times\snn\{\ssp 0\ssp\}\cupss22
 (\sp A\times\snn\{\ssp 1\ssp\}\sp\sbig)0$} \inskipline{.3}0

and \math{\eightroman X\,A=\uniqset\smb\Psii:\roman x\,A\in\smb\Psii\in\vecs F} 
putting \math{\lambda=\seqss33{w\fvalss02\text{\erm X\,}A:A\in\dom\mu} } we 
see that now \math{\lambda} is a real or complex measure that is absolutely \mathss37{
\mu}--\,continuous.

Indeed, given \math{\eps\in\rbb R^+} by (\sp$*$\sp) there is \math{ \delta \in 
 \rbb R^+} such that for all \math{ A \in
 \mu\invss44\image\spp[\KPp1.1 0\,,\spp\delta\KP1{[} } we have \math{
|\KP1 z\fvalss02\text{\erm X\,}A\KP1|\le\eps } for all \math{ z \in 
 \Iota\sp\image\sn K} and hence also \math{|\KP1\lambda\fvalue\ssn A\KP1| = 
 |\KP1 w\fvalss02\text{\erm X\,}A\KP1| \le \eps } holds. 
Note that \math{\lambda} is trivially finitely additive, and that countable 
additivity then follows from the established absolute continuity. Now by 
Radon\,--\,Nikodym there is some \mathss03{\varphi\in\bigcup\ssp\vecs E} with \math{
w\fvalss02\text{\erm X\,}A = \lambda\fvalue\ssn A = 
 \int_{\,A\sp}\varphi\rmdss11\mu } for all \mathss36{A\in\dom\mu}. Then there 
is \mathss30{\smb\Phii} with \mathss35{\varphi\in\smb\Phii\in\vecs E}, \,and 
noting that the linear \mathss37{\sigrd F}--\,span of \math{
\{\KPt8\text{\erm X\,}A:A\in\dom\mu\KPt8\} } is \mathss37{\taurd F}--\,dense, 
we first see that \mathss30{w\fvalue\smb\Psii=
\int_{\KPp1.1\Omega\,}\varphi\cdot\psi\rmdss11\mu} holds for \mathss31{\psi\in
 \smb\Psii\in\vecs F}. Then we get \math{w=\Iota\fvalue\smb\Phii\in\rng\Iota} 
from \inskipline{.5}{19.6}

$  w\fvalue\smb\Psii
= \int_{\KPp1.1\Omega\,}\varphi\cdot\psi\rmdss11\mu
= \Iota\ar 1\ssn\fvalue\smb\Psii\fvalue\smb\Phii
= \twEps\fvalue\smb\Phii\fvalue(\ssp\Iota\ar 1\ssn\fvalue\smb\Psii\sp) $ \inskipline{.4}{26}

${}
= \twEps\fvalue\smb\Phii\circss00\Iota\ar 1\ssn\fvalue\smb\Psii
= \Iota\ar 2\ssn\fvalue(\ssp\twEps\fvalue\smb\Phii\ssp)\fvalue\smb\Psii
= \Iota\ar 2\circ\twEps\fvalue\smb\Phii\fvalue\smb\Psii
= \Iota\fvalue\smb\Phii\fvalue\smb\Psii\,$.
  \end{proof}

% ----------------------------------------------------------------------------

\Ssubhead C          Lifting and integral representations                 \label{Sec C}

As auxiliary results for the proof of Theorem \nfss A\,\ref{main Th} we 
reformulate some forms of the Dunford\,--\,Pettis theorem 
\cite[8.17.6\,--\,8\ssp, p.\ 584]{Edw} in Propositions \ref{Pro Edw 8.17.6} 
and \ref{Pro Edw 8.17.8}  below. The essential content of 
\cite[Lemma 8.17.1\,(a)\,, p.\ 579]{Edw} is in the following

\begin{proposition}\label{Pro lift}

Let $\,\mu$ be an almost decomposable positive measure on $\,\Omega
${\,\rm, }and with \œ$\,\bosy K\in\setRC$ and \œ$\,
G=\mvLrs23^\plusinftyy(\ssp\mu\,,\spp\bosy K\ssp)$ let $\,S$ be a 
vector subspace in $\,\sigrd G$ such that $\,\taurd G\leiss32 S$ is a 
separable topology. \vskip.2mm\centerline{%
Then a choice function $\,c\in
\Cal L\,(\ssp G_{\sp/\ssp S}\ssp,\sp\lll^\plusinftyy\sp
(\ssp\Omega\,,\spp\bosy K\ssp))$ exists.}
  \end{proposition}

\begin{proof} Letting 
\math{\sigrd G_{\ssp|\ssp S}=(\ssp a\,,\sp c\ssp)} and 
\math{R=
\{\,s + t\KP1\imag:s\ssp,\sp t\in\ssbb04 Q\,\} } we first put 
\œ$\ssp X={\ssn}$ %new
\linebreak 
\œ$(\ssp a\,,\sp c\KP1|\KP1(\sp R\times S\ssp)) \ssp$ and 
consider vector subspaces in the 
possibly complex rational vector space \mathss31{X}. 
Thus, letting \math{D} be countable and \mathss37{\taurd G\leiss32 S
}--\,dense, 
let $S\ar 1$ be the linear 
\mathss37{X}--\,span of $D$. Then let $B$ be a linear basis of $
X_{\ssp|\ssp\aars S_1}$. 
By {\sl countable choice\sp} there is a choice function 
$c\ar 0$ of $B$, and we let $\bar c\ar 0$ be its unique linear extension $
X_{\ssp|\ssp\aars S_1}\to
\sigrd\bosy K\expnota^\,\Omega\ssp]_{vs}\ssp_{|
\,\bigcup\,S} \,$. Letting $\scrmt A$ and $N$ be as in 
Definitions \ref{df decomp}\,(2) on page \pageref{decos A} above, 
for every fixed $A\in\scrmt A$ we then see existence of some $N\aar 1\in
\mu\invss44\image\snn\{\ssp 0\ssp\}$ such that 
$|\KP{1.2}\bar c\ar 0\sn\fvalue\smb\Phii\fvalue\eta\KP1|\le
\sup\ssp\big\{\,|\KP1\varphi\fvalue\eta\KP1|:
\eta\in A\KP1\}$ holds for $\varphi\in\smb\Phii\in S\ar 1$ and 
$\eta\in A\sp\setminus N\aar 1\ssp$. Then by 
the {\sl axiom of choice\sp} from the property of being almost 
decomposable we see existence of a \mathss37{\mu}--\,negligible 
$N\spp\yplk$ such that 
$|\KP{1.2}\bar c\ar 0\sn\fvalue\smb\Phii\fvalue\eta\KP1|\le
\sup\ssp\big\{\,|\KP1\varphi\fvalue\eta\KP1|:
\eta\in\Omega\KP1\}$ holds for 
$\varphi\in\smb\Phii\in S\ar 1$ and 
$\eta\in\Omega\spp\setminus N\spp\yplk\sp$. Now taking 
\mathss38{c\ar 1=
\seqss40{
\bar c\ar 0\sn\fvalue\smb\Phii\KP1|\KP1(\ssp\Omega\spp\setminus N\spp\yplk\spp
)\cupss22(\sp N\spp\yplk\timesn\{\ssp 0\ssp\}\sp\sbig)0
:\smb\Phii\in S\ar 1} }, \,we have $c\ar 1$ a 
linear map $X_{\ssp|\ssp\aars S_1}\to
\sigrd\lll^\plusinftyy\sp(\ssp\Omega\,,\spp\bosy K\ssp) $ 
and also 
\math{|\KP{1.2} c\ar 1\sn\fvalue\smb\Phii\fvalue\eta\KP1|\le
\sup\ssp\big\{\,|\KP1\varphi\fvalue\eta\KP1|:
\eta\in\Omega\KP1\} } holds for $\varphi\in\smb\Phii\in S\ar 1$ and all 
$\eta\in\Omega\,$. Then by density of $S\ar 1$ and completeness of 
$\lll^\plusinftyy\sp
(\ssp\Omega\,,\spp\bosy K\ssp)$ letting $c$ be the unique continuous 
extension of $c\ar 1$ we first get 
$c\in
\Cal L\,(\ssp G_{\sp/\ssp S}\ssp,\sp\lll^\plusinftyy\sp
(\ssp\Omega\,,\spp\bosy K\ssp)) \,$, and further using classical 
convergence results for sequences of measurable functions we see that also 
$c$ is a choice function.
  \end{proof}

For a positive measure \math{\mu} on \math{\Omega} and for \mathss38{ X = 
 \sigrd\sn\mLrs23^\plusinftyy(\ssp\mu\sp) }, \,by a {\it lift\ssp} of a linear 
subspace \math{S} in \math{X} one means a linear map \math{c:X_{\ssp|\,S} \to 
 \sigrd\lll^\plusinftyy\sp(\sp\Omega\sp) } that is also a choice function such 
that for \math{(\ssp\smb\Phii,\spp\varphi\ssp)\in c} and for every \math{
\varphi\ar 1\in\smb\Phii } and \math{A\in\mu\invss44\image\spp\lbb R_+} we \linebreak
have \mathss38{ \sup\KP1(\ssp\varphi\ssp\image\ssn A\ssp) \le \sup\KP1(\ssp
 \varphi\ar 1\sn\image\ssn A\ssp) }. So from the proof of 
Proposition \ref{Pro lift} we see that we could have more specifically stated 
that a lift exists. However, below we shall have no essential use of this 
additional information encoded in the definition of lift. \vskip.3mm

Essentially the content of \cite[Theorem 8.17.2\ssp, p.\ 582]{Edw} is in the 
following

\begin{proposition}\label{Pro Edw 8.17.2}

Let $\,\mu$ be an almost decomposable positive measure on $\,\Omega${\,\rm, }%
and with \œ$\,\bosy K\in\setRC$ let \œ$\,\vPi\in\LCSps0(K)$ be normable and 
such that $\,\taurd\vPi$ is a separable topology. Also let \œ$\,\smb U \in 
 \Cal L\,(\sp\vPi\sp,\mvLrs23^\plusinftyy(\ssp\mu\,,\spp\bosy K\ssp))$ be such 
that there is a choice function \œ$\, c \in 
 \Cal L\,(\ssp\mvLrs23^\plusinftyy(\ssp\mu\,,\spp\bosy K\ssp)
      _{\ssp/\ssp\rng\snn\sixmath U} \, , \ssp
         \lll^\plusinftyy\sp(\ssp\Omega\,,\spp\bosy K\ssp)) \, $. Then there 
is \œ$\, y \in \bigcup\ssp\vecs\sn
 \mvsLrs23^\plusinftyy(\ssp\mu\,,\spp\vPi\dlsigss00\spp)$ with $\, \smb U = 
 \vecs\vPi\times\vecs\sn\mvLrs23^\plusinftyy(\ssp\mu\,,\spp\bosy K\ssp)\capss31
  \{\,(\ssp\xi\,,\sp\smb\Phii\spp) : \roman{ev}\KPt2\sbi\xi\snn\circ\sp y \in
   \smb\Phii\,\} \KP1 $.
  \end{proposition}

\begin{proof} With \math{ y = \seqss33{
 \roman{ev}\sp\sbi\eta\snn\circ\sp c\circss11\smb U:\eta\in\Omega} } we have \math{
y} a function \mathss38{\Omega\to\Cal L\,(\sp\vPi\sp,\spp\bosy K\ssp) }, \,and 
for fixed \math{\xi\in\vecs\vPi} noting that \math{c} is a choice function, we 
  obtain \inskipline{.5}2

(\sp$*$\sp) $ \KP{10.55}
  \roman{ev}\KPt2\sbi\xi\snn\circ\sp y 
= \seqss33{\roman{ev}\sp\sbi\eta\snn\circ\sp c\circss11\smb U\fvalss10\xi : 
   \eta\in\Omega}
= \seqss33{c\fvalue(\ssp\smb U\fvalss10\xi\ssp)\fvalue\eta:\eta\in\Omega}$ \inskipline{.5}{29}

${}
=   c\fvalue(\ssp\smb U\fvalss10\xi\ssp)
\in \smb U\fvalss10\xi
\in \vecs\sn\mvLrs23^\plusinftyy(\ssp\mu\,,\spp\bosy K\ssp) \KP1 $. \KP6 From 
  this, \inskipline{.5}0

we see that \math{\smb U = 
 \vecs\vPi\times\vecs\sn\mvLrs23^\plusinftyy(\ssp\mu\,,\spp\bosy K\ssp)\capss31
  \{\,(\ssp\xi\,,\sp\smb\Phii\spp) : \roman{ev}\KPt2\sbi\xi\snn\circ\sp y \in
   \smb\Phii\,\} } holds.

It remains to verify that \math{y \in \bigcup\ssp\vecs\sn
 \mvsLrs23^\plusinftyy(\ssp\mu\,,\spp\vPi\dlsigss00\spp) } holds. First, to 
prove that $(\KPt5 y\,;\sp\mu\,,\spp\vPi\dlsigss00\spp)$ 
is finitely almost scalarly measurable, let 
$\scrmt A$ and $N\sprim1$ be as in 
Definitions \ref{df decomp}\,(2) on page \pageref{decos A} above, and let 
$D$ be countable and $\taurd\vPi\,$--\,dense. 
Then for every fixed 
$A\in\scrmt A$ and $\xi\in D$ from (\sp$*$\sp) we see existence of 
$N\in\mu\invss44\image\snn\{\ssp 0\ssp\}$ such that 
with 
\linebreak
\œ$B=A\sp\setminus N\ssp$ we have 
\math{(\,\roman{ev}\KPt2\sbi\xi\snn\circ\sp y\KP1|\KP1 B\,;
\sp\mu\KP1|\KP1\Pows B\ssp,\spp\bosy K\ssp) } measurable. 
By countability of $D$ we can here take $N$ independent of 
\mathss34{\xi\in D}. 
Then by the {\sl axiom of choice\sp} 
in conjunction with the decomposability property we get 
existence of a $\mu\,$--\,negligible $N\sprimm1$ such that 
for all $A\in\mu\invss44\image\spp\lbb R_+$ and $\xi\in D$ 
with $B=A\sp\setminus N\sprimm1$ we have 

$(\,\roman{ev}\KPt2\sbi\xi\snn\circ\sp y\KP1|\KP1 B\,;
\sp\mu\KP1|\KP1\Pows B\ssp,\spp\bosy K\ssp)$ 

\noin
almost measurable. 
By countability and 
density of $D$ we then see that 
$(\KPt5 y\,;\sp\mu\,,\spp\vPi\dlsigss00\spp)$ 
is finitely almost scalarly measurable.

To complete the proof of \mathss38{y \in \bigcup\ssp\vecs\sn
 \mvsLrs23^\plusinftyy(\ssp\mu\,,\spp\vPi\dlsigss00\spp) }, \,by \math{
\bouSet(\sp\vPi\dlbetss01\sp)\inc\bouSet(\sp\vPi\dlsigss00\spp)} it suffices 
to show that \math{\rng y\in\bouSet(\sp\vPi\dlbetss01\sp) } holds. For this, 
we first note that there is some \mathss03{\smb A\in\lbb R_+} such that for \math{
\varphi\in\smb\Phii\in\rng\smb U} we have \mathss36{
 \|\KP1 c\fvalss11\smb\Phii\,\|\Lnorss33^\plusinftyy_\mu \le \smb A \KP1 
 \|\KP1\varphi\KP1\|\Lnorss33^\plusinftyy_\mu}. Then with \math{ B\ar 1 = 
 \rng\smb U\capss21\{\KPt8\smb\Phii \snn : \aall{\varphi\in\smb\Phii}\,
 \|\KP1\varphi\KP1\|\Lnorss33^\plusinftyy_\mu\le 1\KPt8\} } taking \math{ U = 
 \smb U\invss44\image\snn B\ar 1 } we have \mathss03{U\in\neiBoo\vPi } and 
hence \vskip.3mm\centerline{$
    |\KPp1.1 y\fvalue\eta\fvalss01\xi\KP1|
=   |\KPp1.1 c\circss00\smb U\fvalss11\xi\fvalss10\eta\KP1|
\le \|\KP1 c\circss00\smb U\fvalss11\xi\KP1\|\Lnorss33^\plusinftyy_\mu
\le \smb A\KP1\|\KP1\varphi\KP1\|\Lnorss33^\plusinftyy_\mu
\le \smb A $} \inskipline{.5}0

for \math{\eta\in\Omega} and \math{\xi\in U} and \mathss35{ \varphi \in 
 \smb U\fvalss11\xi}. Consequently \math{ \rng y \in 
 \bouSet(\sp\vPi\dlbetss01\sp) } holds.
  \end{proof}

The essential content of \cite[Theorem 8.17.6\ssp, p.\ 584]{Edw} is in the 
following

\begin{proposition}\label{Pro Edw 8.17.6}

Let $\,\mu$ be an almost decomposable positive measure on $\,\Omega${\,\rm, }%
and with \œ$\,\bosy K\in\setRC$ let \œ$\,\vPi\in\LCSps0(K)$ be normable and 
such that $\,\taurd\vPi$ is a separable topology. Then for every \œ$\,\smb V
\in\Cal L\,(\ssp\mvLrs42^1(\ssp\mu\,,\spp\bosy K\ssp)\,,\spp\vPi\dlbetss01\sp)$ 
there is \œ$\, y \in \bigcup\ssp\vecs\sn
 \mvsLrs23^\plusinftyy(\ssp\mu\,,\spp\vPi\dlsigss00\spp)$ such that $\,
\smb V\fvalss60\smb\Phii\fvalss00\xi = \int_{\KP{1.1}\Omega\,} 
 \roman{ev}\KPt2\sbi\xi\snn\circ\sp y\cdot\varphi\rmdss21\mu$ holds for $\,
\varphi\in\smb\Phii\in\dom\smb V$ and $\,\xi\in\vecs\vPi\sp$.
  \end{proposition}

\begin{proof} We first get a continuous bilinear map \math{ \beta : 
 \mvLrs42^1(\ssp\mu\,,\spp\bosy K\ssp)\ssp\sqcap\sp\vPi\to\bosy K } defined by \mathss35{
(\ssp\smb\Phii\spp,\spp\xi\ssp)\mapsto\smb V\fvalss70\smb\Phii\fvalss01\xi }, \,%
and then a continuous linear map \math{ \smb U : \vPi \to 
 \mvLrs23^\plusinftyy(\ssp\mu\,,\spp\bosy K\ssp) } de- fined by \math{\xi
 \mapsto\Iota\ar 1\sn\inve\sp\fvalue(\ssp\beta\,(\,\cdot\,,\spp\xi\ssp)) } 
where \math{\Iota\ar 1} is as in Proposition \ref{Pro L^1'=L^i} on page \pageref{Pro L^1'=L^i} 
above. Then by Propositions \ref{Pro lift} and \ref{Pro Edw 8.17.2} there is \math{
y \in \bigcup\ssp\vecs\sn
 \mvsLrs23^\plusinftyy(\ssp\mu\,,\spp\vPi\dlsigss00\spp) } with \vskip.4mm\centerline{$
\smb U = 
 \vecs\vPi\times\vecs\sn\mvLrs23^\plusinftyy(\ssp\mu\,,\spp\bosy K\ssp)\capss31
  \{\,(\ssp\xi\,,\sp\smb\Psii\sp) : \roman{ev}\KPt2\sbi\xi\snn\circ\sp y \in
   \smb\Psii\,\} \KP1 $.} \inskipline{.4}0

Now for \math{\xi\in\vecs\vPi} with \math{\varphi\in\smb\Phii\in\dom\smb V } 
and \math{\roman{ev}\KPt2\sbi\xi\snn\circ\sp y \in
   \smb\Psii\in\vecs\sn\mvLrs23^\plusinftyy(\ssp\mu\,,\spp\bosy K\ssp) } we 
have \mathss03{\Iota\ar 1\sn\inve\sp\fvalue(\ssp
 \roman{ev}\KPt2\sbi\xi\snn\circ\sp\smb V\,) =
 \Iota\ar 1\sn\inve\sp\fvalue(\ssp\beta\,(\,\cdot\,,\spp\xi\ssp)) =
 \smb U\fvalue\xi = \smb\Psii } and hence \math{
\roman{ev}\KPt2\sbi\xi\snn\circ\sp\smb V = 
 \Iota\ar 1\ssn\fvalue\ssp\smb\Psii } whence finally \mathss36{
\smb V\fvalss60\smb\Phii\fvalss00\xi = 
 \roman{ev}\KPt2\sbi\xi\snn\circ\sp\smb V\,\fvalue\smb\Phii = 
 \Iota\ar 1\ssn\fvalue\ssp\smb\Psii\fvalue\smb\Phii =
 \int_{\KP{1.1}\Omega\,}
 \roman{ev}\KPt2\sbi\xi\snn\circ\sp y\cdot\varphi\rmdss21\mu}.
  \end{proof}

The content of \cite[Lemma 8.17.8 \erm A\ssp, p.\ 584]{Edw} is in the 
  following

\begin{lemma}\label{Le 8.17.8 A}

With \œ$\,\bosy K\in\setRC$ let \œ$\,E\in\LCSps0(K)$ be normable{\ssp\rm, }and 
let \œ$\,F=E\dlbetss12\,$. Also let $\,S\ar 1$ be a closed linear subspace in $\,
F$ such that \œ$\,\taurd F\leiss22 S\ar 1$ is a separable topology. Then there 
is a closed linear subspace $\,S$ in $\,E$ with \œ$\,\taurd E\leiss32 S$ a 
separable topology and such that \œ$\,\seq{\KP1 u\KPt9|\KP1 S:u\in S\ar 1\,}$ 
is a strict morphism \œ$\, F_{\sp/\ssp\aars S_1} \to 
 (\sp E_{\ssp/\ssp S}\spp)\dlbetss12 $ in the sense of 
  {\,\rm\cite[Definition 2.5.1\sp, p.\ 100]{Ho}\,.}
  \end{lemma}

\begin{proof} Fixing a compatible norm \math{\Nu} for \mathss35{E}, \,let \math{
\bosy u\in\sp^\sbbNo\,S\ar 1} be such that $\rng\bosy u$ is \mathss34{
\taurd F\leiss22 S\ar 1}--\,dense, and let $\bosy x\in\sp
 ^\sbbNo\,(\ssp\Nu\invss44\image\ssbb24 I)$ be such that $
(\ssp 1 + i\ssp\yplus\sp\ydot\ssp^{\mminus 1}\sp\big)\KP1
 (\ssp\bosy u\,.\KPt8\bosy x\fvalss01 i\ssp)
=\sup\ssp\big\{\KPt8|\KP1\bosy u\fvalss01 i\fvalss10 x\KP1|:
 x\in\Nu\invss44\image\ssbb23 I\,\}$ holds for all $i\in\bbNo\,$. Then we let 
$S$ be the closed linear span of $\rng\bosy x$ in $E\,$, and take \mathss38{
\Iota=\seqss40{u\KPt9|\KP1 S:u\in S\ar 1} }.

One easily verifies that \math{
\sup\ssp\big\{\KPt8|\KP1 u\fvalue x\KP1|:x\in\Nu\invss44\image\ssbb23 I\,\} 
\le \sup\ssp\big\{\KPt8|\KP1 u\fvalue x\KP1|:x\in\rng\bosy x\KPt9\} } holds 
for every fixed $u\in S\ar 1\ssp$, and hence we get \vskip.6mm

$
\sup\ssp\big\{\KPt8|\KP1\Iota\fvalss10 u\fvalue x\KP1|:x\in
   \Nu\invss44\image\ssbb20 I\capss32 S\KP1\}
=   \sup\ssp\big\{\KPt8|\KP1 u\fvalue x\KP1|:x\in
   \Nu\invss44\image\ssbb20 I\capss32 S\KP1\}
$ \inskipline{.4}{10.6}

${}\le \sup\ssp\big\{\KPt8|\KP1 u\fvalue x\KP1|:x\in
    \Nu\invss44\image\ssbb23 I\,\}
\le \sup\ssp\big\{\KPt8|\KP1 u\fvalue x\KP1|:x\in\rng\bosy x\KPt9\}
$ \inskipline{.4}{10.6}

${}\le \sup\ssp\big\{\KPt8|\KP1 u\fvalue x\KP1|:x\in
    \Nu\invss44\image\ssbb20 I\capss32 S\KP1\}
=   \sup\ssp\big\{\KPt8|\KP1\Iota\fvalss10 u\fvalue x\KP1|:x\in
    \Nu\invss44\image\ssbb20 I\capss32 S\KP1\} \KP1 $, \inskipline{.6}0

from which the assertion easily follows.
  \end{proof}

The content of \cite[8.17.8\ssp, pp.\ 584\,--\,586]{Edw} is in the following

\begin{proposition}\label{Pro Edw 8.17.8}

Let $\,\mu$ be an almost decomposable positive measure on $\,\Omega${\,\rm, }%
and with \œ$\,\bosy K\in\setRC$ let \œ$\,\vPi\in\BaSps0(K)$ and \œ$\,\smb V\in
 \Cal L\,(\ssp\mvLrs42^1(\ssp\mu\,,\spp\bosy K\ssp)\,,\spp\vPi\dlbetss01\sp)$ 
be such that $\,\taurd(\spp\vPi\dualbeta\spp)\leiss32\rng\smb V$ is a 
separable topology. Then there is \œ$\, y \in \bigcup\ssp\vecs\sn
 \mvsLrs23^\plusinftyy(\ssp\mu\,,\spp\vPi\dlsigss00\spp)$ with $\,
\rng y\inc\CltaurdvPidualbeta\rng\smb V$ and such that \vskip.1mm\centerline{$
\smb V\fvalss60\smb\Phii\fvalss00\xi = \int_{\KP{1.1}\Omega\,} 
 \roman{ev}\KPt2\sbi\xi\snn\circ\sp y\cdot\varphi\rmdss21\mu $} \inskipline{.5}0

holds for $\,\varphi\in\smb\Phii\in\dom\smb V$ and $\,\xi\in\vecs\vPi\sp$.
  \end{proposition}

\begin{proof} Taking \math{S\ar 1=\CltaurdvPidualbeta\rng\smb V } and \math{E=
 \vPi} in Lemma \ref{Le 8.17.8 A} above, there is a closed linear subspace \math{
S} in \math{\vPi} with \math{\taurd\vPi\leiss32 S } a separable topology and 
such that for \math{\Iota=\seq{\KP1 u\KPt9|\KP1 S:u\in S\ar 1\,} } we have \math{
\Iota} a strict morphism \mathss37{\vPi\dlbetss01{}_{\sp/\ssp\aars S_1} \to
 (\sp\vPi_{\ssp/\ssp S}\spp)\dlbetss12 }.

Now we have \mathss38{ \Iota\circ\sp\smb V \in 
 \Cal L\,(\ssp\mvLrs42^1(\ssp\mu\,,\spp\bosy K\ssp) \, , \sp
 (\sp\vPi_{\ssp/\ssp S}\spp)\dlbetss12\sp) }, \,and by separability of the 
topol- \linebreak
       ogy \math{ \taurd(\sp\vPi_{\ssp/\ssp S}\spp) = 
 \taurd\vPi\leiss32 S } we can apply Proposition \ref{Pro Edw 8.17.6} to 
deduce existence of some \linebreak
                         $y\ar 1 \in \bigcup\ssp\vecs\sn
 \mvsLrs23^\plusinftyy(\ssp\mu\,,\spp(\sp\vPi_{\ssp/\ssp S}\spp)\dlsigss12) \ssp$ 
 with \inskipline{.5}2

(\sp r\sp)               $      \KP{12}
  \smb V\fvalss60\smb\Phii\fvalss00\xi 
= \Iota\fvalue(\ssp\smb V\fvalss60\smb\Phii\ssp)\fvalue\sp\xi
= \Iota\circ\spp\smb V\fvalss60\smb\Phii\fvalss00\xi = \int_{\KP{1.1}\Omega\,} 
   \roman{ev}\KPt2\sbi\xi\snn\circ\sp y\ar 1\ssn\cdot\varphi\rmdss21\mu$ \inskipline{.5}0

whenever \math{\varphi\in\smb\Phii\in\dom\smb V} and \math{\xi\in S} hold. 
Then taking \math{ y = \Iota\inve\snn\circ\sp y\ar 1 } we get a function \mathss38{
y:\Omega\to S\ar 1\inc\Cal L\,(\sp\vPi\sp,\spp\bosy K\ssp) }, \,and it remains 
to establish \œ$\ssp y \in \bigcup\ssp\vecs\sn
 \mvsLrs23^\plusinftyy(\ssp\mu\,,\spp\vPi\dlsigss00\spp) $ and to show that \inskipline{.5}2

(\sp s\sp)               $      \KP{12}
\smb V\fvalss60\smb\Phii\fvalss00\xi = \int_{\KP{1.1}\Omega\,} 
 \roman{ev}\KPt2\sbi\xi\snn\circ\sp y\cdot\varphi\rmdss21\mu$ \inskipline{.5}0

holds for \math{\varphi\in\smb\Phii\in\dom\smb V} and \mathss31{ \xi \in 
 \vecs\vPi}.

Noting that \math{\bouSet((\sp\vPi_{\ssp/\ssp S}\spp)\dlsigss12) \inc
 \bouSet((\sp\vPi_{\ssp/\ssp S}\spp)\dlbetss12\sp) } by Banach\,--\,Steinhaus, 
we see that \vskip.5mm\centerline{$
\Iota\invss44\images\bouSet((\sp\vPi_{\ssp/\ssp S}\spp)\dlsigss12) \inc
\Iota\invss44\images\bouSet((\sp\vPi_{\ssp/\ssp S}\spp)\dlbetss12\sp) \inc
\bouSet(\sp\vPi\dlbetss01{}_{\sp/\ssp\aars S_1}\spp\sbig)0 \inc
\bouSet(\sp\vPi\dlbetss01\sp) \inc \bouSet(\sp\vPi\dlsigss12) \KP1 $,} \inskipline{.5}0

and hence \math{y} is similarly \q{finitely almost bounded} as \math{y\ar 1} 
is. So, in order to establish \linebreak
                            \œ$y \in \bigcup\ssp\vecs\sn
 \mvsLrs23^\plusinftyy(\ssp\mu\,,\spp\vPi\dlsigss00\spp) \ssp$ we only need 
to show that \math{(\KPt5 y\,;\sp\mu\,,\spp\vPi\dlsigss00\spp) } is finitely 
almost scalarly measurable. For this, arbitrarily fixing \mathss31{ \xi \in 
 \vecs\vPi }, \,we first observe that \linebreak
                                      $\roman{ev}\sp\sbi\xi\,|\KP1 S\ar 1 \in 
 \Cal L\,(\sp\vPi\dlbetss01{}_{\sp/\ssp\aars S_1},\spp\bosy K\ssp) \ssp$ and 
hence also \mathss38{
\roman{ev}\sp\sbi\xi\,|\KP1 S\ar 1\snn\circ\spp\Iota\inve \in 
 \Cal L\,(
 (\sp\vPi_{\ssp/\ssp S}\spp)\dlbetss12{}_{\ssp/\ssp\rng\snn\emath\iotaa}\ssp , \spp
   \bosy K\ssp) }. 

By separability of the topology \math{
\taurd((\sp\vPi_{\ssp/\ssp S}\spp)\dlbetss12\sp)\leiss32\rng\Iota } we are 
able to apply Lemma \ref{Le 8.17.8 B} on page \pageref{Le 8.17.8 B} above to 
get existence of \math{ \bosy\xi \in \sp ^\sbbNo\,S } with \math{\rng\bosy\xi
 \in\bouSet\vPi } and \vskip.4mm\centerline{$
  \Iota\inve\fvalss02\zeta\ar 1\sn\fvalue\sp\xi 
= \roman{ev}\sp\sbi\xi\,|\KP1 S\ar 1\snn\circ\spp\Iota\inve\fvalss02\zeta\ar 1
=\lim\,(\ssp\zeta\ar 1\sn\circ\ssp\bosy\xi\ssp) $} \inskipline{.4}0

for all \math{\zeta\ar 1\in\rng\Iota }, \,and hence \math{\zeta\fvalss11\xi = 
 \lim\,(\ssp\Iota\fvalss11\zeta\circss11\bosy\xi\ssp) } for all \mathss34{
\zeta\in S\ar 1}. In particular for \linebreak
                                  \œ$(\ssp\eta\ssp,\spp\zeta\ssp)\in y \ssp $ 
we obtain \mathss38{\roman{ev}\KPt2\sbi\xi\snn\circ\sp y\fvalue\eta = 
 \lim\,(\ssp y\ar 1\sn\fvalue\eta\circss11\bosy\xi\ssp) }, \,giving the 
required measurability. To get (\sp s\sp) we note that by the above we also 
have \inskipline{.7}{8.5}

$ \smb V\fvalss60\smb\Phii\fvalss00\xi 
= \lim\,(\ssp\Iota\circ\spp\smb V\fvalss60\smb\Phii\circ\sp\bosy\xi\ssp) 
= \lim\,(\ssp\smb V\fvalss60\smb\Phii\circ\sp\bosy\xi\ssp)$ \inskipline{.5}{19}

${} 
= \lim\sp\sbi{i\ssp\to\ssp\infty}\sp\int_{\KP{1.1}\Omega\,}
   y\ar 1\sn\fvalue\eta\fvalue(\ssp\bosy\xi\fvalss11 i\ssp)\KP1(\ssp
    \varphi\fvalue\eta\ssp)\rmdss11\mu\,(\sp\eta\sp)
= \int_{\KP{1.1}\Omega\,}
   \roman{ev}\KPt2\sbi\xi\snn\circ\sp y\cdot\varphi\rmdss21\mu \,$, \inskipline{.7}0

where we used (\sp r\sp) with dominated convergence, noting that it is 
legitimate by the above established boundedness and measurability properties.
  \end{proof}

% ¤¤¤¤¤¤¤¤¤¤¤¤¤¤¤¤¤¤¤¤¤¤¤¤¤¤¤¤¤¤¤¤¤¤¤¤¤¤¤¤¤¤¤¤¤¤¤¤¤¤¤¤¤¤¤¤¤¤¤¤¤¤¤¤¤¤¤¤¤¤¤¤¤¤¤¤

\insubsubhead   Dunford\,--\,Pettis property of $\,\mLrs42^1(\ssp\mu\sp)$  \label{Ss Dun-Pet}

When treating the reflexive case in the proof of 
Theorem \nfss A\,\ref{main Th}\sp, we need to know that \mathss03{
\mvLrs42^1(\ssp\mu\,,\spp\bosy K\ssp) } has the Dunford\,--\,Pettis property, 
or in our terminology introduced below, is a {\eightsl DP}{\sl\,--\,space\sp}. 
This is equivalent to \math{\mLrs42^1(\ssp\mu\sp) } being a 
\erm{DP}\,--\,space, and for this reason we here consider this matter to some 
extent. Although we shall need the result only for positive measures \math{\mu} 
with \mathss36{\mu\fvalue\bigcup\ssp\dom\mu < \plusinfty}, \,we anyhow 
consider the situation for general positive measures.

In what follows, note 
that \math{\pi} is said to be a {\it probability measure\ssp} if{}f \math{
\pi} is a positive measure with \mathss34{\pi\fvalue\bigcup\ssp\dom\pi=1}.

\begin{lemma}\label{Le L^p(m)=L^p(p)}

Let $\,1\le p\le\plusinfty${\,\rm, }and let $\,\mu$ be a \rsigma0finite 
positive measure with 

$\,\rng\mu\not=\{\ssp 0\ssp\}\KPt8$. Then 
there is a probability measure $\,\pi$ with 

$\,\mLrs03^p(\ssp\mu\sp)$ and $\,\mLrs03^p(\sp\pi\sp)$ 
linearly homeomorphic.
  \end{lemma}

\begin{proof} Letting \math{\scrmt A\inc\mu\invss44\image\spp\rbb R^+} be a 
finite or countably infinite partition of \mathss36{\Omega}, \,let \math{
\bosy a:\scrmt A\to\rbb R^+} be any function with \math{\sum\,\bosy a=1} and 
take \vskip.5mm\centerline{$
\pi =
\ssp\big\langle\,\sum_{\KPt8 B\ssp\in\ssp\scrm7 A\,}
(\ssp\bosy a\fvalue B\KP1(\ssp\mu\fvalue B\ssp)^{\ssp\mminus 1}\,
(\ssp\mu\fvalue(\sp A\capss31 B\ssp))) : A\in\dom\mu\KP1\rangle \KP1 $.} \inskipline{.5}0

Then \math{\pi} is a probability measure with \math{\dom\pi=\dom\mu} and we 
define 

$\Iota:\vecs\mLrs03^p(\ssp\mu\sp)
\to\vecs\mLrs03^p(\sp\pi\sp)$ 

by $\smb\Phii\mapsto\smb\Psii$ when $x\in\smb\Phii$ and 
$\bigcup\KPt8\{\KPt9\roman b\,B\KP1(\ssp x\KP1|\KP1 B\ssp) 
:B\in\scrmt A\KP1\}\in\smb\Psii$

where  $\roman b\,B=
(\ssp\bosy a\fvalue B\ssp)\ssp^{\mminus p^{-1}}
(\ssp\mu\fvalue B\ssp)\KP1^{p^{-1}}$ 

for $1\le p < \plusinfty$ 
and \math{\roman b\,B=1} for 
\mathss36{p = \plusinfty}.
  \end{proof}

\begin{definitions}\label{df D-P}

(1) \ Say that \math{E} is a {\it\eit{DP\,}--\,space over\ssp} \math{\bosy K} 
    if{}f \math{\bosy K\in\setRC} and \mathss30{E \in \sn} \mathss03{
 \LCSps0(K)} and for all \math{F\in\BaSps0(K) } and \math{ \smb U \in 
  \Cal L\,(\sp E\ssp,\spp F\ssp) } and for \inskipline07

 $\scrmt K=\{\,A:A\ssp\text{ is absolutely \mathss37{\sigrd E}--\,convex and \mathss37{
  \taurd(\sp E\subsigrs04) }--\,compact }\} \hfill $ from \KP3 \inskipline0{11}

 $\smb U\ssp\image\bouSet E\inc\{\,B:B\ssp\text{ is relatively \mathss37{
  \taurd(\sp F\subsigrs00) }--\,compact }\} \hfill $ it follows that \KP3 \inskipline0{14.8}

 $\smb U\ssp\image\sp\scrmt K\inc\{\,B:B\ssp\text{ is relatively \mathss37{
  \taurd F}--\,compact }\} \KP1 $ holds, \inskipline{.5}2

(2) \ Say that \math{E} is a {\it\eit{DP\,}--\,space\ssp} if{}f there is \math{
    \bosy K\in\setRC} \inskipline0{7.7}

 such that \math{E} is a \erm{DP\,}--\,space over \mathss32{\bosy K}.
  \end{definitions}

Instead of saying that \math{E} is a \erm{DP}\,--\,space, we may also say that 
it is a {\it Dunford\,--\,Pettis\ssp} space. For application of the 
Dunford\,--\,Pettis property one should note that by 
\cite[Remarks 9.4.1\,(3)\,, p.\ 634]{Edw} for \math{E} complete in 
Definitions \ref{df D-P}\,(1) we get an equivalent condition if we instead 
take \vskip.2mm\centerline{$
 \scrmt K = \{\,A:A\ssp\text{ is relatively \mathss37{
  \taurd(\sp E\subsigrs04) }--\,compact }\} \KP1 $.} \inskipline{.4}0

In particular, this holds if with \math{\bosy K\in\setRC} and \math{\mu} a 
positive measure we take \mathss08{E=\mvLrs42^1(\ssp\mu\,,\spp\bosy K\ssp) }. 
For the proof in the general case one possibly uses {\sl Krein's theorem\sp} 
\cite[9.8.5\ssp, p.\ 192]{Jr}\,. The Lebesgue case with \math{
\mu\fvalue\snn\bigcup\ssp\dom\mu < \plusinfty } also follows from Lemma \ref{Le L^1_si-compa} 
on page \pageref{Le L^1_si-compa} above, and only this will be needed in the 
  sequel.

\begin{proposition}\label{Pro DP-seq}

Let \œ$\,E\in\BaSps0(K)$ with \œ$\,\bosy K\in\setRC\KP1$. Then $\,E$ is a 
Dunford\,--\,Pettis space if and only if \œ$\,
\roman{ev}\circ\spp[\KP1\bosy x\ssp,\spp\bosy y\rbrakf\to 0$ holds for all $\,
\bosy x\ssp,\sp\bosy y$ with $\,\bosy x\to\Bnull_E$ in top $\,\taurd(\sp 
E\subsigrs04)$ and $\,\bosy y\to\vecs E\times\snn\{\ssp 0\ssp\}$ in top $\,
\taurd((\sp E\dlbetss22\sp)\subsigrs03) \KP1 $.
  \end{proposition}

\begin{proof} See \cite[Proposition 20.7.1\sp, p.\ 473]{Jr} or 
\cite[p.\ 636]{Edw}\,.
  \end{proof}

\begin{lemma}\label{Le E_{/S} DP imp ...}

With \œ$\,\bosy K\in\setRC$ let \œ$\,E\in\BaSps0(K)$ be such that for every $\,
\bosy x$ with \œ$\,\bosy x\to\Bnull_E$ in top $\,\taurd(\sp E\subsigrs04)$ 
there is a closed linear subspace $\,S$ in $\,E$ with \œ$\,\rng\bosy x\inc S$ 
and such that $\,E_{\,/\,S}$ is a \eit{DP}\,--\,space. Then $\,E$ is a 
  \eit{DP}\,--\,space.
  \end{lemma}

\begin{proof} Given \math{\bosy x\ssp,\sp\bosy y } with \math{\bosy x\to
 \Bnull_E } in top \math{\taurd(\sp E\subsigrs04) } and \math{\bosy y\to
 \vecs E\times\snn\{\ssp 0\ssp\} } in top \mathss08{\taurd((\sp 
 E\dlbetss22\sp)\subsigrs03) }, \,by Proposition \ref{Pro DP-seq} it suffices 
to show that \math{\roman{ev}\circ\spp[\KP1\bosy x\ssp,\spp\bosy y\rbrakf\to 0 } 
holds. To get this, putting \math{F=E_{\,/\,S} } and \math{\bosy z=\seqss30{
 \bosy y\fvalss01 i\KP1|\KP1 S:i\in\bbNo} } we now have \mathss30{
\roman{ev}\circ\spp[\KP1\bosy x\ssp,\spp\bosy y\rbrakf = \sn} \mathss03{
\roman{ev}\circ\spp[\KP1\bosy x\ssp,\spp\bosy z\rbrakf } and hence we are done 
{\sl if\sp} we can show (\sp a\sp) that \math{\bosy x\to\Bnull_E } in top \mathss30{
\taurd(\sp F\subsigrs00) } holds, and (\sp b\sp) that \math{\bosy z\to
 S\times\snn\{\ssp 0\ssp\} } in top \math{\taurd((\sp 
 F\dlbetss20\sp)\subsigrs03) } holds. Now (\sp a\sp) follows trivially from 
Hahn\,--\,Banach since given \math{v\in\Cal L\,(\sp F\spp,\spp\bosy K\ssp) } 
there is \math{u} with \mathss30{v\inc u\in\Cal L\,(\sp E\ssp,\spp\bosy K\ssp) } 
and hence \math{ v\circss00\bosy x = u\circss00\bosy x\to 0 } holds. For 
(\sp b\sp) taking the annihilators \inskipline{.2}{16.5}

$N\aar 0 = \Cal L\,(\sp E\ssp,\spp\bosy K\ssp)\capss31\{\,u : 
         u\sp\image\sn S\inc\{\ssp 0\ssp\}\sp\} \KP{13.5} $  and \inskipline0{17}

$S\ar 1 = \Cal L\,(\sp E\dlbetss22\sp,\spp\bosy K\ssp)\capss31\{\,w : 
        w\sp\image\sn N\aar 0\inc\{\ssp 0\ssp\}\sp\} \KP9 $ and putting \inskipline{.2}0

\œ$F\aar 1=E\dlbetss22\ssp/\tvsquotient N\aar 0 \KPt7 $, \,for \math{ w \ar 1 
 \in \Cal L\,(\sp F\dlbetss20,\spp\bosy K\ssp) } from 
\cite[3.13\ssp, pp.\ 261\,--\,263]{Ho} we first get ex- istence of \math{
w\ar 2 \in \Cal L\,(\sp F\aar 1\sp,\spp\bosy K\ssp) } with \math{
w\ar 2\ssn\fvalue\sp\smb U = w\ar 1\ssn\fvalue(\ssp u\KP1|\KP1 S\ssp) } for \mathss34{
u\in\smb U\in\vecs F\aar 1 }. Then we get existence of \math{w\in S\ar 1} with \math{
w\fvalue u=w\ar 2\ssn\fvalue\sp\smb U=w\ar 1\ssn\fvalue(\ssp u\KP1|\KP1 S\ssp) } 
for \math{u\ssp,\sp\smb U} as above. Hence we finally get \math{
w\ar 1\sn\circ\spp\bosy z=w\circss00\bosy y\to 0 } as required.
  \end{proof}

\begin{proposition}\label{Pro L^(p) is DP}

If $\,\pi$ is any probability measure{\sp\rm, }then $\,\mLrs42^1(\sp\pi\sp)$ 
is a \eit{DP}\,--\,space.
  \end{proposition}

\begin{proof} See \cite{Brg} and \cite{Sch}\,, and take \math{T=1\spp\adot} in 
  the latter.
  \end{proof}

\begin{corollary}\label{Cor L^1 is D-S}

For $\,\bosy K\in\setRC$ and for any \inskipline0{21.5}

positive measure $\,\mu$ it holds that $\,
\mvLrs42^1(\ssp\mu\,,\spp\bosy K\ssp)$ is a \eit{DP}\,--\,space.
  \end{corollary}

\begin{proof} From Proposition \ref{Pro L^(p) is DP} we first see that also \math{
\mvLrs42^1(\ssp\pi\sp,\sn\tfbbC\ssp) } is a \erm{DP\,}--\,space for any 
probability measure \mathss32{\pi}. Indeed, for \math{ G = 
 \mLrs42^1(\sp\pi\sp)\ssp\sqcap\sp\mLrs42^1(\sp\pi\sp) } from 
\cite[9.4.3\,(a)\,, p.\ 635]{Edw} we first see that \math{G} is a 
\erm{DP\,}--\,space, and since \math{\mLrs42^1(\sp\pi\sp)\flbb_C} and \math{G} 
are naturally linearly homeomorphic, also \math{\mLrs42^1(\sp\pi\sp)\flbb_C} 
is a \erm{DP\,}--\,space. If now \math{\smb U} is a continuous linear map \mathss30{
\mvLrs42^1(\ssp\pi\sp,\sn\tfbbC\ssp)\to F}, \,it is also a continuous real 
linear map \mathss30{\mLrs42^1(\sp\pi\sp)\flbb_C\to F\Reit0} whence the 
assertion follows by noting that the equality \math{
\taurd(\sp E\RHB{.15}{\subsigma}\spp) = 
 \taurd(\sp E\Reit1\snn\RHB{.15}{\subsigma}\spp) } holds for every \mathss38{
E\in\LCSps5(\tfbbC) }.

By Lemma \ref{Le L^p(m)=L^p(p)} from the above we know that \math{
\mvLrs42^1(\ssp\mu\,,\spp\bosy K\ssp) } is a \erm{DP\,}--\,space for any 
\rsigma5finite positive measure \mathss36{\mu}. Then by Lemma \ref{Le E_{/S} DP imp ...} 
we get the general case as follows. Putting \math{ E = 
 \mvLrs42^1(\ssp\mu\,,\spp\bosy K\ssp) } and letting \math{
\bmii8\Phii\to\Bnull_E} in top \math{\taurd(\sp E\subsigrs04) } we first find 
some countable \math{\scrmt A\inc\mu\invss44\image\lbb R_+} such that \math{
\|\KP1\varphi\KP1|\KP1 A\KP1\|\Lnorss33^1_\mu=0 } holds for all \mathss30{A\in
 \scrmt A} and \math{\varphi\in\bigcup\ssp\rng\sp\bmii8\Phii}. Then taking \math{
\mu\ar 1=\mu\KP1|\KP1\Pows\sn\sbig(1\bigcup\,\scrmt A\ssp) } we have \math{
\mvLrs42^1(\ssp\mu\ar 1\sp,\spp\bosy K\ssp) } a \erm{DP\,}--\,space. Moreover, 
we have an obvious strict morphism \math{ \Iota : 
 \mvLrs42^1(\ssp\mu\ar 1\sp,\spp\bosy K\ssp)\to E } with \mathss03{
\rng\sp\bmii8\Phii\inc\rng\Iota} whence Lemma \ref{Le E_{/S} DP imp ...} gives 
  the conclusion.
  \end{proof}

To fill the gap \q{exists $\ldots\,\eightroman M_{\ssp\roman t}\,\ldots$} in 
\cite[p.\ 3]{Sch} we give the following

\begin{lemma}

With $\,\pi$ a probability measure let \œ$\,E=\mLrs42^1(\sp\pi\sp)$ and also 
let \œ$\,\bmii8\Phii\to\Bnull_E$ in top $\, \taurd(\sp E\subsigrs04)\KP1$. 
Further let $\,\eps\in\rbb R^+$ and $\, \bosy\varphi \in 
 \prodc\ssp\bmii8\Phii\ssp$ with \vskip.3mm\centerline{$
\rng\bosy\varphi\inc\{\,\varphi:(\KPt5\varphi\KPt8;\spp\pi\sp,\ssn\tfbbR\ssp)\ssp\text{ 
 is measurable }\} \KP1 $.} \inskipline{.5}0

Then there is some \œ$\,\smb M\in\rbb R^+$ such that \œ$\,\|\KP1\varphi\KP1|\KP1
 ((\ssp\Abrs03^1\snn\circ\sp\varphi\ssp)\invss24\image\sp[\KP1\smb M\sp,
 \plusinfty\KP1{[\KPt8}\sbig)3\,
  \big\|\LHB{.4}{\Lnorss33^1_{\emath\pii}} < \eps$ holds for all $\,
\varphi \in \rng\bosy\varphi \, $.
  \end{lemma}

\begin{proof} If the assertion is false, by {\sl dependent choice\sp} there is 
a stricly increasing \mathss03{\bosy n:\bbNo\to\bbNo } such that \math{\eps\le
 \|\KP1\varphi\KP1|\KP1((\ssp\Abrs03^1\snn\circ\sp\varphi\ssp)\invss24\image\sp
 [\KP1 i\ssp\yplus\sp\ydot,\plusinfty\KP1{[\KPt8}\sbig)3\,
 \big\|\LHB{.4}{\Lnorss33^1_{\emath\pii}} } holds for \mathss30{
(\ssp i\ssp,\spp\varphi\ssp) \in \sn} \mathss06{\bosy\varphi\circss00\bosy n 
}. Noting that \math{\rng\ssp\bmii8\Phii } is relatively \mathss37{
\taurd(\sp E\subsigrs04)}--\,compact, then from Lemma \ref{Le L^1_si-compa} 
on page \pageref{Le L^1_si-compa} above it follows indirectly that there is \math{
\delta\in\rbb R^+ } with the property that for \math{
(\ssp i\ssp,\spp\varphi\ssp)\in\bosy\varphi\circss00\bosy n } we have \mathss38{
\delta\le\pi\fvalue((\ssp\Abrs03^1\snn\circ\sp\varphi\ssp)\invss24\image\sp[\KP1
 i\ssp\yplus\sp\ydot,\plusinfty\KP1{[\KPt8}\sbig)0 }. This implies that \mathss03{
i\ssp\yplus\,\delta\le\|\,\varphi\,\|\Lnorss33^1_{\emath\pii} } holds, giving 
a {\sl contradiction\sp} with \mathss36{\rng\ssp\bmii8\Phii\in\bouSet E }.
  \end{proof}

% ¤¤¤¤¤¤¤¤¤¤¤¤¤¤¤¤¤¤¤¤¤¤¤¤¤¤¤¤¤¤¤¤¤¤¤¤¤¤¤¤¤¤¤¤¤¤¤¤¤¤¤¤¤¤¤¤¤¤¤¤¤¤¤¤¤¤¤¤¤¤¤¤¤¤¤¤

\insubsubhead       Absolutely continuous vector measures                 \label{Ss abs conti}

We here give some basic definitions for vector measures in order to be able to 
present a decent proof for Proposition \ref{Pro Phi5.4} below that is needed 
as an auxiliary result for the proof of Theorem \nfss A\,\ref{main Th} above.

\begin{definitions}\label{df vec mea}

(1) \ Say that \math{E} is a {\it topologized conoid\,} if{}f there are \math{
    a\ssp,\sp c\,,\sp o\ssp,\sp R\,,\sp S\ssp,\sp\scrmt T}  with \math{
\lbb R_+\inc R\inc\mathbb C } and \math{o\in S} and \math{
(\sp S\ssp,\spp\scrmt T\,) } a Hausdorff topological space and \mathss30{ E = {\ssn}} \mathss03{
(\ssp a\ssp,\spp c\,,\spp\scrmt T\,) } and \math{a} a function \math{S\times S
 \to S} and \math{c} a function \math{R\times S\to S} and such that for all \math{
x\ssp,\sp y\ssp,\sp z\in S } and for all \math{s\ssp,\sp t\in R} it holds that \math{
a\,(\ssp x\ssp,\sp\cdot\,) } and \math{c\KPt8(\ssp t\ssp,\sp\cdot\,) } are 
continuous \math{\scrmt T\to\scrmt T } and in addition \inskipline09

$a\fvalue(\ssp a\fvalue(\ssp x\ssp,\spp y\ssp)\,,\sp z\ssp)=
a\fvalue(\ssp x\ssp,\spp a\fvalue(\ssp y\ssp,\spp z\ssp))\,$ and 
$\,a\fvalue(\ssp x\ssp,\spp y\ssp)=a\fvalue(\ssp y\ssp,\spp x\ssp)\,$ and \inskipline09

$a\fvalue(\ssp x\ssp,\spp o\ssp)=c\fvalue(\ssp 1\sp,\spp x\ssp)=x\,$ and 
$\,c\fvalue(\ssp s\KPt8 t\ssp,\spp x\ssp)=
c\fvalue(\ssp s\ssp,\sp c\fvalue(\ssp t\ssp,\spp x\ssp ))\,$ and \inskipline09

$c\fvalue(\ssp s + t\ssp,\spp x\ssp)=
a\fvalue(\,c\fvalue(\ssp s\ssp,\spp x\ssp)\,,\sp 
           c\fvalue(\ssp t\ssp,\spp x\ssp)) \,$ and \inskipline09

$c\fvalue(\ssp s\ssp,\spp a\fvalue(\ssp x\ssp,\spp y\ssp))=
a\fvalue(\,c\fvalue(\ssp s\ssp,\spp x\ssp)\,,\spp
           c\fvalue(\ssp s\ssp,\spp y\ssp)) \KP1 $, \inskipline{.5}2

(2) \ Say that \math{m} is an \mathss35{E}{\it--\,measure\ssp} if{}f \math{E} 
    is a topologized conoid and \math{(\ssp\emptyset\,,\spp\Bnull_E) \in {\ssn}} \inskipline09

$m\in\sp^{\roman{dom}\KPt8 m}\,\vecs E\ssp$ and for all \math{ A\,,\sp B \in
 \dom m} it holds that \math{\{\,A\cupss31 B\ssp,\sp A\setminus B\,\}\inc {\ssn}} \inskipline0{8.7}

$\dom m \ssp $ and \mathss36{A\capss31 B=\emptyset\impss33 
m\fvalue(\sp A\cupss31 B\ssp) = (\ssp m\fvalue\ssn A + m\fvalue B\ssp)\svs E}, \inskipline{.5}2

(3) \ Say that \math{m} is {\it countably \mathss37{E}--\,additive\ssp} if{}f \math{
    m} is an \mathss37{E}--\,measure and \inskipline09

for all countable disjoint \math{\scrmt A\inc\dom m} with \math{
\bigcup\,\scrmt A\in\dom m} \inskipline0{36.5}

it holds that \mathss38{m\fvalue\sn\bigcup\,\scrmt A = 
 E\vtopsum3\KP1(\ssp m\KP1|\KP1\scrmt A\ssp) }, \inskipline{.5}2

(4) \ Say that \math{m} is {\it absolutely \mathss57{\mu}--\,continous\ssp} in 
    \math{E} if{}f \math{\mu} is a positive measure \inskipline09

and \math{m} is an \mathss37{E}--\,measure and for every \math{ U\sn \in 
 \neiBoo E} there is some \math{\delta\in\rbb R^+} \inskipline0{76}

with \mathss30{\mu\invss44\image\sp[\KPp1.1 0\,,\spp\delta\KP1{[\sp} \inc 
 m\invss44\image U}, \inskipline{.5}2

(5) \ Say that \math{m} has {\it bounded \mathss57{\mu}--$\KP2^p\,
    $variation\ssp} in \math{E} if{}f \math{1\le p < \plusinfty} and \math{\mu} 
is a positive measure and \math{m} is an \mathss35{E}--\,measure with \math{
\mu\invss44\image\spp\rbb R^+\inc\dom m} and for ev- ery \math{ \Nu \in 
 \Bqnorm E} there is \math{\smb M\in\lbb R_+} such that \mathss30{
\sum_{\,A\ssp\in\ssp\scrm7 A\,}( (\ssp\mu\fvalue\ssn A\ssp)
 \KP1^{1\ssp-\,p}\,(\ssp\Nu\circss00 m\spp\fvalue\sn A\ssp)
  \RHB{.3}{\KP1^p}\ssp\sbig)0\le\smb M } holds for all finite disjoint \mathss30{
\scrmt A\inc\mu\invss44\image\spp\rbb R^+}.
  \end{definitions}

In Example \ref{Exa sign mea} on page \pageref{Exa sign mea} we demonstrate 
how also the concepts of positive measure and of {\sl signed measure\sp} in 
the sense of \cite[5.6\ssp, p.\ 137]{Du} can be subsumed in Definitions \ref{df vec mea} 
above. By a {\it real measure\ssp} we mean any countably \mathss37{\tfbbR
}--\,additive \linebreak 
              $m\ssp$ such that \math{\dom m} is a \rsigma3algebra. The 
definition of {\it complex measure\ssp} is obtained by taking here \math{
\tfbbC} in place of \mathss36{\tfbbR}.                           \vskip.3mm

The essential content of \cite[Lemma 5.3\ssp, p.\ 133]{Phil} is reformulated 
in the next

\begin{lemma}\label{Le Phi5.3}

Let \œ$\,E\in\LCSps0(K)$ be normable with \œ$\,\bosy K\in\setRC$ and $\,\Nu$ a 
compatible norm for $\,E${\,\rm, }and let $\,m$ be absolutely $\,\mu\,$--\,%
continuous in $\,E$ with \œ$\,\dom\mu\inc\dom m$ and \œ$\,\mu\fvalss01\Omega < 
 \infty$ for the set \œ$\,\Omega=\bigcup\,\dom\mu\,$. Then for every \œ$\,
\smb M\in\rbb R^+$ there exist some \œ$A\ar 0\in\dom\mu$ and a countable set \œ$\,
\scrmt A\inc\dom\mu$ with \œ$\,\scrmt A\cupss31\{\,A\ar 0\sp\}$ disjoint and 
also with \œ$\Omega = \bigcup\,\scrmt A\cupss31 A\ar 0$ and such that for all $\,
  A\,,\sp B$ {\rm\inskipline{.7}2

(1) \ }$A\in\dom\mu\capss22\Pows A\ar 0 \impss33 
       \Nu\circss01 m\sp\fvalue\sn A\le\smb M\KP1(\ssp\mu\sp\fvalue\sn A\ssp) \KP1 
       ${\rm,\inskipline{.4}2

(2) \ }$A\in\dom\mu$ and $\,A\inc B\in\scrmt A \impss33 
       \Nu\circss01 m\sp\fvalue\sn A\ge\smb M\KP1(\ssp\mu\sp\fvalue\sn A\ssp) \KP1 
       $.
  \end{lemma}

\begin{proof} We first note that in the case \math{\bosy K=\tfbbC} we have \math{
\Nu} also a compatible norm for the realification \math{E\Reit4} of \math{E} 
and hence we may without loss of generality assume that \math{\bosy K=\tfbbR} 
holds. Now we let \math{\roman P\,B\,u} mean that \math{B\in\dom\mu} and \œ$\ssp
u\in\Cal L\,(\sp E\ssp,\snn\tfbbR\ssp) $ \linebreak
                                         with \math{
 \sup\ssp\big\{\KPt8|\KP1 u\fvalue x\KP1|:x\in\Nu\invss44\image\ssbb24 I\,\} 
 \le 1 } and that for all \math{A\in\dom\mu\capss22\Pows B } we have \linebreak
                                                                   \œ$
u\circss01 m\sp\fvalue\sn A\ge\smb M\KP1(\ssp\mu\sp\fvalue\sn A\ssp) \ssp $ 
and that \math{ u\circss01 m\sp\fvalue\sn A \le 
 \smb M\KP1(\ssp\mu\sp\fvalue\sn A\ssp) } holds for all \math{A\in\dom\mu} 
with \linebreak
   \œ$A\capss31 B=\emptyset \,$. Also let \mathss38{ \scrmt A\sp\ar 0 = 
 \mu\invss44\image\spp\rbb R^+\sn\cap\sp\{\,A:\eexi{B\ssp,\sp u}\,A\inc B\ssp
  \text{ and }\ssp\roman P\,B\,u\KPt8\} }. By con- sidering the set \math{
\scrmt P} of disjoint subsets \math{\scrmt A} of \math{\scrmt A\sp\ar 0} 
partially ordered by inclusion, from \linebreak
                                {\sl Zorn's lemma\sp} we get existence of some 
maximal \math{\scrmt A} of \mathss30{\scrmt P}. Then by \math{
\mu\fvalss02\Omega < \plusinfty } we \linebreak
                                     see that \math{\scrmt A} is countable, 
and we take \mathss36{ A\ar 0 = \Omega\spp\setminus\bigcup\,\scrmt A }.

Now, for the proof (1) letting \math{A\in\dom\mu\capss22\Pows A\ar 0 } we 
first note (\sp$*$\sp) that by \linebreak
                               maximality of \math{\scrmt A} there cannot 
exist \math{B\ssp,\sp u} with \math{\roman P\,B\,u} and \mathss31{A\capss31 B
 \in \mu\invss44\image\spp\rbb R^+ }. Fur- \linebreak
                                           thermore, by Hahn\,--\,Banach it 
suffices for arbitrarily fixed \math{u\in\Cal L\,(\sp E\ssp,\snn\tfbbR\ssp) } 
with norm \math{ \sup \ssp \big\{ \KPt8 |\KP1 u\fvalue x\KP1| : x \in 
 \Nu\invss44\image\ssbb24 I\,\} \le 1 } to verify that \math{
u\circss01 m\fvalue\sn A \le \smb M\KP1(\ssp\mu\fvalue\sn A\ssp) } holds. \linebreak
Since \math{u\circss01 m\KP1|\KP1\dom\mu } is absolutely \mathss37{\mu
}--\,continuous, there is a Radon\,--\,Nikodym deri- vative of it, and by 
considering one such we see existence of \math{B} with \mathss36{
\roman P\,B\,u}. Then by (\sp$*$\sp) we have \math{ A\capss31 B \in 
 \mu\invss44\image\snn\{\ssp 0\ssp\} } whence with \math{ A\ar 1 = 
 A\spp\setminus B } we finally get \vskip.3mm\centerline{$
u\circss01 m\fvalue\sn A = u\circss01 m\fvalue\sn A\ar 1 \le 
\smb M\KP1(\ssp\mu\fvalue\sn A\ar 1) = 
\smb M\KP1(\ssp\mu\fvalue\sn A\ssp) \KP1 $.} \vskip.3mm

For the proof of (2) letting \math{A\in\dom\mu} with \mathss36{ A \inc B \in 
 \scrmt A\inc\scrmt A\sp\ar 0 }, \,there are some \linebreak
                                             \œ$B\ar 1\ssp$ and \math{u} with \math{
\roman P\,B\ar 1\ssp u} and \mathss34{B\inc B\ar 1}. Then we also have \math{
A\inc B\ar 1} and consequently \linebreak
                               $u\circ m\fvalue\sn A \ge 
 \smb M\KP1(\ssp\mu\fvalue\sn A\ssp) \ssp $ whence further \math{
\Nu\circss01 m\sp\fvalue\sn A\ge\smb M\KP1(\ssp\mu\sp\fvalue\sn A\ssp) } 
  trivially follows.
  \end{proof}

The essential content of \cite[Lemma 5.4\ssp, p.\ 133]{Phil} is reformulated 
in the next

\begin{proposition}\label{Pro Phi5.4}

Let \œ$\,E\in\LCSps0(K)$ be normable with \œ$\,\bosy K\in\setRC$ and $\,\Nu$ a 
compatible norm for $\,E${\,\rm, }and let $\,m$ be absolutely $\,\mu\,$--\,%
continuous in $\,E$ with \œ$\,\mu\fvalss01\Omega < \infty$ for \œ$\, \Omega = 
 \bigcup\,\dom\mu\,$. Also let \œ$\,1\le p < \plusinfty$ and let $\,m$ have 
bounded $\,\mu\,$--$\RHB{.3}{\KP{1.5} ^p \ssp}$variation in $\,E\,$. Then 
there is a decreasing \œ$\,\bmii8 A\in\sp^\sbbNo\,\dom\mu$ with \œ$\,
\lim\,(\ssp\mu\spp\circ\bmii8 A\ssp)=0$ and such that $\,
\Nu\circss01 m\fvalue\sn A\le i\ssp\yplus\,(\ssp\mu\fvalue\sn A\ssp)$ holds 
for $\,i\in\bbNo$ and $\,A\in\dom\mu$ with $\,
A\capss32(\ssp\bmii8 A\fvalss01 i\ssp)=\emptyset\,$.
  \end{proposition}

\begin{proof} We first note that the requirement \math{\dom\mu\inc\dom m} in 
Lemma \ref{Le Phi5.3} holds since from (4) and (5) in Definitions \ref{df vec mea} 
we get \math{\mu\invss44\image\snn\{\ssp 0\ssp\}\inc\dom m } and \mathss30{
\mu\invss44\image\spp\rbb R^+} \mathss04{\inc\dom m }. Now, for each fixed \math{
i\in\bbNo} taking \math{ \smb M = i\ssp\yplus\sp\ydot } in Lemma \ref{Le Phi5.3} 
above, by {\sl countable choice\sp} we get existence of \math{ \scrb8 A \in \sp
 ^\sbbNo\,\Pows\dom\mu } with the property that for every \math{i\in\bbNo} we 
have \math{\scrb8 A\fvalss11 i} countable and disjoint and such that for all \math{
A\,,\sp B} and for \math{ A\sp\ar 0 = 
 \Omega\spp\setminus\bigcup\KP1(\sp\scrb8 A\fvalss11 i\ssp) } we have \inskipline{.7}3

(a) \ $A\in\dom\mu\capss22\Pows A\ar 0 \impss33 
    \Nu\circss01 m\sp\fvalue\sn A \le
    i\ssp\yplus\sp\ydot\,(\ssp\mu\sp\fvalue\sn A\ssp) \KP1 $, \inskipline{.4}3

(b) \ $A\in\dom\mu$ and $\,A\inc B\in\scrb8 A\fvalss11 i \impss33 
    \Nu\circss01 m\sp\fvalue\sn A \ge
    i\ssp\yplus\sp\ydot\,(\ssp\mu\sp\fvalue\sn A\ssp) \KP1 $. \inskipline{.7}0

Then we take \ $\bmii8 B \ssp = \ssp 
 \big\langle\KPt8\bigcup\KP1(\sp\scrb8 A\fvalss11 i\ssp) : i\in\bbNo\,\rangle \,$ 
and \inskipline{.5}{23}

$ \bmii8 A \ssp = \ssp 
 \big\langle\KPt8\bigcup\KP1(\ssp\bmii8 B\KP1|\KP1(\ssp
 \bbNo\sn\setminus i\ssp)) : i\in\bbNo\,\rangle \KP1 $. \inskipline{.4}0

It is now clear that \math{\bmii8 A} is decreasing with \mathss36{\bmii8 A \in
 \sp^\sbbNo\,\dom\mu}, \,and for the proof of the remaining required 
properties we proceed as follows.

By the bounded variation property there is \math{\smb M\aR 1\in\rbb R^+ } such 
that for all \math{i\in\bbNo} and for all finite \math{\scrmt A \inc
 \mu\invss44\image\spp\rbb R^+\sn\cap\sp(\sp\scrb8 A\fvalss11 i\ssp) } in view 
of (b) above we have \vskip.5mm\centerline{$
    i\ssp\yplus\sp\ydot\RHB{.3}{\,^p}\,(\ssp\mu\fvalue\bigcup\,\scrmt A\ssp)
\le \sum_{\,A\ssp\in\ssp\scrm7 A\,}((\ssp\mu\fvalue\sn A\ssp
    )\KP1^{1\ssp-\ssp p}\,(\ssp\Nu\circss01 m\fvalue\sn A\ssp
    )\RHB{.3}{\KP1^p}\sp\big)
\le \smb M\aR 1 \, $,} \inskipline{.5}0

and hence \mathss36{\mu\circss11\bmii8 B\fvalss21 i \le 
 \smb M\aR 1\,i\ssp\yplus\sp\ydot\RHB{.3}{\,^{\mminus p}} }, \,whence further \mathss36{
\lim\,(\ssp\mu\spp\circ\bmii8 B\ssp) = 0 }. Next considering \mathss03{B\ar 1=
 \bmii8 B\fvalss21 i\ssp\yplus\snn\setminus(\ssp\bmii8 B\fvalss21 i\ssp) } for 
all \math{A\in\scrb8 A\fvalss11 i\ssp\yplus } by both (a) and (b) above we 
  have \vskip.5mm\centerline{$
   (\ssp i\ssp\yplus\sp\ydot\snn + 1\ssp)\KP1(\ssp\mu\fvalue(\sp 
    A\capss31 B\ar 1))
\le \Nu\circss01 m\fvalue(\sp A\capss31 B\ar 1)
\le i\ssp\yplus\,(\ssp\mu\fvalue(\sp A\capss31 B\ar 1)) $} \inskipline{.5}0 

and hence \math{\mu\fvalue(\sp A\capss31 B\ar 1) = 0 } whence further \mathss36{
\mu\fvalue B\ar 1 = 0 }. Now for every \math{i\in\bbNo} we have \math{
     \mu\spp\circ\bmii8 A\fvalue\sp i 
 \le \mu\spp\circ\bmii8 B\fvalss21 i + 
      \sum_{\,j\ssp\in\KPt5\sbbNo\sp\setminus\ssp i\,}(\,\mu\fvalue(\sp
       \bmii8 B\fvalss20 j\ssp\yplus\snn\setminus(\ssp\bmii8 B\fvalss20 j\ssp)
        )) 
 = \mu\spp\circ\bmii8 B\fvalss21 i} and hence we obtain \mathss36{
\lim\,(\ssp\mu\spp\circ\bmii8 A\ssp) = 0 }. For the remaining property letting \math{
i\in\bbNo} and \math{A\in\dom\mu} with \mathss38{ \emptyset = 
 A\capss32(\ssp\bmii8 A\fvalss01 i\ssp) = 
 A\capss34\bigcup\KP1(\ssp\bmii8 B\KP1|\KP1(\ssp\bbNo\sn\setminus i\ssp)) 
}, \,we hence also have \mathss30{ \emptyset = 
 A\capss32(\ssp\bmii8 B\fvalss21 i\ssp) } \mathss03{{\KN{.7}} =
 A\capss34\bigcup\KP1(\sp\scrb8 A\fvalss11 i\ssp) } and consequently by (a) we 
obtain \mathss38{\Nu\circss01 m\fvalue\sn A \le 
 i\ssp\yplus\,(\ssp\mu\fvalue\sn A\ssp) }.
  \end{proof}

\begin{proposition}\label{Pro mA=int ev_x c mu}

Let $\,\mu$ be a positive measure with \œ$\,\mu\fvalue\snn\bigcup\ssp\dom\mu 
 < \plusinfty${\,\rm, }and with \œ$\,\bosy K\in\setRC$ let \œ$\, \vPi \in 
 \BaSps0(K)$ be such that either $\,\vPi$ is reflexive or $\,\taurd\vPi$ is a 
separable topology. Also let \œ$\,
\Nu\aR 1=\seqss44{
\sup\KPt8(\ssp\Abrs00^1\circ\sp u\circss01\Nu\invss44\image\ssbb15 I)
:u\in\Cal L\,(\sp\vPi\sp,\spp\bosy K\ssp)}$ where $\,\Nu$ is a compatible norm 
for $\,\vPi${\sp\rm, }and let $\,m$ be a $\,
\vPi\dlbetss01\,$--\,measure with \œ$\,\dom m=\dom\mu$ and such that 
\œ$\,\Nu\aR 1\snn\circ\sp m\fvalue\ssn A \le \mu\fvalue\ssn A$ holds 
for all \œ$\,A\in\dom\mu\,$. Then there are $\,y\ssp,\sp S$ such that 
$\,S$ is a separable closed linear subspace in $\,\vPi\dlbetss01$ and 
\œ$\,
(\KPt5 y\,;\spp\mu\,,\spp\vPi\dlsigss00\spp)$ is 
simply measurable and Pettis with 
\œ$\, m\fvalue\ssn A\fvalss12\xi = 
 \int_{\,A}\ssp\roman{ev}\KPt2\sbi\xi\snn\circ\sp y\rmdss11\mu$ for all 
\œ$\, A 
 \in\dom m$ and \œ$\,\xi\in\vecs\vPi${\sp\rm, }and in addition such that also 
\œ$\,\rng y\inc S$ holds and 
\œ$\,(\KPt5 y\,;\spp\mu\,,\spp\vPi\dlbetss01\sp)$ is simply 
measurable in the case where $\,\vPi$ is reflexive.
  \end{proposition}

\begin{proof} Let \math{\Omega=\bigcup\ssp\dom\mu} and \mathss38{ E = 
 \mvLrs42^1(\ssp\mu\,,\spp\bosy K\ssp) }. Putting 
\newline
\math{\roman x\,A=
(\ssp\Omega\spp\setminus A\ssp)\times\snn\{\ssp 0\ssp\}\cupss22
(\sp A\times\snn\{\ssp 1\ssp\}\sp\sbig)0} and 
\math{\eightroman X\,A=\uniqset\smb\Phii:
\roman x\,A\in\smb\Phii\in\vecs E} let \math{D} 
be the linear \mathss37{\sigrd E}--\,span of 
\mathss38{\{\KPt8\eightroman X\,A:
A\in\mu\invss44\image\spp\rbb R^+\sp\big\} }. Thus 
\math{D} is the set of all \math{\smb\Phii\in\vecs E} such that 
there is \math{\varphi\in\smb\Phii} with 
\math{\rng\varphi} finite. We know that 
\math{D} is \mathss37{\taurd E}--\,dense. Then 
we let \math{\smb V\aR 0} be the unique linear extension 
\math{\sigrd E_{\KPt8|\,D}\to\sigrd(\sp\vPi\dlbetss01\sp) } of 
\newline
\mathss38{\{\,(\ssp\eightroman X\,A
\,,\spp m\fvalue\ssn A\ssp):
A\in\mu\invss44\image\spp\rbb R^+\sp\big\} }, \,noting 
that by the assumptions on \math{m} we indeed get a linear map.

Now for finite functions \math{\bosy s\inc
\dom\mu\times\vecs\bosy K} with \math{\dom\bosy s} disjoint and 
\newline
\math{\varphi=
\sigrd\bosy K\expnota^\ssp\Omega\sp]_{vs}\text{\KPt8-}
\sum_{\,A\ssp\in\ssp\dom\sn\bmii8 s\,}
(\ssp\bosy s\fvalue\ssn A\ssp)\KP1\roman x\,A
\in\smb\Phii\in\vecs E } we obtain \vskip.5mm\centerline{$
|\KP1\smb V\aR 0\sn\fvalue\smb\Phii\fvalss01\xi\KP1|=
\big|\,
\sum_{\,A\ssp\in\ssp\dom\sn\bmii8 s\,}
(\ssp\bosy s\fvalue\ssn A\ssp)\KP1
(\ssp m\fvalue\ssn A\fvalss01\xi\ssp)\KP1|\le
\sum_{\,A\ssp\in\ssp\dom\sn\bmii8 s\,}|\KP1
\bosy s\fvalue\ssn A\KP1|\KP1|\KPp1.2
m\fvalue\ssn A\fvalss01\xi\KP1| $} \inskipline1{21.6}

${}\le
\sum_{\,A\ssp\in\ssp\dom\sn\bmii8 s\,}|\KP1
\bosy s\fvalue\ssn A\KP1|\KP1
(\ssp\mu\fvalue\ssn A\ssp)\KP1(\ssp\Nu\fvalss01\xi\ssp)=
(\ssp\Nu\fvalss01\xi\ssp)\KP1\|\,\varphi\,\|\Lnorss33^1_\mu \, $. \inskipline{.7}0

Consequently \math{\smb V\aR 0} has a unique continuous extension \mathss38{
\smb V\in\Cal L\,(\sp E\ssp,\spp\vPi\dlbetss01\sp) }. 

Now assuming that \math{\vPi} is reflexive and  taking 
\mathss38{K=\{\KPt8\eightroman X\,A:A\in\dom\mu\KPt8\} }, \,from 
Lemma \ref{Le L^1_si-compa} on page \pageref{Le L^1_si-compa} above 
we see that 
\math{K} 
is relatively \mathss37{\taurd(\sp E\subsigrs04)}--\,compact. Since by 
Corollary \ref{Cor L^1 is D-S} on page \pageref{Cor L^1 is D-S} above 
\math{E} is a \erm{DP\,}--\,space, noting that by reflexivity 
of \math{\vPi} all bounded sets in \math{\vPi\dlbetss01} are relatively 
\mathss37{\taurd((\sp\vPi\dlbetss01\sp)\subsigma)}--\,compact, we 
see that \math{\smb V\KPt8\image\snn K} is 
relatively \mathss37{\taurd(\sp\vPi\dlbetss01\sp)}--\,compact. 
Since the linear \mathss37{\sigrd E}--\,span of 
\math{K} is \mathss37{\taurd E}--\,dense, it follows that 
\math{\taurd(\spp\vPi\dlbetss01\sp)\leiss32\rng\smb V } is a 
separable topology. Taking \math{S=
\CltaurdvPidualbeta\rng\smb V} hence by 
Proposition \ref{Pro Edw 8.17.8} on page \pageref{Pro Edw 8.17.8} above 
there is $y\ar 1 \in \bigcup\ssp\vecs\sn
 \mvsLrs23^\plusinftyy(\ssp\mu\,,\spp\vPi\dlsigss00\spp)$ with \math{
\rng y\ar 1\inc S} and (\sp$*$\sp) that \math{
\smb V\fvalss60\smb\Phii\fvalss00\xi=
\int_{\KP{1.1}\Omega\,} 
 \roman{ev}\KPt2\sbi\xi\snn\circ\sp y\ar 1\ssn\cdot\spp\varphi\rmdss21\mu 
} holds for \math{
\varphi\in\smb\Phii\in\vecs E} and \mathss31{\xi\in\vecs\vPi }. By Pettis' 
theorem and reflexivity of \math{\vPi} in fact \math{
y\ar 1 \in \bigcup\ssp\vecs\sn
 \mvLrs23^\plusinftyy(\ssp\mu\,,\spp\vPi\dlbetss01\sp)
} holds. Hence 
there is some \math{N\in\mu\invss44\image\snn\{\ssp 0\ssp\} } such that 
for \math{B=\Omega\spp\setminus N} we have \math{
(\ssp y\ar 1\ssp|\KP1 B\,;\spp\mu\KP1|\KP1\Pows B\ssp,\spp\vPi\dlbetss01\sp) } 
simply measurable, and so taking 
\newline
\math{y=N\timesn\{\KPt8\vecs\vPi\timesn\{\ssp 0\ssp\}\sp\}\cupss22
 (\ssp y\ar 1\ssp|\KP1 B\ssp) } we get 
\math{(\KPt5 y\,;\spp\mu\,,\spp\vPi\dlbetss01\sp) } simply measurable. To 
conclude the proof in the reflexive case, it suffices to take 
\math{\varphi=\roman x\,A} in (\sp$*$\sp) above.

In the separable case we instead apply Proposition \ref{Pro Edw 8.17.6} on 
page \pageref{Pro Edw 8.17.6} above to get existence of \math{y\ar 1} with 
(\sp$*$\sp) above. To see that \math{y\ar 1} can be modified on a set of 
measure zero to get some \math{y} with \math{
(\KPt5 y\,;\spp\mu\,,\spp\vPi\dlsigss00\spp) } simply measurable, we proceed as 
follows. We take a countable \math{D\ar 1} such that \math{D\ar 1} is \mathss37{
\taurd\vPi}--\,dense. For every fixed \linebreak
                                      \mathss03{\xi\in D\ar 1} we now know 
that \math{
(\,\roman{ev}\KPt2\sbi\xi\snn\circ\sp y\ar 1\ssp;\spp\mu\,,\spp\bosy K\ssp) } 
is almost measurable, and hence there is some \math{ N\aar 1 \in 
 \mu\invss44\image\snn\{\ssp 0\ssp\} } such that \math{
(\,\roman{ev}\KPt2\sbi\xi\snn\circ\sp y\ar 1\,|\KP1 B\,;\spp
 \mu\KP1|\KP1\Pows B\ssp,\spp\bosy K\ssp) } for \math{ B = 
 \Omega\spp\setminus N\aar 1 } is measurable. Since \math{D\ar 1} is 
countable, by {\sl countable choice\sp} we then find \mathss30{ N \in 
 \mu\invss44\image\snn\{\ssp 0\ssp\} } \linebreak
                                       such that with \math{ B = 
 \Omega\spp\setminus N } we have \math{
(\,\roman{ev}\KPt2\sbi\xi\snn\circ\sp y\ar 1\,|\KP1 B\,;\spp
 \mu\KP1|\KP1\Pows B\ssp,\spp\bosy K\ssp) } measurable for all \linebreak \mathss04{
\xi\in D\ar 1 }. By density, we can extend this to hold for all \mathss31{
\xi\in\vecs\vPi}. By Proposition \ref{pro-mea-equ} on page \pageref{pro-mea-equ} 
above, this gives that \math{
(\ssp y\ar 1\,|\KP1 B\,;\spp\mu\KP1|\KP1\Pows B\ssp,\spp\vPi\dlsigss00\spp) } 
is simply measurable, and so it suffices to take \math{y} as in the reflexive 
  case above.
  \end{proof}

From the logical point of view, note that in the nonreflexive case in 
Proposition \ref{Pro mA=int ev_x c mu} we may trivially take for example \mathss38{
S=\{\KPt8\vecs\vPi\timesn\{\ssp 0\ssp\}\sp\} }. We give below an alternative 
proof for the existence of \math{y} above. It has the drawback of not giving 
existence of the separable \math{S} that allowed us to deduce the stronger 
measurability in the reflexive case. The underlying argument of applying 
Alaoglu's theorem is already shortly sketched in \cite[p.\ 131]{Phil}\,, and 
in a more explicit manner it is also utilized in 
\cite[pp.\ 594\,--\,595]{Edw}\,. This alternative in fact was our first 
approach but then we noticed that using Propositions \ref{Pro Edw 8.17.8} and 
\ref{Pro Edw 8.17.6} offers a more uniform way to treating the cases (5) and 
(6) in Theorem \nfss A\,\ref{main Th} together.

\vskip1mm

{\it$\null\hfill$
Let $\,\mu$ be a positive measure on 
$\,\Omega$ with $\,\mu\fvalss02\Omega < \plusinfty${\,\rm, }and 
with 
\linebreak
$\,\bosy K\in\setRC$ let \œ$\,E\in\BaSps0(K)$ 
with $\,\Nu\aR 1$ a compatible dual norm for 
$\,F=E\dlbetss12\,$. 

Also let $\,m$ be an $\,F\,$--\,measure with $\,\dom m=\dom\mu$ and 

such that $\,\Nu\aR 1\snn\circ\sp m\fvalue\ssn A \le \mu\fvalue\ssn A$ 
holds for all $\,A \in\dom\mu\,$.

Then there is $\, c \in \sp ^\Omega\,\Cal L\,(\sp E\ssp,\spp\bosy K\ssp)$ such 
that $\,(\,c\KPt8;\sp\mu\,,\spp E\dlsigss12\spp)$ is Pettis and such that 

$\, m\fvalue\ssn A\fvalue x = 
 \int_{\,A}\ssp\roman{ev}\sp\sbi{\emath x}\snn
\circ\sp c\rmdss11\mu$ holds for all \œ$\, A 
 \in\dom\mu$ and $\,x\in\vecs E\,$.
} % ENDs italic!!!

\begin{proof} The assertion being trivial if \math{\mu\fvalss02\Omega=0} 
holds, assuming \math{\mu\fvalss02\Omega > 0} we consider the net \math{
(\sp\varDelta\,,\spp\bosy c\ssp) } obtained as follows. 
Let $\varDelta$ be the set of all pairs 
$(\sp\scrmt A\,,\spp\scrmt B\ssp)$ where 
\math{\scrmt A\,,\sp\scrmt B\inc\dom m\setminus\sp
\mu\invss44\image\snn\{\ssp 0\ssp\} } are finite partitions of 
$\Omega$ such that for every $B\in\scrmt B$ there is some $A\in\scrmt A$ 
with $B\inc A \,$. Then $\varDelta$ is a direction, and we take \vskip.3mm\centerline{$
\bosy c=\langle\KP1 
\Omega\times\Univ\capss31\{\,(\ssp\eta\ssp,\spp u\ssp):
\all A\,\eta\in A\in\scrmt A\impss33
u=(\ssp\mu\fvalue\ssn A\ssp)^{\,\mminus 1}\,
(\ssp m\fvalue\ssn A\ssp)\KPt9\}:\scrmt A\in\dom\varDelta
\KP{1.3}\rangle $} \inskipline{.3}0

thus obtaining a function $\dom\varDelta\to
\sp^\Omega\,\Cal L\,(\sp E\ssp,\sp\bosy K\ssp)$ 
such that for every $\scrmt A\in\dom\varDelta$ we have 
$\bosy c\sp\fvalue\sn\scrmt A$ the function $\Omega\owns\eta\mapsto
(\ssp\mu\fvalue\ssn A\ssp)^{\,\mminus 1}\,(\ssp m\fvalue\ssn A\ssp) $ 
when $\eta\in A\in\scrmt A$ holds.

We further let \math{\varLambda} be the set of all pairs \math{
(\ssp\eta\ssp,\spp u\ssp)\in\Omega\times\Cal L\,(\sp E\ssp,\spp\bosy K\ssp) } 
such that \math{u} is a \mathss37{\taurd(\sp E\dlsigss12\spp) }--\,limit point 
of the net \mathss38{ (\sp \varDelta \, , \spp
 \roman{ev}\sbi\eta\snn\circ\spp\bosy c\ssp) }. Then by {\sl Alaoglu's theorem\sp} 
we have \mathss36{\Omega\inc\dom\varLambda }, \,and hence by the {\sl axiom of 
choice\sp} there is a function \math{c\inc\varLambda} with \mathss06{ \dom c = 
 \Omega}. Arbitrarily fixing \mathss34{x\in\vecs E}, \,it remains to show that 
\math{\roman{ev}\sp\sbi{\emath x}\snn\circ\sp c } is inte- grable over every \mathss36{
A\in\dom\mu}, \,and that \math{m\fvalue\ssn A\fvalue x = 
 \int_{\,A}\ssp\roman{ev}\sp\sbi{\emath x}\snn\circ\sp c\rmdss11\mu } holds.

To see this, we let \math{\varphi} be a Radon\,--\,Nikodym derivative 
with respect to $\mu$ of 
$\dom\mu\owns A\mapsto 
m\fvalue\ssn A\fvalss10 x \,$, 
noting that some such exist since by our assumption 
for some $\smb M\in\rbb R^+$ we have  
$|\KP1 m\fvalue\ssn A\fvalss10 x\KP1|
\le\smb M\KP1(\ssp\mu\fvalue\sn A\ssp)$ for all 
$A\in\dom\mu\,$. For the same reason we may assume that 
$|\KP1\varphi\fvalue\eta\KP1|\le\smb M$ holds for all $
\eta\in \Omega\,$. We now have 
$m\fvalue\ssn A\fvalss10 x=
\int_{\,A}\ssp\varphi\rmdss11\mu$ 
for $A\in\dom\mu\, $, and it suffices to show 
existence of some $N\in\mu\invss44\image\snn\{\ssp 0\ssp\}$ such that 
$\roman{ev}\sp\sbi{\emath x}\snn\circ\sp c\fvalue\eta=
\varphi\fvalue\eta$ holds for all 
$\eta\in \Omega\spp\setminus N\ssp$. 

By taking inverse images under $\varphi$ of partitions of 
$\mathbb C\capss31\{\,z:|\,z\,|\le\smb M\KPt8\}$ into sets of diameter 
${}<i\ssp\yplus\sp\ydot\,^{\mminus 1}$, we obtain a sequence 
$\bosy s\ar 1$ of simple functions such that 
$|\KP{1.1}\bosy s\ar 1\sn\fvalue\sp i\fvalss10\eta - \varphi\fvalue\eta\KP{1.2}|
< i\ssp\yplus\sp\ydot\,^{\mminus 1}$ holds for all 
$i\in\bbNo$ and $\eta\in \Omega\,$.

If $\sigma\ar 1\in\rng\bosy s\ar 1$ is such that $
\sigma\ar 1\sn\inve\ssp\image\snn\{\ssp s\ssp\}
\in\mu\invss44\image\snn\{\ssp 0\ssp\}$ holds for 
some $s\in\rng\sigma\ar 1\,$, on a set of measure zero we can modify 
$\sigma\ar 1$ to get another simple function $\sigma$ 
such that for every 
$s\in\rng\sigma$ we have $\sigma\invss44\image\snn\{\ssp s\ssp\}
\in\mu\invss44\image\spp\rbb R^+\,$. 
Using this observation in conjunction with 
{\sl countable choice\sp} we obtain another sequence $\bosy s$ of 
simple functions and some $N\in\mu\invss44\image\snn\{\ssp 0\ssp\}
$ such that for all $\eta\in \Omega\spp\setminus N$ and $i\in\bbNo$ we have 
$\bosy s\fvalss01 i\fvalss10\eta=
 \bosy s\ar 1\sn\fvalue\sp i\fvalss10\eta\,$.

Now arbitrarily given \math{\eta\in \Omega\spp\setminus N} and \math{\eps\in
 \rbb R^+} we pick some $\sigma\in\rng\bosy s$ such that for all $\eta\ar 1\in 
 \Omega\spp\setminus N$ we have 
$|\KP1\sigma\fvalue\eta\ar 1\snn - \varphi\fvalue\eta\ar 1\,| < \eps \,$. Then 
with $A=\sigma\invss44\image\snn\{\,\sigma\fvalue\eta\,\}$ we take either 
$\scrmt A=\{\,A\,,\spp \Omega\spp\setminus A\KPt9\}$ 
or $\scrmt A=\{\ssp A\,\}$ 
according to whether 
$A\not=\Omega$ or $A=\Omega$ holds, getting then 
$\scrmt A\in\dom\varDelta$ by construction. 
If now $\eta\in B\in\scrmt B\in\varDelta\spp\image\snn\{\,\scrmt A\,\}$ 
holds, we have $B\inc A$ and hence \vskip.5mm\centerline{$
\roman{ev}\sp\sbi{\emath x}\snn\circ\sp 
\roman{ev}\sbi\eta\snn\circ\spp\bosy c\fvalss02\scrmt B
=\bosy c\fvalss02\scrmt B\fvalue\eta\fvalss02\xi=
(\ssp\mu\fvalue B\ssp)^{\,\mminus 1}\,
(\ssp m\fvalue B\fvalss11\xi\ssp)=
(\ssp\mu\fvalue B\ssp)^{\,\mminus 1}\int_{\KP1 B}\ssp\varphi\rmdss11\mu $} \inskipline{.5}0

further giving 
$|\KP{1.2}
\roman{ev}\sp\sbi{\emath x}\snn\circ\sp 
\roman{ev}\sbi\eta\snn\circ\spp\bosy c\fvalss02\scrmt B
 - \varphi\fvalue\eta\KP{1.2}| < 2\KP1\eps \,$. 
Since $c\fvalue\eta\fvalue x$ is a $\taurd\bosy K\,$--\,limit point of 
the net 
$(\sp\varDelta\,,\spp
\roman{ev}\sp\sbi{\emath x}\snn\circ\sp
\roman{ev}\sbi\eta\snn\circ\spp\bosy c\ssp) \KPt8 $, this gives 
$c\fvalue\eta\fvalue x=\varphi\fvalue\eta\,$, and having 
here $\eta\in \Omega\spp\setminus N$ arbitrarily fixed, we see that 
$\roman{ev}\sp\sbi{\emath x}\snn\circ\sp c\fvalue\eta=
\varphi\fvalue\eta$ holds for all 
$\eta\in \Omega\spp\setminus N\ssp$.
  \end{proof}

\begin{corollary}\label{Coro q-var}

Let \œ$\,1\le q<\plusinfty$ and let $\,\mu$ be a positive measure on $\,\Omega
${\,\rm, }and with \œ$\,\bosy K\in\setRC$ let \œ$\,\vPi\in\BaSps0(K)$ be such 
that either \œ$\,\vPi$ is reflexive or \œ$\,\taurd\vPi$ is a separable 
topology. Also let $\,m$ be a $\,\vPi\dlbetss01\,$--\,measure with \œ$\,\dom m
 = \mu\invss44\image\spp\lbb R_+$ and such that $\,m$ is absolutely $\,\mu\,
$--\,continuous in $\,\vPi\dlbetss01$ with $\,m$ having bounded $\,\mu\,$--$
 \RHB{.3}{\KP{1.5} ^q \ssp}$variation in $\,\vPi\dlbetss01\,$. Then there are 
some countable disjoint $\,\scrmt A$ and $\,y$ with \œ$\, \scrmt A \inc 
 \mu\invss44\image\spp\rbb R^+$ and \œ$
(\KPt5 y\,;\spp\mu\,,\spp\vPi\dlsigss00\spp)$ simply measurable and such that 
{\,\rm(1)} and {\,\rm(2)} and {\,\rm(3)} and {\,\rm(4)} below hold for all $\,
A\in\dom m$ and $\,A\spp\ar 1\in\scrmt A$ and $\,\xi\in\vecs\vPi$ and $\,\eta
 \in\Omega\,$. {\rm\inskipline14

(1)} \ $\eta\not\in\bigcup\,\scrmt A\impss33 
     y\fvalue\eta=\vecs\vPi\times\snn\{\ssp 0\ssp\} \KP1 ${\rm, \inskipline{.5}4

(2)} \ $\bigcup\,\scrmt A\capss31 A=\emptyset\impss33 m\fvalue\ssn A = 
     \vecs\vPi\times\snn\{\ssp 0\ssp\} \KP1 ${\rm, \inskipline{.5}4

(3)} \ $A\inc A\spp\ar 1\impss33 m\fvalue\ssn A\fvalss12\xi = 
     \int_{\,A}\ssp\roman{ev}\KPt2\sbi\xi\snn\circ\sp y\rmdss11\mu \KPt6 ${\rm, \inskipline{.5}4

(4)} \ $\vPi$ is reflexive $\impss33
     (\KPt5 y\,;\spp\mu\,,\spp\vPi\dlbetss01\sp)$ is simply measurable. \vskip1mm
  \end{corollary}

\begin{proof} Let \math{\Nu\aR 1=\seqss44{
\sup\KPt8(\ssp\Abrs00^1\circ\sp u\circss01\Nu\invss44\image\ssbb15 I)
:u\in\Cal L\,(\sp\vPi\sp,\spp\bosy K\ssp)} } where 
\math{\Nu} is some fix- ed compatible norm 
for \mathss31{\vPi}. We first show that there is a countable disjoint 
\math{\scrmt C\inc\mu\invss44\image\spp\rbb R^+ } such that 
\math{m\fvalue\ssn A=\vPi\times\snn\{\ssp 0\ssp\} } for all 
\math{A\in\dom m} with 
\mathss36{\bigcup\,\scrmt C\capss31 A=\emptyset}. To see 
this, with \math{
\scrmt A\ar 1=
\dom m\capss21\{\,A:\Nu\aR 1\sn\circ\sp
  m\spp\fvalue\sn A\not=0\KPt9\} } 
\newline 
we let \mathss38{\Cal P=\{\,(\sp\scrmt A\,,\spp\scrmt B\ssp)
:\scrmt A\,,\sp\scrmt B\ssp\text{ are disjoint and }\ssp
\scrmt A\inc\scrmt B\inc\scrmt A\ar 1\,\} }. 
Then $\Cal P$ is a nonempty partial order, and if 
$\Cal C$ is a $\Cal P\,$--\,chain, then 
$\bigcup\KP1\Cal C$ is an upper $\Cal P\,$--\,bound. Hence by 
{\sl Zorn's lemma\sp} there exists some $\Cal P\,$--\,maximal $\scrmt C\ssp$. 
Clearly $\scrmt C$ is as required {\sl if it is countable\sp}. To verify this, 
we note that $\scrmt C
=\{\KP1\roman C\,n:n\in\rbb Z^+\sp\big\}$ when 
$\roman C\,n$ is the set of all $A\in\scrmt C$ 
with \mathss31{n^{\,\mminus 1} < 
(\ssp\mu\fvalue\ssn A\ssp)\KP1^{1\ssp-\,q}\,(\ssp\Nu\aR 1\sn\circ\sp
  m\spp\fvalue\sn A\ssp)\RHB{.3}{\KPt8^q} }. If $\roman C\,n$ 
is finite for every $n\in\rbb Z^+$, then $\scrmt C$ is countable. 
If $\roman C\,n$ is 
infinite for some $n\in\rbb Z^+$, we 
get a contradiction with the assumption that 
\math{m} has bounded \mathss37{\mu}--$
 \RHB{.3}{\KP{1.5} ^q \ssp}$variation in \mathss34{\vPi\dlbetss01}.

Next, using Proposition \ref{Pro Phi5.4} on page \pageref{Pro Phi5.4} 
above we find a countable disjoint 

$\scrmt A\inc\mu\invss44\image\spp\rbb R^+$ with 
$\bigcup\,\scrmt A=\bigcup\,\scrmt C$ and such that 

$\sup\KPt8\{\KPt8(\ssp\mu\fvalue\ssn A\ssp)^{\,\mminus 1}\,
(\ssp\Nu\aR 1\sn\circ\sp m\spp\fvalue\sn A\ssp) :
 A\in\mu\invss44\image\spp\rbb R^+\sn\cap\ssp\Pows A\ar 1\ssp\} < \plusinfty$

holds for every fixed \mathss36{A\ar 1\in\scrmt A}. 
Indeed, we just apply Proposition \ref{Pro Phi5.4} separately to 
\math{\mu\KP1|\KP1\Pows C}
for every fixed \math{C\in\scrmt C} and then take 
the union of the thus obtained partitions. 

Finally we let \math{\scrmt Y} be the set of all 
pairs \math{
(\sp A\ar 1\sp,\spp y\ar 1) } with \mathss03{ A\ar 1 \in 
 \scrmt A} and such that for \mathss03{
\mu\ar 1=\mu\KP1|\KP1\Pows A\ar 1} and 
\math{m\ar 1 = m\KP1|\KP1\Pows A\ar 1} we have \math{
(\ssp y\ar 1\KPt2;\spp\mu\ar 1\sp,\spp\vPi\dlsigss00\spp) } 
simply measurable and 
Pettis with \math{m\ar 1\ssn\fvalue\ssn A\fvalss12\xi = 
 \int_{\,A}\ssp\roman{ev}\KPt2\sbi\xi\snn\circ\sp y\ar 1\rmdss01\mu } 
for all \math{
A\in\dom m\ar 1} and \mathss31{\xi\in\vecs\vPi }, \,and such that also 
\math{(\ssp y\ar 1\KPt2;\spp\mu\ar 1\sp,\spp\vPi\dlbetss01\sp) } is simply 
measurable if \math{\vPi} is reflexive. Then 
considering arbitrarily fixed \math{
A\ar 1\in\scrmt A} and with \vskip.3mm\centerline{$
 \smb M=\sup\KPt8\{\KPt8(\ssp\mu\fvalue\ssn A\ssp)^{\,\mminus 1}\,
(\ssp\Nu\aR 1\sn\circ\sp m\spp\fvalue\sn A\ssp) :
 A\in\mu\invss44\image\spp\rbb R^+\sn\cap\ssp\Pows A\ar 1\ssp\} $} \inskipline{.3}0

taking \mathss03{ \Nu\sprim2 = 
 \{\,(\ssp\xi\,,\spp\smb M\KPt8 t\ssp) : 
 (\ssp\xi\,,\spp t\ssp)\in\Nu\KPt8\} } in place of 
\math{\Nu} in 
Proposition \ref{Pro mA=int ev_x c mu} 
on page \pageref{Pro mA=int ev_x c mu} above, 
we see that \mathss03{ \scrmt A 
 \inc \dom\scrmt Y } holds, and hence by {\sl countable choice\sp} there 
is 
a function \mathss30{\scrmt Y\ar 1\inc\scrmt Y} with \mathss34{ \scrmt A \inc 
 \dom\scrmt Y\ar 1}. Now taking \vskip.3mm\centerline{$
 y = (\ssp\Omega\setminus\bigcup\,\scrmt A\,)
 \times\snn\{\KPt8\vecs\vPi\times\snn\{\ssp 0\ssp\}\sp\}\cupss24
   \bigcup\ssp\rng\scrmt Y\ar 1 \ssp $,} \inskipline{.3}0

it is clear that all the asserted properties hold.
  \end{proof}

Although we shall not need below the result, as an application of the 
Dunford\,--\,Pettis property of \math{\mLrs42^1(\ssp\mu\sp) } we reformulate 
the a bit mysterious looking assertion \vskip.5mm\centerline{%
\q{if $|\tau_0|\tau<\infty$, then $[x(\tau)|\tau\subset\tau_0]$ is compact 
 valued}} \inskipline{.5}0

from \cite[p.\ 131]{Phil} in the following

\begin{proposition}

Let $\,\mu$ be a positive measure with \œ$\,\sup\rng\mu < \plusinfty${\,\rm, }%
and with \œ$\,\bosy K\in{\ssn}$ $\setRC$ let \œ$\,F\in\BaSps0(K)$ be reflexive 
with $\,\Nu$ a compatible norm. Also let \œ$\,m \in{\ssn}$ \œ$ %\sp 
 ^{\roman{dom\,}\mu}\,\vecs E$ be such that \œ$\,m\fvalue(\sp A\cupss31 B\ssp)
 = (\ssp m\fvalue\ssn A + m\fvalue B\ssp)\svs E$ and \œ$\,
\Nu\circss01 m\fvalue\ssn A \le \mu\fvalue\ssn A$ hold for all \œ$\,A\,,\sp B
 \in\dom\mu$ with \œ$\,A\capss31 B=\emptyset\,$. Then $\,\rng m$ is relatively $\,
  \taurd F\,$--\,compact.
  \end{proposition}

\begin{proof} Let \math{\Omega=\sp\bigcup\ssp\dom\mu } and \mathss38{ E = 
 \mvLrs42^1(\ssp\mu\,,\spp\bosy K\ssp) }. Also 

putting \math{\roman x\,A=
(\ssp\Omega\spp\setminus A\ssp)\times\snn\{\ssp 0\ssp\}\cupss22
(\sp A\times\snn\{\ssp 1\ssp\}\sp\sbig)0} 

and \math{\eightroman X\,A=\uniqset\smb\Phii:
\roman x\,A\in\smb\Phii\in\vecs E} let \math{S} 
be the linear \mathss37{\sigrd E}--\,span 

of 
\mathss38{\{\KPt8\eightroman X\,A:
A\in\mu\invss44\image\spp\rbb R^+\sp\big\} }. Thus 
\math{S} is the set of all \math{\smb\Phii\in\vecs E} such that 
there is \math{\varphi\in\smb\Phii} with 
\math{\rng\varphi} finite. We know that 
\math{S} is \mathss37{\taurd E}--\,dense. Then 
we let \math{\smb V\aR 0} be the unique linear extension 
\math{\sigrd E_{\KPt8|\,S}\to\sigrd F } of 
\mathss38{\{\,(\ssp\eightroman X\,A
\,,\spp m\fvalue\ssn A\ssp):
A\in\mu\invss44\image\spp\rbb R^+\sp\big\} }, \,noting that by the assumptions 
on \math{m} we indeed get a linear map.

For finite functions \math{\bosy s\inc
\dom\mu\ar 1\sn\times\vecs\bosy K} with \math{\dom\bosy s} disjoint 
and 
\newline
\math{\varphi=
\sigrd\bosy K\expnota^\sp\aars A_1]_{vs}\text{\KPt8-}
\sum_{\,A\ssp\in\ssp\dom\sn\bmii8 s\,}
(\ssp\bosy s\fvalue\ssn A\ssp)\KP1\roman x\,A
\in\smb\Phii\in\vecs E } 

and for \math{u\in\Cal L\,(\sp F\spp,\spp\bosy K\ssp) } with \math{
\sup\KPt8(\ssp\Abrs00^1\circ\sp u\circss01\Nu\invss44\image\ssbb15 I)
\le 1 } we obtain \vskip.5mm

$|\KP1 u\circss00 \smb V\aR 0\sn\fvalue\smb\Phii\KP1|=
\big|\,
\sum_{\,A\ssp\in\ssp\dom\sn\bmii8 s\,}
(\ssp\bosy s\fvalue\ssn A\ssp)\KP1
(\ssp u\circss00 m\fvalue\ssn A\ssp)\KP1|\le
\sum_{\,A\ssp\in\ssp\dom\sn\bmii8 s\,}|\KP1
\bosy s\fvalue\ssn A\KP1|\KP1|\KP1
u\circss00 m\fvalue\ssn A\KP1| $ \inskipline1{11}

${}\le
\sum_{\,A\ssp\in\ssp\dom\sn\bmii8 s\,}|\KP1
\bosy s\fvalue\ssn A\KP1|\KP1
(\ssp\Nu\circss01 m\fvalue\ssn A\ssp)\le
\sum_{\,A\ssp\in\ssp\dom\sn\bmii8 s\,}|\KP1
\bosy s\fvalue\ssn A\KP1|\KP1
(\ssp\mu\fvalue\ssn A\ssp)=
\|\,\varphi\,\|\Lnorss33^1_{\aars\mu_1} \, $. \inskipline{.7}0

Consequently \math{\smb V\aR 0} has a unique continuous extension \mathss38{
\smb V\in\Cal L\,(\sp E\ssp,\spp F\ssp) }.

Taking \mathss38{K=\{\KPt8\eightroman X\,A:A\in\dom\mu\KPt8\} }, \,from Lemma \ref{Le L^1_si-compa} 
on page \pageref{Le L^1_si-compa} above we see that \math{K} is relatively \mathss37{
\taurd(\sp E\subsigrs04)}--\,compact. Since by 
Corollary \ref{Cor L^1 is D-S} on page \pageref{Cor L^1 is D-S} above 
\math{E} is a \erm{DP\,}--\,space, noting that by reflexivity of 
\math{F} all bounded sets in \math{F} are relatively 
\mathss37{\taurd(F\!\RHB{.25}{\subsigma})}--\,compact, we 
see that \math{\smb V\KPt8\image\sn K} is 
relatively \mathss37{\taurd F}--\,compact. Noting that also 
\math{\rng m\inc\smb V\KPt8\image\sn K} holds, 
we are done.
  \end{proof}

% ----------------------------------------------------------------------------

\Ssubhead D               Duality of Bochner spaces                       \label{Sec D}

Proceeding by a sequence of lemmas, we here give the proof of %% \label{beginmpf}
Theorem \nfss A\,\ref{main Th} on page \pageref{main Th} above. From now on 
untill the end of the proof of Lemma \nfss A\,\ref{final lemma} on 
page \pageref{endmpf} below, without further mention we let \math{p\,,\sp
\bosy K\ssp,\sp\vPi\sp,\sp\mu\,,\sp\Omega
\,,\sp F\sp,\sp F\aar 1\ssp,\sp\Iota } be as in 
Theorem \nfss A\,\ref{main Th}\sp. For short, we call this assumption 
together with the temporary shorthands below 
{\it Assumptions %%{\eit A\sp}
                 \hbox{\font\Å=cmssi9\ÅA}\sp}. From Corollary \ref{Coro Io inj etc} 
on page \pageref{Coro Io inj etc} above, we see that \math{\Iota} is an 
injective continuous linear map \mathss34{F\aar 1\to F\dlbetss10 }. Since by 
Theorem \ref{Th L_s^p Ba} and Corollary \ref{Cor L^p Ban} the spaces \math{
F\aar 1} and \math{F\dlbetss10} are \erm Banachable, by the open mapping 
theorem we only need to verify the surjectivity \mathss34{
\Cal L\,(\sp F\spp,\spp\bosy K\ssp)\inc\rng\Iota                          \label{page surj}
}. This we shall do separately for \mathss30{p=1} under (1) or (2) or (3) or 
(4) and for \math{1 < p < \plusinfty} under (5) or (6)\,.

Fixing a compatible norm \math{\Nu} 
for \math{\vPi} and letting \math{\Nu\aR 1} be the dual norm, we introduce the 
following shorthands \vskip.6mm

$\|\,\smb X\ssp\|\sNorF = \inf\sp\big\{\KPt8\|\KP1\Nu\circss01 x\KP1
 \|\Lnorss33^p_\mu\snn:x\in\smb X\,\} \hfill $ and \KP7 \vskip.4mm

$\|\,\smb Y\KPt8\|\sNorFp = \inf\sp\big\{\KPt8\|\KP1\Nu\ar 1\snn\circ\sp y\KP1
 \|\Lnorss50^{p^*}_\mu\snn:y\in\smb Y\KPp1.1\} \hfill $ and \KP7 \vskip.4mm

$\|\KPt8\smb U\,\| = \sup\ssp\big\{\KP1|\KP{1.1}\smb U\sp\fvalue\smb X\KPt9| : 
 \smb X\in\vecs F\ssp\text{ and }\ssp \|\,\smb X\ssp\|\sNorF\le 1\KPt9\} \hfill $ 
  and \KP7 \vskip.4mm

$\roman f\,u\,\xi=\uniqset\smb X:{}$ \inskipline0{18.5}

$(\ssp\Omega\setminus\snn\dom u\ssp)\times\snn\{\,\Bnull_\vPi\}\cupss22
\seqss33{((\ssp u\fvalue\eta\ssp)\KP1\xi\ssp)\svs\vPi\sn:\eta\in\dom u}
\in\smb X\in\vecs F \sp $. \inskipline{.6}0

Note that by the discussion after the proof of Lemma \ref{Le 0_{L^p}} on page \pageref{discus inf N = N} 
above we in fact have \math{\|\,\smb X\ssp\|\sNorF = 
 \|\KP1\Nu\circss01 x\KP1 \|\Lnorss33^p_\mu } for \mathss30{x\in\smb X\in
  \vecs F}.

\begin{Alemma}\label{LeA(1)}

If under {\,\rm Assumptions \nfss A} also {\,\rm(1)} holds{\ssp\rm, }then $\,
\Cal L\,(\sp F\spp,\spp\bosy K\ssp)\inc\rng\Iota\ssp$.
  \end{Alemma}

\begin{proof} Arbitrarily fix \math{ \smb U \in 
 \Cal L\,(\sp F\spp,\spp\bosy K\ssp) } and let \math{\scrmt A} and \math{
N\sprim1} be as in Definitions \ref{df decomp}\,(2) on page \pageref{decos A} 
above. Then let \math{\scrmt Y} be the set of all pairs \math{
(\sp A\ar 1\KPt2;\sp y\ar 1\sp,\spp S\ar 1) } with \mathss03{ A\ar 1 \in 
 \scrmt A} and \math{\rng y\ar 1\inc S\ar 1} and \math{S\ar 1} a separable 
closed linear subspace in \math{\vPi\dlbetss01} and such that for \math{
\mu\ar 1=\mu\KP1|\KP1\Pows A\ar 1} and \mathss30{m = \langle\ssp\seqss33{ 
 \smb U\fvalss11\lfloor\,^{1\sp,\ssp\aars\mu_1\spp,\ssp\vPi}\ssp\xi\sp\sbi A \snn 
 : \xi\in\vecs\vPi} : A \in \mu\ar 1\sn\inve\sp\image\spp\lbb R_+\ssp
 \big\rangle } we have \math{
(\ssp y\ar 1\KPt2;\spp\mu\ar 1\sp,\spp\vPi\dlbetss01\sp) } measurable and 
Pettis with \math{m\fvalue\ssn A\fvalss12\xi = 
 \int_{\,A}\ssp\roman{ev}\KPt2\sbi\xi\snn\circ\sp y\ar 1\rmdss01\mu } for all \linebreak \mathss03{
A\in\dom m} and \mathss31{\xi\in\vecs\vPi }. Now considering arbitrarily fixed \math{
A\ar 1\in\scrmt A} and choosing \mathss03{ \Nu\sprim2 = 
 \{\,(\ssp\xi\,,\sp t\KPt8\|\KPt8\smb U\,\|\sp\sbig)0 : 
 (\ssp\xi\,,\spp t\ssp)\in\Nu\KPt8\} } by Proposition \ref{Pro mA=int ev_x c mu} 
on page \pageref{Pro mA=int ev_x c mu} above we see that \mathss03{ \scrmt A 
 \inc \dom\scrmt Y } holds, and hence by the {\sl axiom of choice\sp} there is 
a function \mathss30{\scrmt Y\ar 1\inc\scrmt Y} with \mathss34{ \scrmt A \inc 
 \dom\scrmt Y\ar 1}. Let \mathss34{ y = (\ssp\Omega\setminus N\sprim1\sp)
 \times\snn\{\KPt8\vecs\vPi\times\snn\{\ssp 0\ssp\}\sp\}\cupss24
   \bigcup\ssp\dom\sn\rng\scrmt Y\ar 1}.

To verify that \math{y\in\bigcup\ssp\vecs F\aar 1} holds, it suffices to get \mathss38{
\|\KP1\Nu\aR 1\snn\circ\sp y\KP1\|\Lnorss33^\plusinftyy_\mu\le
 \|\KPt8\smb U\,\| }. This in turn follows if for every fixed \math{A\ar 1\in
 \scrmt A} we show existence of some \math{ N \in 
 \mu\invss44\image\snn\{\ssp 0\ssp\} } such that \math{
\Nu\aR 1\snn\circ\sp y\fvalue\eta\le\|\KPt8\smb U\,\| } holds for \mathss30{
\eta\in A\ar 1\ssn\setminus N}. Now for \math{A\in\dom\mu\capss22\Pows A\ar 1 } 
and \math{\xi\in\vecs\vPi} we have \math{
  \smb U\fvalss11\lfloor\,^{1\sp,\ssp\mu\sp,\ssp\vPi}\ssp\xi\sp\sbi A
= m\fvalue\ssn A\fvalss01\xi
= \int_{\,A}\ssp\roman{ev}\KPt2\sbi\xi\snn\circ\sp y\rmdss11\mu } and hence \vskip.2mm\centerline{$
   \big|\sp\int_{\,A}\ssp\roman{ev}\KPt2\sbi\xi\snn\circ\sp y\rmdss11\mu\KP1|
\le \|\KPt8\smb U\,\|\KP1(\ssp\Nu\fvalss02\xi\ssp)\KP1
    (\ssp\mu\fvalue\ssn A\ssp) \KP1 $.} \inskipline{.6}0

Then by Corollary \ref{Coro |f|<M} on page \pageref{Coro |f|<M} 
above for every \math{\xi\in\vecs\vPi} there is \math{ N\aar 1 \in 
 \mu\invss44\image\snn\{\ssp 0\ssp\} } such that \math{
|\KPp1.1 y\fvalue\eta\fvalss01\xi\KP1| \le
 \|\KPt8\smb U\,\|\KP1(\ssp\Nu\fvalss02\xi\ssp) } holds for \mathss34{ \eta 
 \in A\ar 1\ssn\setminus N\aar 1 }. 

Now taking \math{S\ar 1=
\roman{pr}\ar 2\circ\scrmt Y\ar 1\ssn\fvalue\ssn A\spp\ar 1 } in place of \math{
S} in Lemma \ref{Le Nu_1 = sup ...} on page \pageref{Le Nu_1 = sup ...} above, 
let \math{D} be as given there. Then considering fixed \math{\xi\in D} we find \math{
N\aar 1\in\mu\invss44\image\snn\{\ssp 0\ssp\} } with \math{
|\KPp1.1 y\fvalue\eta\fvalss01\xi\KP1| \le
 \|\KPt8\smb U\,\|\KP1(\ssp\Nu\fvalss02\xi\ssp) } for all \mathss34{ \eta \in 
 A\ar 1\ssn\setminus N\aar 1}. By {\sl countable choice\sp} taking as \math{N} 
the union of these \math{N\aar 1} we get \math{
|\KPp1.1 y\fvalue\eta\fvalss01\xi\KP1| \le
 \|\KPt8\smb U\,\|\KP1(\ssp\Nu\fvalss02\xi\ssp) } for all \math{ \eta \in 
 A\ar 1\ssn\setminus N } and \mathss34{\xi\in D}. Now having \mathss38{
\Nu\aR 1\sn\circ\sp y\fvalue\eta = 
 \sup \KPt8(\ssp\Abrs00^1\circ\sp(\ssp y\fvalue\eta\ssp)\spp\image\sn D\ssp)
 \le\|\KPt8\smb U\,\|} for all \mathss30{\eta\in A\ar 1\ssn\setminus N}, \,the 
assertion follows.

Thus having \math{y\in\bigcup\ssp\vecs F\aar 1} there is \math{\smb Y} with \mathss34{
y\in\smb Y\in\vecs F\aar 1}. To proceed, we first note that we now have \math{
m\fvalue\ssn A\fvalss12\xi =                                              \label{Le A2 final ded}
 \int_{\,A}\ssp\roman{ev}\KPt2\sbi\xi\snn\circ\sp y\rmdss11\mu } for all \math{A\in
 \mu\invss44\image\spp\rbb R^+} and \mathss31{\xi\in\vecs\vPi}. To see this, 
let \math{\scrmt C=\scrmt A\capss31\{\,A\spp\ar 1\ssn:A\spp\ar 1\snn\cap\sp A
 \in\mu\invss44\image\spp\rbb R^+\sp\big\} } and \mathss36{N = N\sprim1\cup\ssp
 \bigcup\KP1(\sp\scrmt A\sp\setminus\scrmt C\ssp)\capss21 A }. Then \math{
\scrmt C} is countable since otherwise \math{A\in\mu\invss44\image\spp\rbb R^+} 
would be contradicted. In addition \math{ N \in \bigcup\KPt8\{\KPt8
 \Pows N\aar 1\sn:N\aar 1\in\mu\invss44\image\snn\{\ssp 0\ssp\} \sp\} } holds 
with \mathss30{ A = \bigcup\KP1(\sp\scrmt C\leiss31 A\ssp)\cupss21 N}. \hfill 
Now by do- \linebreak minated convergence we obtain \inskipline{.6}{20.2}

$ m\fvalue\ssn A\fvalss12\xi
= \smb U\fvalss11\lfloor\,^{1\sp,\ssp\mu\sp,\ssp\vPi}\ssp\xi\sp\sbi A
= \sum_{\,\aars A_1\sp\in\ssp\scrm7 A\,}(\ssp
   \smb U\fvalss11\lfloor\,^{1\sp,\ssp\mu\sp,\ssp\vPi}\ssp\xi\sp\sbi
        {\aars A_1\capss25 A}\sp\sbig)0 $ \inskipline{.6}{31}

${}
= \sum_{\,\aars A_1\sp\in\ssp\scrm7 A\,}
   m\fvalue(\sp A\spp\ar 1\sn\cap\sp A\ssp)\fvalss01\xi $ \inskipline{.4}{31}

${}
= \sum_{\,\aars A_1\sp\in\ssp\scrm7 A\ssp}\int_{\,\aars A_1\capss25 A}\ssp
        \roman{ev}\KPt2\sbi\xi\snn\circ\sp y\rmdss11\mu
= \int_{\,A}\ssp\roman{ev}\KPt2\sbi\xi\snn\circ\sp y\rmdss11\mu \,$. \inskipline{.8}0

Then by Lemma \ref{Le-first} on page \pageref{Le-first} above we have \vskip.3mm\centerline{$
U=
\vecs F\snn\times\mathbb C\capss31\{\,(\ssp\smb X\sp,\spp t\ssp) : 
 \aall{x\in\smb X}\,
     t = \int_{\KPp1.1\Omega\,}y\,.\KPt8 x\rmdss11\mu\KPt9\} $} \inskipline{.5}0

and hence \math{ \smb U = \Iota\snn\fvalue\sp\smb Y} holds and so \math{
\smb U\in\rng\Iota} is established.
  \end{proof}

\begin{Alemma}\label{LeA(2)}

If under {\,\rm Assumptions \nfss A} also {\,\rm(2)} holds{\ssp\rm, }then $\,
\Cal L\,(\sp F\spp,\spp\bosy K\ssp)\inc\rng\Iota\ssp$.
  \end{Alemma}

\begin{proof}  Given \mathss38{\smb U
\in\Cal L\,(\sp F\spp,\spp\bosy K\ssp) }, \,letting \vskip.5mm\centerline{$
\smb V=
\ssp\big\langle\KPt8
\{\,(\ssp\xi\,,\spp\smb U\fvalss22\roman f\KPt8\varphi\KP1\xi
\ssp):\vPi:\xi\in\vecs\vPi\ssp\text{ and }\ssp\varphi\in\smb\Phii\,\}
:
\smb\Phii \in 
 \vecs\sn\mvLrs42^1(\ssp\mu\,,\spp\bosy K\ssp) 
\KP1\big\rangle \KP1 $,} \inskipline{.5}0

we easily see \math{\smb V \in 
 \Cal L\,(\ssp\mvLrs42^1(\ssp\mu\,,\spp\bosy K\ssp)\,,\spp\vPi\dlbetss01\sp) } 
to hold. Hence by 
Proposition \ref{Pro Edw 8.17.6} on page \pageref{Pro Edw 8.17.6} above there 
exists some \math{y \in \bigcup\ssp\vecs\sn
 \mvsLrs23^\plusinftyy(\ssp\mu\,,\spp\vPi\dlsigss00\spp) } such that \vskip.5mm\centerline{$
\smb U\fvalss22\roman f\KPt8\varphi\KP1\xi
=\smb V\fvalss60\smb\Phii\fvalss00\xi=
\int_{\KP{1.1}\Omega\,} 
 \roman{ev}\KPt2\sbi\xi\snn\circ\sp y\cdot\varphi\rmdss21\mu $} \inskipline{.5}0

holds for \math{\varphi\in\smb\Phii\in\dom\smb V} and \mathss31{ \xi \in 
 \vecs\vPi }. Noting that from \mathss30{ y \in \bigcup\ssp\vecs\sn
 \mvsLrs23^\plusinftyy(\ssp\mu\,,\spp\vPi\dlsigss00\spp) } we directly get \math{
\|\KP1\Nu\aR 1\snn\circ\sp y\KP1\|\Lnorss33^\plusinftyy_\mu < \plusinfty } now 
Lemma \ref{Le-first} gives the conclusion similarly as in the proof of Lemma \nfss A\,\ref{LeA(1)} 
  above.
  \end{proof}

\begin{Alemma}\label{LeA(3)}

If under {\,\rm Assumptions \nfss A} also {\,\rm(3)} holds{\ssp\rm, }then $\,
\Cal L\,(\sp F\spp,\spp\bosy K\ssp)\inc\rng\Iota\ssp$.
  \end{Alemma}

\begin{proof} Given \mathss38{\smb U
\in\Cal L\,(\sp F\spp,\spp\bosy K\ssp) }, \,define $\smb V:
\mvLrs42^1(\ssp\mu\,,\spp\bosy K\ssp)\to\vPi\dlbetss01$ by 
$\smb V\fvalss60\smb\Phii\fvalss00\xi=
\smb U\fvalss22\roman f\KPt8\varphi\KP1\xi$ for 
$\varphi \in \smb\Phii \in 
 \vecs\sn\mvLrs42^1(\ssp\mu\,,\spp\bosy K\ssp)$ and $\,\xi\in\vecs\vPi\sp$. 
Then by 
Proposition \ref{Pro Edw 8.17.8} on page \pageref{Pro Edw 8.17.8} above there 
is $y \in \bigcup\ssp\vecs\sn
 \mvsLrs23^\plusinftyy(\ssp\mu\,,\spp\vPi\dlsigss00\spp)$ with 
$\smb U\fvalss22\roman f\KPt8\varphi\KP1\xi
=\smb V\fvalss60\smb\Phii\fvalss00\xi=
\int_{\KP{1.1}\Omega\,} 
 \roman{ev}\KPt2\sbi\xi\snn\circ\sp y\cdot\varphi\rmdss21\mu$ . 
The rest proceeds as in the proof of Lemma \nfss A\,\ref{LeA(1)} above.
  \end{proof}

\begin{Alemma}\label{LeA(4)}

If under {\,\rm Assumptions \nfss A} also {\,\rm(4)} holds{\ssp\rm, }then $\,
\Cal L\,(\sp F\spp,\spp\bosy K\ssp)\inc\rng\Iota\ssp$.
  \end{Alemma}

\begin{proof}                                                          \newcommand\erU{\hbox{\font\Å=cmr8\ÅU}\kern.7mm}%
              Arbitrarily fix \mathss38{\smb U \in 
\Cal L\,(\sp F\spp,\spp\bosy K\ssp) }. Putting \math{ G = 
 \mvLrs23^\plusinftyy(\ssp\mu\,,\spp\bosy K\ssp) } and letting \math{c} be as 
given by (4) in Theorem \nfss A\,\ref{main Th} let \math{c\ar 1 = c } if \math{
\bosy K=\tfbbR} holds and in the complex case let \math{ c\ar 1 = 
 \Iota\ar 3\circ\sp(\ssp c\ftimes c\ssp)\circ\sp\Iota\ar 2} where with \math{
S = \vecs\mLrs23^\plusinftyy(\ssp\mu\sp) \times
    \vecs\mLrs23^\plusinftyy(\ssp\mu\sp) } we have \inskipline{.5}{10}

$\Iota\ar 3 = \{\,(\ssp x\ssp,\spp y\ssp,\spp x + \imag\KPt8 y\ssp) : 
  x\ssp,\sp y\in\vecs\lll^\plusinftyy\sp(\sp\Omega\sp)\KPt8\} \KP{33} $ and \inskipline{.2}{10}

$\Iota\ar 2 = \vecs G\times S\capss31\{\,(\ssp\smb\Psii\,;\sp
  \smb\Psii\ar 1\sp,\spp\smb\Psii\aR 2\spp) : \aall{\psi\ar 1\in\smb\Psii\ar 1\KPt2
 ,\sp\psi\ar 2\in\smb\Psii\ar 2}\,
     \psi\ar 1\snn + \imag\KPt8\psi\ar 2\in\smb\Psii\,\} \KP1 $. \inskipline{.5}0

Then \math{c\ar 1} is a continuous linear choice function \math{G\to
\lll^\plusinftyy\sp(\ssp\Omega\,,\spp\bosy K\ssp) } and hence there is some \math{
\smb A\in\lbb R_+} with the property that \math{
    \|\KP1 c\fvalss01\smb\Psii\KP1\|\lllnor_\plusinftyy 
\le \smb A\KP1\|\,\psi\,\|\Lnorss33^\plusinftyy_\mu } holds for \mathss30{
\psi\in\ssn} \mathss02{\smb\Psii\in\vecs G}. Further let \math{\Iota\ar 1} be 
as in Proposition \ref{Pro L^1'=L^i} on page \pageref{Pro L^1'=L^i} above. 
Then for \mathss30{ E = \ssn } \mathss03{\mvLrs42^1(\ssp\mu\,,\spp\bosy K\ssp) } 
with \math{\erU\xi = \vecs E\times\ssbb00 C\capss41\{\,
 (\ssp\smb\Phii,\spp t\ssp) : \aall{\varphi\in\smb\Phii}\, t = 
 \smb U\fvalss11\roman f\,\varphi\KP1\xi\KPt9\} } we obtain \math{ \smb V \in \ssn } \mathss03{
 \Cal L\,(\sp\vPi\sp,\sp\lll^\plusinftyy\sp(\ssp\Omega\,,\spp\bosy K\ssp)) } 
by taking \mathss38{\smb V = c\ar 1\sn\circ\sp\Iota\ar 1\sn\inve\circ\sp
 \seqss33{\erU\xi:\xi\in\vecs\vPi} }.

Now for \math{\xi\in\vecs\vPi} and \math{\varphi\in\smb\Phii\in\vecs E} we 
have \math{\smb V\fvalss50\xi\in\Iota\ar 1\sn\inve\sp\fvalue\sp\erU\xi } and 
hence \mathss30{\smb U\fvalss11\roman f\,\varphi\KP1\xi } \mathss03{ = 
 \erU\xi\sp\fvalss10\smb\Phii = 
 \int_{\KP{1.1}\Omega}\sp\smb V\fvalss50\xi\cdot\varphi\rmdss11\mu }. Taking \inskipline{.5}{12}

$\smb B = \inf\,\{\KPt8\sup\ssp\big\{\,\big|\sp\int_{\KPp1.1\Omega\,}
          \psi\cdot\varphi\rmdss21\mu\KP1| : \varphi\in\bigcup\ssp\vecs E\ssp\text{ 
 and }\ssp \|\,\varphi\,\|\Lnorss33^1_\mu\le 1\KPt8\}$ \inskipline{.5}{58.6}

${} : \psi\in\bigcup\ssp\vecs G\ssp\text{ and }\ssp
  \|\,\psi\,\|\Lnorss33^\plusinftyy_\mu = 1\KPt8\} \KP1 $, \inskipline{.5}0

we have \math{\smb B\in\rbb R^+} unless \math{\Omega} is \mathss37{\mu
}--\,negligible in which case the assertion of the lemma to be proved 
trivially holds. Then for \math{\xi\in\vecs\vPi} we get \inskipline{.4}4

(\sp$*$\sp) \ $\|\KP1\smb V\fvalss50\xi\KP1\|\lllnor_\plusinftyy \le
            \smb A\KP1\|\KP1\Iota\ar 1\sn\inve\sp\fvalue\sp\erU\xi
 \KP1\|\Lnorss33^\plusinftyy_\mu \le \smb A\,\smb B^{\ssp\mminus 1}\,
 \|\KPt8\smb U\,\|\KP1(\ssp\Nu\fvalss12\xi\ssp) \KP1 $. \vskip.5mm

Now taking \mathss38{ y = 
%%\seqss04{\sn\seqss33{\roman v\,\xi\fvalss20\eta : \xi\in\vecs\vPi}
\seqss33{\roman{ev}\sp\sbi{\eta}\snn\circ
%\sp
\smb V
:\eta\in\Omega} }, \,trivially \math{
(\KPt5 y\,;\spp\mu\,,\spp\vPi\dlsigss00\spp) } is finitely almost sca- larly 
measurable, and by (\sp$*$\sp) above having \math{
\|\KP1\Nu\aR 1\snn\circ\sp y\KP1\|\Lnorss33^\plusinftyy_\mu \le
 \smb A\,\smb B^{\ssp\mminus 1}\,\|\KPt8\smb U\,\| } we get \mathss02{ y \in 
 \bigcup\ssp\vecs F\aar 1 }. Then the conclusion follows from Lemma \ref{Le-first} 
  similarly as above.
  \end{proof}

As opposed to the case (1) in Lemma \nfss A\,\ref{LeA(1)} above, note that in 
the cases (2) and (3) and (4) in Lemmas \nfss A\,\ref{LeA(2)} and 
\nfss A\,\ref{LeA(3)} and \nfss A\,\ref{LeA(4)} we only got \math{
\|\,\smb Y\KPt8\|\sNorFp\le\smb A\KP1\|\KP1\Iota\snn\fvalue\sp\smb Y\KPp1.2\| } 
for all \mathss03{\smb Y\in\vecs F\aar 1} for some \math{\smb A} with \math{
1\le\smb A<\plusinfty} and possibly \math{1<\smb A}.

\begin{Alemma}\label{final lemma}

If under {\,\rm Assumptions \nfss A} also 
{\,\rm(5)} or {\,\rm(6)} holds{\ssp\rm, } \inskipline{.2}{54.3}

then $\,\Cal L\,(\sp F\spp,\spp\bosy K\ssp)\inc\rng\Iota\ssp$.
  \end{Alemma}

\begin{proof} Since the verification is quite long ending on page \pageref{endmpf} 
below, we devide it into Steps 1\ssp$,\ldots\,$4\ssp. Now, arbitrarily fixing \math{
\smb U \in \Cal L\,(\sp F\spp,\spp\bosy K\ssp) } let \vskip.5mm\centerline{$
m = \langle\ssp\seqss33{ \smb U\fvalss11
     \lfloor\,^{p\sp,\ssp\mu\sp,\ssp\vPi}\ssp\xi\sp\sbi A\snn : \xi \in 
   \vecs\vPi} : A\in\mu\invss44\image\spp\lbb R_+\ssp\big\rangle \KP1 $.}

\Step 1.0 We first show that \math{m} has bounded \mathss37{
          \mu}--$\KP{1.5} ^{p\sast}\ssp$variation in \mathss34{\vPi\dlbetss01
}. Indeed, we show that \math{
\sum_{\,A\ssp\in\ssp\scrm7 A\,}( (\ssp\mu\fvalue\ssn A\ssp)
 \KP1^{1\ssp-\,p\sast}\ssp(\ssp\Nu\aR 1\sn\circ\sp m\spp\fvalue\sn A\ssp)
  \KP1^{p\sast}\sp\sbig)0
\le\|\KPt8\smb U\,\|\KP1^{p\sast} } holds for arbitrarily given finite 
disjoint \mathss30{\scrmt A\inc\mu\invss44\image\spp\rbb R^+}. In order to get 
this, we first note that for arbitrarily given \math{
\bosy\xi\in\sp^{\scrm7 A}\,(\ssp\Nu\invss44\image\ssbb15 I) } it suffices to 
show that \vskip.0mm\centerline{$
\big(\sp
\sum_{\,A\ssp\in\ssp\scrm7 A\,}( (\ssp\mu\fvalue\ssn A\ssp)
 \KP1^{1\ssp-\,p\sast}\ssp
|\KPp1.1 m\,.\KPt9\bosy\xi\KPt4\fvalue\KN{.4}A
\KP1|\KP1^{p\sast}\ssp\sbig)0\sbig)0\KP1^{p\sast\sp^{-1}}\sn\le\sp
\|\KPt8\smb U\,\| $} \inskipline{.5}0

holds since otherwise we could easily get a contradiction. 

$\null\hfill$
In order to get this, taking 
\math{s=p^{\,*\sn}-1} and with the short- 
  \linebreak
hand \math{
\roman v\ssp A=(\ssp\mu\fvalue\ssn A\ssp)
 \KP1^{p\sast\sp^{-1}\spp -\ssp 1}\,(\ssp
 m\,.\KPt9\bosy\xi\KPt4\fvalue\KN{.4}A\ssp) } putting \math{v=
\seqss33{\roman v\ssp A:A\in\scrmt A} } we have 
\math{v\in\sp^{\scrm7 A}\,\ssbb10 C } and we need to show that \math{
\|\,v\,\|\lllnor_{p\sast}\le\|\KPt8\smb U\,\| } holds. We may assume that 
\math{\|\,v\,\|\lllnor_{p\sast}\not=0 } holds, and then taking \math{u=
\seqss33{\roman u\,A:A\in\scrmt A} } where 
%%we have 
\math{\roman u\,A=
\|\,v\,\|\lllnor_{p\sast}\snn^{\mminus\emath s}\,
|\KP1\roman v\ssp A\KP1|\KPt8^{\emath s\ssp - \ssp 1}\,
\overline{\roman v\ssp A\RHB{.0}{\KN{.99}\phantom{'}}}
} if \math{\roman v\ssp A\not=0} holds, otherwise having 
\mathss36{\roman u\,A=0}, \,we now have \math{
\|\,u\,\|\lllnor_p=1 } and \vskip.5mm\centerline{$
\big|\sp\sum\KP1(\ssp u\cdot v\ssp)\KP1|=
\sum\KP1(\ssp u\cdot v\ssp)=
\sum_{\,A\ssp\in\ssp\scrm7 A\,}(\ssp\roman u\,A\KP1\roman v\ssp A\ssp)=
\|\,v\,\|\lllnor_{p\sast} \KP1 $.} \inskipline{.3}0

Furthermore, with the shorthand \math{ \roman t\,A = \roman u\,A\KP1(\ssp
 \mu\fvalue\ssn A\ssp)\,^{\mminus p}\LHB{.2}{^{^{-1}}} } we have \inskipline1{13}

$ \big|\sp\sum\KP1(\ssp u\cdot v\ssp)\KP1|
= \big|\sp\sum_{\,A\ssp\in\ssp\scrm7 A\,}(\ssp
   \roman u\,A\KP1\roman v\ssp A\ssp) \KP1 | $ \inskipline{.6}{31.5}

${} = \big|\sp\sum_{\,A\ssp\in\ssp\scrm7 A\,}(\,\roman t\,A\KP1(\ssp
       \mu\fvalue\ssn A\ssp)\KP1^p\LHB{.2}{^{^{-1}}} (\ssp
       \mu\fvalue\ssn A\ssp)\KP1^{p\sast\sp^{-1}\spp -\ssp 1}\,(\ssp
        m\,.\KPt9\bosy\xi\KPt4\fvalue\KN{.4}A\ssp )\ssp )\KP1 | $ \inskipline{.6}{31.5}

${} = \big|\sp\sum_{\,A\ssp\in\ssp\scrm7 A\,}(\,\roman t\,A\KP1(\ssp
      m\,.\KPt9\bosy\xi\KPt4\fvalue\KN{.4}A\ssp )\ssp )\KP1 |$ \inskipline{.6}{31.5}

${} = \big|\sp\sum_{\,A\ssp\in\ssp\scrm7 A\,}(\,\roman t\,A\KP1(\ssp
       \smb U\fvalss20
        \lfloor\,^{p\sp,\ssp\mu\sp,\ssp\vPi}\ssp(\ssp
         \bosy\xi\KPt4\fvalue\KN{.4}A\ssp)\sp\sbi A\ssp))\KP1|$ \inskipline{.6}{31.5}

${} = |\KP1\smb U\fvalue(\ssp\sigrd F\text{\KPt8-}
       \sum_{\,A\ssp\in\ssp\scrm7 A\,}\roman t\,A\KP1
        \lfloor\,^{p\sp,\ssp\mu\sp,\ssp\vPi}\ssp(\ssp
         \bosy\xi\KPt4\fvalue\KN{.4}A\ssp)\sp\sbi A\,|$ \inskipline{.6}{31.5}

${}\le\|\KPt8\smb U\,\|\KP1\|\KP1\sigrd F\text{\KPt8-}
       \sum_{\,A\ssp\in\ssp\scrm7 A\,}\roman t\,A\KP1
        \lfloor\,^{p\sp,\ssp\mu\sp,\ssp\vPi}\ssp(\ssp
         \bosy\xi\KPt4\fvalue\KN{.4}A\ssp)\sp\sbi A\,\|\sNorF$ \inskipline{.6}{31.5}

${}\le\|\KPt8\smb U\,\|\KP1\big(\sp\sum_{\,A\ssp\in\ssp\scrm7 A}\sp\sbig(3
       |\KP1\roman t\,A\KP1|\RHB{.3}{\KP1^p}\,(\ssp
         \mu\fvalue\ssn A\ssp)))\KP1^p\LHB{.2}{^{^{-1}}} $ \inskipline{.6}{31.5}

${} = \|\KPt8\smb U\,\|\KP1\|\,u\,\|\lllnor_p=\|\KPt8\smb U\,\| \KP1 
      $, \,giving the assertion.

\Step 2.0 Noting that the requirement of absolute continuity holds since we 
trivially have \math{\Nu\aR 1\sn\circ\sp m\spp\fvalue\sn A \le 
 \|\KPt8\smb U\,\|\KP1(\ssp\mu\fvalue\ssn A\ssp)\KPt8\RHB{.2}{^p}{^{^{\sp-1}}} } 
for any \math{A\in\mu\invss44\image\spp\lbb R_+}, \,now let \math{\scrmt A} 
and \math{y} be as given by Corollary \ref{Coro q-var} on page \pageref{Coro q-var} 
above. Then we have \math{(\KPt5 y\,;\spp\mu\,,\spp\vPi\dlsigss00\spp) } simply 
measurable and such that \math{y\fvalue\eta=\vPi\times\snn\{\ssp 0\ssp\} } 
holds for \math{\eta\in\Omega\sp\setminus\bigcup\,\scrmt A } and such that we 
also have \math{m\fvalue\ssn A\fvalss12\xi = 
 \int_{\,A}\ssp\roman{ev}\KPt2\sbi\xi\snn\circ\sp y\rmdss11\mu } for \math{
A\spp\ar 1 \in\scrmt A} and \math{A\in\dom\mu\capss33\Pows A\ar 1} and \mathss31{
\xi\in\vecs\vPi}. In addition \math{
(\KPt5 y\,;\spp\mu\,,\spp\vPi\dlbetss01\sp) } is simply measurable if \math{
\vPi} is reflexive.

\Step 3.0 Under (5) or (6) to prove that \math{
\|\KP1\Nu\aR 1\sn\circ\sp y\KP1\|\Lnorss40^{p\sast}_\mu\le\|\KPt8\smb U\,\| } 
holds, noting that in the reflexive case now \math{
(\ssp\Abrs33^{p\sast}\KN1\circ\KPt2\Nu\aR 1\sn\circ\sp y\,;\spp
 \mu\,,\sn\tfbbR\ssp) } is trivially measurable, and that by Lemma \ref{Le Nu_1 ci y meas} 
on page \pageref{Le Nu_1 ci y meas} above the same holds also in the separable 
case, it suffices to show that \math{\int_{\KPp1.1\Omega\,}\Abrs33^{p\sast}\KN1
 \circ\KPt2\Nu\aR 1\sn\circ\sp y\rmdss11\mu \le 
 \|\KPt8\smb U\,\|\RHB{.2}{\KP1^{p\sast}} } holds. For this in turn for every 
fixed $A\sp\ar 0\in\mu\invss44\image\spp\rbb R^+\sp$ it suffices to show that \math{
\int_{\,\aars A_0\sp}(\ssp\Nu\aR 1\sn\circ\sp y\fvalue\eta\ssp)
  \RHB{.2}{\KP1^{p\sast}}\ssn\rmdss01\mu\,(\sp\eta\sp)
 \le\|\KPt8\smb U\,\|\RHB{.2}{\KP1^{p\sast}} } holds. 

Now we can express \math{A\sp\ar 0} as the union of an increasing sequence of 

$A\in\mu\invss44\image\spp\rbb R^+$ such that 
$\Nu\aR 1\sn\circ\sp y$ is bounded on every $A\,$, 

say \math{\Nu\aR 1\sn\circ\sp y\sp\image\ssn A\inc
 [\KP{1.1} 0\,,\spp\smb M\KPt9] } with \mathss30{\smb M\in\rbb R^+}, \,and 

it further suffices to show that for every such $A$ with 
$0 < \mu\fvalue\ssn A$ we have

$
\int_{\,A\KPt8}(\ssp\Nu\aR 1\sn\circ\sp y\fvalue\eta\ssp)
  \RHB{.2}{\KP1^{p\sast}}\ssn\rmdss01\mu\,(\sp\eta\sp)
\le\|\KPt8\smb U\,\|\RHB{.2}{\KP1^{p\sast}} \sp$.

To proceed indirectly, supposing that 

$\|\KPt8\smb U\,\|\RHB{.2}{\KP1^{p\sast}}
<\int_{\,A\KPt8}(\ssp\Nu\aR 1\sn\circ\sp y\fvalue\eta\ssp)
  \RHB{.2}{\KP1^{p\sast}}\ssn\rmdss01\mu\,(\sp\eta\sp)$ holds, we let 

$\eps=
 \frac 14\KP1
(\ssp\mu\fvalue\ssn A\ssp)
\RHB{.2}{\,^{\mminus p^{\sp-1}}}
\big(\sp
\int_{\,A\KPt8}(\ssp\Nu\aR 1\sn\circ\sp y\fvalue\eta\ssp)
  \RHB{.2}{\KP1^{p\sast}}\ssn\rmdss01\mu\,(\sp\eta\sp)
)
  \RHB{.2}{\KP1^{p\sast\ssp^{-1}}}
\snn - \|\KPt8\smb U\,\|\ssp\sbig)0 \KP1 $. \vskip1mm

Since \math{\Nu\aR 1\sn\circ\sp y\KP1|\KPt9 A } is positive \mathss37{\mu
}--\,measurable with \mathss35{
\sup\KPt8(\ssp\Nu\aR 1\sn\circ\sp y\ssp\image\ssn A\ssp) < \plusinfty}, \,we 
can find a finite partion \math{\scrmt A\sp\ar 0\inc\dom\mu} of \math{A} such 
that \math{ |\KP{1.2}\Nu\aR 1\sn\circ\sp y\fvalue\eta - 
                     \Nu\aR 1\sn\circ\sp y\fvalue\eta\ar 1\,| < \eps } holds 
for all \mathss36{\eta\ssp,\sp\eta\ar 1\in A\ar 1\in\scrmt A\sp\ar 0 }. Taking \math{
S=\Nu\invss44\image\snn\{\ssp 1\ssp\} } and \vskip.4mm\centerline{$
 P = 
A\times S\capss31\{\,(\ssp\eta\ssp,\spp\xi\ssp):
0 \le y\fvalue\eta\fvalss01\xi \ssp \text{ and }\ssp
\Nu\aR 1\sn\circ\sp y\fvalue\eta 
< y\fvalue\eta\fvalss01\xi + \eps \KPt9\} \KP1 $,} \inskipline{.4}0

we first see that \math{A\inc\dom P} holds. In the reflexive case letting \math{
S\ar 0} be the closed linear span in \math{\vPi\dlbetss01} of \math{\rng y} we 
take \math{\scrmt T\aR 1=\taurd(\sp\vPi\dlbetss01\sp)\leiss42(\ssp
 \Nu\aR 1\sn\inve\ssp\image\sp[\KPp1.1 0\,,\spp\smb M\KPt9]\capss42 
S\ar 0\spp) } whereas in the separable case we put \mathss38{\scrmt T\aR 1 = 
\taurd(\sp\vPi\dlsigss00\spp)
\leiss42(\ssp\Nu\aR 1\sn\inve\ssp\image\sp[
\KP{1.1} 0\,,\spp\smb M\KPt9]\sp
\sbig)0 }. Noting that in both cases now 
\math{\scrmt T\aR 1} is a separable and metrizable and hence 
second countable topology, we find some 
$\bmii8 U\in\sp^{\sbbNo}\,\scrmt T\aR 1$ with 
$y\sp\image\ssn A\inc\bigcup\,\rng\ssp\bmii8 U$ and 
such that for every 
$U\in\rng\bmii8 U$ there are $\xi\,,\sp\eta$ with 
$(\ssp\eta\ssp,\spp\xi\ssp)\in P$ and 
\mathss38{U\inc\Cal L\,(\sp\vPi\sp,\spp\bosy K\ssp)
\capss31\{\,\zeta:
|\KP{1.1}(\ssp\zeta - y\fvalue\eta\ssp)\fvalss01\xi\KP{1.1}| 
< \eps\KPt8\} }.

We next fix some bijection $
\bmii8 A\sp\ar 0:k\to\scrmt A\sp\ar 0$ with $k\in\bbN\,$ and construct the 
countable finite or infinite sequence $\bmii8A$ as follows. Indeed, we first
 let $\bmii8 A\ar 1$ be the infinite 
sequence of possibly empty finite sequences obtained as follows. For 
every fixed $i\in\bbNo$ with 
$B=y\invss46[\KP{1.2}
\bmii8 U\fvalss51 i\sp\setminus\bigcup\KP1(\ssp\bmii8 U\KPt8\image\spp i\ssp)
\KP{1.1}]$ let 
$\bmii8 A\ar 1\sn\fvalue\sp i$ with 
$l\in\bbNo$ be the unique bijection 
$l\to\scrmt A\sp\ar 0\leiss02 B\sp\setminus 1\sp\adot$ ordered by 
$\bmii8 A\ar 0 \,$. Then let 
$\bmii8 A$ be the infinite concatenation of 
$\bmii8 A\ar 1 \ssp $. Now 
$\bmii8 A$ is injective with $\rng\bmii8 A\inc
\dom\mu\setminus 1\sp\adot$ and such that 
$\rng\bmii8 A$ is a partition of $A$ refining $\scrmt A\sp\ar 0\,$, i.e.\ for 
every $i\in\dom\bmii8 A$ there is $A\ar 1\in\scrmt A\sp\ar 0$ with 
$\bmii8 A\fvalss51 i\inc A\ar 1 \ssp$. Possibly by 
{\sl countable choice\sp} we 
take any 
$\bosy\eta\in\prodc\bmii8 A$ and any $\bosy\xi\in\sp^{\dom\sn\bmii6 A}\,S$ 
such that 
$(\ssp\bosy\eta\fvalss01 i\ssp,\spp\bosy\xi\fvalss01 i\ssp)
\in P$ holds for all $i\in\dom\bmii8 A \,$. Now by construction 

$|\KP{1.1}(\ssp y\fvalue\eta - y\fvalue\eta\ar 1)\fvalss01\xi\KP{1.1}| 
< \eps \,$ and $\,
\Nu\aR 1\sn\circ\sp y\fvalue\eta\ar 1 
< y\fvalue\eta\ar 1\sn\fvalue\sp\xi + \eps 
\,$ 
and $0\le y\fvalue\eta\ar 1\sn\fvalue\sp\xi\,$ hold 

whenever we have
$(\ssp i\ssp,\spp A\ar 1)\in\bmii8 A$ and $\eta\in A\ar 1$ and 
$(\ssp i\ssp,\spp\eta\ar 1)\in\bosy\eta$ and 
$(\ssp i\ssp,\spp\xi\ssp)\in\bosy\xi \,$.

With \math{N\aar 0=\dom\bmii8 A} we next compute \vskip1mm

$
\int_{\,A\KPt8}(\ssp\Nu\aR 1\sn\circ\sp y\fvalue\eta\ssp)
  \RHB{.2}{\KP1^{p\sast}}\ssn\rmdss01\mu\,(\sp\eta\sp)
=\sum_{\KPt8 i\ssp\in\ssp\aars N_0}\spp
\int_{\,\bmii6 A\snn\ffvalue i\KPt8}(\ssp\Nu\aR 1\sn\circ\sp y\fvalue\eta\ssp)
  \RHB{.2}{\KP1^{p\sast}}\ssn\rmdss01\mu\,(\sp\eta\sp) $ \inskipline{.7}{27}

${}\le
\sum_{\KPt8 i\ssp\in\ssp\aars N_0}\ssp
(\ssp\Nu\aR 1\sn\circ\sp y\circss00\bosy\eta\fvalss01 i + \eps\ssp)
  \RHB{.2}{\KP1^{p\sast}}\ssp
(\ssp\mu\circss10\bmii8 A\fvalss01 i\ssp)$ \inskipline{.7}{20}

${}\le
\sum_{\KPt8 i\ssp\in\ssp\aars N_0}\ssp
(\ssp y\circss00\bosy\eta\fvalss01 i
\fvalss20(\ssp\bosy\xi\fvalss11 i\ssp) + 2\KPt8\eps\ssp)
  \RHB{.2}{\KP1^{p\sast}}\ssp
(\ssp\mu\circss10\bmii8 A\fvalss01 i\ssp)$ \vskip.7mm\centerline{$
{}=\lim\sbi{\sp\ssmb N\ssp\to\ssp\infty\ssp}
\sum_{\KPt8 i\ssp\in\ssp\ssmb N}\,
(\ssp y\circss00\bosy\eta\fvalss01 i
\fvalss20(\ssp\bosy\xi\fvalss11 i\ssp) + 2\KPt8\eps\ssp)
  \RHB{.2}{\KP1^{p\sast}}\ssp
(\ssp\mu\circss10\bmii8 A\fvalss01 i\ssp) \KP1 $,} \inskipline10

where the last limit expression is valid and needed only in the case where 
\math{\bmii8 A} is infinite. According to whether 
\math{\bmii8 A} is finite or infinite, with 
\math{\smb N=N\aar 0} or 
for arbitrarily fixed \math{\smb N\in\bbN} considering 
\math{u\in\sp^{\ssmb N}\KPt8\lbb R_+} given by \vskip.5mm\centerline{$
u=\seqss33{
(\ssp y\circss00\bosy\eta\fvalss01 i
\fvalss20(\ssp\bosy\xi\fvalss11 i\ssp) + 2\KPt8\eps\ssp)\KP1
(\ssp\mu\circss10\bmii8 A\fvalss01 i\ssp)
  \RHB{.2}{\KP1^{p\sast\ssp^{-1}}}\ssn:i\in\smb N} \KP1 $,} \inskipline{.5}0

we know that for some \math{v\in\sp^{\ssmb N}\KPt8\lbb R_+} with \math{ 
\|\,v\,\|\lllnor_p=1} we have \mathss30{ \|\,u\,\|\lllnor_{p\sast} =      \label{express p-norm}
 \sum\KP1(\ssp u\cdot v\ssp) } where \math{u\cdot v} is the pointwise product \mathss39{
\smb N\owns i\mapsto u\fvalss01 i\cdot(\ssp v\fvalss01 i\ssp) }. Using this, 
we get \vskip1mm

$
\sum_{\KPt8 i\ssp\in\ssp\ssmb N}\,
(\ssp y\circss00\bosy\eta\fvalss01 i
\fvalss20(\ssp\bosy\xi\fvalss11 i\ssp) + 2\KPt8\eps\ssp)
  \RHB{.2}{\KP1^{p\sast}}\ssp
(\ssp\mu\circss10\bmii8 A\fvalss01 i\ssp) 
= \big(\sp\sum\KP1
(\ssp u\cdot v\ssp))\,^{\,p\sast}$ \vskip.7mm

${}=\big(\sp\sum_{\KPt8 i\ssp\in\ssp\ssmb N}\,(
(\ssp y\circss00\bosy\eta\fvalss01 i
\fvalss20(\ssp\bosy\xi\fvalss11 i\ssp) + 2\KPt8\eps\ssp)\KP1
(\ssp\mu\circss10\bmii8 A\fvalss01 i\ssp)
  \RHB{.2}{\KP1^{p\sast\ssp^{-1}}}
(\ssp v\fvalss01 i\ssp)))\RHB{.2}{\KP1^{p\sast}}$ \vskip.7mm

${}=\big(\sp\sum_{\KPt8 i\ssp\in\ssp\ssmb N}\spp
\int_{\,\bmii6 A\snn\ffvalue i\KPt8}(
(\ssp y\circss00\bosy\eta\fvalss01 i
\fvalss20(\ssp\bosy\xi\fvalss11 i\ssp) + 2\KPt8\eps\ssp)\KP1
(\ssp\mu\circss10\bmii8 A\fvalss01 i\ssp)
  \RHB{.2}{\KP1^{p\sast\ssp^{-1} - \ssp 1\,}}
(\ssp v\fvalss01 i\ssp))
\rmdss11\mu\,(\sp\eta\sp))\RHB{.2}{\KP1^{p\sast}}$ \vskip.7mm

${}=(\ssp\smb I\aR 1\snn + 2\KP1\eps\KP1\smb I\aR 2\spp
)\RHB{.2}{\KP1^{p\sast}}$ where we have \vskip.7mm

$\smb I\aR 1 = \sum_{\KPt8 i\ssp\in\ssp\ssmb N}\spp
\int_{\,\bmii6 A\snn\ffvalue i\KPt8}y\circss00\bosy\eta\fvalss01 i
\fvalss20(\ssp\bosy\xi\fvalss11 i\ssp)\KP1
(\ssp\mu\circss10\bmii8 A\fvalss01 i\ssp)
  \RHB{.2}{\KP1^{p\sast\ssp^{-1} - \ssp 1\,}}(\ssp v\fvalss01 i\ssp)
\rmdss11\mu\,(\sp\eta\sp) \,$  and

$\smb I\aR 2 = \sum_{\KPt8 i\ssp\in\ssp\ssmb N}\spp
\int_{\,\bmii6 A\snn\ffvalue i\KPt8}(
\ssp\mu\circss10\bmii8 A\fvalss01 i\ssp)
  \RHB{.2}{\KP1^{p\sast\ssp^{-1} - \ssp 1\,}}
(\ssp v\fvalss01 i\ssp)\rmdss11\mu\,(\sp\eta\sp) \KP1 $. \inskipline10

Now with \math{A\sp\ar 1=\bigcup\KP1(\sp\bmii8 A\spp\image\snn\smb N\ssp) } 
a direct computation using H\"older's inequality 

gives 
\math{\smb I\aR 2\le
(\ssp\mu\fvalue\ssn A\sp\ar 1)
\RHB{.2}{\KP1^{p^{\sp-1}}}\le
(\ssp\mu\fvalue\ssn A\ssp)
\RHB{.2}{\KP1^{p^{\sp-1}}} }, \,and to estimate \math{\smb I\aR 1}, \,taking 

$\bosy\xi\ar 1=
\seqss33{((\ssp\mu\circss10\bmii8 A\fvalss01 i\ssp)
  \RHB{.2}{\KP1^{p\sast\ssp^{-1} - \ssp 1\,}}(\ssp v\fvalss01 i\ssp)\KP1
(\ssp\bosy\xi\fvalss11 i\ssp))\svs\vPi\sn
:i\in\smb N} \,$ and 

$\smb X
=\sigrd F\text{\KPt8-}\sum_{\KPt8 i\ssp\in\ssp\ssmb N}\,
\lfloor\,^{p\sp,\ssp\mu\sp,\ssp\vPi}\ssp
(\ssp\bosy\xi\fvalss11 i\ssp)\sbi{\,\bmii6 A\snn\ffvalue i} \KP1 $, we get \vskip1mm

$\smb I\aR 1 = \sum_{\KPt8 i\ssp\in\ssp\ssmb N}\spp
\int_{\,\bmii6 A\snn\ffvalue i\KPt8}y\circss00\bosy\eta\fvalss01 i
\fvalss20(\ssp\bosy\xi\ar 1\sn\fvalue\sp i\ssp)
\rmdss11\mu\,(\sp\eta\sp) 
=\smb U\sp\fvalue\snn\smb X + \smb I\aR 3 $ where \vskip.7mm

$\smb I\aR 3 = 
\sum_{\KPt8 i\ssp\in\ssp\ssmb N}\spp
\int_{\,\bmii6 A\snn\ffvalue i\KPt8}(\ssp y\circss00\bosy\eta\fvalss01 i 
- y\fvalue\eta\ssp)\fvalue(\ssp\bosy\xi\ar 1\sn\fvalue\sp i\ssp)
\rmdss11\mu\,(\sp\eta\sp) \KP1 $. \inskipline10

A direct computation gives \math{\|\,\smb X\ssp\|\sNorF \le 1} whence we get \mathss38{
|\KP1\smb U\sp\fvalue\snn\smb X\,|\le\|\KPt8\smb U\,\|}, \,and further \vskip1mm

$|\KP1\smb I\aR 3\,| \le \eps\,
\sum_{\KPt8 i\ssp\in\ssp\ssmb N}\spp
\int_{\,\bmii6 A\snn\ffvalue i\KPt8}
(\ssp\mu\circss10\bmii8 A\fvalss01 i\ssp)
  \RHB{.2}{\KP1^{p\sast\ssp^{-1} - \ssp 1\,}}(\ssp v\fvalss01 i\ssp)
\rmdss11\mu\,(\sp\eta\sp)=
\eps\KP1\smb I\aR 2
\le\eps\KP1(\ssp\mu\fvalue\ssn A\ssp)
\RHB{.2}{\KP1^{p^{\sp-1}}} $. \inskipline10

Putting these results together, and letting 
\math{\smb N\to\infty} or taking \math{\smb N=\dom\bmii8 A} if 
\math{\bmii8 A} is finite, \,we finally obtain \inskipline1{11}

$
\int_{\,A\KPt8}(\ssp\Nu\aR 1\sn\circ\sp y\fvalue\eta\ssp)
  \RHB{.2}{\KP1^{p\sast}}\ssn\rmdss01\mu\,(\sp\eta\sp)
\le
\big(\ssp \|\KPt8\smb U\,\| + 3\KP1
\eps\KP1(\ssp\mu\fvalue\ssn A\ssp)
\RHB{.2}{\KP1^{p^{\sp-1}}}
\sp\big)\RHB{.2}{\KP1^{p\sast}} $ \inskipline{.7}{47}

${}<
\big(\ssp \|\KPt8\smb U\,\| + 4\KP1
\eps\KP1(\ssp\mu\fvalue\ssn A\ssp)
\RHB{.2}{\KP1^{p^{\sp-1}}}
\sp\big)\RHB{.2}{\KP1^{p\sast}} $ \inskipline{.7}{47}

${}=
\int_{\,A\KPt8}(\ssp\Nu\aR 1\sn\circ\sp y\fvalue\eta\ssp)
  \RHB{.2}{\KP1^{p\sast}}\ssn\rmdss01\mu\,(\sp\eta\sp)
\KP1 $, \,a {\sl contradiction\sp}. \vskip1mm

\Step 4.0 Now having obtained \math{
\|\KP1\Nu\aR 1\sn\circ\sp y\KP1\|\Lnorss40^{p\sast}_\mu\le\|\KPt8\smb U\,\| } 
we know that \math{y\in\bigcup\ssp\vecs F\aar 1} holds, and hence there is 
some \math{\smb Y} with \math{y\in\smb Y\in\vecs F\aar 1}. Then we get \math{ 
\smb U \in \rng\Iota} from Lemma \ref{Le-first} similarly as in the proof of 
Lemma \nfss A\,\ref{LeA(1)} on page \pageref{Le A2 final ded} above.      \label{endmpf}
  \end{proof}

We have now established Theorem \nfss A\,\ref{main Th} since as noted at the 
beginning of this section on page \pageref{page surj} above, only the 
surjectivity \mathss34{\Cal L\,(\sp F\spp,\spp\bosy K\ssp)\inc\rng\Iota } 
remained to be verified, and this is done in the various cases in 
Lemmas \nfss A\,\ref{LeA(1)}$\,,\ldots\KPt8$\nfss A\,\ref{final lemma} above. 
Note also that as opposed to the treatments in \cite{Phil} and \cite{Edw}\,, 
we succeeded to handle the cases (5) and (6) simultaneously. In \cite{Phil} 
only the case (5) is considered, and the text also contains some quite obscure 
passages. In \cite{Edw} the case (6) is treated under the additional 
assumption that \math{\mu} be at least positive \erm Radonian.

% ----------------------------------------------------------------------------

\Ssubhead E               Examples and open problems                      \label{Sec E}

Below, we have collected some examples in order to make more concrete some 
points of the abstract theory given above. We also point out some related open 
problems. In the first example we demonstrate that in Theorem \nfss A\,\ref{main Th} 
the case (3) {\sl does not\sp} cover (1) and (2) even when \math{\mu} is a 
  probability measure.

\begin{example}\label{Exa big compact}                               \renewcommand\sNorF{\sNor{\fivemath F}}\renewcommand\erm[1]{\hbox{\font\Å=cmr8\Å#1}}

For \mathss37{\Omega=\sp^\bbI\ssbb70 I}, \,we construct a probability measure \math{
\mu} on \math{\Omega} such that for the space \math{ F = 
 \mLrs42^1(\ssp\mu\ssp) } the topology \math{\taurd F} is not separable. 
Indeed, for details referring to \cite[199\,--\,203]{Du} let \math{ \mu = 
 \otimes_{\fiveroman{mea}\,}(\ssbb60 I\times\snn\{\,\LeBmef^{}\,|\KP1
 \Pows\bbI\KP1\}\sp\sbig)0 } be the uncountable product measure of the 
Borel\,--\,Lebesgue measure on the closed unit interval.

Now with 
$\roman A\,s=
\Omega\capss41\{\,\eta:\frac 12 \le \eta\fvalue s \le 1\KPt9\} \, $ let

$\roman x\,s=
(\ssp\Omega\sp\setminus\roman A\,s\ssp)\times\snn\{\ssp 0\ssp\}
\cupss22(\sp\roman A\,s\times\snn\{\ssp 2\ssp\}\sp\sbig)0 \, $ and 
$\, \erm X\,s=\uniqset\smb X:\roman x\,s\in\smb X\in\vecs F \sp $. 

For $\smb X\in\vecs F$ letting 
$\|\,\smb X\ssp\|\sNorF = \uniqset s:\aall{x\in\smb X}\, s = 
 \int_{\KP{1.1}\Omega\,}|\KP1 x\fvalue\eta\KP1|\suba
  \rmdss01\mu\,(\spp\eta\spp)$ 

and 
$\|\KPt9\smb X - \smb Y\KP{1.1}\|\sNorF=
\|\KP1(\sp\smb X - \smb Y\,)\svs F\,\|\sNorF$ , 

\noin
then $\{\KPt8\erm X\,s:s\in\bbI\KP1\}$ is uncountable, and 
for $s\ssp,\sp t\in\bbI$ with $s\not=t$ by a simple computation we get 
$\|\KP1\erm X\,s - \erm X\,t\KP1\|\sNorF=1 \,$, \,giving the 
assertion on nonseparability.
  \end{example}

Equally well in Example \ref{Exa big compact} above we could have taken the 
uncountable \q{coin tossing} measure \math{ \mu = 
 \otimea3(\ssp I\times\snn\{\ssp\pi\ssp\}\sp\sbig)0 } for any uncountable set \math{
I} when \vskip.3mm\centerline{$
\pi = \{\,1\spp\adot\ssp,\sp\{\,1\spp\adot\spp\}\sp\}\timesn\big\{\ssp
 \frac 12\ssp\big\}\cupss22\{\,(\ssp\emptyset\,,\spp 0\ssp)\,,\sp
 (\ssp 2\sp\adot\ssp,\spp 1\ssp)\,\} \KP1 $.}

\begin{problem}

{\it Does {\,\rm(4)} hold\,} in Theorem \nfss A\,\ref{main Th} when \math{\mu} 
is the probability measure constructed in Example \ref{Exa big compact} above? 
Observe that \cite[Lemma 8.17.1\,(\sp b\spp)\,, p.\ 580]{Edw} would give a 
positive answer only if \math{
(\ssp\nsTbb_R\leiss02\ssbb05 I)\expnota^\ssbb44 I]_{ti} } were a metrizable 
topology, thus requiring the set \math{\mathbb I=[\KPp1.1 0\,,\spp 1\KPt9] } 
  to be countable.
  \end{problem}

\begin{example}\label{Exa not trul deco}

For \math{\Omega=\bbR\times\bbR} we construct a {\sl decomposable\sp} positive 
measure 
\math{\mu} on \math{\Omega} that is {\sl not truly decomposable\sp}. We also 
get a 
function \math{u:\Omega\to\{\KPt8 0\,,\spp 1\KPt5\} } with \math{
\upint u\rmdss11\mu=
\plusinfty} 
but \math{
\int_{\,A}\ssp u\rmdss11\mu=
0} for all \mathss34{A\in\mu\invss44\image\spp\lbb R_+}. Indeed, 

let $\mu$ be the set of all pairs 
$(\sp A\,,\spp s\ssp)$ 
with $A\inc\Omega$ and such that there are 
$B\in\{\,A\,,\spp\Omega\sp\setminus A\KPt8\}$ and a 
countable $C\inc\bbR$ such that 
$B\,\image\snn\{\ssp t\ssp\}\in\dom\Lebmef^{}$ holds 
for all $t\in C$, and that 
$B\,\image\snn\{\ssp t\ssp\}=\emptyset$ for 
$t\in\ssbb02 R\setminus C$, and that $s=\sum\KP1
\seqss33{\Lebmef^{}\ssn\fvalue(\sp 
A\ssp\image\snn\{\ssp t\ssp\}\ssp\big):t\in\bbR} \KP1 $. For $N=
\bbR\times\snn\{\ssp 0\ssp\}$ then $N\not\in\dom\mu$ but 
$A \capss31 N \in  \mu\invss44\image\snn\{\ssp 0\ssp\} $ 
for all $A\in\mu\invss44\image\sp\lbb R_+$. 
It follows that $\mu$ cannot be truly decomposable. To see that $\mu$ is 
decomposable, just take \mathss39{\scrmt A=
\big\{\ssp\{\ssp t\ssp\}\snn\times
{]}\KP{1.2} n\ssp,\spp n + 1\KP1]:
t\in\bbR\ssp$ and $\ssp n\in\mathbb Z\KP1 \} }. One also 
observes that for \math{u\ar 0=N\timesn\{\ssp 1\ssp\} } and \math{u
=(\ssp\Omega\sp\setminus N\ssp)\times\snn\{\ssp 0\ssp\}\cupss22 u\ar 0 } 
we have $\upint u\rmdss11\mu=\upint u\ar 0\rmdss01\mu=
\plusinfty$ 
but $\int_{\,A}\ssp u\rmdss11\mu=\int_{\,A}\ssp u\ar 0\rmdss01\mu=
0$ for all $A\in\mu\invss44\image\sp\lbb R_+$.
  \end{example}

Decomposable but not \rsigma6finite positive measures are given in the next

\begin{example}\label{Exa Haar}

Let \math{g\in\sp^{S\ssp\times\ssp S}\,S } be a group operation with \math{S} 
uncountable. Then with \mathss03{\Omega=S\snn\times\bbR } and \math{ \scrmt T 
 = \Pows S\sp\ttimes\nsTbb_R } and \inskipline{.2}{4.4}

$a=\{\,(\ssp s\ar 1\sp,\spp t\ar 1\KPt2;\spp s\ar 2\ssp,\spp t\ar 2\,;\spp 
             s\ar 3\ssp,\spp t\ar 3\spp) : 
 (\ssp s\ar 1\sp,\spp s\ar 2\ssp,\spp s\ar 3\spp) \in g\ssp\text{ and }\ssp 
 (\ssp t\ar 1\sp,\spp t\ar 2\ssp,\spp t\ar 3\spp) \in\sigrd\snn\fbbR\KP1\} \hfill $ 
putting \inskipline{.5}{4}

$\mu \sp = \ssp \big\langle\ssp\sum_{\KPt8\emath s\ssp\in\ssp S\KPt8}
 (\ssp\Lebmef^{}\ssn\fvalue(\sp A\sp\image\snn\{\ssp s\ssp\}\sp\sbig)0\big) : 
 A\inc\Omega\ssp\text{ and }\ssp\aall{s\in S}\,A\sp\image\snn\{\ssp s\ssp\} 
 \in \dom\Lebmef^{} \KPt8 \rangle $ \inskipline{.7}0

and \mathss38{\mu\ar 1=\mu\KP1|\KP1\{\,A:\dom A\ssp\text{ is countable or 
 \math{\dom(\ssp\Omega\spp\setminus A\ssp) } is countable }\} }, \hfill we 
have \linebreak \mathss03{
     (\ssp a\,,\spp\scrmt T\,) } a locally compact Hausdorff topological group 
with \math{\mu} a modified {\sl Haar me- asure\sp} for it and \mathss38{
\mu\ar 1 = \mu\KP1|\KP1\sigmAlg3\{\,K:K\ssp\text{ is \mathss37{\scrmt T
 }--\,compact }\} }. With \vskip.2mm\centerline{$
\scrmt A = \{\sp\{\ssp s\ssp\}\snn\times[\KP1 n\ssp,\spp n + 1 \KP1 {[\sp} : 
           s\in S\ssp\text{ and }\ssp n\in\ssbb04 Z\,\} $} \inskipline{.2}0

one checks that \math{\mu} is {\sl truly decomposable\sp} and that \math{
\mu\ar 1} is {\sl decomposable\sp}. Note that for \math{ g = \sigrd\snn\fbbR = 
 \ssigrd\tfbbR } we have \math{\mu\ar 1} precisely the \math{\mu} given in 
Example \ref{Exa not trul deco} above.
  \end{example}

\begin{problem}\label{Prblm z-z mea}

{\it Is $\,\mu$ almost decomposable\ssp} in the following situation\sp? Let \math{
\Omega=\sp^{2.}\ssbb60 R } and with \math{\roman S\,a\,b=\{\,a + t\KP1(\ssp 
 b - a\ssp):0\le t\le 1\KPt8\} } let \mathss30{\mu\ar 0=\{\,(\KPt5
 \roman S\,a\,b\,,\spp\|\KPt8 a - b\KP1\|\lllnor_2\sp):a\ssp,\sp b\in\Omega
 \KP1\} } and \ $\mu\sp = \sp \upCth\ssp \seqss43{ \inf \sp \big \{\ssp \sum\KP1
 (\ssp\mu\ar 0\circ\bmii8 A\ssp) : \mu\ar 0\snn : \ebit A \in \sp 
 ^\sbbNo\,\dom\mu\ar 0\ssp\text{ and }$ \inskipline0{39}

 $ A\inc\bigcup\ssp\rng\ebit A\,\} :
   A\inc\Omega } \KP1|\KP1\sigmAlg3\dom\mu\ar 0 \KPt8 $. \inskipline{0.2}0

Since from \cite[Proposition 3.2.4\ssp, p.\ 72]{Du} we know that \math{\mu} is 
a positive measure, the problem is whether there exist \math{\scrmt A\,,\sp 
 N\sprim1 } as required in Definitions \ref{df decomp}\,(2) on page \pageref{decos A} 
above. An appeal to intuition suggests that \math{\mu} is {\sl not\sp} almost 
decomposable, but a possible proof does not seem to be simple.

Note that if we above instead had written \vskip.25mm\centerline{$
\mu=\uniqset m:m\ssp$ is a positive measure and 
\math{\dom m=\sigmAlg3\dom\mu\ar 0 } and \mathss36{\mu\ar 0\inc m },} \inskipline{.25}0

then it might have happened that \math{\mu=\Univ} holds, and hence the answer 
to the above question would trivially have been \q{no}, noting that by the 
lacking \rsigma5finiteness the uniqueness in \cite[Theorem 3.1.10\ssp, p.\ 68]{Du} 
is not applicable in this situation.
  \end{problem}

{                                                                      \newcommand\elvrB{\lower.2mm\hbox{\font\Å=cmr11\ÅB\kern.7mm}} %
Similarly as in Problem \ref{Prblm z-z mea} above we might ask whether with \math{
k\ssp,\smb N\in\bbN } and \mathss30{ \Omega = \sn} \mathss30{\sp
 ^{k\ssp+\KPt4\ssmb N}\ssbb80 R } and suitably fixed \math{\lambda\in\rbb R^+ } 
and \vskip.25mm\centerline{$
\elvrB r=\{\KPt8\Omega\capss31\{\,\eta:\|\KP1\eta - \eta\ar 0\,\|\lllnor_2 < 
 \smb R\,\}:\eta\ar 0\in\Omega\ssp\text{ and }\ssp 
 0 < \smb R \le r \, \} $} \inskipline{.25}0

and \math{ \alpha = \seqss30{ t\KPt8^{\ssmb N\,(\sp k \ssp + \ssp\ssmb N\ssp)
 \yydot\sp^{-1}}\KN{.8}:t\in\lbb R_ +\ssn} } the \mathss36{\smb N
}--\,dimensional {\sl Hausdorff measure\sp} \inskipline{.3}{24}

$\upCth\ssp\seqss43{\lim_{\KP1\emath r\ssp\to\,0^+}\sp\inf\sp\big\{\,\lambda\sp
 \sum\,(\ssp\alpha\circss00\Lebmef^{\ssp k\ssp+\ssp\ssmb N}\sn\circ\ebit B\ssp
 ) : {} $ \inskipline0{38.3}

$ \ebit B\in\sp^\sbbNo\,\elvrB r\ssp\text{ and }\ssp A \inc \bigcup \sp \rng 
 \ebit B\,\} : A\inc\Omega}$ \inskipline{.3}0

is almost decomposable. } % END of \elvrB

In search for an example of a positive measure that would not be almost 
decomposable we noticed the positive measure \math{\mu} in the following

\begin{example}\label{Exa non-Rad?}

Let \math{\mu\ar 1=\scrmt N\times\snn\{\ssp 0\ssp\}\cupss22\{\,(\ssp
 \varOmega\aar 1\ssn\setminus N\sp,\spp 1\ssp):N\in\scrmt N\KP1\} } where \math{
\varOmega\aar 1} is an uncountable set and \math{\scrmt N} is a \rsigma7ideal 
in \math{\varOmega\aar 1} with \math{\varOmega\aar 1\not\in\scrmt N } and \mathss34{
\{\sp\{\ssp\eta\ssp\}:\eta\in\varOmega\aar 1\ssp \}\inc\scrmt N }. For 
example, we might have \math{\varOmega\aar 1=\bbI } and \mathss36{ \scrmt N = 
 \Lebmef^{}\KN1\inve\ssp\image\snn\{\ssp 0\ssp\}\leiss22\bbI }, \,or \vskip.2mm\centerline{$
\scrmt N=\sp\big\{\ssp\bigcup\,\scrmt A:\scrmt A\ssp\text{ is countable and }\ssp
 \scrmt A\inc\Pows\varOmega\aar 1\cap\sp\{\,A:\roman{Int\,}\sbi{\scrm7 T\KP1}
 \roman{Cl\,}\sbi{\scrm7 T\KPt8}A = \emptyset \KPt9 \} \sp \} $} \inskipline{.2}0

where \math{\scrmt T} is a regular locally compact or completely metrizable 
topology for \mathss34{\varOmega\aar 1}. 

With a fixed \math{s\ar 0\in\varOmega\aar 1 } for \math{ \eta\ar 0 = 
 (\ssp s\ar 0\KPt5,\spp s\ar 0) } and for \math{ \Omega = 
 \varOmega\aar 1\sn\times\varOmega\aar 1 } we then construct a positive 
measure \math{\mu} on \math{\Omega} such that \math{\mu} is {\sl decomposable\sp} 
but not \rsigma5finite and such that \math{ \bigcap\,\scrmt A\sp\ar 0 = 
 \{\,\eta\ar 0\spp\}\not\in\dom\mu } and \math{ \bigcap\,\scrmt A \in
 \mu\invss44\image\snn\{\ssp 1\ssp\} } hold for \vskip.2mm\centerline{$
\scrmt A\sp\ar 0 = \mu\invss44\image\snn\{\ssp 1\ssp\}\capss22\{\, A : 
                   \eta\ar 0 \in A \KP1 \} $} \inskipline{.2}0

and for all nonempty countable \mathss36{\scrmt A\inc\scrmt A\sp\ar 0 }.

Indeed, with \math{ \roman P\ssp\ebit A = \bigcup\KPt8\{\sp\{\ssp s\ssp\}
 \timesn A\spp\ar 1\ssn:(\ssp s\ssp,\spp A\spp\ar 1)\in\ebit A\,\} } and \math{
\scrmt P} the set of all countable functions \math{ \ebit A \inc 
 \varOmega\aar 1\sn\times\dom\mu\ar 1 } with \math{[\KPp1.4 s\ar 0 \in 
 \dom\ebit A\ssp\text{ and }\ssp\ebit A\fvalue s\ar 0\not\in\scrmt N\impss33 
 s\ar 0\in\ebit A\fvalue s\ar 0 \KP1 ] } we let \math{ \mu = 
 \{\KPt8(\sp A\,,\spp t\ssp):\eexi{\ebit A\in\scrmt P}\,[\KPp1.4 A = 
 \roman P\ssp\ebit A\ssp\text{ and }\ssp t = 
 \sum\KP1(((\ssp\mu\ar 1\sn\circ\spp\ebit A\ssp
 )\invss24\image\snn\{\ssp 1\ssp\}\sp\sbig)0\times\snn
                    \{\ssp 1\ssp\}\sp\sbig)0\KPp1.4\big] } \inskipline0{71.4}

 or $\ssp[\KPp1.4 A = \Omega\spp\setminus\roman P\ssp\ebit A\ssp\text{ and }\ssp 
 t = \plusinfty\KPp1.4]\KP1 % " sbig } " !!!
                       \lower.1mm\hbox{\font\Å=cmsy11\Å\char'147}
                                        \KP1 $. \inskipline{.25}0

Note that \math{\{\KPt9\roman P\ssp\ebit A:\ebit A\in\scrmt P\KP1\} } is a 
\rsigma3ring and hence that \math{\dom\mu} is a \rsigma3algebra.

One sees that for \math{0 < p\le\plusinfty } and for \math{\vPi\in\LCSps0(K) } 
with \math{\bosy K\in\setRC } the spaces \math{
\mvLrs03^p(\ssp\mu\,,\spp\vPi\ssp) } and \math{
\lll^p\spp(\sp\varOmega\aar 1\sp,\spp\vPi\ssp) } are linearly homeomorphic 
under \math{\smb X\mapsto y } when \mathss03{ y \in \sp
 ^{\aars\varOmega_1}\ssp\vecs\vPi } is such that for \math{x\in\smb X} and \math{
u\in\Cal L\,(\sp\vPi\sp,\spp\bosy K\ssp) } and for finite \math{ \ebit A \in 
 \scrmt P } with \mathss03{ \rng\ebit A \inc 
 \mu\ar 1\ssn\inve\ssp\image\snn\{\ssp 1\ssp\} } we have \mathss38{
  \int_{\KPp1.1\roman P\KPt2\bmii6 A}\sp u\circss00 x\rmdss11\mu 
= \sum\KP1(\ssp u\circss00 y\KP1|\KP1\dom\ebit A\ssp) }.
  \end{example}

From Example \ref{Exa non-Rad?} above we arrive at the following

\begin{problem}

{\it Is $\,\mu$ positive \eit Radonian\ssp} when \math{\mu = 
 \Lebmef^{}\,|\KP1(\ssp\Pows\bbI\capss42(\ssp
 \Lebmef^{}\KN1\inve\ssp\image\snn\{\KPt8 0\,,\spp 1\,\}\sp\sbig)0\sbig)0 } 
holds\ssp? Note that if there is \math{\scrmt T} that 
positively \erm Radonizes \math{\mu} above, then necessarily \mathss30{
\mu\fvalue\snn K\in\sn} \mathss03{
                   \{\KPt8 0\,,\spp 1\,\} } holds when \math{K} is \mathss37{
\scrmt T}--\,compact. Furthermore, there is some \mathss30{ N \in 
 \mu\invss44\image\snn\{\ssp 0\ssp\} } such that \math{
\scrmt T\leiss32(\ssbb62 I\setminus N\ssp) } is a compact topology. Then we 
get \math{\mu\sp\image(\ssp\scrmt T\leiss32(\ssbb62 I\setminus N\ssp)) \inc \sn} \mathss03{
 \{\KPt8 0\,,\spp 1\,\} } and hence \mathss38{ \mu\sp\image\ssp\scrmt T \inc 
 \{\KPt8 0\,,\spp 1\,\} }.
  \end{problem}

Observe that if we take the trivially positive \erm Radonian \mathss38{
\mu\ar 0 = \{\KPt8(\ssp\emptyset\,,\spp 0\ssp)\,,\sp(\ssp 1\spp\adot\ssp,\spp 
 1\ssp)\KPt8 \} }, then for \math{ q = \bbI\times 1\spp\adot } and \math{ 
\mu\ar 2 = \{\,(\,q\invss44\image\ssn A\,,\spp t\ssp):(\sp A\,,\spp t\ssp)\in
 \mu\ar 0\,\} } and for \math{0 \le p \le \plusinfty } and \mathss03{ E = 
 \mLrs03^p(\ssp\mu\sp) } and \math{F = \mLrs03^p(\ssp\mu\ar 2\spp) } and \math{
\Iota = \vecs E\times\vecs F\capss21\{\,(\ssp\smb\Phii\spp,\spp\smb\Psii\sp) : 
 \smb\Phii\capss12\smb\Psii\not=\emptyset \KPt9 \} } we \linebreak
                                                        have \math{
\Iota:E\to F} a linear homeomorphism. This leads us to the following

\begin{definitions}\label{Df Leb equ}

(1) \ $q\meastss33\mu = \{\KPt7(\,q\invss44\image\ssn A\,,\spp t\ssp) : 
    (\sp A\,,\spp t\ssp)\in\mu\KPt8\} \KP1 $, \inskipline{.5}2

(2) \ Say that \math{N\sprim1} is {\it finitely \mathss37{\mu}--\,negligible\ssp} 
    if{}f \math{N\sprim1} is \mathss37{\mu}--\,negligible \inskipline0{37.5}

 and \math{\mu\invss44\image\snn\{\sp\plusinfty\,\}\capss13\Pows N\sprim1 = 
  \emptyset} holds, \inskipline{.5}2

(3) \ Say that \math{\mu\ar 1\sp,\sp\mu\ar 2} are {\it Lebesgue equal\,} if{}f \math{
    \mu\sbi{\sixmath\nuu} } for \math{\sbi{ \sixmath\nuu \sixroman{\ssp =\ssp
 1\spp,\ssp 2}}} is a positive measure and there are \math{N\aar 1\sp,\sp 
 N\aar 2\ssp,\sp q\ar 1\sp,\sp q\ar 2\ssp,\sp Q } with  \math{
N\sn\sbi{\sixmath\nuu} } finitely \mathss37{\mu\sbi{\sixmath\nuu} 
}--\,negligible and \math{q\sbi{\sixmath\nuu} } is a surjection \mathss03{Q\to 
 \bigcup\ssp\dom\mu\sbi{\sixmath\nuu}\sn\setminus N\sn\sbi{\sixmath\nuu} } for \math{\sbi{
\sixmath\nuu \sixroman{\ssp =\ssp 1\spp,\ssp 2}} } and \math{ \Iota \in 
 \Lis(\sp E\ar 1\sp,\spp E\ar 2\spp) } holds for \mathss30{E\sbi{\sixmath\nuu}
 = \mLrs42^1(\,q\sbi{\sixmath\nuu}\ssn\meastss03\mu\sbi{\sixmath\nuu}) } and \mathss38{
\Iota = \vecs E\ar 1\sn\times\vecs E\ar 2\capss01\{\,(\ssp\smb\Phii\spp , \spp
 \smb\Psii\sp) : \smb\Phii\capss12\smb\Psii\not=\emptyset \KPt9 \} }, \inskipline{.5}2

(4) \ Say that \math{\mu} is {\it essentially positive \eit Radonian\ssp} 
    if{}f \inskipline0{9}

 there is a positive \erm Radonian \math{\mu\ar 0} such that \math{
 \mu\,,\sp\mu\ar 0} are Lebesgue equal.
  \end{definitions}

Note that if \math{\mu} is a positive measure and \math{q} is a small function 
such that for the set \math{N\sprim1=\bigcup\ssp\dom\mu\setminus\rng q } it 
holds that \math{N\sprim1} is \mathss37{\mu}--\,negligible, then \math{
q\meastss33\mu } {\sl need not\sp} be a positive measure since it may even 
fail to be a function. For example, taking \mathss03{q=\emptyset} and \math{
\mu = \{\KPt8(\ssp\emptyset\,,\spp 0\ssp)\,,\sp(\ssp 1\spp\adot\ssp ,
 \plusinfty\ssp)\KPt8\} } we get \mathss38{ q\meastss33\mu = 
 1\spp\adot\timesn\{\KPt7 0\,,\plusinfty\,\} }. However, if we know that \math{
q\meastss33\mu } is a function, then it is also a positive measure as one 
quickly verifies. A sufficient condition to guarantee that \math{
q\meastss33\mu } be a function is that \mathss30{N\sprim1} be finitely \mathss37{
\mu}--\,negligible as we have required in Definitions \ref{Df Leb equ}\,(3) 
  above.

Note also that our Definition \ref{Df Leb equ}\,(3) is not entirely 
satisfactory since for example taking \math{ \mu\ar 0 = 
 \{\,(\ssp\emptyset\,,\spp 0\ssp)\,\} } we have both \math{
\mLrs42^1(\ssp\mu\sp) } and \math{\mLrs42^1(\ssp\mu\ar 0\spp) } linearly 
homeomorphic to the trivial space \math{
(\ssp 1\spp\adot\snn\times 1\spp\adot\snn\times 1\spp\adot\ssp , \spp
 \bbR\times 1\spp\adot\snn\times 1\spp\adot\ssp , \spp\Pows 1\spp\adot\spp) } 
but \math{\mu\,,\sp\mu\ar 0} are not Lebesgue equal by the above. We further 
remark that the relation of being Lebesgue equal is not an equivalence since 
for example taking the positive \erm Radonian \math{ \mu\ar 1 = 
 2\sp\adot\timesn\{\ssp 0\ssp\} } that is positively \erm Radonized by \math{
\Pows 1\spp\adot} we have both \math{\mu\,,\sp\mu\ar 1} Lebesgue equal and \mathss30{
\mu\ar 0\,,\sp\mu\ar 1} Lebesgue equal. Now we can pose the following

\begin{problem}\label{Prblm all Rado?}

{\sl Is every positive measure essentially positive \eit Radonian\ssp}?
  \end{problem}

If the answer to the question in Problem \ref{Prblm all Rado?} is positive, 
then one might be able to remove from Theorem \nfss A\,\ref{main Th} the 
assumption on \math{\mu} being almost decomposable in the case where \math{p=1} 
holds. It seems that possibly by using {\sl Kakutani's theorem\sp} 
\cite[4.23.2\ssp, p.\ 287]{Edw} one might be able to prove that this indeed is 
the case. However, we leave these matters open here.

For example the {\sl Wiener\sp} probability {\sl measure\sp} in \cite{Pi-Po} 
on a non\ssp-\ssp locally compact separably metrizable topological space 
{\sl is essentially positive \esl Radonian\sp} directly by its construction 
since it is obtained by restricting a \erm Radonian probability measure on a 
compact topological space to a subset of measure unity. More specifically, one 
first constructs a probability measure \math{\pi} that is positively 
\erm Radonized by a compact topology \mathss30{\scrmt T}. Then for a certain 
separably metrizable topological space \math{(\ssp\Omega\,,\spp\scrmt U\ssp) } 
one shows that \math{\sigmAlg3\scrmt U=\dom\pi\leiss23\Omega } and \math{
\Omega\in\pi\invss44\image\snn\{\ssp 1\ssp\} } hold, and finally one defines \linebreak
$\id\sp\Omega\meastss33\pi\ssp$ to be the Wiener measure.

We remark that there is some confusion in \cite[pp.\ 12\,--\,25]{Pi-Po} and 
that the above is not a review but rather an interpretation of how it could 
  have been done.

\begin{example}\label{Exa pos Rad C(I)}

It holds that \math{\Pows\Omega} {\sl positively almost \esl Radonizes\sp} \math{
\pi} in the following situation. Let \math{(\ssp\Omega\ssp,\spp\scrmt T\,) } 
be a separably metrizable and not locally compact topological space with \math{
\Omega} uncountable, and let \math{D} countable and \mathss37{\scrmt T
}--\,dense with \math{\bosy a\in\sp^D\KPt8\rbb R^+ } and \mathss04{
\sum\,\bosy a=1}. Then let \mathss38{\pi = \sp \big\langle\ssp\sum\KP1(\ssp
 \bosy a\KPt9|\KPt8 A\ssp) : A \in \sigmAlg3\scrmt T\KPp1.2\rangle }. For 
example, we might have \mathss03{\scrmt T=\taurd C\,(\ssbb55 I) } and \math{D} 
the set of all polynomial functions with rational coefficients, or the set of 
all piecewise affine functions with rational \q{break} points.
  \end{example}

\begin{example}\label{Exa Sum Lebm^N}

For \math{\Omega = \sp^\sbbNo\,\bbI } we obtain a \hfill {\sl decomposable\sp} 
non\ssp-\sp\rsigma5finite positive mea- \linebreak
                                        sure \math{\mu} on \math{\Omega} by 
taking \math{\mu=\sum\KP1\seqss30{\roman m\KPt8\alpha:\alpha\inc\bbNo} } where 
with \math{m\ar 1 = \Lebmef^{}\,|\KP1\Pows\bbI } and \mathss03{\delta\ar 0 = 
 \dom m\ar 1\snn\times\{\KPt7 0\,,\spp 1\KPt6\}\capss22\{\,
 (\sp A\,,\spp t\ssp) : t = 1 \equivss33 0 \in A \KP1 \} } we have \vskip.3mm\centerline{$
\roman m\KPt8\alpha=
\otimea0((\ssp\bbNo\sn\setminus\alpha\ssp)\times\snn\{\KPt8 m\ar 1\}\cupss22
(\ssp\alpha\times\snn\{\,\delta\ar 0\spp\}\sp\sbig)0\sbig)0 \KP1 $.} \inskipline{.3}0

Indeed, with \math{\roman A\KPt8\alpha = \Omega\capss41\{\,\eta : 
 \eta\invss44\image\snn\{\ssp 0\ssp\}=\alpha\KPt8\} } taking \math{ \scrmt A = 
 \{\,\roman A\KPt8\alpha:\alpha\inc\bbNo\,\} } we have \mathss03{\scrmt A} 
uncountable and disjoint with \math{\Omega=\bigcup\,\scrmt A} and \mathss38{
\scrmt A\inc\mu\invss44\image\snn\{\ssp 1\ssp\} }. Furthermore, for \mathss03{
\alpha\KPt5,\sp\kappa\in\Pows\bbNo} we have \math{
\roman m\KPt8\alpha\fvalue\sn\roman A\KPt8\alpha = 1 } and \mathss36{
\alpha\not=\kappa\impss33\roman m\KPt8\kappa\fvalue\sn\roman A\KPt8\alpha=0}. 
To see these, by straightforward inspection one first verifies the last 
assertion which then directly implies that \math{ \scrmt A \inc 
 \mu\invss44\image\snn\{\ssp 1\ssp\} } holds. Consequently \math{\mu} is not 
  \rsigma5finite.

To show that \math{\mu} is decomposable, let \math{A\in
 \mu\invss44\image\sp\rbb R^+} and \math{N\sprim1\inc\Omega} with \mathss30{
\scrmt A\leiss42 N\sprim1 \inc \sn} \mathss08{ 
 \bigcup\KPt8\{\KPt8\Pows N \sn : N\in\mu\invss44\image\snn\{\ssp 0\ssp\}\sp\} 
}. Then taking \mathss30{ \varLambda\sp\ar 0 = \Pows\bbNo\capss01\{\,\alpha : 
 A\capss32\roman A\KPt8\alpha\in\mu\invss44\image\sp\rbb R^+\sp\big\} } we 
have \math{\varLambda\sp\ar 0} countable, and also putting \math{ N\aar 1 = 
 A\spp\setminus\bigcup\KPt8\{\,\roman A\KPt8\alpha:\alpha\in\varLambda\sp\ar 0
 \,\} } we get \mathss30{N\aar 1 \in \sn } \mathss08{
 \mu\invss44\image\snn\{\ssp 0\ssp\} }. Indeed, trivially \math{ N\aar 1 \in
 \dom\mu } holds, and if we have \mathss38{ N\aar 1 \not \in 
 \mu\invss44\image\snn\{\ssp 0\ssp\} }, \,then \math{0 < 
 \mu\fvalue\snn N\aar 1 = \sum\KP1\seqss30{
 \roman m\KPt8\alpha\fvalue\snn N\aar 1\sn : \alpha\in\Pows\bbNo} = 
 \sum\KP1\seqss30{\roman m\KPt8\alpha\fvalue\snn N\aar 1\sn : \alpha \in 
 \Pows\bbNo\sn\setminus\varLambda\sp\ar 0} } when- ce there is some \math{
\alpha\in\bbNo\sn\setminus\varLambda\sp\ar 0 } with \math{ 0 < 
 \roman m\KPt8\alpha\fvalue\snn N\aar 1 \le \roman m\KPt8\alpha\fvalue\ssn A } 
and hence \mathss36{\alpha\in\varLambda\sp\ar 0}, \,a {\sl contradiction\sp}. 
Now by {\sl countable choice\sp} there is some countable \mathss30{ \scrmt N
 \inc\mu\invss44\image\snn\{\ssp 0\ssp\} } with \mathss34{
\bigcup\KPt8\{\,N\sprim1\cap\sp\roman A\KPt8\alpha : \alpha \in 
 \varLambda\sp\ar 0\,\}\inc\bigcup\,\scrmt N }. Then taking \math{ N = 
 \bigcup\,\scrmt N\cupss42 N\aar 1} we finally get \mathss08{
A\capss31 N\sprim1\inc N\in\mu\invss44\image\snn\{\ssp 0\ssp\} }.
  \end{example}

Observe that if in Example \ref{Exa Sum Lebm^N} in place of \math{\bbNo} we 
take any uncountable set, then we obtain a trivial measure \math{\mu} in the 
sense that \math{\rng\mu=\{\KPt7 0\,,\plusinfty\,\} } holds.

\begin{problem}\label{Prblm Sum Lebm^N}

{\it Is $\,\mu$ positive \eit Radonian\ssp} in Example \ref{Exa Sum Lebm^N} 
above\sp? Note that at least we cannot take \math{ \scrmt T = 
 (\ssp\nsTbb_R\lei\ssbb25 I)\expnota^\sp\sbbNo\sp]_{ti} } in order to 
positively \erm Radonize \math{\mu} since \math{\Omega} is \mathss37{\scrmt T
}--\,compact with \mathss35{\mu\fvalss01\Omega=\plusinfty}.
  \end{problem}

\begin{example}\label{Exa sign mea}

We say that \math{\mu} is a {\it signed measure\ssp} if{}f \math{ \mu \in \sp
 ^{\dom\sn\mu}\,\ovbbR} with \math{\dom\mu} being a \rsigma3algebra and \math{
\mu\fvalue\sn\bigcup\,\scrmt A = \sum\KP1(\ssp\mu\KP1|\KP1\scrmt A\ssp) } for 
all countable disjoint \mathss35{\scrmt A\inc\dom\mu}. Then \linebreak 
                                                            we cannot have \math{
\{\sp\minusinfty\,,\plusinfty\KPt8\}\inc\rng\mu } since otherwise we could 
find \math{A\,,\sp B} with \mathss03{
\{\KPt8(\sp A\,,\minusinfty\ssp)\,,\sp(\sp B\ssp,\plusinfty\ssp)\KPt8\} \inc 
 \mu} and \math{A\capss32 B=\emptyset} whence we would get \vskip.3mm\centerline{$
\Univ = \sum\KP1(\,\mu\KP1|\KPt8\{\,A\,,\sp B\KPt8\}\sp\sbig)0
      = \mu\fvalue(\sp A\cupss32 B\ssp) \in \ovbbR \KP1 $,} \inskipline{.3}0

a contradiction following from our sum conventions in \cite{Hif}\,. Now the 
positive measures are precisely the signed measures \math{\mu} with \mathss38{
\rng\mu\inc[\KPp1.1 0\,,\plusinfty\KPt9] }, \,and a signed measure \math{\mu} 
we say to be {\it positively signed\,} if{}f \math{\minusinfty\not\in\rng\mu} 
holds. Similarly the con- dition \math{\plusinfty\not\in\rng\mu} defines being 
{\it negatively signed\,}. Real measures are now those that are both 
positively and negatively signed\sp.

We next construct the topologized conoid \math{ \tcbbR_+ = 
 (\ssp a\ssp,\sp c\,,\spp\scrmt T\,) } so that \math{m} is positively signed 
if{}f \math{m} is countably \mathss37{\tcbbR_+}--\,additive and such that \math{
\dom m} is a \rsigma3algebra. Indeed, taking \math{R=\lbb R_+} and \math{ S = 
 {\ssp]}\,\minusinfty\ssp,\plusinfty\KP1] } let \math{ \scrmt T = 
 \barscTbb_R\leiss02 S} and \inskipline{.3}{10.3}

$a = \ssigrd\tfbbR\cupss31\{\,(\ssp s\ssp,\spp t\ssp,\plusinfty\ssp) : 
     \plusinfty\in\{\,s\ssp,\sp t\,\}\inc S\KP1\} \KP{26.5} $ and \inskipline{.3}{10.7}

$c = \tsigrd\tfbbR\KP1|\KP1(\sp R\times\Univ\ssp)\cupss21\{\,
     (\ssp 0\,,\plusinfty\ssp,\spp 0\ssp)\,\}\cupss22\{\,
     (\ssp s\ssp,\plusinfty\ssp,\plusinfty\ssp):s\in\rbb R^+\sp\big\} \KP1 $. \inskipline{.3}0

Making the obvious modifications we similarly get the topologized conoid \math{
\tcbbR_-} so that \math{m} is negatively signed if{}f \math{m} is countably \mathss37{
\tcbbR_-}--\,additive and such that \mathss30{\dom m} is a \rsigma3algebra. 
Likewise, we can construct the topologized conoid \math{\tcovbbRplus} 
characterizing the positive measures.
  \end{example}

\begin{problem}\label{Prblm Io sur}

In Theorem \nfss A\,\ref{main Th} taking for example \math{ \mu = 
 \LeBmef^{}\,|\KP1\Pows\mathbb I } and \vskip.5mm\centerline{$
\vPi \in \{\KP1\co(\ssbb44 I) \, , \sp \ell\KPt8^1\sp(\ssbb44 I) \, , \sp
          \LLrs23^\plusinftyy(\ssbb44 I) \KPt9 \} \KP1 $,} \inskipline{.3}0

hence \math{\vPi} being nonreflexive with \math{\taurd\vPi} a nonseparable 
topology, for \math{p=1} we see \linebreak
                                that (3) holds, and then with \math{ F\aar 1 =
 \mvsLrs23^{p\sast}\ssn(\ssp\mu\,,\spp\vPi\dlsigss00\spp) } we obtain that \math{
\Iota} is a linear ho- \linebreak
                       meomorphism \mathss30{F\aar 1\to F\dlbetss10}, \,and 
hence in particular that \math{\Cal L\,(\sp F\spp,\sn\tfbbR\ssp)\inc\rng\Iota} 
holds. However, if we instead take \math{1 < p < \plusinfty} for example with \mathss35{
p=2}, \,then we {\sl do not\sp} know whether \math{
\Cal L\,(\sp F\spp,\sn\tfbbR\ssp)\inc\rng\Iota } holds. So under these 
circumstances we may ask\sp: {\it Is there\ssp} \math{\smb U} such that \math{
\smb U\in\Cal L\,(\sp F\spp,\sn\tfbbR\ssp)\setminus\rng\Iota} holds\,{\bf?}

We remark that by suitably adapting the proof of Corollary \ref{Coro q-var} 
above it seems \linebreak
               to be possible to deduce existence of some \math{y} and a 
countable disjoint \mathss30{\scrmt A\inc\mu\invss44\image\spp\rbb R^+ \ssn} 
with \mathss30{\bigcup\,\scrmt A = \mathbb I } and such that \math{
(\KPt5 y\,;\spp\mu\,,\spp\vPi\dlsigss00\spp) } is scalarly measurable and such 
that we have \math{\smb U\sp\fvalue\snn\smb X=
 \int_{\,A}\ssp y\,.\KPt8 x\rmdss11\mu } for all \math{x\ssp,\sp\smb X} with \math{
x\in\smb X\in\vecs F} and \math{\rng x} finite and such that for some \math{
A\in\scrmt A} we have \mathss36{
x\invss46[\KPp1.1\Univ\sp\setminus\{\,\Bnull_\vPi\}\KP1]\inc A }. However, we 
do not know whether \math{
\|\KP1\Nu\aR 1\sn\circ\sp y\KP1\|\Lnorss40^{p\sast}_\mu < \plusinfty } holds. 
If we could get \math{y} with these properties together with \mathss36{
\|\KP1\Nu\aR 1\sn\circ\sp y\KP1\|\Lnorss40^{p\sast}_\mu < \plusinfty }, \,then 
we would also get \mathss34{\smb U\in\rng\Iota}.
  \end{problem}

% ----------------------------------------------------------------------------

\end{document}